\documentclass[a4paper,12pt]{article}
\usepackage{latexsym,amsmath,amssymb,amsfonts,amsthm,bbm,color,xcolor,breakcites,dsfont}
\usepackage[round]{natbib}
\usepackage[normalem]{ulem}
\usepackage{algorithm, algpseudocode}
\usepackage{dsfont,url}
\usepackage{enumerate}
\RequirePackage[colorlinks,citecolor={blue!90!black},urlcolor={blue!70!black},linkcolor={red!60!black},breaklinks,hypertexnames=false]{hyperref}
\usepackage{graphicx}
\usepackage{graphics}
\usepackage{psfrag}
\usepackage{caption, subcaption}
\usepackage{comment}

\newtheorem{theorem}{Theorem}
\newtheorem{lemma}[theorem]{Lemma}
\newtheorem{proposition}[theorem]{Proposition}
\newtheorem{corollary}[theorem]{Corollary}
\newtheorem*{remark*}{Remark}

\DeclareMathOperator*{\argmax}{argmax}

\newenvironment{definition}[1][Definition:]{\begin{trivlist}
\item[\hskip \labelsep {\bfseries #1}]}{\end{trivlist}}

\newenvironment{remark}[1][Remark:]{\begin{trivlist}
\item[\hskip \labelsep {\bfseries #1}]}{\end{trivlist}}

\title{High-dimensional nonparametric density estimation via symmetry and shape constraints}
\author{Min Xu\footnote{Email address: mx76@stat.rutgers.edu} \ and Richard J. Samworth\footnote{Email address: r.samworth@statslab.cam.ac.uk} \\ Rutgers University and \\
Statistical Laboratory, University of Cambridge}
\date{\today}

\begin{document}

\maketitle

\begin{abstract}
  We tackle the problem of high-dimensional nonparametric density estimation by taking the class of log-concave densities on $\mathbb{R}^p$ and incorporating within it symmetry assumptions, which facilitate scalable estimation algorithms and can mitigate the curse of dimensionality.  Our main symmetry assumption is that the super-level sets of the density are $K$-homothetic (i.e.~scalar multiples of a convex body $K \subseteq \mathbb{R}^p$).  
  When $K$ is known, we prove that the $K$-homothetic log-concave maximum likelihood estimator based on $n$ independent observations from such a density has a worst-case risk bound with respect to, e.g., squared Hellinger loss, of $O(n^{-4/5})$, independent of $p$.  Moreover, we show that the estimator is adaptive in the sense that if the data generating density admits a special form, then a nearly parametric rate may be attained.  We also provide worst-case and adaptive risk bounds in cases where $K$ is only known up to a positive definite transformation, and where it is completely unknown and must be estimated nonparametrically.  Our estimation algorithms are fast even when $n$ and $p$ are on the order of hundreds of thousands, and we illustrate the strong finite-sample performance of our methods on simulated data.
\end{abstract}

\section{Introduction}

Density estimation emerged as one of the fundamental challenges in Statistics very soon after its inception as a field.  Up until halfway through the last century, approaches based on parametric (often Gaussian) assumptions or histograms/contingency tables were dominant \citep{Fisher1922,Fisher1925}.  However, the restrictions of these techniques has led, since the 1950s, to an enormous research effort devoted to exploring nonparametric methods, primarily based on smoothness assumptions, but also on shape constraints.  These include kernel density estimation \citep{rosenblatt1956remarks,WandJones1995}, wavelets \citep{donoho1996density} and other orthogonal series methods, splines \citep{gu1993smoothing}, as well as techniques based on monotonicity \citep{grenander1956theory}, log-concavity \citep{CSS2010} and others.  Although highly successful for low-dimensional data, these approaches all encounter two serious difficulties in moderate- or high-dimensional regimes: first, theoretical performance is limited by minimax lower bounds that characterise the `curse of dimensionality' \citep[e.g.][]{ibragimov1983estimation}; and second, computational issues may become a bottleneck, often exacerbated by the need to choose (multiple) smoothing parameters.  

In parallel to these developments, modern technology now allows the routine collection of extremely high-dimensional data sets, leading to a great demand for reliable and scalable density estimation algorithms.  To emphasise the challenge here, let $\mathcal{F}_p$ denote the class of upper semi-continuous, log-concave densities on $\mathbb{R}^p$, and let $X_1,\ldots,X_n$ be independent and identically distributed random vectors with density $f_0 \in \mathcal{F}_p$.  \citet{kim2016global} proved that for each $p \in \mathbb{N}$, there exists $c_p > 0$ such that\footnote{In fact, just prior to completion of this work, \citet{dagan2019logconcave} showed that $c_p$ may be chosen independent of $p$.} 
\[
\inf_{\tilde{f}_n} \sup_{f_0 \in \mathcal{F}_p} \mathbb{E}d_{\mathrm{H}}^2(\tilde{f}_n,f_0) \geq \left\{ \begin{array}{ll} c_1 n^{-4/5} & \mbox{if $p=1$} \\
c_p n^{-2/(p+1)} & \mbox{if $p \geq 2$}, \end{array} \right.
\]
where $d_{\mathrm{H}}^2(f,g) := \int_{\mathbb{R}^p} (f^{1/2}-g^{1/2})^2$ denotes the squared Hellinger distance between densities $f$ and $g$ on $\mathbb{R}^p$, and where the infimum is taken over all estimators $\tilde{f}_n$ of $f_0$ based on $X_1,\ldots,X_n$.  This suggests that very large sample sizes would be required for an adequate approximation to the true density, even for $p=5$.  In view of these fundamental theoretical limitations, it is natural to consider imposing additional structure on the problem, while simultaneously seeking to retain the desirable flexibility of the nonparametric paradigm.

In this paper, we propose a new method for high-dimensional, nonparametric density estimation by incorporating symmetry constraints into the shape-constrained class.  We demonstrate that this approach facilitates efficient algorithms that in some cases can even evade the curse of dimensionality in terms of its rate of convergence.  The particular type of symmetry constraint that we consider is what we call \emph{homotheticity}, where the super-level sets of the density are scalar multiples of each other.  Thus, any elliptically symmetric density, for instance, is homothetic, but the class is of course much broader than this.  We combine homotheticity with the shape constraint of log-concavity, which in particular ensures that the super-level sets are convex, compact sets, to yield a flexible yet practical class that facilitates nonparametric density estimation even in moderate or high dimensional problems.

To introduce our contributions, let $\mathcal{K}$ denote the set of compact, convex sets $K \subseteq \mathbb{R}^p$ containing $0$ as an interior point, and let $\Phi$ denote the set of upper semi-continuous, decreasing, concave functions $\phi:[0,\infty) \rightarrow [-\infty,\infty)$. In Section~\ref{Sec:Basic}, we show that we can write any homothetic log-concave density $f$ as $f(\cdot) = e^{ \phi(\| \cdot - \mu\|_K)}$ for some \emph{super-level set} $K \in \mathcal{K}$, \emph{centering vector} $\mu \in \mathbb{R}^p$ and \emph{generator} $\phi \in \Phi$, and where $\| \cdot \|_K$ denotes the Minkowski functional with respect to $K$, whose definition we recall in~\eqref{Eq:Minkowski} below. We thus write $\mathcal{F}^{\mathcal{K}}_p$ for the class of homothetic log-concave densities and, for fixed $K$ and $\mu$, write $\mathcal{F}^{K, \mu}_p$ for the subclass of $\mathcal{F}^{\mathcal{K}}_p$ consisting of homothetic log-concave densities with super-level set $K$ and centering vector $\mu$.  Writing $\mathcal{P}_p$ for the class of probability distributions $P$ on $\mathbb{R}^p$ with finite mean $\mu \in \mathbb{R}^p$ and $P(\{\mu \}) < 1$, we further prove that for fixed $K$ and $\mu$, there exists a well-defined homothetic, log-concave projection $\psi^* \equiv \psi_{K,\mu}^*:\mathcal{P}_p \rightarrow \mathcal{F}_p^{K,\mu}$.  Thus, if $X_1,\ldots,X_n \stackrel{\mathrm{iid}}{\sim} P \in \mathcal{P}_p$, with empirical distribution $\mathbb{P}_n$, then a natural estimator of $\psi^*(P)$ is given by $\hat{f}_n := \psi^*(\mathbb{P}_n)$.  In particular, if $P$ has a density $f_0 \in \mathcal{F}_p^{K,\mu}$, then $\psi^*(P) = f_0$, and the first main aim of our theoretical contribution is to study the performance of $\hat{f}_n$ as an estimator of $f_0$.  On the other hand, if $K$ and $\mu$ are unknown, then we investigate a computationally-efficient plug-in approach where we first construct estimators $\hat{K}$ of $K$ and $\hat{\mu}$ of $\mu$, and then (with a slight abuse of notation) compute $\hat{f}_n := \psi_{\hat{K}, \hat{\mu}}^*(\mathbb{P}_n)$.  

To this end, we focus in Section~\ref{Sec:Kknown} on the case where $K$ and $\mu$ are assumed to be known and where an attractive feature of our estimator is that, even in high-dimensional problems, it does not require the choice of any tuning parameters.  Our results on the theoretical performance of $\hat{f}_n$ are presented in terms of the divergence measure
\[
d_X^2(\hat{f}_n,f_0) := \frac{1}{n}\sum_{i=1}^n \log \frac{\hat{f}_n(X_i)}{f_0(X_i)}.
\]
We show in Proposition~\ref{Prop:ProjectionBasicProperties}(iv) that $d_X^2(\hat{f}_n,f_0) \geq \int_{\mathbb{R}^p} \hat{f}_n \log (\hat{f}_n/f_0) =: d_{\mathrm{KL}}^2(\hat{f}_n,f_0)$, so that our upper bounds on $\mathbb{E} d_X^2(\hat{f}_n,f_0)$ immediately yield the same upper bounds on the expected Kullback--Leibler divergence (as well as the risks in the squared total variation and squared Hellinger distances, for instance).  One of our main results in this section (Theorem~\ref{Thm:KKnownWorstCase}) is that, if $X_1,\ldots,X_n \stackrel{\mathrm{iid}}\sim f_0 \in \mathcal{F}_p^{K,\mu}$, then there exists a universal constant $C > 0$ such that
\[
  \mathbb{E}d_X^2(\hat{f}_n,f_0) \leq Cn^{-4/5}.
\]
Thus, there is no dependence on $p$ (or $K$) in this worst case risk bound, which is verified by our empirical studies in Section~\ref{Sec:Simulations}.  We also elucidate the adaptation behaviour of $\hat{f}_n$.  More precisely, for $k \in \mathbb{N}$, we let $\Phi^{(k)}$ denote the set of $\phi \in \Phi$ that are piecewise linear on $\mathrm{dom}(\phi) := \{r:\phi(r) > -\infty\}$, with at most $k$ linear pieces.  In Theorem~\ref{Thm:KknownAdaptation}, we show that if $f_0 \in \mathcal{F}_p^{K,\mu}$ is of the form $f_0(\cdot) = e^{\phi_0(\|\cdot-\mu\|_K)}$ for some $\phi_0 \in \Phi^{(k)}$, then 
\begin{equation}
  \label{Eq:AdaptKknown}
  \mathbb{E}d_X^2(\hat{f}_n,f_0) \leq C\frac{k}{n}\log^{5/4}\Bigl(\frac{en}{k}\Bigr)
\end{equation}
for a universal constant $C > 0$.  This result reveals that $\hat{f}_n$ adapts to densities $f_0$ whose corresponding $\phi_0$ is $k$-affine with $k$ not too large; in particular, an almost parametric rate can be attained for small $k$.  In fact, we prove a stronger, oracle inequality, version of~\eqref{Eq:AdaptKknown}, where $f_0 \in \mathcal{F}_p^{K,\mu}$ need only be close to a density of the form $e^{\phi(\|\cdot - \mu\|_K)}$ for some $\phi \in \Phi^{(k)}$, and the bound incurs an additional approximation error term.

In Section~\ref{Sec:KEstimated}, we consider the case where the super-level set $K$ and the centering vector $\mu$ are unknown.  We first obtain a general purpose bound on the squared Hellinger risk of $\hat{f}_n$ for arbitrary estimators $\hat{K}$ and $\hat{\mu}$ in terms of deviations between $\hat{K}$ and $K$ and $\hat{\mu}$ and $\mu$.  As an initial application, we take the semiparametric setting and suppose that $K = \Sigma_0^{1/2}K_0$ for some known balanced $K_0 \in \mathcal{K}$ and some unknown positive definite matrix $\Sigma_0 \in \mathbb{R}^{p \times p}$; one important example of this setting is \emph{elliptical symmetry}, where $K_0$ is the unit Euclidean ball. Then we estimate $K$ by $\hat{K} = \hat{\Sigma}^{1/2}K_0$ where $\hat{\Sigma}$ is the sample covariance matrix and estimate $\mu$ by the sample mean $\hat{\mu}$. This yields a worst-case squared Hellinger risk bound of order $p^{3/2}/n^{1/2}$ up to polylogarithmic factors, and moreover we obtain adaptation rates of order $n^{-4/5} + p^3/n$ and $p^3/n$ in cases where $f_0(\cdot) = e^{\phi_0(\|\cdot-\mu\|_K)}$ with $\phi_0$ smooth and $1$-affine respectively, again up to polylogarithmic factors.  In a second application, we consider the nonparametric setting where $K \in \mathcal{K}$ is arbitrary.  Here, we propose a new algorithm to estimate $K$ as the convex hull of estimates of its boundary at a set $\{\theta_m:m = 1,\ldots,M\}$ of randomly chosen directions, where these boundary estimates are obtained as the average Euclidean norm of observations lying in a cone around $\theta_m$.  The resulting estimator $\hat{f}_n$ is shown to have a worst-case squared Hellinger risk bound of order $n^{-1/(p+1)}$, which improves to $n^{-2/(p+1)}$ in cases where $f_0(\cdot) = e^{\phi_0(\|\cdot - \mu\|_K)}$ with $\phi_0$ smooth, up to polylogarithmic factors.  Importantly, this estimator is computable even in high-dimensional settings because there is no need to enumerate the facets of $\hat{K}$; in order to evaluate $\hat{f}_n(x^*)$ for some $x^* \in \mathbb{R}^p$, we need only compute $\|x^*\|_{\hat{K}}$, which can be achieved by a simple linear programme owing to the representation of $\hat{K}$ as a convex hull of $M$ points.

Section~\ref{Sec:Simulations} provides a simulation study that illustrates our theoretical results and confirms the computational feasibility of our estimators.  Proofs of some of our main results are given in the Appendix; several other proofs, as well as auxiliary results, are given in the online supplement. Results in the online supplement have an `S' prefixing their label numbers. 

Other recent work on estimation over the class $\mathcal{F}_p$ of log-concave densities on $\mathbb{R}^p$ includes \citet{robeva2018maximum}, \citet{carpenter2018near}, \citet{feng2018multivariate} and \citet{dagan2019logconcave}; see \citet{samworth2018recent} for a review.  Multivariate shape-constrained density estimation has also been considered over other classes, including the set of block decreasing densities on $[0,1]^p$ \citep{polonik1995density,polonik1998silhouette,biau2003risk,gao2007entropy}, and the class of $s$-concave densities on $\mathbb{R}^p$ \citep{doss2016global,han2016approximation}.  In the former case, for uniformly bounded densities, \citet{biau2003risk} established a minimax lower bound in total variation distance of order $n^{-1/(p+2)}$, while in the latter case, the main interest has been in the classes with $s < 0$, which contain the class $\mathcal{F}_p$, so the same minimax lower bounds apply as for $\mathcal{F}_p$.  Various other simplifying structures and methods have also been considered for nonparametric high-dimensional density estimation, including kernel approaches for forest density estimation \citep{liu2011forest} and star-shaped density estimation \citep{liebscher2016estimation}, as well as nonparametric maximum likelihood methods for independent component analysis \citep{samworth2012independent}.  Perhaps most closely related to this work is the approach of \citet{bhattacharya2012adaptive}, who consider a maximum likelihood approach (as well as spline approximations) to estimating the generator of an elliptically symmetric distribution with decreasing generator.

\textbf{Notation}: Given $n \in \mathbb{N}$, we write $[n] := \{1,\ldots, n\}$. Given $a, b \in \mathbb{R}$, we write $a \vee b := \max(a,b)$ and $a\wedge b := \min(a,b)$. We also say that $a \lesssim b$ if there exists a universal constant $C > 0$ such that $a \leq C b$ and, given also some quantity $s$, that $a \lesssim_s b$ if there exists some $C_s > 0$, depending only on $s$, such that $a \leq C_s b$. For a given function $f \,:\, A \rightarrow \mathbb{R}$ on some domain $A$, let $\| f \|_\infty := \sup_{x \in A} |f(x)|$; for a Borel measurable function $f \,:\, A \rightarrow \mathbb{R}$, we let $\| f \|_{\textrm{esssup}}$ denote the (Lebesgue) essential supremum. Additionally, if $A = \mathbb{R}$ and $f$ is a density, then we let $\mu_f := \int_{-\infty}^\infty x f(x) \, dx$ and $\sigma_f := \bigl\{ \int_{-\infty}^\infty (x - \mu_f)^2 f(x) \,dx \bigr\}^{1/2}$ be the mean and standard deviation of $f$ respectively. If $A \subseteq \mathbb{R}^p$, we let $\mathcal{B}(A)$ be the set of Borel measurable subsets of $A$. If $A \in \mathcal{B}(\mathbb{R}^p)$, we let $\lambda_p(A)$ denote the Lebesgue measure of $A$. We let $B_p(0,1)$ denote the unit ball in $\mathbb{R}^p$ and write $\kappa_p := \lambda_p(B_p(0,1))$. If $x \in \mathbb{R}^n$ is a vector, then we let $\|x \|$ denote its $\ell_2$ norm. If $B \in \mathbb{R}^{m \times n}$, then we let $\| B \|_{\mathrm{op}} := \sup_{x \in \mathbb{R}^n \,:\, \|x\| = 1} \| Bx \|$ be its operator norm. We let $\mathbb{S}^{p \times p}$ denote the set of positive definite $p \times p$ matrices. For a set $D \subseteq \mathbb{R}^p$, we let $\mathrm{conv}(D)$ and denote its convex hull, and if $D$ is convex, we let $\partial D$ denote its boundary.

\section{Minkowski functionals, homothetic log-concave densities and projections}
\label{Sec:Basic}

\subsection{Minkowski functionals}

In this section, we introduce notation and basic results that will be used throughout the paper.
\begin{definition}
  Let $\mathcal{K} := \{ K \subseteq \mathbb{R}^p \,:\, K \textrm{ convex, compact},\, 0 \in \mathrm{int}(K) \}$. For $K \in \mathcal{K}$, we define the \emph{Minkowski functional} $\| \cdot \|_K \,:\, \mathbb{R}^p \rightarrow [0,\infty)$ as
  \begin{equation}
    \label{Eq:Minkowski}
    \| x \|_K := \inf \{ t \in [0, \infty) \,:\, x \in tK \}.
  \end{equation}
\end{definition}
The proposition below presents some standard properties of the Minkowski functional.  In particular, $\| \cdot \|_K$ is not necessarily a norm but it is subadditive and positively homogeneous.
\begin{proposition}
  \label{Prop:BasicK}
Let $K \in \mathcal{K}$ and $x,y \in \mathbb{R}^p$. Then
\begin{enumerate}[(i)]
\item $\|x \|_K < \infty$,
\item $x \in K$ if and only if $\|x \|_K \leq 1$,
\item $x \in \partial K$ if and only if $\|x \|_K = 1$, where $\partial K := K \setminus \mathrm{int}(K)$.
\item $\| x + y \|_K \leq \| x \|_K + \| y \|_K$ and, if $\alpha \geq 0$, then $\| \alpha x \|_K = \alpha \| x \|_K$. 
\end{enumerate}
\end{proposition}
In fact, if $K = -K$, then $\|\cdot\|_K$ is a norm; as a special case, if $K$ is the closed Euclidean unit ball, then $\|\cdot\|_K$ coincides with the Euclidean norm. Conversely, a norm is also a Minkowski functional: let $\| \cdot \|$ be a norm and define a convex body $K = \{ x \in \mathbb{R}^p \,:\, \| x \| \leq 1 \}$, then we have that $\| x \| = \|x \|_K$ for all $x \in \mathbb{R}^p$.

\subsection{Homothetic log-concave densities}
\label{Sec:Homothetic}

We say that a density $f$ on $\mathbb{R}^p$ is \emph{homothetic} if there exist a decreasing function $r:(0,\|f\|_\infty) \rightarrow [0,\infty)$, a measurable subset $A$ of $\mathbb{R}^p$ with $0 \in \mathrm{int}(A)$ and $\mu \in \mathbb{R}^p$ such that $\{x:f(x) \geq t\} = r(t)A + \mu$ for every $t \in (0,\|f\|_\infty)$.  Note that any such set $A$ has the property that if $0 < t_1 \leq t_2 < \|f\|_\infty$, then $r(t_1)A \supseteq r(t_2)A$.  

In fact, this definition also characterises the level set of $f$ at $\| f \|_{\infty}$ since, for any sequence $t_n \in (0, \|f\|_{\infty})$ such that $t_n \nearrow \|f \|_{\infty}$,
\[
  \{x \,:\, f(x) = \| f \|_\infty\} = \bigcap_{n=1}^\infty \{ x \,:\, f(x) \geq t_n \} = \bigcap_{n=1}^\infty r(t_n)A + \mu.
\]
We have that $\bigcap_{n=1}^\infty r(t_n)A + \mu \supseteq \bigl( \lim_n r(t_n) \bigr) A + \mu$, with equality in certain cases. For example, if $\lim_n r(t_n) = 0$, then equality holds when $A$ is bounded; if $\lim_n r(t_n) > 0$, then equality holds when $A$ is closed, which occurs when, e.g.,~$f$ is upper semi-continuous.   

Recall that $\Phi$ denotes the set of upper semi-continuous, concave, decreasing functions $\phi:[0,\infty) \rightarrow [-\infty, \infty)$. The following proposition characterises densities on $\mathbb{R}^p$ that are simultaneously homothetic and log-concave. 
\begin{proposition}
  \label{Prop:ExistencePhi}
  Let $f$ be an upper semi-continuous density on $\mathbb{R}^p$. Then $f$ is homothetic and log-concave if and only if there exist $K \in \mathcal{K}$, $\mu \in \mathbb{R}^p$ and $\phi \in \Phi$ such that $f(\cdot) = e^{\phi(\|\cdot - \mu \|_K)}$.  If $f$ has an alternative representation as $f(\cdot) = e^{\tilde{\phi}(\|\cdot - \tilde{\mu} \|_{\tilde{K}})}$, where $\tilde{K} \in \mathcal{K}$, $\tilde{\mu} \in \mathbb{R}^p$ and $\tilde{\phi} \in \Phi$, then there exist $\sigma > 0$, $\sigma' > 0$ such that $\tilde{K} = \sigma K + \sigma'(\mu - \tilde{\mu})$ and $\tilde{\phi}(\cdot) = \phi(\sigma\cdot)$; moreover, if $f$ is not the uniform distribution, then $\tilde{\mu} = \mu$.
\end{proposition}

Proposition~\ref{Prop:ExistencePhi} states that any upper semi-continuous, homothetic, and log-concave density may be parametrised by a generator $\phi \in \Phi$, a super-level set $K \in \mathcal{K}$, and a centering vector $\mu \in \mathbb{R}^p$. Moreover, as long as $f$ is not the uniform distribution, the only degree of non-identifiability is that we may scale $K$ and horizontally dilate $\phi$ by the same scalar $\sigma > 0$. This degree of non-identifiability is in fact helpful for density estimation because we need only estimate $K$ \emph{up to a scaling factor} in order to estimate the density $f$.

We let $\mathcal{F}^{\mathcal{K}}_p$ denote the set of all upper semi-continuous, homothetic, log-concave densities on $\mathbb{R}^p$, and for $K \in \mathcal{K}$ and $\mu \in \mathbb{R}^p$, let $\mathcal{F}_p^{K,\mu}$ denote the set of $K$-homothetic, log-concave densities of the form given in Proposition~\ref{Prop:ExistencePhi}.  We also write $\mathcal{F}_p^K := \mathcal{F}_p^{K,0}$.  The following proposition can be regarded as an analogue of a known characterisation of elliptically symmetric densities (where $K$ is taken to be an ellipsoid) to the general homothetic, log-concave case.
\begin{proposition}
  \label{Prop:NormDirectionIndep}
Let $f \in \mathcal{F}_p^{\mathcal{K}}$ be of the form $f(\cdot) = e^{\phi(\|\cdot\|_K)}$, for some $K \in \mathcal{K}$ and $\phi \in \Phi$.  Let $R$ be a random variable taking values in $[0,\infty)$ with density $h$, where $h(r) := p\lambda_p(K)r^{p-1}e^{\phi(r)}$ for $r \in [0,\infty)$, and let $Z$ be a random vector, independent of $R$, uniformly distributed on $K$. Then $\frac{Z}{\| Z\|_K} R$ has density $f$. 
\end{proposition}

\begin{remark}
The random vector $Z/\|Z\|_K$ is supported on the boundary of $K$. When $K$ is the unit Euclidean ball in $\mathbb{R}^p$, we have that $Z/\|Z\|_K$ is uniformly distributed on the surface of the unit Euclidean sphere.  However, when $K$ is an arbitrary convex body, $Z/\| Z \|_K$ is generally not distributed uniformly on the surface $\partial K$. As a simple example in $\mathbb{R}^2$, we may take $K = B_2(0,1) \cap \{ (x_1,x_2) \,:\, |x_1| \leq 2^{-1/2}\}$. The probability that $Z/\|Z\|_K$ lies on the line segment $\{ (2^{-1/2}, x_2) \,:\, x_2 \in [-2^{-1/2}, 2^{1/2}] \}$ is $1/(2+ \pi)$, whereas the length of the line segment divided by the perimeter of $K$ is $1/(2 + 2^{-1/2} \pi)$.
\end{remark}

\subsection{Projections onto the class of homothetic, log-concave densities}
\label{Sec:Projection}

In this section, we fix $K \in \mathcal{K}$ and consider projections onto $\mathcal{F}_p^K$. For $\phi \in \Phi$ and a probability measure $P$ on $\mathbb{R}^p$, we define
\begin{equation}
  \label{Eq:RpProjection}
\mathcal{L}(\phi, P) \equiv \mathcal{L}_K(\phi, P) = \int_{\mathbb{R}^p} \phi(\| x \|_K) \, dP(x) - p\lambda_p(K)\int_0^{\infty} r^{p-1} e^{\phi(r)} \, dr + 1,
\end{equation}
and write $\varphi^*(P) \equiv \varphi_K^*(P) := \argmax_{\phi \in \Phi} \mathcal{L}(\phi,P)$.  Since $\mathcal{L}(\cdot,P)$ is strictly concave, any maximiser of $\mathcal{L}(\cdot,P)$ over $\Phi$ is unique.

If $\mathcal{L}(\phi,P) \in \mathbb{R}$, then 
\[
\frac{\partial}{\partial c} \mathcal{L}(\phi + c, P) = 1 - e^c p\lambda_p(K)\int_0^\infty r^{p-1}e^{\phi(r)} \, dr,
\]
so $\mathcal{L}(\phi + c, P)$ is maximised by choosing $c = -\log \bigl(p\lambda_p(K)\int_0^\infty r^{p-1}e^{\phi(r)} \, dr\bigr)$.  It follows that if $\varphi^*(P)$ exists, and if $\mathcal{L}(\varphi^*(P),P) \in \mathbb{R}$, then we can define the \emph{$K$-homothetic log-concave projection} $f^*(P) \equiv f_K^*(P) \in \mathcal{F}_p^K$ by $f^*(P)(\cdot) := e^{\varphi^*(P)(\|\cdot\|_K)}$. When the centering vector $\mu$ is not the origin, the projection of a probability measure $P$ onto $\mathcal{F}^{K,\mu}_p$ may be reduced to the case where $\mu = 0$ by translating the probability measure $P$ by $-\mu$, projecting the translated distribution onto $\mathcal{F}^{K}_p$, and then translating the resulting log-concave density back by $\mu$. 

By Proposition~\ref{Prop:ExistencePhi}, for any $\alpha > 0$, it holds that $\mathcal{F}_p^{K} = \mathcal{F}_p^{\alpha K}$; we therefore need to check that $f^*_K(P)$ does not depend on the choice of $K$. To see this, fix $\alpha > 0$ and $\phi \in \Phi$, and define $\phi_\alpha \,:\, [0, \infty) \rightarrow [-\infty, \infty)$ by $\phi_{\alpha}(r) := \phi( \alpha r)$ for $r \in [0,\infty)$. Observe then that $\mathcal{L}_K(\phi, P) = \mathcal{L}_{\alpha K}(\phi_{\alpha}, P)$ and hence, if we write $\phi^* := \varphi^*_K(P)$, then $\phi^*_{\alpha} = \varphi^*_{\alpha K}(P)$ and therefore, $f^*_K(P) = f^*_{\alpha K}(P)$ as desired.

In fact, in order to study $\varphi^*(P)$, it will be convenient also to define a related projection of a one-dimensional probability distribution.  To this end, for $a \geq 0$,  let $\Phi_{a} := \{\phi(\cdot-a):\phi \in \Phi\}$ and set
\begin{align}
  \mathcal{H}_{a} \equiv \mathcal{H}_{a,K} := \biggl\{r \mapsto p\lambda_p(K)r^{p-1}e^{\phi(r)} : \phi \in \Phi_{a}, p\lambda_p(K)\int_{a}^\infty r^{p-1}e^{\phi(r)} \, dr = 1\biggr\}.
  \label{Eqn:HclassDefn}
\end{align}
Here, we incorporate the greater generality of the translation by $a$ in order to facilitate our analysis of the adaptivity properties of the $K$-homothetic log-concave MLE in Section~\ref{Sec:KknownAdaptation}. We continue to write $\Phi = \Phi_0$ and also write $\mathcal{H} = \mathcal{H}_0$ as shorthand. For a probability measure $Q$ on $[a, \infty)$ and $\phi \in \Phi_{a}$, let
\begin{equation}
  \label{Eq:QProjection}
  L(\phi, Q) \equiv L_{a,K}(\phi, Q) := \int_{[a,\infty)} \phi \, dQ - p\lambda_p(K) \int_{a}^\infty r^{p-1} e^{\phi(r)} \, dr + 1.
\end{equation}
  Similarly to before, we let $\phi^*(Q) \equiv \phi_{a,K}^*(Q) := \argmax_{\phi \in \Phi_{a}} L(\phi,Q)$.  Again, any maximiser $\phi^*(Q)$ of $L(\cdot,Q)$ over $\Phi_{a}$ is unique, and if $\phi^*(Q)$ exists with $L(\phi^*(Q), Q) \in \mathbb{R}$, then, writing $h^*(Q)(r) \equiv h_{a,K}^*(Q)(r) := p\lambda_p(K)r^{p-1}e^{\phi^*(Q)(r)}$, we have that $h^*(Q) \in \mathcal{H}_{a}$, so in particular, $h^*(Q)$ is a (log-concave) density.  

The following proposition gives necessary and sufficient conditions for the $K$-homothetic log-concave projection to be well-defined.  We write $\mathcal{P}_p$ for the set of probability distributions on $\mathbb{R}^p$ with $\int_{\mathbb{R}^p} \| x \|_K \, dP(x) < \infty$ and $P(\{0 \}) < 1$; the first of these conditions is equivalent to $\int_{\mathbb{R}^p} \| x \| \, dP(x) < \infty$.  We let $\mathcal{Q}_{a}$ denote the class of probability measures $Q$ on $[a,\infty)$ with $\int_{a}^\infty r \, dQ(r) < \infty$ and $Q(\{a\}) < 1$, and let $\mathcal{Q} := \mathcal{Q}_0$.
\begin{proposition}
\label{Prop:ProjectionExistence}
We have 
\begin{enumerate}[(i)]
\item if $\int_{a}^\infty r \, dQ(r) = \infty$, then $L(\phi, Q) = - \infty$ for all $\phi \in \Phi_{a}$;
\item if $Q(\{a\}) = 1$, then $\sup_{\phi \in \Phi_{a}} L(\phi, Q) = \infty$;
\item if $Q \in \mathcal{Q}_{a}$, then $\sup_{\phi \in \Phi_{a}} L(\phi, Q) \in \mathbb{R}$ and $\phi_{a}^*$ is a well-defined map from $\mathcal{Q}_{a}$ to $\Phi_{a}$;
\item if $P$ is a probability measure on $\mathbb{R}^p$ and we define a probability measure $Q$ on $[0,\infty)$ by $Q\bigl([0,r)\bigr) := P(\{x:\|x\|_K < r\})$, then $\mathcal{L}(\phi,P) = L(\phi,Q)$ for every $\phi \in \Phi$.  In particular, if $P \in \mathcal{P}_p$, then $Q \in \mathcal{Q}$ and $\varphi^*(P) = \phi^*(Q)$.
\end{enumerate}
\end{proposition}
\begin{remark}
From Proposition~\ref{Prop:ProjectionExistence}(iv), we see that the conditions on $P$ required for the $K$-homothetic log-concave projection to exist, namely $\int_{\mathbb{R}^p} \| x \|_K \, dP(x) < \infty$ and $P(\{0 \}) < 1$, are weaker than the corresponding conditions for the ordinary log-concave projection to exist, namely $\int_{\mathbb{R}^p} \| x \| \, dP(x) < \infty$ and $P(H) < 1$ for every hyperplane $H$; cf.~\citet[][Theorem~2.2]{dumbgen2011approximation}.
\end{remark}

The next proposition gives some basic properties of the $K$-homothetic log-concave projection. 

\begin{proposition}
  \label{Prop:ProjectionBasicProperties}
Let $Q \in \mathcal{Q}_{a}$, let $\phi^* := \phi_{a}^*(Q)$ and let $h^*(r) := p \lambda_p(K) r^{p-1}e^{\phi^*(r)}$ for $r \in [a,\infty)$.  
  \begin{enumerate}[(i)]
  \item The projection $\phi_{a}^*(\cdot)$ is scale equivariant in the sense that if $\alpha > 0$, and $Q_\alpha \in \mathcal{Q}_{\alpha a}$ is defined by $Q_{\alpha}\bigl( [\alpha a, r)\bigr) := Q\bigl( [a, r/\alpha) \bigr)$ for all $r \in [\alpha a, \infty)$, then $\phi^*_{\alpha a}(Q_\alpha)(r) = \phi^*_{a}(Q)(r/\alpha) - p\log \alpha$, and therefore $h^*_{\alpha a}(Q_\alpha)(r) = (1/\alpha) h^*_a(Q)(r/\alpha)$.
  \item Let $\Delta : [a, \infty) \rightarrow [-\infty, \infty)$ be a function satisfying the property that there exists $t > 0$ such that
    $r \mapsto \phi^*(r) + t \Delta(r) \in \Phi_{a}$. Then
    \[
      \int_a^\infty \Delta \, dQ(r) \leq \int_{a}^\infty \Delta(r) h^*(r) \, dr.
    \]
  \item $\int_{a}^\infty r h^*(r) \, dr \leq \int_a^\infty r \, dQ(r)$.
  \item For any $h_0 \in \mathcal{H}_a$, we have $\int_a^\infty h^* \log (h^*/h_0) \leq \int_a^\infty \log (h^*/h_0) \, dQ$.
  \end{enumerate}
\end{proposition}
\begin{remark}
Proposition~\ref{Prop:ProjectionBasicProperties}(iii) reveals a difference between the $K$-homothetic log-concave projection, and the ordinary log-concave projection, which preserves the mean \citep[][Remark~2.3]{dumbgen2011approximation}.  Lemma~\ref{Lem:MLEMeanPreservation} provides control on the extent to which the mean is shrunk by the $K$-homothetic log-concave projection in the special case where $Q$ is an empirical distribution.
  \end{remark}
In particular, consider $X_1,\ldots,X_n \stackrel{\mathrm{iid}}{\sim} P \in \mathcal{P}_p$ with empirical distribution $\mathbb{P}_n$, and for $A \in \mathcal{B}(\mathbb{R})$, let $Q(A) := P(\{x:\|x\|_K \in A\})$.  Let $Z_i := \|X_i\|_K$ for $i \in [n]$ and let $\mathbb{Q}_n$ denote the empirical distribution of $Z_1, \ldots, Z_n$.  Writing $\hat{f}_n := f^*(\mathbb{P}_n)$, $f^* := f^*(P)$, $\hat{h}_n := h^*(\mathbb{Q}_n)$ and $h^* := h^*(Q)$, we have by Proposition~\ref{Prop:ProjectionBasicProperties}(iv) and Lemma~\ref{Lem:ChangeOfVar} that
\begin{align*}
d_X^2(\hat{f}_n,f^*) = \int_{\mathbb{R}^p} \log \frac{\hat{f}_n}{f^*} \, d\mathbb{P}_n = \int_0^\infty \log \frac{\hat{h}_n}{h^*} \, d\mathbb{Q}_n \geq \int_0^\infty \hat{h}_n \log \frac{\hat{h}_n}{h^*} &= \int_{\mathbb{R}^p} \hat{f}_n \log \frac{\hat{f}_n}{f^*} \\
&= d_{\mathrm{KL}}^2(\hat{f}_n,f^*).
\end{align*}
As a final basic property of our projections, we establish continuity with respect to the Wasserstein distance.  Recall that if $P,P'$ are probability measures on a Euclidean space with finite first moments, then the Wasserstein distance between $P$ and $P'$ is defined as
\[
d_\mathrm{W}(P,P') := \inf_{(X,X') \sim (P,P')} \mathbb{E}\|X-X'\|,
\]
where the infimum is taken over all pairs of random vectors $X,X'$, defined on the same probability space, with $X \sim P$ and $Y \sim Q$.  We also recall that if $P$ has a finite first moment, then $d_\mathrm{W}(P_n,P) \rightarrow 0$ if and only if both $P_n \stackrel{d}{\rightarrow} P$ and $\int_{\mathbb{R}^p} \|x\| \, dP_n(x) \rightarrow \int_{\mathbb{R}^p} \|x\| \, dP(x)$.
\begin{proposition}
\label{Prop:Continuity}
Suppose that $P \in \mathcal{P}_p$ and $d_\mathrm{W}(P_n,P) \rightarrow 0$.  Then $\sup_{\phi \in \Phi} \mathcal{L}(\phi,P_n) \rightarrow \sup_{\phi \in \Phi} \mathcal{L}(\phi,P)$, $f^*(P_n)$ is well-defined for large $n$ and 
\begin{equation}
\label{Eq:TVConv}
\int_{\mathbb{R}^p} |f^*(P_n)(x) - f^*(P)(x)| \, dx \rightarrow 0.
\end{equation}
Moreover, given any compact set $D$ contained in the interior of the support of $f^*(P)$, we have $\sup_{x \in D} |f^*(P_n)(x) - f^*(P)(x)| \rightarrow 0$. 
\end{proposition}
\begin{remark}
Proposition~\ref{Prop:Continuity} immediately yields a consistency (and robustness to misspecification) result for the $K$-homothetic log-concave MLE.  In particular, suppose that  $X_1,\ldots,X_n \stackrel{\mathrm{iid}}{\sim} P \in \mathcal{P}_p$ with empirical distribution $\mathbb{P}_n$, and let $\hat{f}_n := f^*(\mathbb{P}_n)$, $f^* := f^*(P)$.  Then, by the strong law of large numbers and Varadarajan's theorem \citep[e.g.][Theorem~11.4.1]{dudley2002real}, we have $d_{\mathrm{W}}(\mathbb{P}_n,P) \stackrel{\mathrm{a.s.}}{\rightarrow} 0$, so
\[
\int_{\mathbb{R}^p} |\hat{f}_n - f^*| \stackrel{\mathrm{a.s.}}{\rightarrow} 0.
\]
\end{remark}
\begin{remark}
In fact, the conclusion of Proposition~\ref{Prop:Continuity} also holds for stronger norms than the total variation norm.  In particular, taking $a_0 > 0$ and $b_0 \in \mathbb{R}$ such that $f^*(P)(x) \leq e^{-a_0\|x\|+b_0}$ for all $x \in \mathbb{R}^p$, we have by, e.g., \citet[][Proposition~2]{cule2010theoretical} that for every $a < a_0$,
\[
\int_{\mathbb{R}^p} e^{a\|x\|}|f^*(P_n)(x) - f^*(P)(x)| \, dx \rightarrow 0.
\]
\end{remark}

\section{Risk bounds when \texorpdfstring{$K$}{K} is known}
\label{Sec:Kknown}

In this section we continue to consider $K \in \mathcal{K}$ as fixed (and known) and $\mu=0$. Let $f_0 \in \mathcal{F}_p^{K}$, and suppose that $X_1,\ldots, X_n \stackrel{\mathrm{iid}}{\sim} f_0$ with empirical distribution $\mathbb{P}_n$. Let $\hat{f}_n := f^*(\mathbb{P}_n)$ be the $K$-homothetic log-concave MLE. By Proposition~\ref{Prop:ProjectionExistence}(iv), we may compute $\hat{f}_n$ efficiently by first computing $Z_i = \| X_i \|_K$ for $i\in [n]$ and then, writing $\mathbb{Q}_n$ for the empirical distribution of $Z_1, \ldots, Z_n$, computing $\hat{\phi}_n := \phi^*(\mathbb{Q}_n)$.  Our final estimate is $\hat{f}_n := e^{\hat{\phi}_n}$. We defer algorithmic details to Section~\ref{Sec:Algorithm}.

\subsection{Worst-case bound}

Our first main result below provides a worst-case risk bound for $\hat{f}_n$ as an estimator of $f_0$ in terms of the $d_X^2$ divergence.
\begin{theorem}
\label{Thm:KKnownWorstCase}
Let $X_1,\ldots, X_n \stackrel{\mathrm{iid}}{\sim} f_0 \in \mathcal{F}_p^K$ with empirical distribution $\mathbb{P}_n$. Let $\hat{f}_n := f^*(\mathbb{P}_n)$ be the $K$-homothetic log-concave MLE.  There exists a universal constant $C > 0$ such that for $n \geq 8$,
  \[
    \mathbb{E}d^2_X(\hat{f}_n,f_0) \leq \frac{C}{n^{4/5}}.
  \]
\end{theorem}
As mentioned in the introduction, the attractive aspect of this bound is that it does not depend on $p$.  The proof relies heavily on the special moment preservation properties of the $K$-homothetic log-concave MLE developed in Lemmas~\ref{Lem:MLEMeanPreservation},~\ref{Lem:MLEVarPreservation} and~\ref{Lem:MLEMeanVarPreservation}.

\subsection{Adaptive bounds}
\label{Sec:KknownAdaptation}

We now turn to the adaptation properties of $\hat{f}_n$.  For $k \in \mathbb{N}$ and $a \geq 0$, we say $\phi \in \Phi_a$ is \emph{$k$-affine}, and write $\phi \in \Phi_a^{(k)}$, if there exist $r_0 \in (a,\infty]$ and intervals $I_1,\ldots,I_{k}$ with $I_j = [a_{j-1},a_j]$ for some $a = a_0 < a_1 < \ldots < a_k = r_0$ and such that $\phi$ is affine on each $I_j$ for $j \in [k]$, and $\phi(r) = -\infty$ for $r > r_0$.  Define $\mathcal{H}_a^{(k)} := \bigl\{ h \in \mathcal{H}_a\,:\, h(r) = p\lambda_p(K) r^{p-1}e^{\phi(r)} \textrm{ for some } \phi \in \Phi_a^{(k)}\bigr\}$.  Again, we write $\Phi^{(k)} := \Phi_0^{(k)}$ and $\mathcal{H}^{(k)} := \mathcal{H}_0^{(k)}$.
\begin{theorem}
\label{Thm:KknownAdaptation}
Let $f_0 \in \mathcal{F}_p^K$ be given by $f_0(\cdot) = e^{\phi_0(\|\cdot\|_K)}$ for some $\phi_0 \in \Phi$, and let $X_1,\ldots, X_n \stackrel{\mathrm{iid}}{\sim} f_0$ with empirical distribution $\mathbb{P}_n$. Let $\hat{f}_n := f^*(\mathbb{P}_n)$ be the $K$-homothetic log-concave MLE.  Define $h_0 \in \mathcal{H}$ by $h_0(r) := p\lambda_p(K)r^{p-1}e^{\phi_0(r)}$ for $r \in [0,\infty)$.  Then, writing $\nu_k := 2^{1/2} \wedge \inf_{h \in \mathcal{H}^{(k)}} d_{\mathrm{KL}}(h_0,h)$, there exists a universal constant $C > 0$ such that for $n \geq 8$,
\[
    \mathbb{E}d_X^2(\hat{f}_n,f_0) \leq C \min_{k \in [n]} \biggl(\frac{k^{4/5}\nu_k^{2/5}}{n^{4/5}} \log \frac{en}{k\nu_k} + \frac{k}{n}\log^{5/4}\frac{en}{k}\biggr).
  \]
\end{theorem}
\begin{remark}
Taking the universal constant $C > 0$ from the conclusion of Theorem~\ref{Thm:KknownAdaptation} and setting $C_* := \max\{(3C/2)^{5/4},1\}$, we see that if $k \in [n]$ and if $\nu_k^2 \geq C_* \frac{k}{n} \log^{5/4} \bigl(\frac{en}{k}\bigr)$, then
\[
C\frac{k^{4/5}\nu_k^{2/5}}{n^{4/5}}\log\frac{en}{k\nu_k} \leq C\nu_k^2\biggl\{\frac{1}{C_*^{4/5}\log(en/k)}\log\biggl(\frac{en^{3/2}}{C_*^{1/2}k^{3/2} \log^{5/8}(en/k)}\biggr)\biggr\} \leq \nu_k^2.
\]
On the other hand, if $k \in [n]$ and if $\nu_k^2 \leq C_* \frac{k}{n} \log^{5/4} \bigl(\frac{en}{k}\bigr)$, then
\[
\frac{k^{4/5}\nu_k^{2/5}}{n^{4/5}}\log\frac{en}{k\nu_k} \lesssim \frac{k}{n}\log^{1/4}\Bigl(\frac{en}{k}\Bigr)\log\biggl(\frac{en^{3/2}}{k^{3/2}\log^{5/8}(en/k)}\biggr) \lesssim \frac{k}{n}\log^{5/4} \frac{en}{k}.
\]
It follows that Theorem~\ref{Thm:KknownAdaptation} implies the following sharp oracle inequality: there exists a universal constant $C > 0$ such that  
\[
\mathbb{E}d^2_X(\hat{f}_n,f_0) \leq \min_{k \in [n]}\biggl(\nu_k^2 + C\frac{k}{n}\log^{5/4} \frac{en}{k}\biggr).
\]
\end{remark}
The proof of Theorem~\ref{Thm:KknownAdaptation} proceeds by first considering the case $k=1$, described in Proposition~\ref{Prop:AdaptiveRateAffine} below, for which we obtain a slightly different approximation error term. 
\begin{proposition}
  \label{Prop:AdaptiveRateAffine}
  Let $a \in [0, \infty)$ and suppose that $Z_1,\ldots,Z_n \stackrel{\mathrm{iid}}{\sim} h_0$ for some $h_0 \in \mathcal{H}_{a}$ with empirical distribution function $\mathbb{Q}_n$, and let $\hat{h}_n := h_{a}^*(\mathbb{Q}_n)$.  Set $\nu := \inf \{  d_{\mathrm{H}}(h_0, h) : h \in \mathcal{H}_{a}^{(1)},\, h_0 \ll  h\}$.  Then there exists a universal constant $C > 0$ such that for $n \geq 8$,
  \[
 \frac{1}{n}\sum_{i=1}^n \mathbb{E} \log \frac{\hat{h}_n(Z_i)}{h_0(Z_i)} \leq  C\biggl(\frac{\nu^{2/5}}{n^{4/5}}\log\frac{en}{\nu} + \frac{1}{n}\log^{5/4}(en)\biggr).
  \]
\end{proposition}
\begin{remark}
  Since $2^{1/2} \leq n e^{-3/2}$ for $n \geq 8$ and the function $x \mapsto x^{2/5} \log(en /x)$ is increasing for $x \leq n e^{-3/2}$, Proposition~\ref{Prop:AdaptiveRateAffine} remains true if we redefine $\nu := 2^{1/2} \wedge \inf_{ h \in \mathcal{H}_a^{(1)}} d_{\mathrm{KL}}(h_0, h) $. Hence, the conclusion of Proposition~\ref{Prop:AdaptiveRateAffine} is stronger than that obtained by specialising Theorem~\ref{Thm:KknownAdaptation} to the case $k=1$.
\end{remark}

Proposition~\ref{Prop:AdaptiveRateAffine} is analogous to Theorem~5 in \cite{kim2016adaptation}.  However the proof does not follow from the local bracketing entropy analysis in \citet[][Theorem~4]{kim2016adaptation} because, for $h \in \mathcal{H}_{a}^{(1)}$, $\log h$ is not an affine function when $p \geq 2$. To prove Proposition~\ref{Prop:AdaptiveRateAffine}, then, we show in Lemma~\ref{Lem:GeneralLocalBracketing} that the bracketing entropy of a local Hellinger ball around an arbitrary $g_0 \in \mathcal{F}_1$ is small if we further restrict the local ball to include only $g \in \mathcal{F}_1$ such that $\log(g/g_0)$ is concave.  \cite[][Theorem~4]{kim2016adaptation} were interested in the case where $g_0$ is log-affine, for which $\log(g/g_0)$ is necessarily concave for every log-concave $g$, so their result can be considered as a special case of Lemma~\ref{Lem:GeneralLocalBracketing}.

\section{Risk bounds when \texorpdfstring{$K$}{K} is estimated}
\label{Sec:KEstimated}

\subsection{General approach}
\label{Subsec:GeneralApproach}

When $K$ is unknown and needs to be estimated, one approach is to attempt to maximise~\eqref{Eq:RpProjection} jointly in $K$ and $\phi$; however, this appears to be computationally infeasible.  We therefore consider the following plug-in procedure, where for simplicity of exposition we assume an even sample size.  Given $p$-dimensional random vectors $X_1,\ldots,X_{2n}$, we use $X_{n+1}, \ldots, X_{2n}$ to estimate $\hat{K}$ and $\hat{\mu}$ (we give specific examples for how to estimate $\hat{K}$ in Sections~\ref{Sec:Affine} and~\ref{Sec:KNonparametric} below), where we think of $\mathcal{K}$ as being metrised by the Hausdorff metric, and equip $\mathcal{K}$ with the $\sigma$-algebra induced by this metric. We then form $\tilde{Z}_i := \| X_i - \hat{\mu} \|_{\hat{K}} $ for $i \in [n]$ and, writing $\tilde{\mathbb{Q}}_n$ for the empirical distribution of $\tilde{Z}_1,\ldots,\tilde{Z}_n$, compute $\hat{\phi}_n := \phi_{\hat{K}}^*(\tilde{\mathbb{Q}}_n)$. Our final density estimate is $\hat{f}_n(\cdot) := e^{ \hat{\phi}_n( \| \cdot - \hat{\mu} \|_{\hat{K}})}$. 

Our goal in this section is to analyse the performance of the plug-in procedure without restricting our attention to any specific estimators $\hat{K}$ and $\hat{\mu}$. To do this, we assume that $X_1,\ldots,X_{2n}$ are generated independently from a density $f_0 \in \mathcal{F}^{\mathcal{K}}_p$ of the form $f_0(\cdot) = e^{\phi_0(\|\cdot - \mu \|_K)}$ and then bound the Hellinger error $d_{\mathrm{H}}(\hat{f}_n, f_0)$ in terms of the deviations between $\hat{K}$ and $K$ and between $\hat{\mu}$ and $\mu$. To that end, for $K_1,K_2 \in \mathcal{K}$, define a pseudo-metric
\begin{equation}
\label{Eq:dscale}
  d_{\mathrm{scale}}(K_1,K_2) := \inf \biggl\{ \epsilon > 0 \,:\, \frac{1}{1 + \epsilon} K_1 \subseteq K_2 \subseteq (1 + \epsilon) K_1 \biggr\}.
\end{equation}
This notion of distance satisfies all of the axioms for being a metric except for the triangle inequality; it is also scale invariant in the sense that $d_{\mathrm{scale}}(\gamma K_1, \gamma K_2) = d_{\mathrm{scale}}(K_1, K_2)$ for any $\gamma  >0$. For $c_1,c_2 > 0$, let
\begin{align}
  \mathcal{E}_{c_1, c_2} := \biggl\{ \frac{\|\hat{\mu} - \mu\|_K}{\mathbb{E}_{f_0}^{1/2}(\| X_1 - \mu \|_K^2)} p^{3/2} \log (ep) \leq c_1 \biggr\} \, \bigcap \, \bigl\{ p\inf_{\alpha > 0} d_{\mathrm{scale}}(\alpha \hat{K},K)  < c_2 \bigr\}. \label{Eqn:Ec1c2}
\end{align}
Our main result in this subsection is the following:
\begin{proposition}
\label{Prop:KEstimatedGeneralRisk}
Let $X_1,\ldots, X_{2n} \stackrel{\mathrm{iid}}{\sim} f_0 \in \mathcal{F}^{K,\mu}_p$, and let $\hat{f}_n$, $\hat{K}$ and $\hat{\mu}$ be as defined above.  Then there exist universal constants $c_1,c_2 > 0$ such that 
\[
\mathbb{E}_{f_0}\bigl\{d_{\mathrm{H}}^2 (\hat{f}_n, f_0) \bigm| \hat{K},\hat{\mu}\bigr\}\mathbbm{1}_{\mathcal{E}_{c_1,c_2}} \lesssim n^{-4/5} +  p \frac{ \| \hat{\mu} - \mu\|_K }{\mathbb{E}_{f_0}^{1/2}(\| X_1 - \mu \|^2_K)} + p \inf_{\alpha > 0} d_{\mathrm{scale}}(\alpha \hat{K},K).
\]
Moreover, if $f_0 \in \mathcal{F}^{K,\mu}_p$ is of the form $f_0(x) = e^{\phi_0(\| x - \mu \|_K)}$ for some $\phi_0 \in \Phi$ such that $\phi_0'$ is absolutely continuous and such that $\inf_{r \in [0,\infty)} \phi_0''(r) \geq - D_0$ for some $D_0 > 0$, then 
\[
  \mathbb{E}_{f_0}\bigl\{d_{\mathrm{H}}^2 (\hat{f}_n, f_0) \bigm| \hat{K},\hat{\mu}\bigr\}\mathbbm{1}_{\mathcal{E}_{c_1, c_2}} \lesssim n^{-4/5} +  (D_0^2 + 1) \biggl\{ p^2 \frac{\| \hat{\mu} - \mu\|^2_K}{\mathbb{E}_{f_0}(\|X_1 - \mu \|_K^2)} + p^2 \inf_{\alpha > 0} d_{\mathrm{scale}}^2(\alpha \hat{K},K)\biggr\}.
\]
Finally, if $f_0 \in \mathcal{F}^{K,\mu}_p$ is of the form $f_0(\cdot) = e^{-a\| \cdot - \mu \|_K + b}$ for some $a > 0$ and $b\in \mathbb{R}$, then
\[
  \mathbb{E}_{f_0}\bigl\{d_{\mathrm{H}}^2 (\hat{f}_n, f_0) \bigm| \hat{K},\hat{\mu} \bigr\}\mathbbm{1}_{\mathcal{E}_{c_1, c_2}} \lesssim \frac{1}{n}\log^{5/4}(en) +  p^2 \frac{\| \hat{\mu} - \mu\|^2_K}{\mathbb{E}_{f_0}(\| X_1 - \mu \|_K^2)} + p^2\inf_{\alpha > 0} d_{\mathrm{scale}}^2(\alpha \hat{K},K).
\]
\end{proposition}

The first term in the error bounds of Proposition~\ref{Prop:KEstimatedGeneralRisk} arises from the estimation of the generator $\phi_0$, the second term arises from the estimation of the centering vector $\mu$, and the third term arises from the estimation of the super-level set $K$. 

\begin{remark}
When $f_0$ is not a uniform density, the bounds in Proposition~\ref{Prop:KEstimatedGeneralRisk} do not depend on the choice of $K$ in the representation of $f_0$.  More precisely, if $f_0(\cdot) = e^{\tilde{\phi}( \| \cdot - \mu \|_{\tilde{K}})}$ for $\tilde{K} \in \mathcal{K}$ and $\tilde{\phi} \in \Phi$, then, by Proposition~\ref{Prop:ExistencePhi}, there exists $\gamma > 0$ such that $\tilde{K} = \gamma K$. We observe then that the quantities on the right-hand sides of the inequalities in Proposition~\ref{Prop:KEstimatedGeneralRisk} do not change if $K$ is replaced with $\tilde{K}$. If $f_0$ is the uniform distribution, then the centering vector $\mu$ is not unique either, and Proposition~\ref{Prop:KEstimatedGeneralRisk} applies to \emph{any} choice of the centering vector $\mu$.  
\end{remark}

The main difficulty in the proof of Proposition~\ref{Prop:KEstimatedGeneralRisk} is that we cannot make any assumptions about the density of the data points $\tilde{Z}_1, \ldots, \tilde{Z}_n$ since these are constructed from $\hat{\mu}$ and $\hat{K}$ instead of the true $\mu$ and $K$. We overcome this problem through Lemma~\ref{Lem:EmpiricalProcess}, where we apply empirical process theory in the presence of model misspecification.

\subsection{Risk bounds when \texorpdfstring{$K$}{K} is known up to a positive definite transformation}
\label{Sec:Affine}

In this section, we let $K_0 \in \mathcal{K}$ be a balanced convex body in an isotropic position, so that $K_0 = -K_0$ and $\frac{1}{\lambda_p(K_0)} \int_{K_0} x x^\top \, dx = I_p$.  Let $r_1, r_2 > 0$ be such that $r_1 B_p(0,1) \subseteq K_0 \subseteq r_2 B_p(0,1)$ and let $r_0 := r_2/r_1$. Let $\mu \in \mathbb{R}^p$, $\Sigma_0 \in \mathbb{S}^{p \times p}$, $K = \Sigma_0^{1/2} K_0$, $\phi_0 \in \Phi$, and let $f_0 \in \mathcal{F}_p^{K,\mu}$ be such that $f_0(\cdot) = e^{\phi_0(\|\cdot - \mu\|_K)}$. We assume that $K_0$ is known but that $\Sigma_0$ is unknown. Let $\Sigma := \int_{\mathbb{R}^p} (x-\mu)(x-\mu)^\top f_0(x) \, dx$, so that $\Sigma \propto \Sigma_0$.

Throughout this subsection, we assume that $X_1,\ldots,X_n \stackrel{\mathrm{iid}}{\sim} f_0$, and denote the sample covariance matrix by $\hat{\Sigma} := n^{-1} \sum_{i=1}^n (X_i - \hat{\mu})(X_i - \hat{\mu})^\top$, where $\hat{\mu} := n^{-1}\sum_{i=1}^n X_i$.  We let $\hat{K} := \hat{\Sigma}^{1/2} K_0$.  The following proposition controls $d_{\mathrm{scale}}(\hat{K}, K)$ and $\| \hat{\mu} - \mu \|_K$; this first part relies heavily on \citet[][Theorem~4.1]{adamczak2010quantitative}, which provides a tail bound for the operator norm of the difference between the sample covariance matrix and the identity matrix, uniformly over all isotropic log-concave densities.
\begin{proposition}
  \label{Prop:dscaleAffineBound}
  There exists a universal constant $C \geq 1$ such that, if $C \sqrt{ \frac{p}{n}} \log^3 (en) \leq 1/2$, then with probability at least $1 - 2/n$,
  \begin{align}
  \inf_{\alpha > 0} d_{\mathrm{scale}}(\alpha \hat{K}, K) \leq C r_0\frac{p^{1/2}}{n^{1/2}} \log^3 (en), \label{Eqn:dscaleAffineBound1}
\end{align}
and
\[
 p^{1/2} \frac{\| \hat{\mu} - \mu \|_K}{\mathbb{E}_{f_0}^{1/2}( \| X_1 - \mu \|^2_K)} \leq C r_0 \frac{p^{1/2}}{n^{1/2}} \log (en).
\]
\end{proposition}
\begin{remark}
Since $K_0$ is balanced and in an isotropic position, we know from John's Ellipsoid Theorem that $r_0 \leq p^{1/2}$, see, e.g., \citet[][Theorem 3.1]{ball1997elementary} or \citet[][Section 3]{john1948extremum}. The extreme case where $r_0 = p^{1/2}$ is realised when $K_0 = [-1,1]^p$. For a general $K_0$ however, $r_0$ may be much smaller. 
\end{remark}
The following corollary is therefore an immediate consequence of Propositions~\ref{Prop:KEstimatedGeneralRisk} and~\ref{Prop:dscaleAffineBound}.
\begin{corollary}
  \label{Cor:Ellipse}
 Let $\hat{K}, \hat{\mu}$ be as in Proposition~\ref{Prop:dscaleAffineBound} and let $\hat{f}_n$ be as in Proposition~\ref{Prop:KEstimatedGeneralRisk}.  Then
\[
  \mathbb{E}_{f_0}d_{\mathrm{H}}^2 (\hat{f}_n, f_0) \lesssim r_0\frac{p^{3/2}}{n^{1/2}} \log^3 (en).
\]
Moreover, if $f_0 \in \mathcal{F}^{K,\mu}_p$ is of the form $f_0(\cdot) = e^{\phi_0(\| \cdot - \mu \|_K)}$ for some $\phi_0 \in \Phi$ such that $\phi_0'$ is absolutely continuous and that $\inf_{r \in [0,\infty)} \phi_0''(r) \geq - D_0$ for some $D_0 > 0$, then 
\[
  \mathbb{E}_{f_0}d_{\mathrm{H}}^2 (\hat{f}_n, f_0) \lesssim \frac{1}{n^{4/5}} + r_0^2(D_0^2+1)\frac{p^3}{n}\log^6 (en).
\]
Finally, if $f_0 \in \mathcal{F}^{K,\mu}_p$ is of the form $f_0(\cdot) = e^{-a\| \cdot - \mu \|_K + b}$ for some $a > 0$ and $b\in \mathbb{R}$, then
\[
  \mathbb{E}_{f_0}d_{\mathrm{H}}^2 (\hat{f}_n, f_0) \lesssim \frac{1}{n}\log^{5/4}(en) +  r_0^2\frac{p^3}{n}\log^6 (en).
\]
\end{corollary}
Thus, in particular, Corollary~\ref{Cor:Ellipse} provides risk bounds for estimating elliptically symmetric log-concave densities, where we may take $r_0 = 1$.

\subsection{Risk bounds for general \texorpdfstring{$K$}{K}}
\label{Sec:KNonparametric}

For simplicity, we assume that $\mu = 0$ in this section. In the case where $\mu$ is unknown and $K$ is balanced, we may estimate $\mu$ with the empirical mean $n^{-1} \sum_{i=1}^n X_i$. Our algorithm for estimating $K$ proceeds by estimating the boundary of $K$ at a set of randomly chosen directions and then outputting the convex hull of the estimated boundary points. 

\begin{algorithm}[htp]
  \caption{Estimating $K$}
  \label{alg:estimateK}
  \textbf{Input:} $M \in \mathbb{N}$ and $X_1, \ldots, X_n, X_{n+1}, \ldots, X_{n+M} \in \mathbb{R}^p$.\\
  \textbf{Output:} $\hat{K} \in \mathcal{K}$.

  \begin{algorithmic}[1]
    \State Set $k \leftarrow \lfloor \log n \rfloor$.
    \State For $m \in [M]$, let $\theta_m = X_{n+m}/\| X_{n+m}\|_2$.
    \For{$m \in [M]$}:
       \State $\mathcal{I}^k_m \leftarrow \bigl\{ i \in [n] : X_i^\top \theta_m \geq \textrm{ $k$-th max} \{ X_i^\top \theta_{m'} \,:\, m' \in [M] \} \bigr \}$.
       \State $t_m \leftarrow | \mathcal{I}^k_m |^{-1} \sum_{i \in \mathcal{I}_m^k} \| X_i \|_2$.
       \EndFor
    \State Output $\hat{K} \leftarrow \mathrm{conv} \{ t_1 \theta_1, \ldots, t_m \theta_m \}$. 
  \end{algorithmic}
\end{algorithm}

The set $\{\theta_m/\| \theta_m \|_K\}_{m=1,\ldots,M}$ contains random points on the boundary of $K$. It is shown in Lemma~\ref{Lem:tmDeviation} that $\max_{m \in [M]}\bigl|t_m - 1/\| \theta_m \|_K\bigr|$ is small, which is allows us to control the error of the approximation of $K$ by $\hat{K}$. Figure~\ref{Fig:estimateK} illustrates the behaviour of the algorithm when $p=2$. Recall the definition of $d_{\mathrm{scale}}$ from~\eqref{Eq:dscale}; the next Proposition bounds the deviation $\inf_{\alpha > 0} d_{\mathrm{scale}}(\alpha \hat{K}, K)$. 

\begin{figure}
  \begin{subfigure}{0.24\textwidth}
        \centering
    \includegraphics[scale=0.3, trim=70 50 70 50]{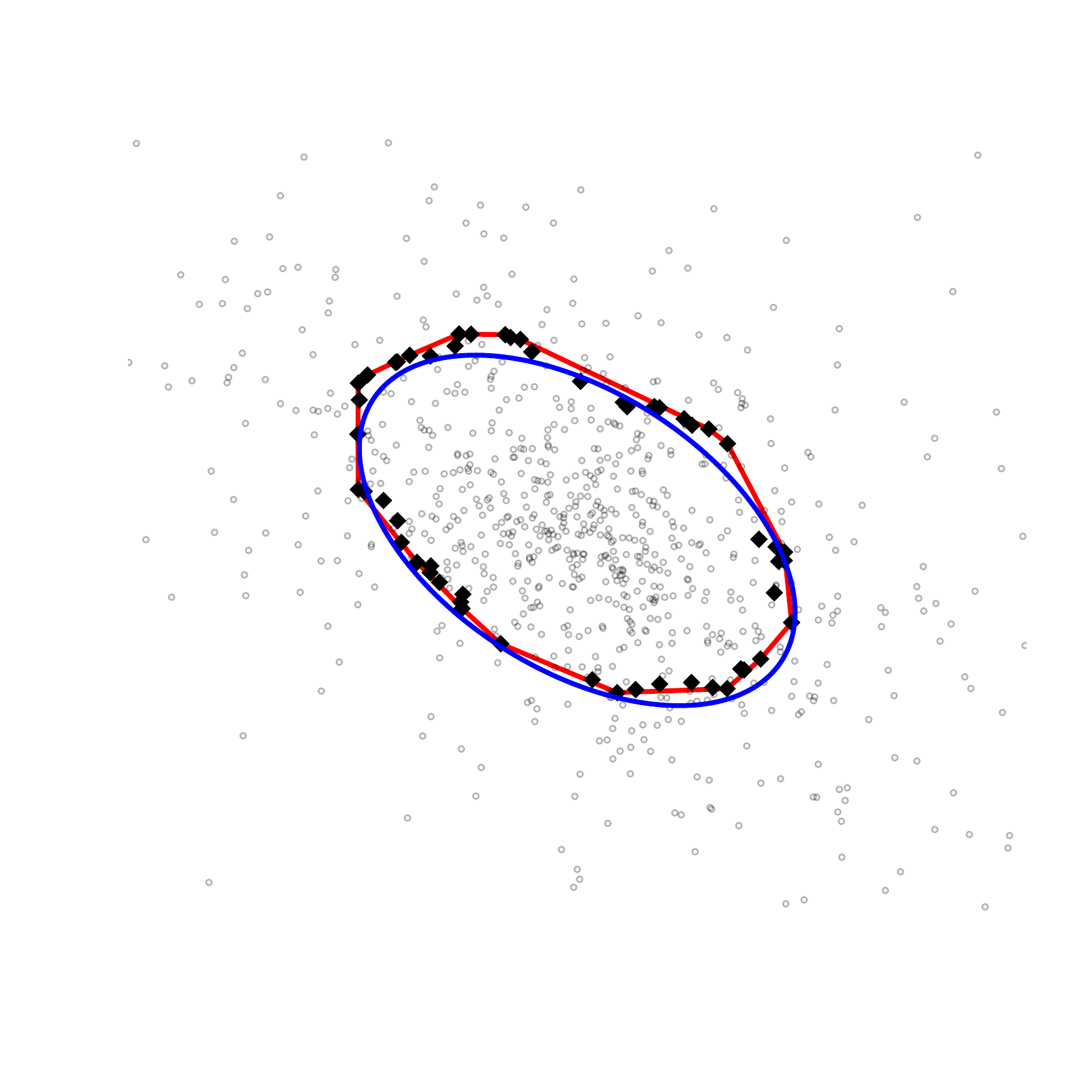}
    \caption{}
    \label{Fig:ball800}
  \end{subfigure}
  \begin{subfigure}{0.24\textwidth}
        \centering
    \includegraphics[scale=0.3, trim=70 50 70 50]{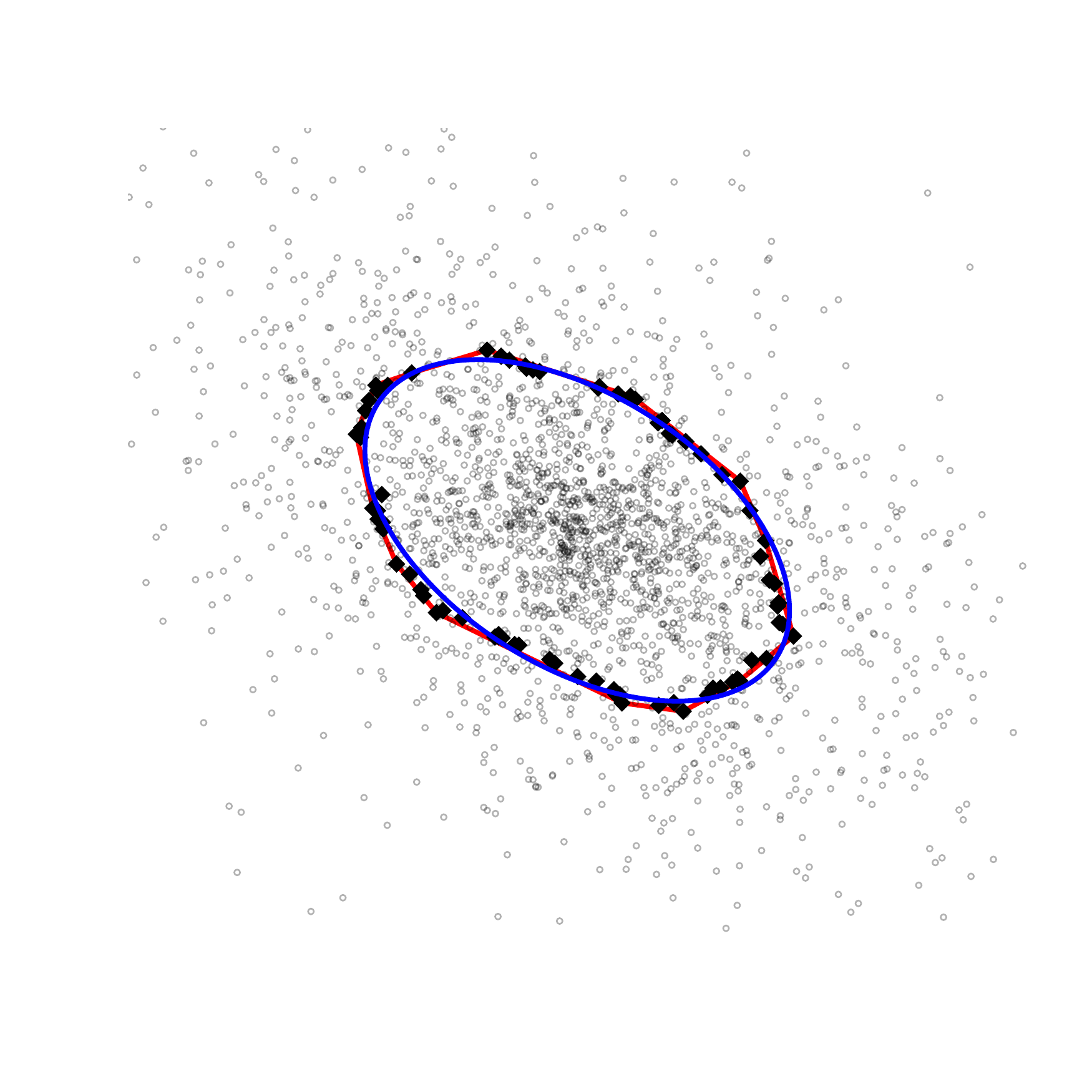}
    \caption{}
    \label{Fig:ball2400}
  \end{subfigure}  
  \begin{subfigure}{0.24\textwidth}
    \centering
    \includegraphics[scale=0.3, trim=70 50 70 50]{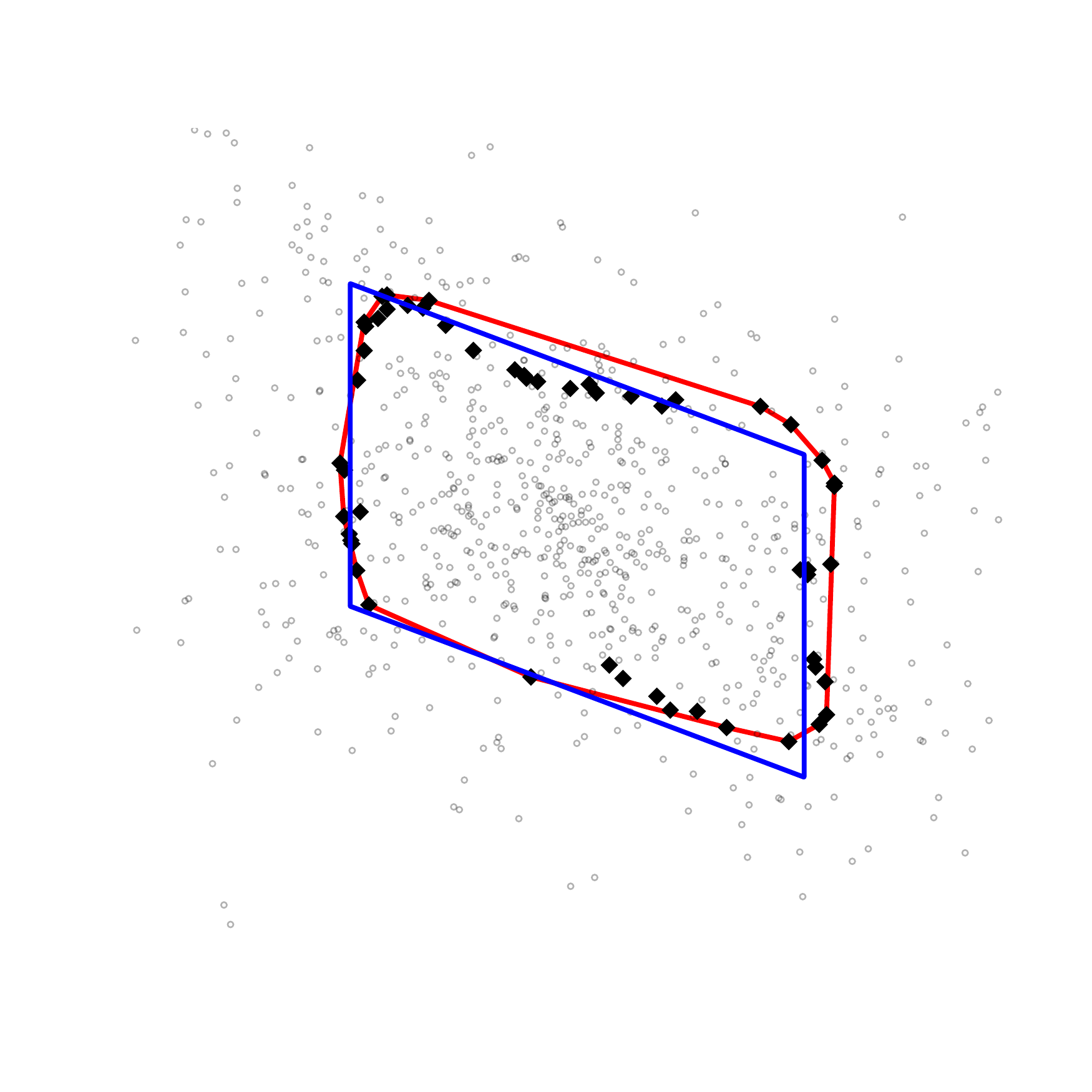}
    \caption{}
    \label{Fig:sq800}
  \end{subfigure}
  \begin{subfigure}{0.24\textwidth}
    \centering
    \includegraphics[scale=0.3, trim=70 50 70 50]{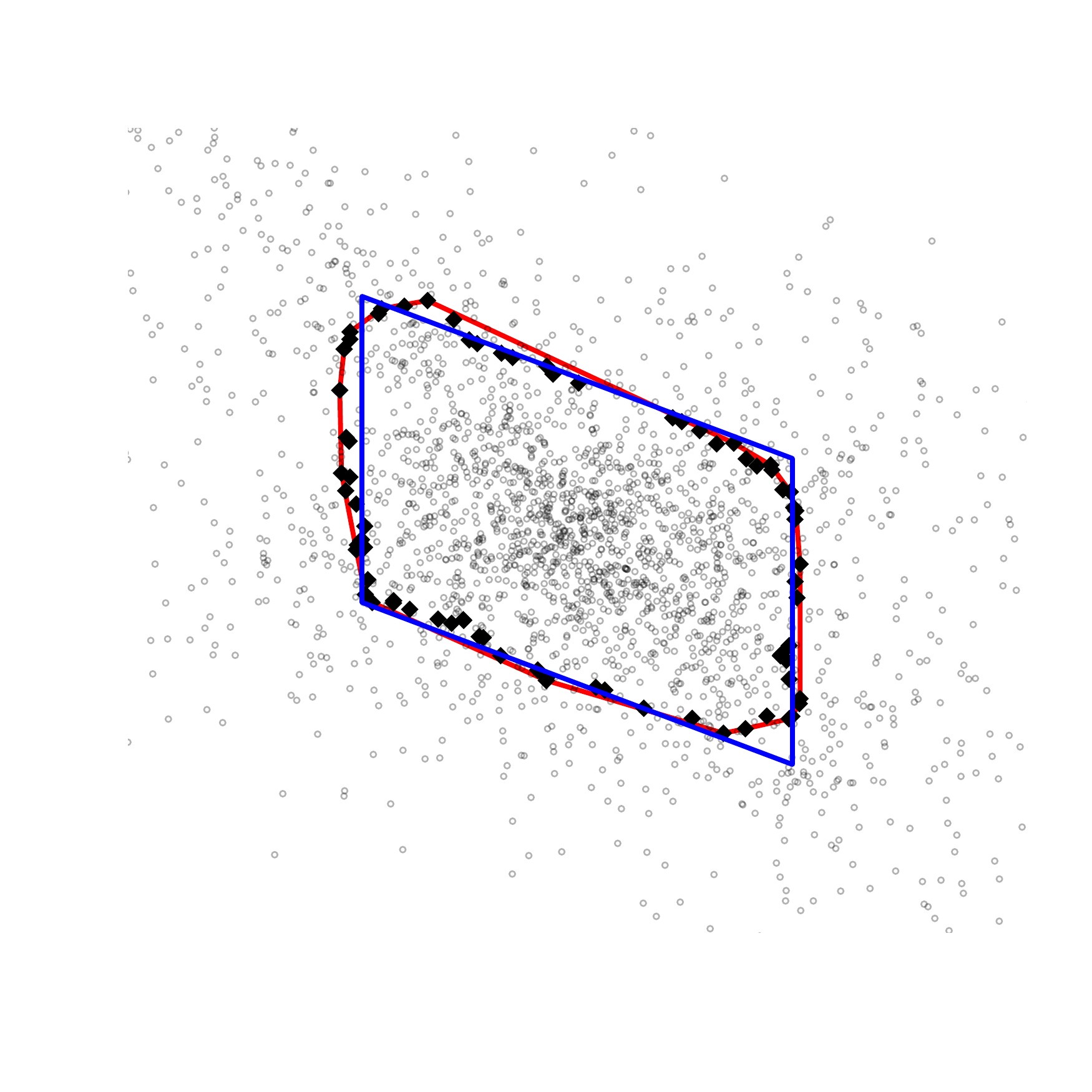}
    \caption{}
    \label{Fig:sq2400}
  \end{subfigure}
  \caption{Illustration of Algorithm~\ref{alg:estimateK} when $p=2$ and $f(\cdot) \propto e^{-\| \cdot \|_K}$. The blue shape is the true $K$; the red shape is the estimated $\hat{K}$.  Points in bold are the estimated boundary points $\{t_m \theta_m\}_{m=1}^M$ with $M = 6 \bigl \lceil n^{\frac{p-1}{p+1}} \bigr \rceil$. Figures~\ref{Fig:ball800} and~\ref{Fig:sq800} use $n=800$, while Figures~\ref{Fig:ball2400} and~\ref{Fig:sq2400} use $n=2400$.}
  \label{Fig:estimateK}
\end{figure}

\begin{proposition}
  \label{Prop:dscaleOverall}
  Suppose that $p \geq 2$, $K \in \mathcal{K}$, that there exist $r_2 \geq r_1 > 0$ such that $r_1 B_p(0,1)\subseteq K \subseteq r_2 B_p(0,1)$, and write $r_0 := r_2/r_1$. Suppose further that $r_0^2 n^{-\frac{1}{p+1}} \log^3 (en) \leq 1/64$ and let $M := \bigl \lceil n^{\frac{p-1}{p+1}} \bigr \rceil$. Let $X_1, \ldots, X_n, X_{n+1}, \ldots, X_{n+M} \stackrel{\mathrm{iid}}{\sim} f_0 \in \mathcal{F}^K_p$ and let $\hat{K}$ be the output of Algorithm~\ref{alg:estimateK}. Then, there exists a constant $C_{p,r_0} > 0$ depending only on $r_0$ and $p$ such that with probability at least $1 - C_{p,r_0} n^{- \frac{p}{p+1}}$,
  \begin{align}
    \inf_{\alpha > 0} d_{\mathrm{scale}}(\alpha \hat{K}, K) \lesssim r_0^2 n^{ - \frac{1}{p+1} } \log^3 (en).
    \label{Eqn:dscaleOverall}
  \end{align}
\end{proposition}

\begin{remark}
  If $K$ is in an isotropic position, then we have that $r_0 \leq p$ by John's Ellipsoid Theorem \citep[][Section 3]{john1948extremum}. If $K$ is additionally balanced, then we know from the same theorem that $r_0 \leq p^{1/2}$ (see the remark after Proposition~\ref{Prop:dscaleAffineBound}). Therefore, we recommend that Algorithm~\ref{alg:estimateK} be applied to whitened data $\tilde{X}_i = \hat{\Sigma}^{-1/2}(X_i - \hat{\mu})$ where $\hat{\Sigma}$ and $\hat{\mu}$ are the sample covariance matrix and the sample mean respectively. This transformation brings $K$ into an approximately isotropic position, and the error in this approximation is given in Proposition~\ref{Prop:dscaleAffineBound}.
\end{remark}
\begin{remark}
Proposition~\ref{Prop:dscaleOverall} has connections with an extensive line of research on the estimation of a convex body $K$ from observations supported on $K$; see, e.g., \citet{bronstein2008approximation} or \citet{brunel2018methods} for introductions to the field. 
\end{remark}

\begin{corollary}
  \label{Cor:General}
Let $\hat{K}$ be defined as in Proposition~\ref{Prop:dscaleOverall} and let $\hat{f}_n$ be defined as in Proposition~\ref{Prop:KEstimatedGeneralRisk}. Then
  \[
  \mathbb{E}_{f_0}d_{\mathrm{H}}^2 (\hat{f}_n, f_0) \lesssim_{p, r_0}
  n^{-\frac{1}{p+1}} \log^3 (en).
\]
Moreover, if $f_0 \in \mathcal{F}^{K,\mu}_p$ is of the form $f_0(\cdot) = e^{\phi_0(\| \cdot - \mu \|_K)}$ for some $\phi_0 \in \Phi$ such that $\phi_0'$ is absolutely continuous and such that $\inf_{r \in [0,\infty)} \phi_0''(r) \geq - D_0$ for some $D_0 > 0$, then 
\[
  \mathbb{E}_{f_0}d_{\mathrm{H}}^2 (\hat{f}_n, f_0) \lesssim_{p,r_0}
  n^{-\frac{2}{p+1}} \log^6 (en).
\]
\end{corollary}

\begin{remark}
It is important to observe that the computation of the estimator $\hat{f}_n$ is scalable for large $n$ and $p$. Computing $\hat{K}$ requires at most $O(pn^2 \log n)$ operations, since we represent $\hat{K}$ implicitly in terms of its hull vertices and have no need to enumerate its facets. For any $x \in \mathbb{R}^p$, we may then compute $\| x \|_{\hat{K}}$ through a straightforward linear programme of at most $n+1$ variables; see~\eqref{Eqn:LinearProgram}. Thus, it is also fast to compute $\hat{\phi}_n$ and to evaluate $\hat{f}_n(\cdot) = e^{\hat{\phi}_n(\| \cdot \|_{\hat{K}})}$ at any $x \in \mathbb{R}^p$. 
\end{remark}

\section{Algorithm}
\label{Sec:Algorithm}

In this section, we assume $\hat{K} \in \mathcal{K}$, $\hat{\mu} \in \mathbb{R}^p$ and data $X_1, \ldots, X_n \in \mathbb{R}^p$ with empirical distribution $\mathbb{P}_n$ are given, and describe an efficient algorithm for computing the $\hat{K}$-homothetic log-concave projection of $\mathbb{P}_n$.  Fixing $x \in \mathbb{R}^p$, we first note that in many cases of interest, the Minkowski functional $\| x - \hat{\mu}\|_{\hat{K}}$ is easy to compute when $\hat{K}$ is constructed using the estimation schemes described in Section~\ref{Sec:KEstimated}.  In particular, if $\hat{K}$ is of the form $\hat{\Sigma}^{1/2} K_0$, where $K_0$ is a known convex body whose Minkowski functional is simple to compute, and where $\hat{\Sigma} \in \mathbb{S}^{p \times p}$, as is the case in Section~\ref{Sec:Affine}, then $\| x - \hat{\mu} \|_{\hat{K}} = \| \hat{\Sigma}^{-1/2} (x - \hat{\mu})\|_{K_0}$, so it may also be computed easily.  As another example, if $\hat{K}$ is the convex hull of a set of points $\{Y_1, \ldots, Y_M\}$ in $\mathbb{R}^p$, as is the case in Section~\ref{Sec:KNonparametric}, then $\| x - \hat{\mu} \|_{\hat{K}}$ is the solution to the following linear programme:
\begin{align}
  \max_{(u_1,\ldots,u_{M+1}) \in [0, \infty)^{M+1}} u_{M+1} \quad  \textrm{s.t. }  \sum_{m=1}^M u_m Y_m = u_{M+1} (x - \hat{\mu}) \,\,\, \textrm{and } 
                 \sum_{m=1}^M u_m \leq 1. \label{Eqn:LinearProgram}
\end{align}

Let $Z_i := \| X_i - \hat{\mu}\|_{\hat{K}}$ for $i\in [n]$, and let $\mathbb{Q}_n$ denote the empirical distribution of $Z_1,\ldots,Z_n$.  Proposition~\ref{Prop:ProjectionExistence} shows that, provided at least one of $Z_1,\ldots,Z_n$ is non-zero, the function $\hat{\phi}_n := \phi_{\hat{K}}^*(\mathbb{Q}_n)$ is well-defined, and we can then set $\hat{f}_n(\cdot) := e^{\hat{\phi}_n(\|\cdot-\hat{\mu}\|_{\hat{K}})}$.  Our aim is therefore to provide an algorithm for computing $\hat{\phi}_n$.

Let $\bar{\Phi}$ denote the set of $\phi \in \Phi$ with the property that $\phi$ is constant on the interval $[0,Z_{(1)}]$ and affine on the intervals $[Z_{(i-1)},Z_{(i)}]$ for $i=2,\ldots,n$, with $\phi(r) = -\infty$ for $r > Z_{(n)}$.  Observe that if we fix $\phi \in \Phi$, and $\bar{\phi} \in \bar{\Phi}$ be such that $\bar{\phi}(0) = \phi(0)$ and $\bar{\phi}(Z_i) = \phi(Z_i)$ for all $i=1,\ldots,n$. Then by concavity of $\phi$, we have $\phi(r) \geq \bar{\phi}(r)$ for all $r \in [0, \infty)$.  Hence
\begin{align}
\label{Eq:phibar}
  L(\phi, \mathbb{Q}_n) &= \frac{1}{n}\sum_{i=1}^n \phi(Z_i) - p\lambda_p(\hat{K}) \int_0^\infty r^{p-1} e^{\phi(r)} \, dr +1 \nonumber \\
  &\leq
    \frac{1}{n}\sum_{i=1}^n \bar{\phi}(Z_i) - p\lambda_p(\hat{K}) \int_0^\infty  r^{p-1} e^{\bar{\phi}(r)} \, dr +1 = L(\bar{\phi},\mathbb{Q}_n).
\end{align}

The volume $\lambda_p(\hat{K})$, in the case where $\hat{K} = \hat{\Sigma}^{1/2} K_0$ where the volume of $K_0$ is known and $\hat{\Sigma} \in \mathbb{S}^{p \times p}$, takes the simple form $\lambda_p(K_0) \mathrm{det}(\hat{\Sigma})^{1/2}$. More generally, $\lambda_p(\hat{K})$ can be computed efficiently if, for any $x \in \mathbb{R}^p$, the query of whether or not $x \in \hat{K}$ can be evaluated efficiently. When $\hat{K}$ is the convex hull of $M$ points in $\mathbb{R}^p$ for example, we may evaluate a query by solving the linear programme~\eqref{Eqn:LinearProgram} and then checking whether the solution $u_{M+1}$ is less than or equal to 1. If we let $q$ denote the number of queries made by an algorithm, then \citet{kannan1997random} give a Markov Chain Monte Carlo algorithm whose query complexity is bounded by $O(q^5)$ up to polylogarithmic factors. In fact, the computation of the volume of a convex body is a deep and beautiful problem that had been studied intensely by the theoretical computer scientists since the seminal paper of \citet{dyer1991random}, who first gave a polynomial time algorithm for the problem. It is one of few instances in computer science where all deterministic algorithms are provably intractable but efficient \emph{randomised} algorithms exist.  We refer readers to \citet{simonovits2003compute} for an accessible tutorial.

We now assume for simplicity of exposition that $Z_1,\ldots,Z_n$ are distinct.  The more general case can be treated similarly by assigning appropriate weights to duplicated points.  Any $\phi \in \bar{\Phi}$ can be identified with $\mathbf{\phi} = (\phi_1,\ldots,\phi_n)^\top \in \mathbb{R}^n$ given by $\phi_i := \phi(Z_i)$ for $i \in [n]$.  For $i\in [n-1]$, let $\delta_i := Z_{(i+1)} - Z_{(i)}$.  Define $v_1 = (v_{1,j})_{j=1}^n \in \mathbb{R}^n$ to have two non-zero entries, namely $v_{1,1} := -1$, $v_{1,2} := 1$.  Further, for $i=2,\ldots,n-1$, let $v_i = (v_{i,j})_{j=1}^n \in \mathbb{R}^n$ have three non-zero entries, namely
\[
v_{i,i-1} := \frac{1}{\delta_{i-1}}, \quad v_{i,i} := -\frac{1}{\delta_i} - \frac{1}{\delta_{i-1}} \quad \text{and} \quad v_{i,i+1} := \frac{1}{\delta_i}.
\]
Finally, let $\bar{\mathbf{\Phi}}_n := \bigl\{\phi \in \mathbb{R}^n: v_i^{\top} \phi \leq 0 \text{ for } i=1,\ldots, n-1\bigr\}$.  By~\eqref{Eq:phibar}, we see that it suffices to compute $\phi^* = (\phi_1^*,\ldots,\phi_n^*)^\top \in \argmax_{\phi \in \bar{\mathbf{\Phi}}_n} F(\phi)$, where
\begin{equation}
\label{Eq:Optimization}
F(\phi) := \frac{1}{n} \sum_{i=1}^n \phi_i - \frac{1}{p}Z_{(1)}^p - p\lambda_p(\hat{K})\sum_{i=1}^{n-1} \int_{Z_{(i)}}^{Z_{(i+1)}}
   r^{p-1} \exp \biggl( \frac{Z_{(i+1)} - r}{\delta_i} \phi_i +
    \frac{r - Z_{(i)}}{\delta_i} \phi_{i+1} \biggr) \,dr + 1. 
\end{equation}
This is a finite-dimensional convex optimisation problem with linear inequality constraints.  We propose an active set algorithm for the optimisation of~\eqref{Eq:Optimization}, a variant of the algorithm used in \cite{dumbgen2007active} to compute the ordinary univariate log-concave MLE. For $\phi \in \bar{\Phi}_n$, we define $A(\phi) := \bigl\{i \in [n-1]:v_i^\top\phi = 0 \bigr\}$ to be the set of `active' constraints.  Note that this is the complement in $\{1,\ldots,n\}$ of the set of `knots' of $\phi$.
Given a set $A \subseteq [n-1]$, we define $V(A) := \bigl \{ \phi \in \mathbb{R}^n : v_i^\top\phi = 0, \, \forall \, i \in A\bigr\}$, and 
  \begin{equation}
    \label{Eq:OptimizationFixedA}
    V^*(A) := \argmax_{\phi \in V(A)} F(\phi).
  \end{equation}
Here, the maximiser is unique because $F(\cdot )$ is strictly concave on $\mathbb{R}^n$ with $F(\phi) \rightarrow -\infty$ as $\|\phi\| \rightarrow \infty$. It is convenient to define, for $i\in [n-1]$, vectors $b_i = (b_{i,j})_{j=1}^n \in \mathbb{R}^n$ by
\[
b_{i,j} := - \sum_{k=i}^{j-1} \delta_k,
\]
where, as usual, we interpret an empty sum as $0$, and also define $b_n := \mathbf{1}_n \in \mathbb{R}^n$, the all-one vector.  It follows from this definition that $b_i^{\top} v_i = -1$ for $i\in [n-1]$ and $ b_i^\top v_j = 0$ for all $i \in [n]$ and $j \in [n-1]$ with $i \neq j$.  Finally, given $\phi \in \bar{\mathbf{\Phi}}_n$ and $\phi' \in \mathbb{R}^n$, we define 
\[
t(\phi, \phi') := \max\biggl\{\frac{v_i^\top \phi'}{ v_i^\top(\phi' - \phi)}:i \in [n-1] \setminus A(\phi)\, , \, v_i^\top\phi' > 0\biggr\}.
\] 
We are now in a position to present the full algorithm; see Algorithm~\ref{Alg:ActiveSet}.  It is guaranteed to terminate in finitely many steps with the exact solution.   

\begin{algorithm}
\caption{Computing the $\hat{K}$-homothetic log-concave MLE}
\label{Alg:ActiveSet}
\textbf{Input}: $\hat{K} \in \mathcal{K}$, $\hat{\mu} \in \mathbb{R}^p$, $X_1,\ldots, X_n \in \mathbb{R}^p$. 
\begin{algorithmic}[1]
\State $Z_i \leftarrow \| X_i-\hat{\mu}\|_{\hat{K}}$ for all $i\in [n]$.
\State $A \leftarrow [n-1]$.
\State $\phi \leftarrow V^*(A)$.
\While{ $\max_{i=1,\ldots,n} b_i^\top \nabla F(\phi) \geq 0$ }
 \State $i^* \leftarrow \argmax_{i=1,\ldots,n} b_i^\top \nabla F(\phi) $
\State $\phi' \leftarrow V^*(A \setminus \{i^*\})$
 \While{ $\phi' \notin \bar{\Phi}_n$}
    \State $\phi \leftarrow t(\phi, \phi') \phi + \{1 - t(\phi, \phi')\} \phi'$ 
    \State $A \leftarrow A(\phi)$
    \State $\phi' \leftarrow V^*(A)$
 \EndWhile
 \State $\phi \leftarrow \phi'$
 \State $A \leftarrow A(\phi)$
\EndWhile   
\end{algorithmic}
\textbf{Output}: $\phi \in \bar{\Phi}_n$.
\end{algorithm}


We complete this section by providing further detail on how to solve the optimisation problem in~\eqref{Eq:OptimizationFixedA}. Given the active set $A \subseteq [n-1]$, let us define $I := [n] \setminus A$. We index the elements of $I$ by $i_1 < \ldots < i_T$ where $T := |I|$.  Given $v \in \mathbb{R}^{(n-1) \times n}$, we also write $v_A$ for the matrix in $\mathbb{R}^{|A| \times n}$ obtained by extracting the rows of $v$ with indices in $A$.  Observe that the set $\{ \phi = (\phi_1,\ldots,\phi_n)^\top \in \mathbb{R}^n \,:\, v_A \phi = 0 \}$ is the subspace of $\mathbb{R}^n$ where for $j < i_1$, we have $\phi_j = \phi_{i_1}$, and for $j \in \{i_t+1,\ldots,i_{t+1}-1\}$, we have 
\[
\phi_j = \frac{Z_{(i_{t+1})} - Z_{(j)}}{Z_{(i_{t+1})} - Z_{(i_t)}} \phi_{i_t} +
         \frac{Z_{(j)} - Z_{(i_t)}}{Z_{(i_{t+1})} - Z_{(i_t)}} \phi_{i_{t+1}}.
\]
It follows  we can solve the optimisation problem~\eqref{Eq:OptimizationFixedA} by solving instead an unconstrained optimisation over $T$ variables, i.e.~by computing
\begin{align*}
\max_{\phi_I \in \mathbb{R}^T}  \frac{1}{n} \biggl( i_1 &\phi_{i_1} + 
          \sum_{t=1}^{T-1} \sum_{j=i_t + 1}^{i_{t+1}}   
         \frac{Z_{(i_{t+1})} - Z_{(j)}}{Z_{(i_{t+1})} - Z_{(i_t)}} \phi_{i_t} +
         \frac{Z_{(j)} - Z_{(i_t)}}{Z_{(i_{t+1})} - Z_{(i_t)}} \phi_{i_{t+1}} \biggr) - \lambda_p(\hat{K})e^{\phi_{i_1}} Z_{(i_1)}^p \\
     &  - p\lambda_p(\hat{K})
       \sum_{t=1}^{T-1} \int_{Z_{(i_t)}}^{Z_{(i_{t+1})}} r^{p-1} \exp
       \biggl( \frac{Z_{(i_{t+1})} -r}{Z_{(i_{t+1})} -Z_{(i_t)}} \phi_i + \frac{r -Z_{(i_t)}}{Z_{(i_{t+1})} -Z_{(i_t)}} \phi_{i_{t+1}}\biggr) \, dr.
\end{align*}
We solve this latter problem via Newton's method.

\section{Empirical performance}
\label{Sec:Simulations}

We perform three sets of simulation studies. In the first set, reported in Figure~\ref{Fig:KKnown}, we choose $p = 100$ suppose that $K = B_p(0,1)$ and $\mu = 0$ are known.  We generate $X_1, \ldots, X_n \stackrel{\mathrm{iid}}{\sim} f_0$ where we take $f_0(\cdot) \propto e^{- \|\cdot \|^2_K/2}$, $f_0(\cdot) \propto \mathbbm{1}_{\{\|\cdot\|_K \leq p\}}$ and $f_0(\cdot) \propto e^{-\| \cdot \|_K}$ in settings~(a),~(b) and~(c) respectively.  We then compute the homothetic log-concave MLE $\hat{f}_n$ and report the average squared Hellinger errors $d^2_{\mathrm{H}}(\hat{f}_n, f_0)$ over 50 repetitions and with $n \in \{2000,4000,6000,8000\}$ in the curve labelled ``HLC'' Figure~\ref{Fig:KKnown}.  For comparison, we also present the corresponding results with $p=100,000$ in the curves labelled ``HLC(p=100k)''.  The simulation results are in line with Theorem~\ref{Thm:KKnownWorstCase}, which gives a bound on $d_X^2(\hat{f}_n,f_0)$ that is independent of~$p$.

\begin{figure}[hbt]
  \begin{subfigure}{0.3\textwidth}
    \centering
    \includegraphics[scale=.28]{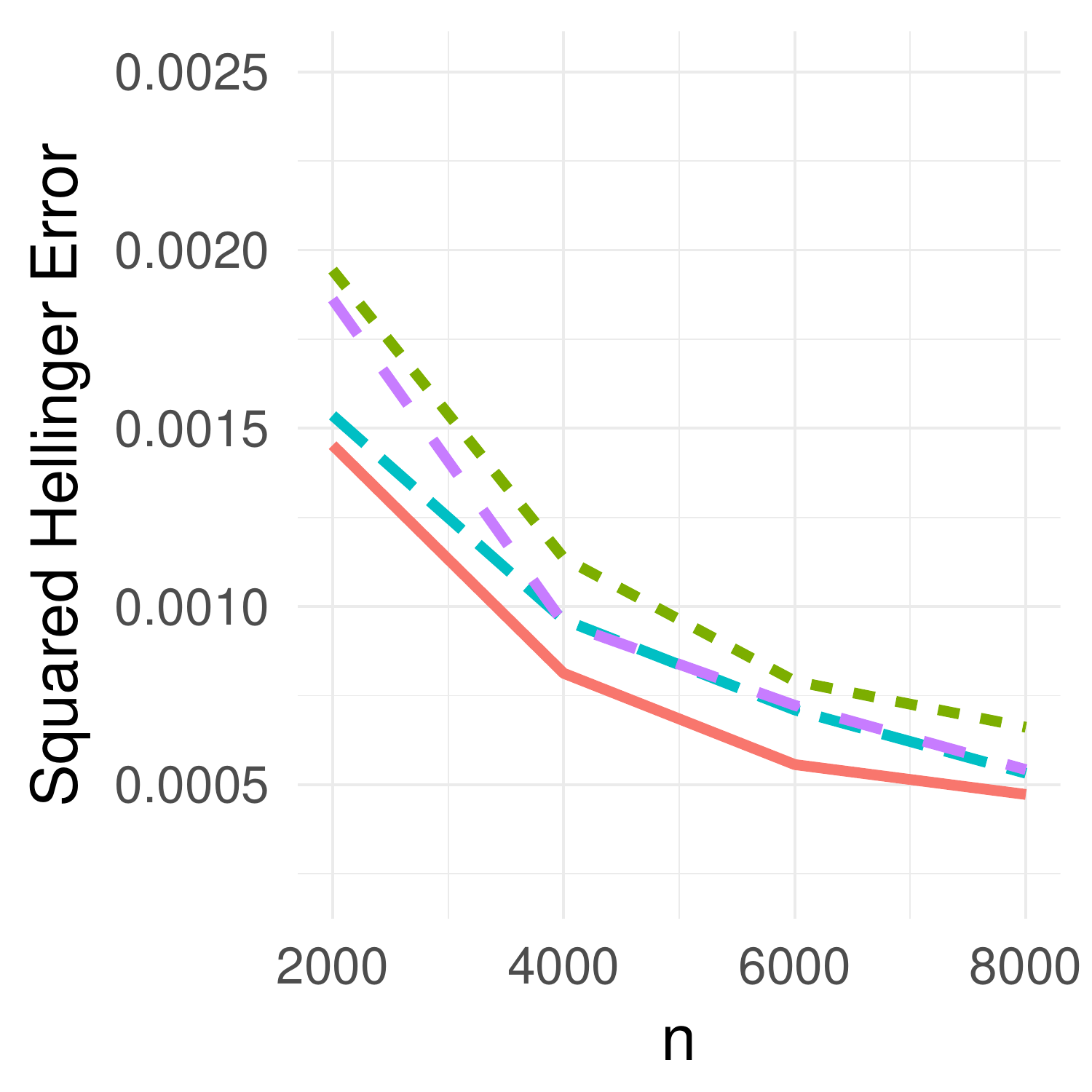}
    \caption{}
    \label{Fig:KKnown1}
  \end{subfigure}
  \begin{subfigure}{0.3\textwidth}
    \centering
    \includegraphics[scale=0.28]{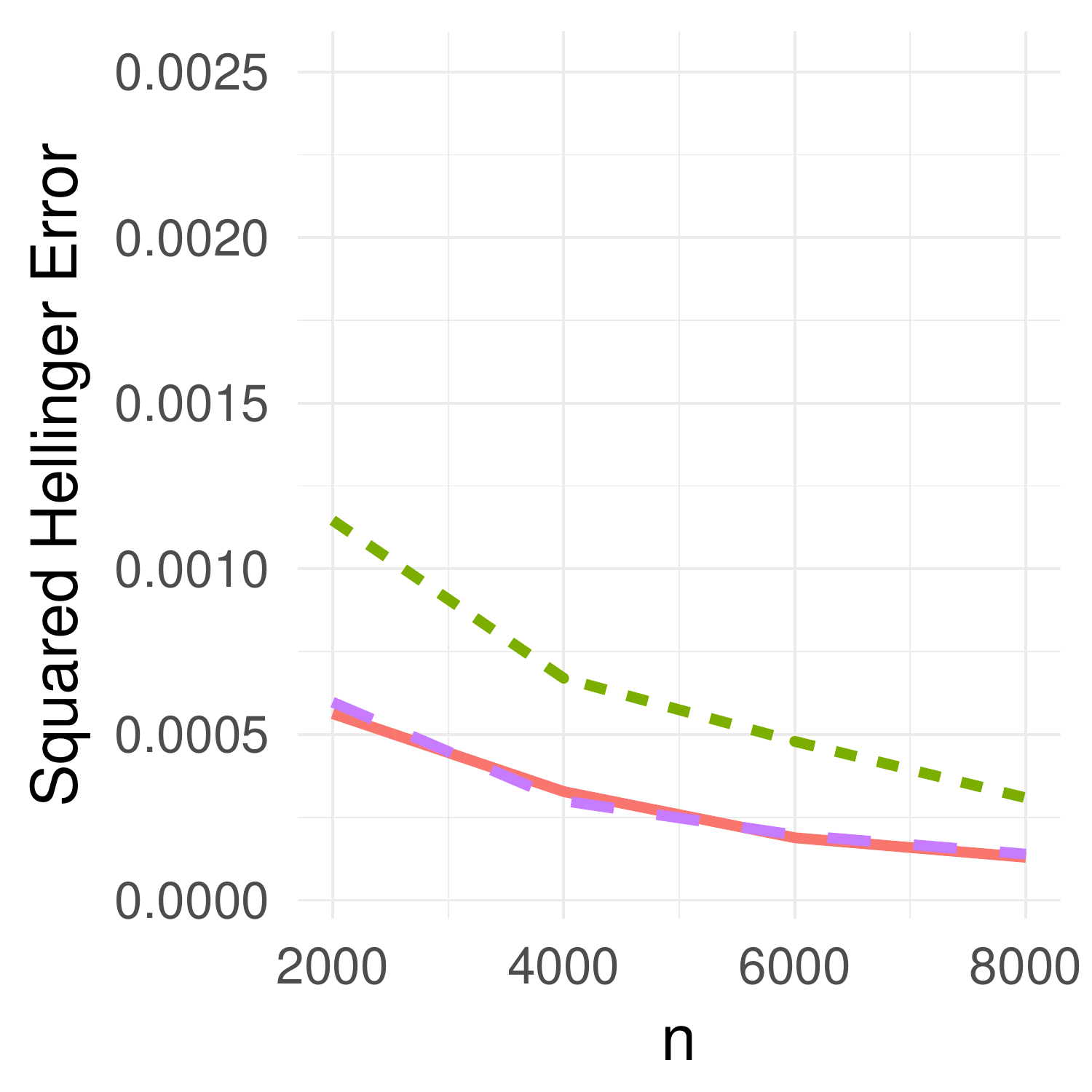}
    \caption{}
    \label{Fig:KKnown2}
  \end{subfigure}
    \begin{subfigure}{0.35\textwidth}
    \centering
    \includegraphics[scale=0.28]{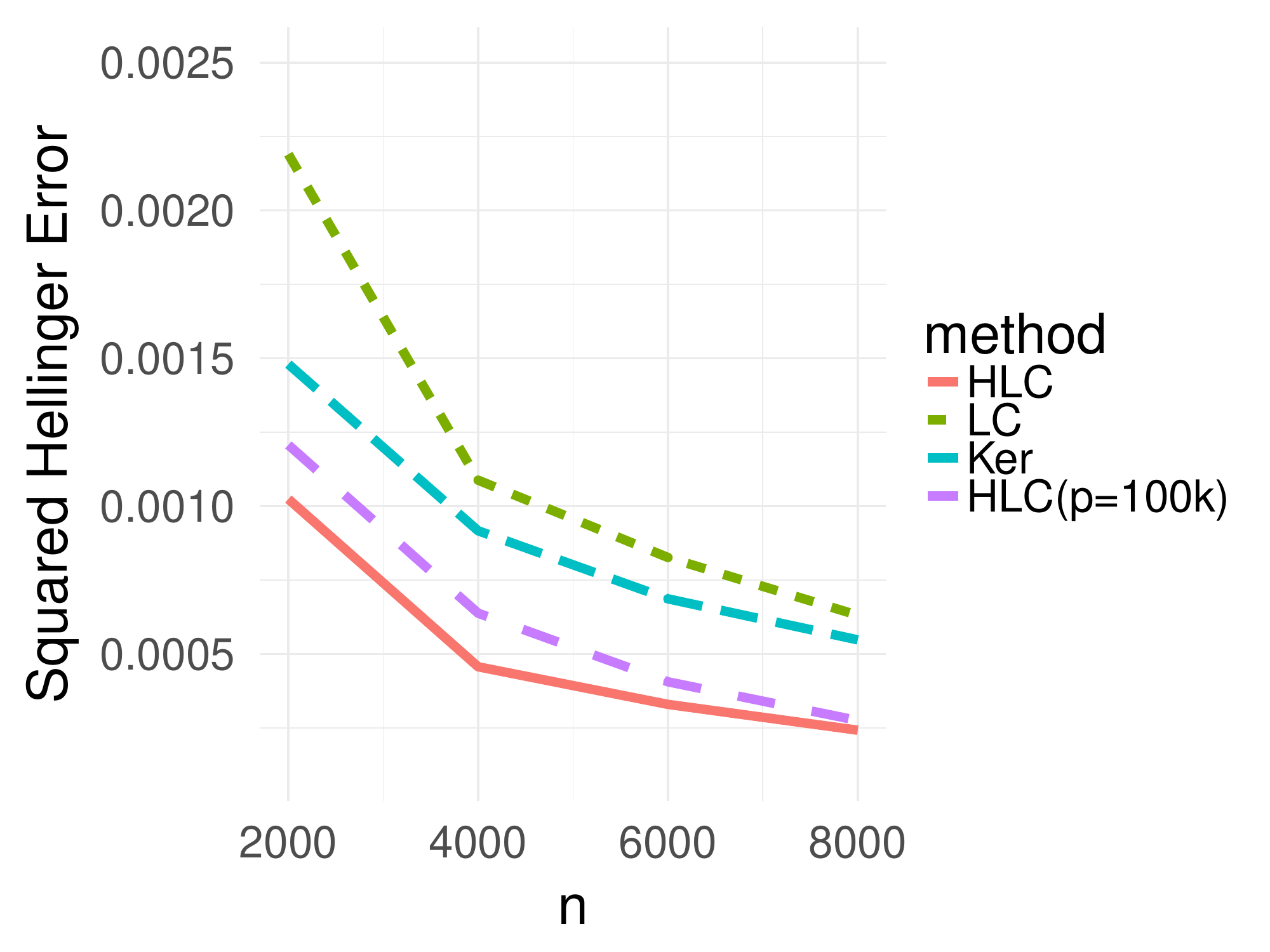}
    \caption{}
    \label{Fig:KKnown3}
  \end{subfigure}
  \caption{\label{Fig:KKnown}Average squared Hellinger errors of different methods when $K = B_p(0,1)$ is known.  (a) $f_0(\cdot) \propto e^{- \|\cdot \|^2_K/2}$; (b) $f_0(\cdot) \propto \mathbbm{1}_{\{\|\cdot\|_K \leq p\}}$; (c) $f_0(\cdot) \propto e^{-\| \cdot \|_K}$.  For the purple line, $p=100,000$, while in other cases, $p=100$.}
\end{figure}

We also compare the $K$-homothetic log-concave MLE against two alternative methods. In the first of these methods, we write $Z_i = \| X_i \|_K$ for $i\in [n]$, apply the ordinary univariate log-concave MLE to $Z_1,\ldots,Z_n$ to obtain a density $\hat{h}^{\mathrm{LC}}_n$ and then estimate $f_0$ by $\hat{f}^{\mathrm{LC}}_n$, where
\begin{equation}
  \label{Eq:fntilde}
  \hat{f}^{\mathrm{LC}}_n(x) := \frac{\hat{h}^{\mathrm{LC}}_n(\|x \|_K)}{p \lambda_p(K)\|x\|_K^{p-1}}.
  \end{equation}
      We compute the squared Hellinger errors $d^2_{\mathrm{H}}(\hat{f}^{\mathrm{LC}}_n, f_0)$ and report them in the curve labelled $\mathrm{LC}$ in Figure~\ref{Fig:KKnown}.  In fact, besides the improved empirical performance of $\hat{f}_n$ observed in Figure~\ref{Fig:KKnown}, we argue that $\hat{f}_n$ has several advantages over $\hat{f}^{\mathrm{LC}}_n$ in this context, and list these in roughly decreasing order of importance:
\begin{enumerate}
\item The estimator $\hat{f}^{\mathrm{LC}}_n$ is inconsistent at $x=0$.  Indeed $\hat{h}^{\mathrm{LC}}_n(x) = 0$ whenever $\|x\|_K < \min_i Z_i$, and the division by $\|x\|_K^{p-1}$ in~\eqref{Eq:fntilde} means that the estimator behaves poorly for small $\|x\|_K$; see~Figure~\ref{Fig:SSLC}.  By contrast, $\hat{f}_n$ is uniformly consistent over compact sets contained in the interior of the support of $f_0$ (Proposition~\ref{Prop:Continuity});
\item As mentioned in Section~\ref{Sec:KknownAdaptation}, the estimator $\hat{f}_n$ attains faster rates of convergence when the true density has a simple structure;
\item The estimator $\hat{f}_n$ takes values in the relevant class $\mathcal{F}_p^{K}$, whereas $\hat{f}^{\mathrm{LC}}_n$ does not;
\item The estimator $\hat{f}_n$ exists in slightly greater generality than $\hat{f}^{\mathrm{LC}}_n$ (cf.~the remark following Proposition~\ref{Prop:ProjectionExistence}(iv)).
\end{enumerate}
\begin{figure}[ht!]
  \centering
  \begin{subfigure}[b]{0.4\linewidth}
    \centering
    \includegraphics[scale=0.31, trim=0 0 0 0]{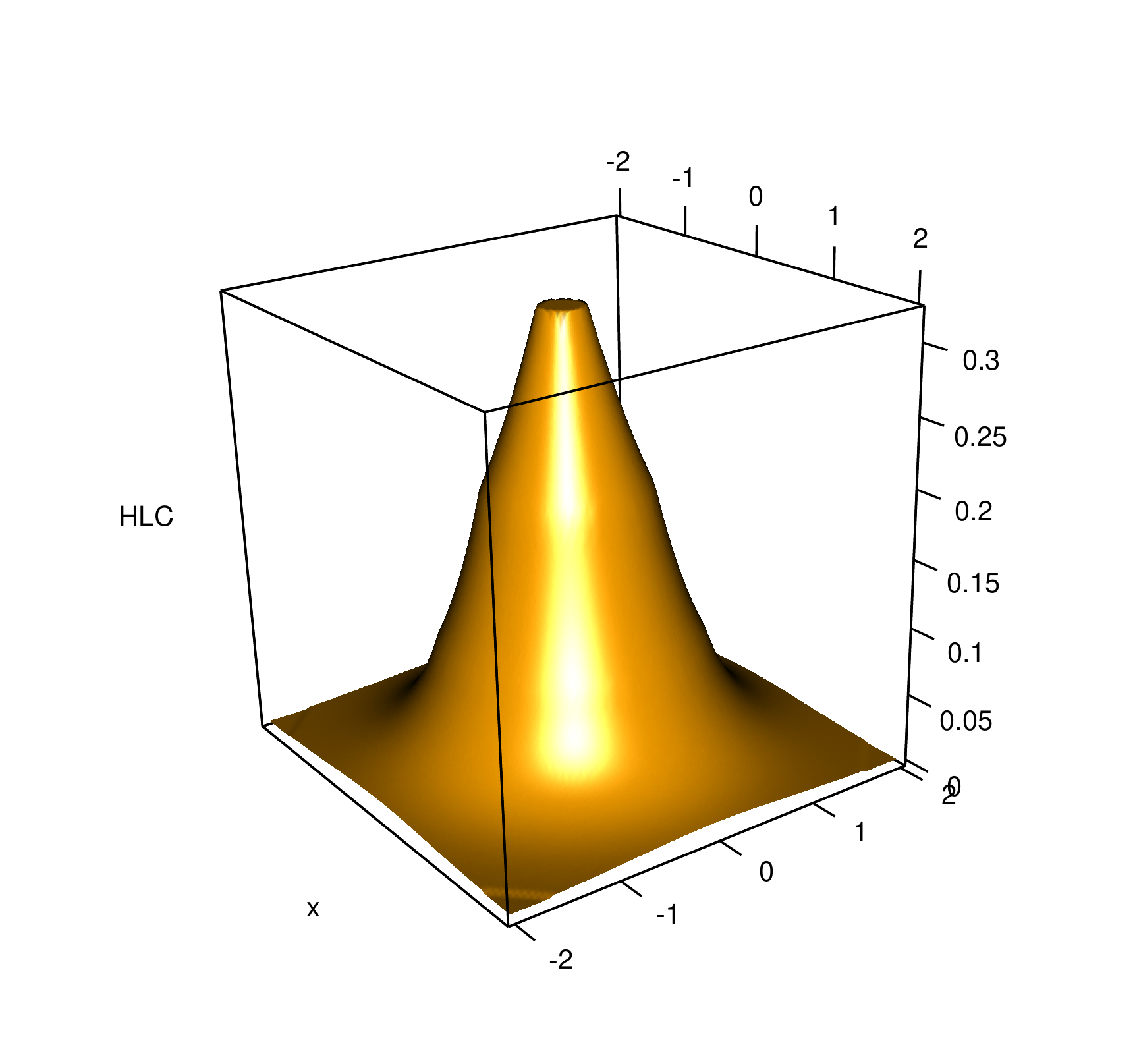}
  \end{subfigure}
  \begin{subfigure}[b]{0.4\linewidth}
    \centering
    \includegraphics[scale=0.34, trim=0 0 0 0]{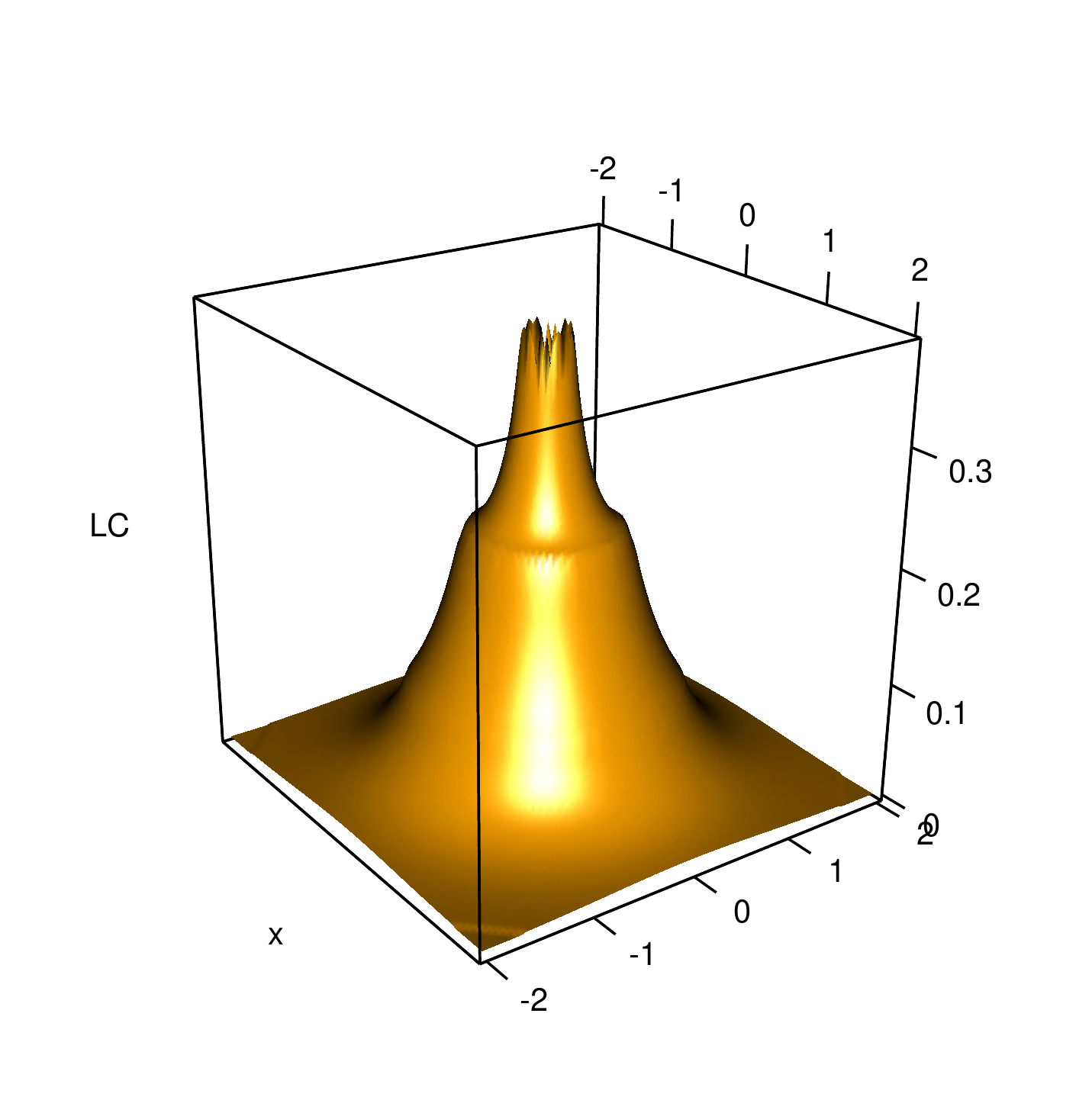}
  \end{subfigure}
  \begin{subfigure}[b]{0.4\linewidth}
    \centering
    \includegraphics[scale=0.26]{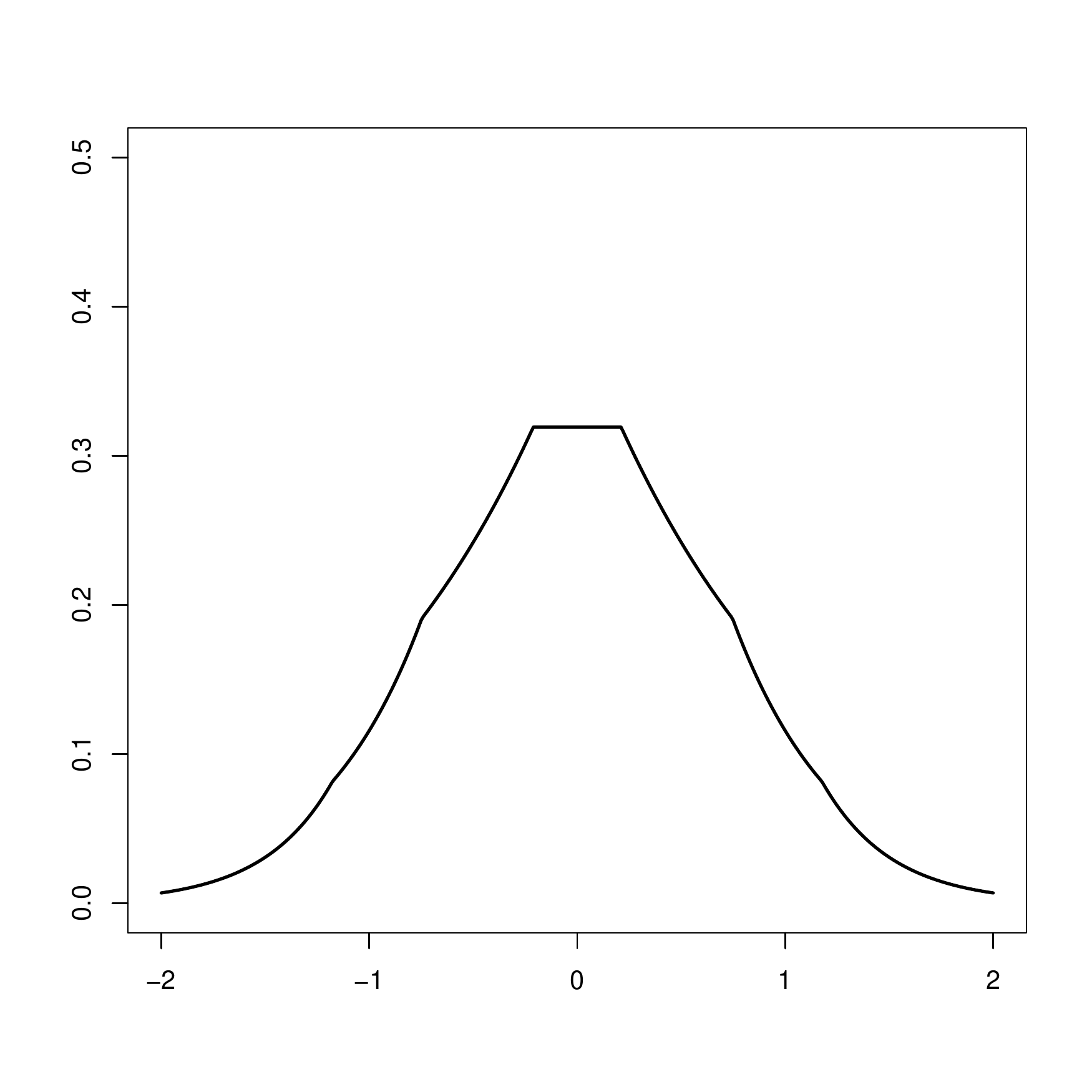}
  \end{subfigure}
  \begin{subfigure}[b]{0.4\linewidth}
    \centering
    \includegraphics[scale=0.26]{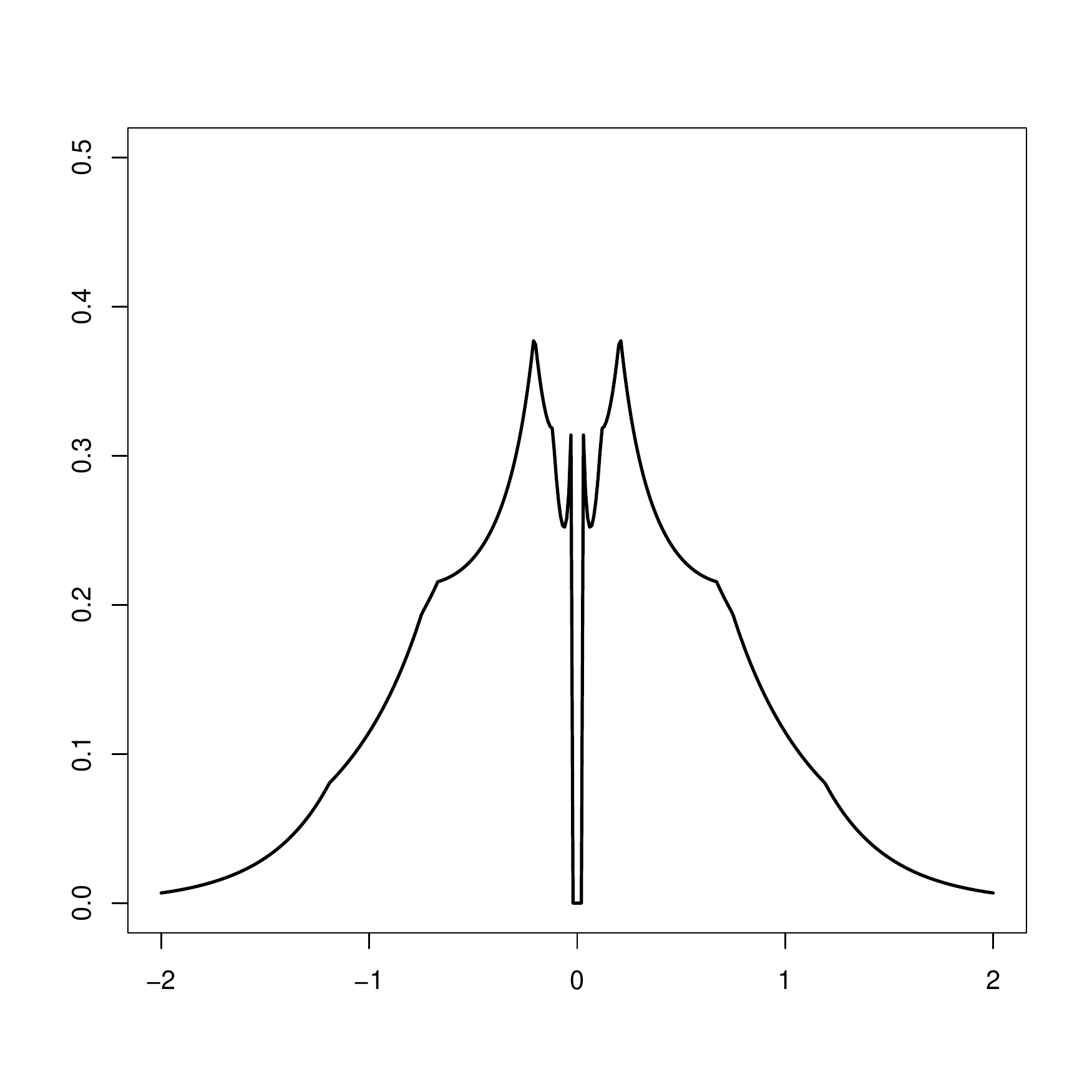}
  \end{subfigure}
\caption{\label{Fig:SSLC}A comparison of the $K$-homothetic log-concave MLE $\hat{f}_n$ with $K = B_p(0,1)$ (left) and the ordinary log-concave MLE $\hat{f}^{\mathrm{LC}}_n$ (right) based on a sample of size $n=1000$ from a standard bivariate normal distribution.  The top plots give the bivariate density estimates, while the bottom ones present the corresponding estimates $\hat{f}_n(|\cdot|)$ and $\hat{f}^{\mathrm{LC}}_n(|\cdot|)$.}
\end{figure}
In the second competing method, we apply a kernel density estimator (with the default settings of the \texttt{density} function in \texttt{R}) to $Z_1, \ldots, Z_n$ to obtain a density $\hat{h}^{\mathrm{ker}}_n$, and then estimate $f_0$ by $\hat{f}^{\mathrm{ker}}_n$, where $\hat{f}^{\mathrm{ker}}_n(x) := \hat{h}^{\mathrm{ker}}_n(\|x \|_K)/\bigl\{p \lambda_p(K)\|x\|_K^{p-1}\bigr\}$. We compute the squared Hellinger errors $d^2_{\mathrm{H}}(\hat{f}^{\mathrm{ker}}_n, f_0)$ and report them in the curve labelled `$\mathrm{ker}$' in Figure~\ref{Fig:KKnown}.  The squared Hellinger errors for the kernel density estimator $\hat{f}^{\mathrm{ker}}$ do not appear in Figure~\ref{Fig:KKnown2} because the errors are greater than 0.007 and therefore much larger than those of $\hat{f}_n$ and~$\hat{f}^{\mathrm{LC}}_n$. 

\begin{figure}[hbt]
  \begin{subfigure}{0.49\textwidth}
    \centering
    \includegraphics[scale=0.35]{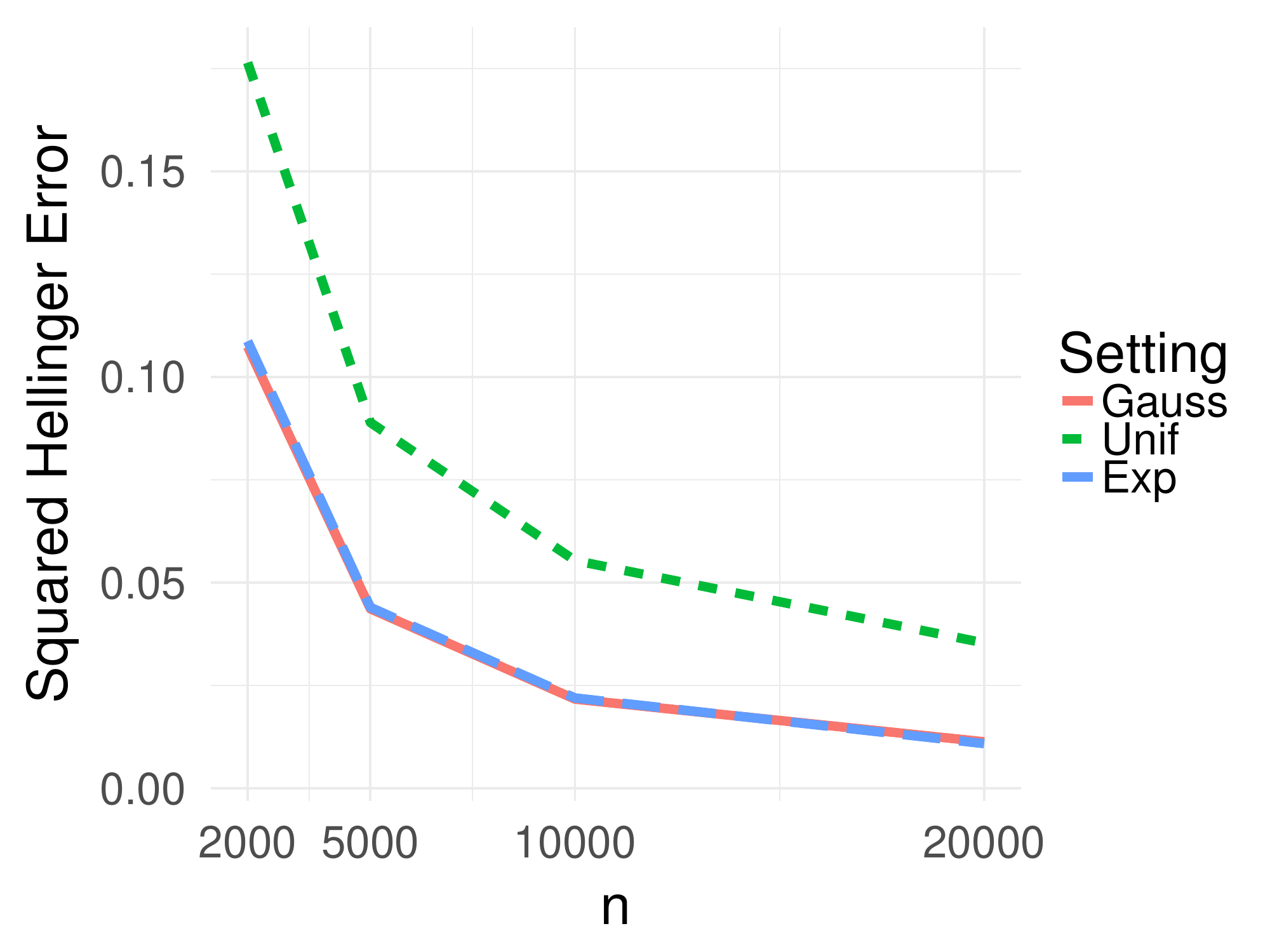}
    \caption{}
    \label{Fig:KAffineN1}
  \end{subfigure}
  \begin{subfigure}{0.49\textwidth}
    \centering
    \includegraphics[scale=0.35]{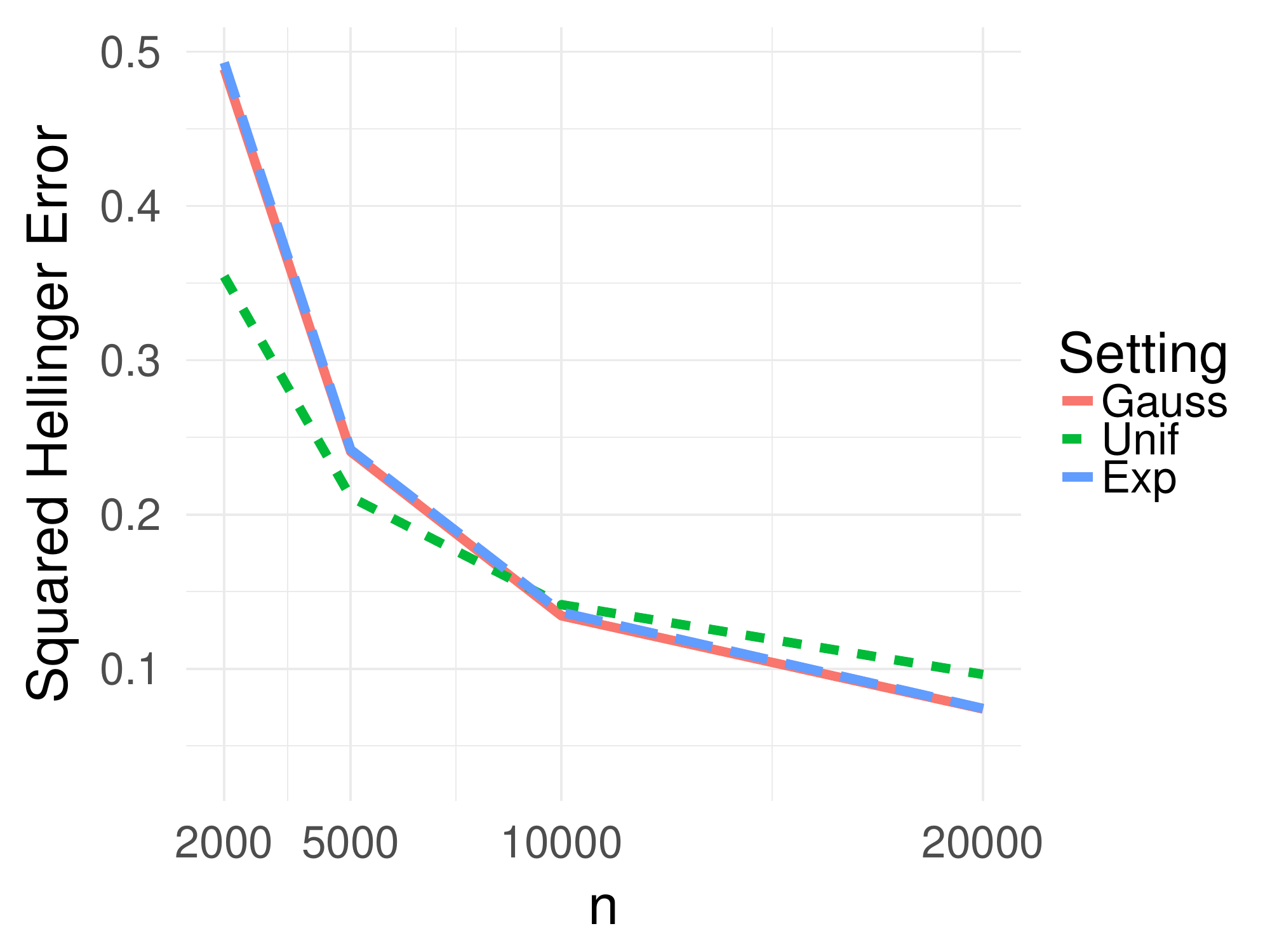}
    \caption{}
    \label{Fig:KAffineN2}
  \end{subfigure}
  \begin{subfigure}{0.49\textwidth}
    \centering
    \includegraphics[scale=0.35]{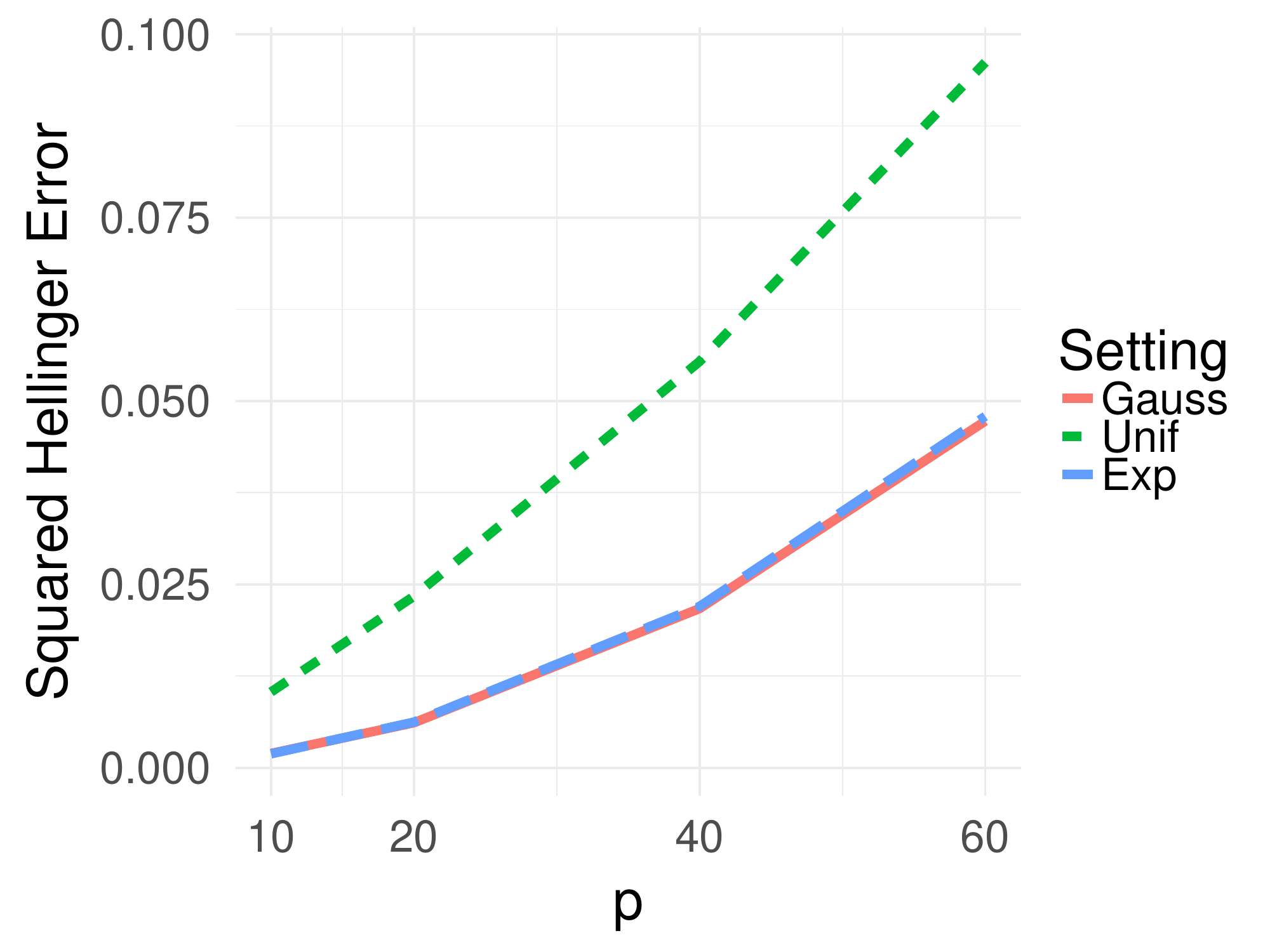}
    \caption{}
    \label{Fig:KAffineP1}
  \end{subfigure}
  \begin{subfigure}{0.49\textwidth}
    \centering
    \includegraphics[scale=0.35]{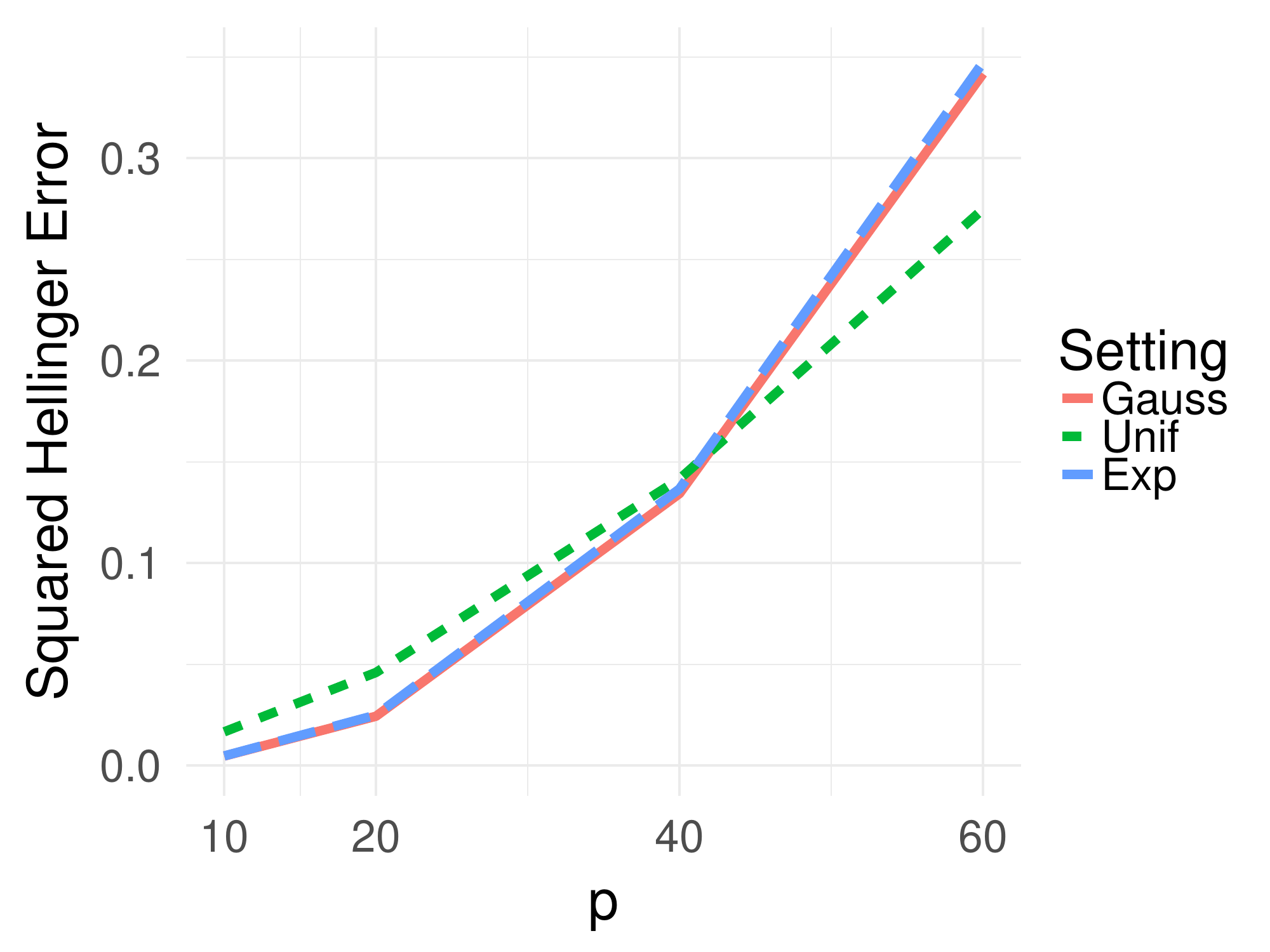}
    \caption{}
    \label{Fig:KAffineP2}
  \end{subfigure}
  \caption{\label{Fig:KAffineN}Average squared Hellinger errors $d_{\mathrm{H}}^2(\hat{f}_n,f_0)$ with $K$ known up to a positive definite transformation of $K_0 = B_p(0,1)$ in (a) and (c), and $K = [-1,1]^p$ in (b) and (d).  For $\mathrm{Gauss}$, we took $f_0(\cdot) \propto e^{- \|\cdot \|^2_K/2}$; for $\mathrm{Unif}$, we took $f_0(\cdot) \propto \mathbbm{1}_{\{\|\cdot\|_K \leq p\}}$; for $\mathrm{Exp}$, we took $f_0(\cdot) \propto e^{-\| \cdot \|_K}$.  In (a) and (b), we fixed $p=40$, while in (c) and (d), we fixed $n =10000$.}
\end{figure}

In the second set of simulations, reported in Figure~\ref{Fig:KAffineN}, we consider the semiparametric setting where $K = \Sigma_0^{1/2} K_0$ for some known $K_0 \in \mathcal{K}$ and unknown $\Sigma_0 \in \mathbb{S}^{p \times p}$. We estimate~$\Sigma_0$ up to a scaling factor by the empirical covariance matrix $\hat{\Sigma}$, take $\hat{K}$ to be $\hat{\Sigma}^{1/2} K_0$, and estimate the centering vector $\mu$ (taken to be 0) by the empirical mean vector $\hat{\mu}$. We then construct $\tilde{Z}_i := \| X_i - \hat{\mu} \|_{\hat{K}}$ for $i\in[n]$, compute $\hat{\phi}_n := \argmax_{\phi \in \Phi} n^{-1}\sum_{i=1}^n \phi(\tilde{Z}_i) - p\lambda_p(\hat{K}) \int_0^\infty r^{p-1} e^{\phi(r)} \, dr$, and construct the density estimate $\hat{f}_n(\cdot) = e^{\hat{\phi}_n(\cdot)}$.  In all cases, we generate $\Sigma_0$ as $UDU^{\top}$, where $U$ is generated according to Haar measure on the set of orthogonal $p \times p$ real matrices, and where $D \in \mathbb{R}^{p \times p}$ is a diagonal matrix whose $j$th diagonal entry is $1.2^j$.  In Figures~\ref{Fig:KAffineN1} and~\ref{Fig:KAffineP1}, we take $K_0 = B_p(0,1)$, while in Figures~\ref{Fig:KAffineN2} and~\ref{Fig:KAffineP2}, we take $K_0 = [-1,1]^p$.  In Figures~\ref{Fig:KAffineN1} and~\ref{Fig:KAffineN2}, we fix $p=40$ and report the squared Hellinger errors $d_{\mathrm{H}}^2(\hat{f}_n, f_0)$ for $n \in \{2000,5000,10000,20000\}$, while in Figures~\ref{Fig:KAffineP1} and~\ref{Fig:KAffineP2}, we fix $n=10000$ and report the corresponding squared errors with $p \in \{10,20,40,60\}$.  In the settings of Figures~\ref{Fig:KAffineN1} and~\ref{Fig:KAffineP1}, we see the advantage conferred by the smoothness of $\phi_0$, in line with our theoretical guarantees from Corollary~\ref{Cor:Ellipse}.  On the other hand, when $K_0 = [-1,1]^p$ in Figures~\ref{Fig:KAffineN2} and~\ref{Fig:KAffineP2}, $r_0$ is much larger than in the $K_0 = B_p(0,1)$ case (it is equal to $p^{1/2}$ instead of $1$), and this makes the problem significantly harder, which is again in agreement with Corollary~\ref{Cor:Ellipse}.


Finally, in the third set of simulations, given in Figure~\ref{Fig:KUnknown}, we take $\mu = 0$ to be known and estimate $K$ nonparametrically using Algorithm~\ref{alg:estimateK}, with $M = \lceil n^{\frac{p-1}{p+1}} \rceil$. As with the previous set of simulations, once we obtain $\hat{K}$, we construct $\tilde{Z}_i := \| X_i \|_{\hat{K}}$ for $i \in [n]$, compute $\hat{\phi}_n := \argmax_{\phi \in \Phi} n^{-1}\sum_{i=1}^n \phi(\tilde{Z}_i) - p\lambda_p(\hat{K}) \int_0^\infty r^{p-1} e^{\phi(r)} \, dr$, and set $\hat{f}_n = e^{\hat{\phi}_n}$.  The choices of $K$ were the same as those for $K_0$ in the corresponding panels of Figure~\ref{Fig:KAffineN}.  In Figures~\ref{Fig:KUnknownN1} and~\ref{Fig:KUnknownN2}, we take $p=6$ and $n \in \{6000,12000,24000,36000\}$, while in Figures~\ref{Fig:KUnknownP1} and~\ref{Fig:KUnknownP2}, we fix $n=24000$ and take $p \in \{2,4,6,8\}$.  We observe similar phenomena to those seen in the case where $K$ is known up to a positive definite transformation.

\begin{figure}
  \begin{subfigure}{0.49\textwidth}
    \centering
    \includegraphics[scale=0.35]{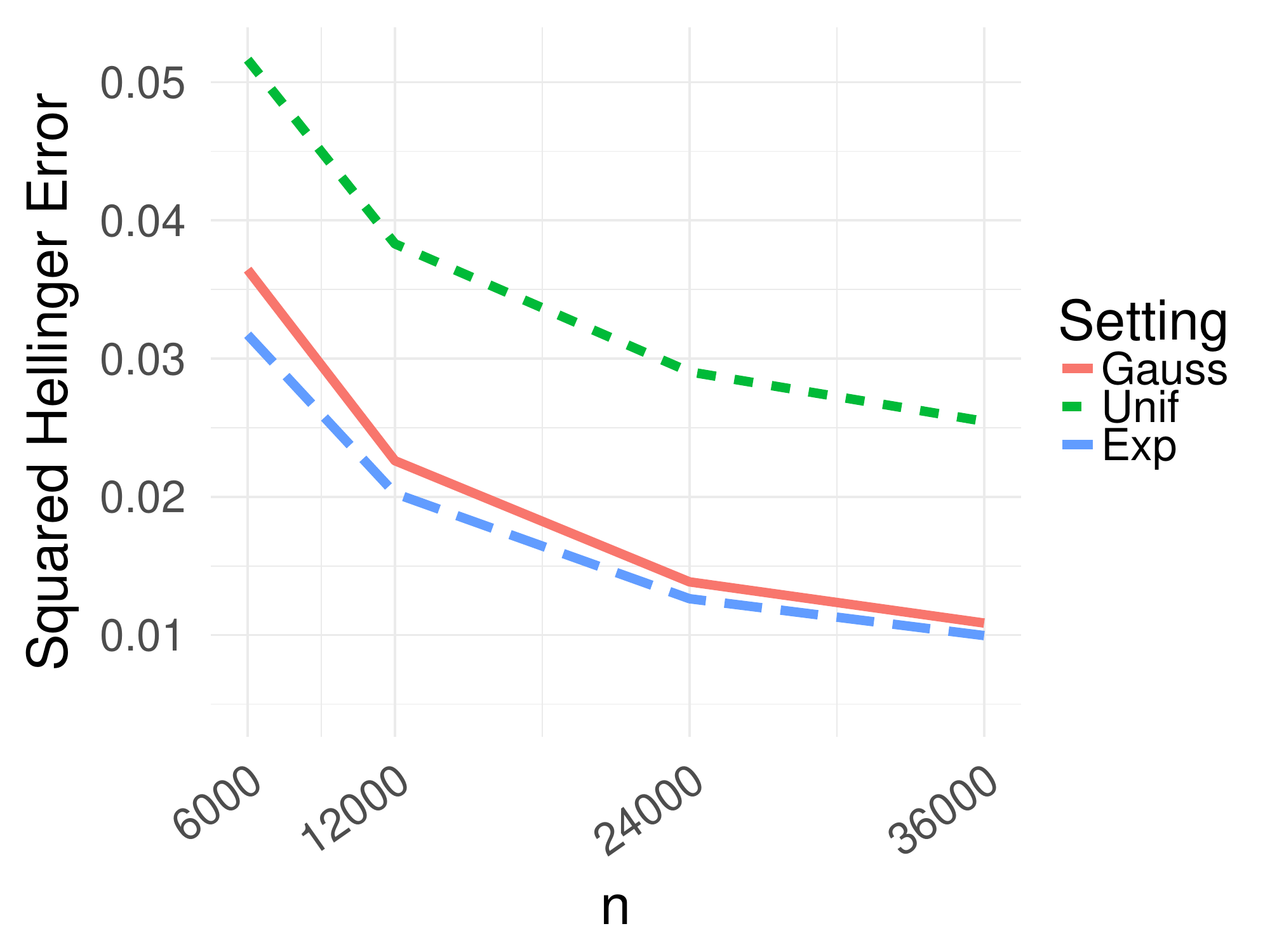}
    \caption{}
    \label{Fig:KUnknownN1}
  \end{subfigure}
  \begin{subfigure}{0.49\textwidth}
    \centering
    \includegraphics[scale=0.35]{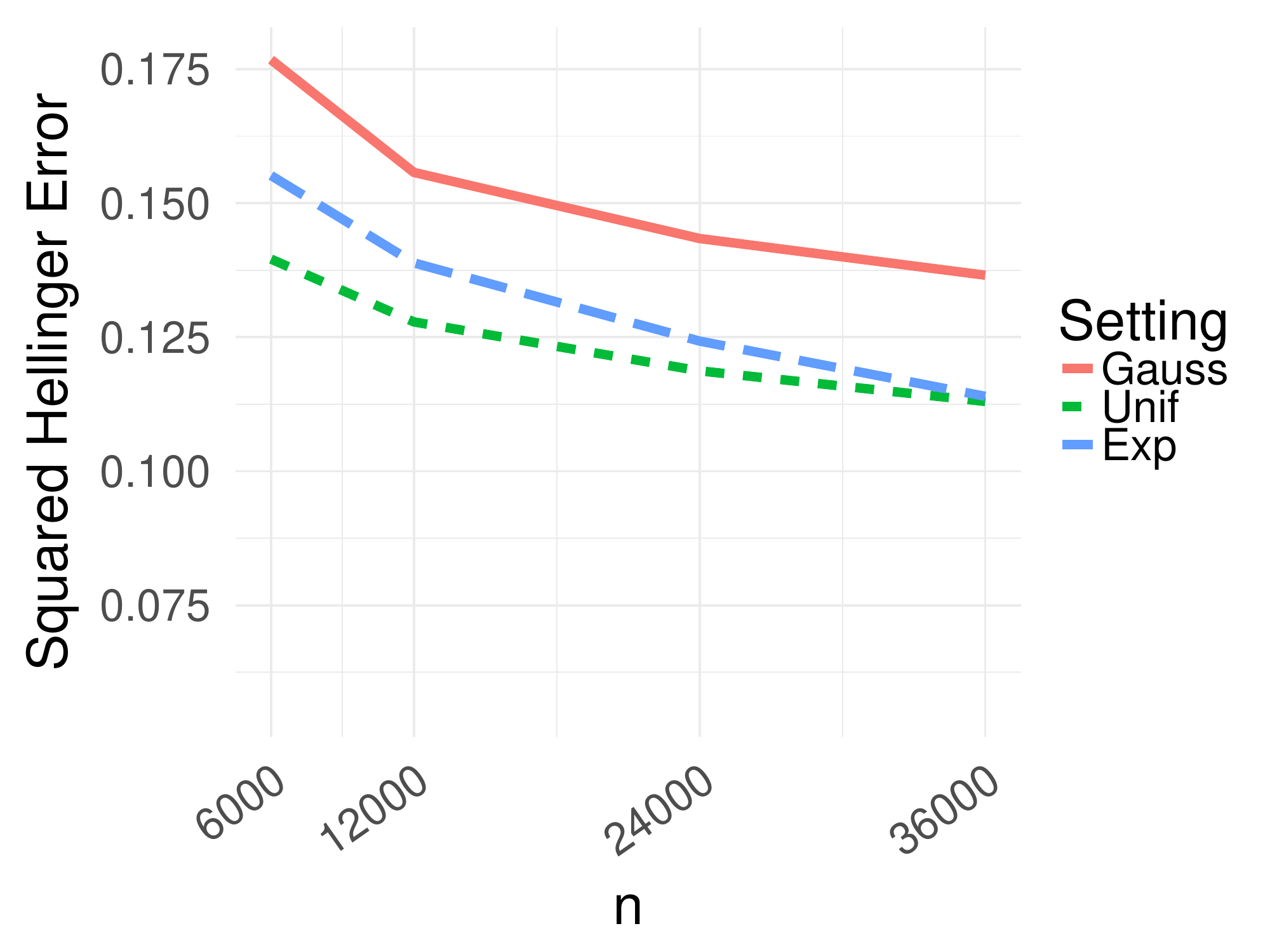}
    \caption{}
    \label{Fig:KUnknownN2}
  \end{subfigure}
  \begin{subfigure}{0.49\textwidth}
    \centering
    \includegraphics[scale=0.35]{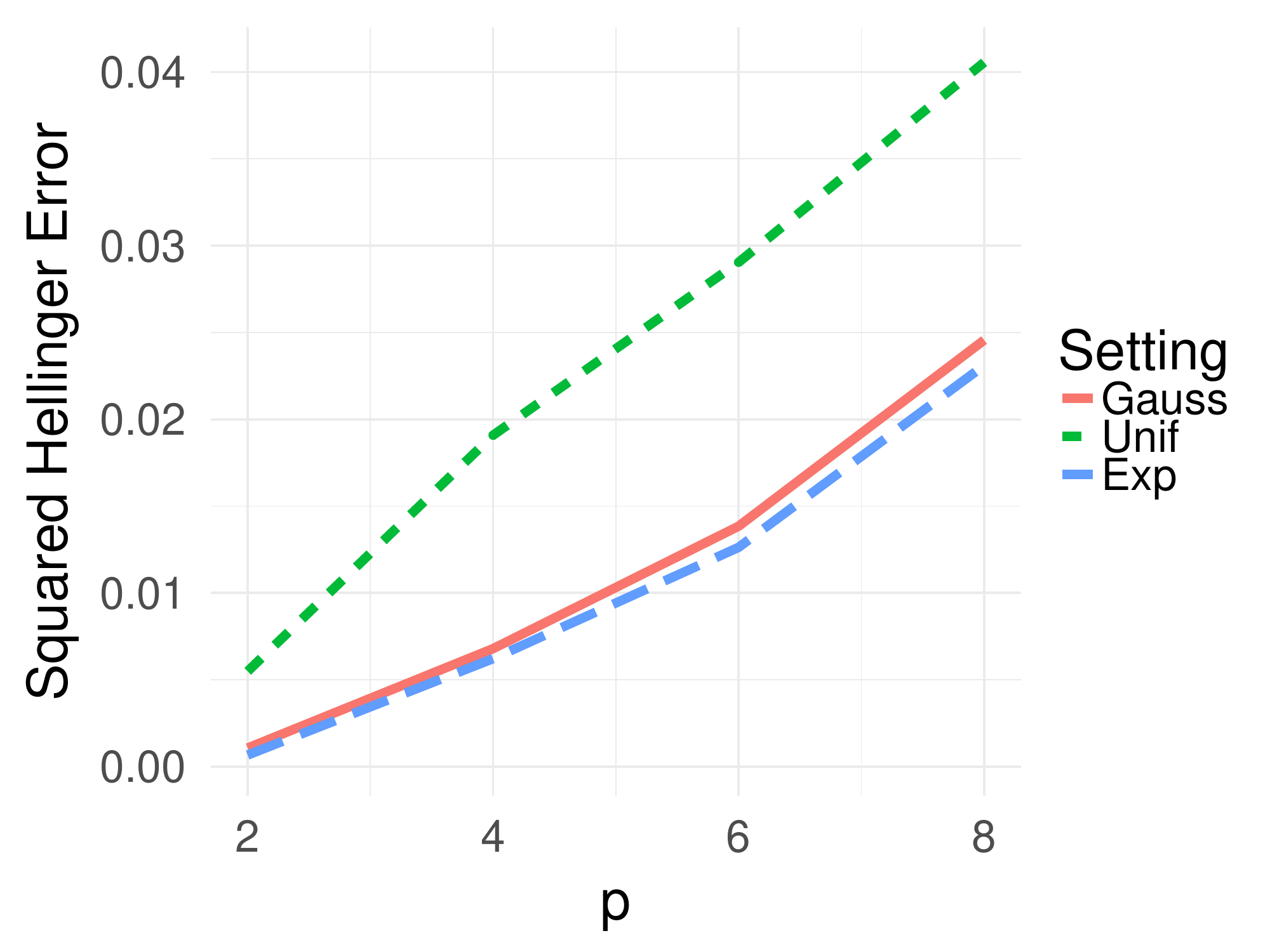}
    \caption{}
    \label{Fig:KUnknownP1}
  \end{subfigure}
   \begin{subfigure}{0.49\textwidth}
    \centering
    \includegraphics[scale=0.35]{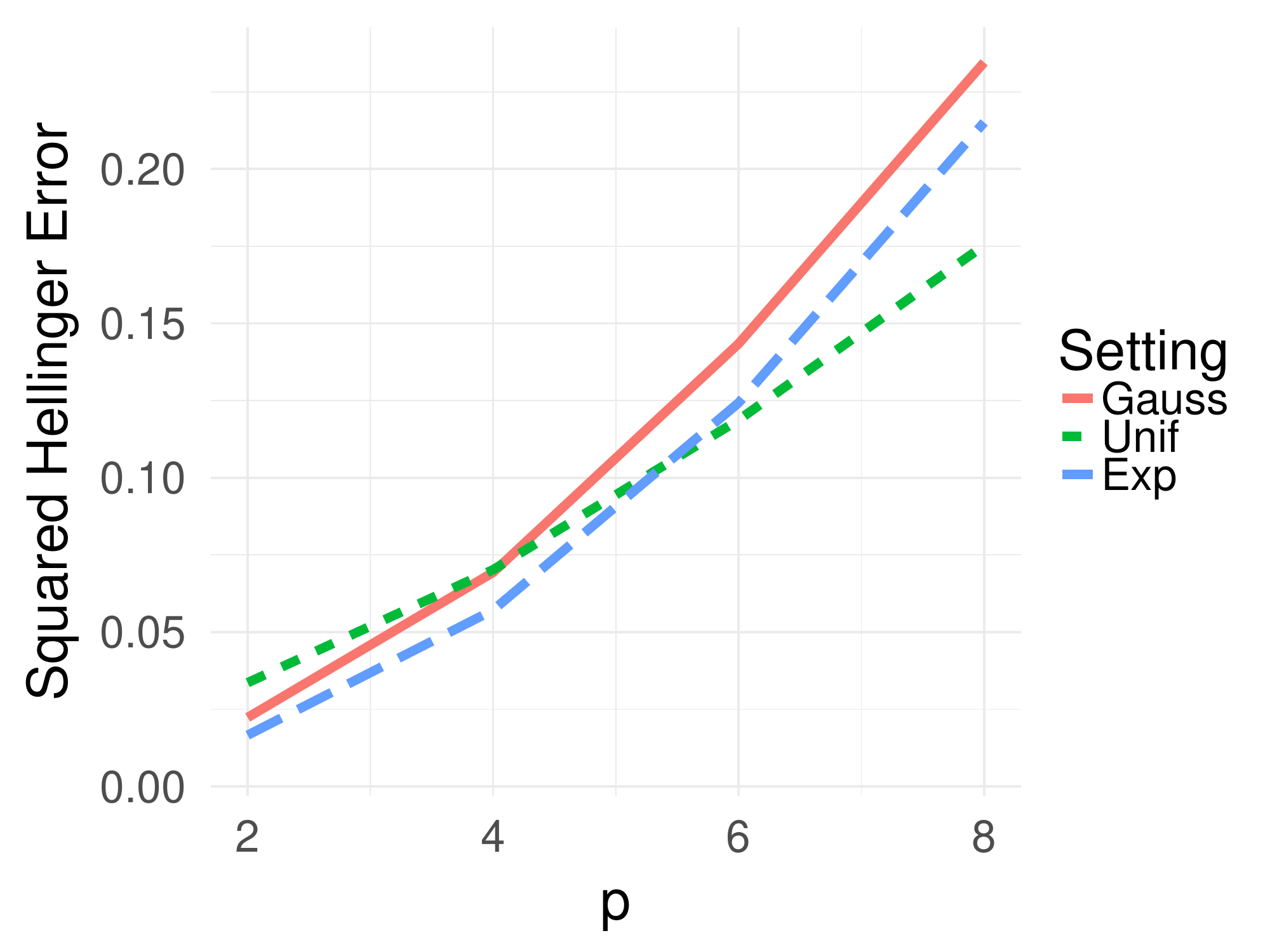}
    \caption{}
    \label{Fig:KUnknownP2}
  \end{subfigure} 
  \caption{\label{Fig:KUnknown}Average squared Hellinger errors $d_{\mathrm{H}}^2(\hat{f}_n,f_0)$ with $K$ unknown and estimated using Algorithm~\ref{alg:estimateK}.  For $\mathrm{Gauss}$, we took $f_0(\cdot) \propto e^{- \|\cdot \|^2_K/2}$; for $\mathrm{Unif}$, we took $f_0(\cdot) \propto \mathbbm{1}_{\{\|\cdot\|_K \leq p\}}$; for $\mathrm{Exp}$, we took $f_0(\cdot) \propto e^{-\| \cdot \|_K}$.  In (a) and (b), we fixed $p=6$, while in (c) and (d), we fix $n = 24000$.}
\end{figure}

\clearpage

\setcounter{section}{0}
\setcounter{equation}{0}
\setcounter{theorem}{0}
\def\theequation{S\arabic{section}.\arabic{equation}}
\def\thesection{S\arabic{section}}
\def\thetheorem{S\arabic{theorem}}

\begin{center}
\Large{Supplementary material to `High-dimensional nonparametric density estimation via symmetry and shape constraints'} \\ \vspace{0.2in}
\large{Min Xu and Richard J. Samworth} 
\end{center}

\section{Proofs from Section~\ref{Sec:Basic}}

\begin{proof}[Proof of Proposition~\ref{Prop:BasicK}]
  (i) This follows from the assumption that $0 \in \mathrm{int}(K)$.

  \medskip
  
  (ii) This follows from, e.g., \citet[][Corollary~9.7.1]{rockafellar1997convex}.

  \medskip
  
  (iii) Let $x \in \mathbb{R}^p$ and suppose that $\|x \|_K < 1$. Then there exists $\epsilon > 0$ such that $x \in (1 - \epsilon)K$ by the second claim.  Since $K$ is convex, we have that $x + \epsilon K \subseteq K$.  Moreover, since $K$ contains an open neighbourhood of 0, we see that $x$ is an interior point of $K$.  Conversely, if $x$ is an interior point of $K$, then there exists $\epsilon > 0$ such that $B_p(x,\epsilon) \subseteq K$.  Hence there exists $x' \in K$ with $\|x'\|_K > \|x\|_K$, so the conclusion follows from~(ii).

  \medskip
  
  (iv) See \citet[][Proposition~14.24]{RoydenFitzpatrick2010}.
\end{proof}
\begin{proof}[Proof of Proposition~\ref{Prop:ExistencePhi}]
  Any density of the form $f(\cdot) = e^{\phi(\|\cdot-\mu \|_K)}$ for some $K \in \mathcal{K}$, $\mu \in \mathbb{R}^p$ and $\phi \in \Phi$ is upper semi-continuous, and is log-concave by Proposition~\ref{Prop:BasicK}(iv).  Moreover, writing $\phi^{-1}(s) := \sup\{r \in [0,\infty):\phi(r) \geq s\}$ for $s \in (-\infty,\log \|f\|_\infty)$, we have
  \[
    \{x:f(x) \geq t\} = \{x:\|x-\mu\|_K \leq \phi^{-1}(\log t)\} = \phi^{-1}(\log t)K + \mu
  \]
  for all $t \in (0,\|f\|_\infty)$ by Proposition~\ref{Prop:BasicK}(ii) and Proposition~\ref{Prop:BasicK}(iv).  Hence $f$ is homothetic, as required.

  Conversely, suppose that $f$ is an upper semi-continuous, homothetic and log-concave density on $\mathbb{R}^p$, so there exist a decreasing function $r:(0,\|f\|_\infty) \rightarrow [0,\infty)$, a set $A \subseteq \mathcal{B}(\mathbb{R}^p)$ with $0 \in \mathrm{int}(A)$ and $\mu \in \mathbb{R}^p$ such that $\{x:f(x) \geq t\} = r(t)A + \mu$ for every $t \in (0,\|f\|_\infty)$.  Then in particular, $A \in \mathcal{K}$.  Since $0 < \lambda_p(\{ x: f(x) > 0 \}) = \lim_{n \rightarrow \infty} \lambda_p( \{ x : f(x) \geq 1/n \})$, there exists $t^* \in (0, \|f \|_{\infty})$ such that $r(t^*) > 0$. Thus, $\lambda_p(r(t^*)A) = r(t^*)^p \lambda_p(A) > 0$ and $\lambda_p(r(t^*)A) \leq (1/t^*) \int_{r(t^*)A + \mu} f(x) \, dx \leq 1/t^* < \infty$. By replacing $A$ with $r(t^*)A$ and $r(\cdot)$ with $r(\cdot)/r(t^*)$, we may therefore assume without the loss of generality that $r(t^*) = 1$. 

  We now claim that $r$ is left continuous. To see this, let $t \in (0, \|f\|_{\infty})$ and let $(t_n \in (0, \|f \|_{\infty}))_{n \in \mathbb{N}}$ be a sequence such that $t_n \nearrow t$. Then, $0 \leq r(t) \leq \lim_{t_n \nearrow t} r(t_n) \leq r(t_{n_0})$ for any $n_0 \in \mathbb{N}$. Since $A \in \mathcal{K}$, we have
    \begin{align}
     r(t)A + \mu = \bigcap_{n=1}^\infty \{ x \,:\, f(x) \geq t_n \} 
      = \bigcap_{n=1}^\infty \bigl(r(t_n)A + \mu\bigr) 
      \supseteq 
      \Bigl( \lim_{t_n \nearrow t} r(t_n) \Bigr)A + \mu. \nonumber 
    \end{align}
    Since $\lambda_p(A) < \infty$ and $A \in \mathcal{K}$, we have $\lim_{t_n \nearrow t} r(t_n) \leq r(t)$, so $r(t) = \lim_{t_n \nearrow t} r(t_n)$. We have thus shown that $r$ is left continuous and may define $r^{-1} \,:\, [0, \infty) \rightarrow [0, \| f\|_{\infty}]$ by $r^{-1}(u) := \sup\{t \in [0,\|f\|_\infty):r(t) \geq u\}$ for $u \in [0, \|r\|_\infty)$ and, if $\|r \|_{\infty} < \infty$, we define $r^{-1}( \| r \|_\infty ) := \sup \{ t \in [0, \|f \|_\infty): r(t) \geq \| r\|_\infty\}$ and $r^{-1}(u) = 0$ for any $u \in (\|r\|_{\infty}, \infty)$.  Notice that for any $u \in [0,\infty)$ and any $t \in (0, \|f \|_{\infty})$, we have $r(t) \geq u$ if and only if $r^{-1}(u) \geq t$.

    We now set $K :=  A$ and $\phi(s) := \log r^{-1}(s)$ for $s \in [0, \infty)$ (with the convention that $\log 0 := -\infty$).  Then the function $x \mapsto e^{\phi(\|x - \mu \|_K)}$ is well-defined on $\mathbb{R}^p$ and moreover for any $u \in [0,\infty)$ and any $t \in (0, \|f \|_{\infty})$, we have
    \[
      \{x:e^{\phi(\|x-\mu\|_K)} \geq t\} = \{x:r^{-1}(\|x-\mu\|_K) \geq t\} = \{x:\|x-\mu\|_K \leq r(t)\} = r(t)A + \mu.
    \]
We therefore conclude that $f(\cdot) = e^{\phi(\|\cdot - \mu\|_K)}$, and hence $\phi \in \Phi$, as desired.

Now suppose that $f(\cdot) = e^{\phi( \|\cdot - \mu \|_K)} = e^{\tilde{\phi}( \|\cdot - \tilde{\mu} \|_{\tilde{K}})}$ for some $\tilde{K} \in \mathcal{K}$, $\tilde{\mu} \in \mathbb{R}^p$ and $\tilde{\phi} \in \Phi$.  Suppose further that $f$ is not a uniform density and that $\|f\|_\infty = e^M$ for some $M \in \mathbb{R}$.  Then, by the log-concavity of $f$, there exist $c < c' < M$ such that $f^{-1}(\{e^c\}) \neq \emptyset$ and $f^{-1}(\{e^{c'}\}) \neq \emptyset$.  Let $a := \sup \{ r \geq 0 \,:\, \phi(r) \geq c\}$, $\tilde{a} := \sup \{ r \geq 0 \,:\, \tilde{\phi}(r) \geq c\}$, and note that $a, \tilde{a} > 0$.  If $x \in \mathbb{R}^p$ satisfies $\phi(\| x - \mu \|_K) \geq c$, then $\| x - \mu \|_K \leq a$ and thus $x \in a K + \mu$.  If on the other hand $x \in a K + \mu$, then $\phi( \| x - \mu \|_K) \geq \phi(a) \geq c$ since $\phi$ is upper semi-continuous. Thus,
    \begin{align}
      \{ x: f(x) \geq e^c \} = \{ x: \phi( \| x - \mu\|_K) \geq c \} = a K + \mu.
      \label{Eqn:LevelSetEqual}
    \end{align}
    By the same reasoning, $\{ x : f(x) \geq e^c \} = \tilde{a} \tilde{K} + \tilde{\mu}$ and we therefore have that $\tilde{K} = (a/\tilde{a}) K + (\mu - \tilde{\mu})/\tilde{a}$.  But, writing $a' := \sup \{ r \geq 0 \,:\, \phi(r) \geq c'\}$ and $\tilde{a}' := \sup \{ r \geq 0 \,:\, \tilde{\phi}(r) \geq c'\}$, we also have $\tilde{K} = (a'/\tilde{a}') K + (\mu - \tilde{\mu})/\tilde{a}'$, and moreover, $\tilde{a}' \neq \tilde{a}$.  We deduce that $\tilde{\mu} = \mu$, and $a\tilde{a}' = a'\tilde{a}$), that $\tilde{K} = (a/\tilde{a})K$ and that $\tilde{\phi}(r) = \phi( (a/\tilde{a}) r)$ for all $r \in [0, \infty)$, as required.

If $f$ is a uniform density, then there must exist $r_0 > 0$ and $s_0 \in \mathbb{R}$ such that $\phi(r) = s_0$ for $r \in [0,r_0]$ and $\phi(r) = -\infty$ for $r > r_0$.  Similarly, there exist $\tilde{r}_0 > 0$ and $\tilde{s}_0 \in \mathbb{R}$ such that $\tilde{\phi}(r) = \tilde{s}_0$ for $r \in [0,\tilde{r}_0]$ and $\tilde{\phi}(r) = -\infty$ for $r > \tilde{r}_0$.  It follows that $\|x - \tilde{\mu}\|_{\tilde{K}} \leq \tilde{r}_0$ if and only if $\|x - \mu\|_K \leq r_0$, so $\tilde{r}_0\tilde{K} + \tilde{\mu} = r_0K + \mu$.  We conclude that $\tilde{K} = (r_0/\tilde{r}_0)K + (\mu - \tilde{\mu})/\tilde{r}_0$ and $\tilde{\phi}(r) = \phi( (r_0/\tilde{r}_0) r)$ for all $r \in [0, \infty)$, as required.
\end{proof}

\begin{proof} [Proof of Proposition~\ref{Prop:NormDirectionIndep}]
  For $n \in \mathbb{N}$, define $K_n := K\setminus (1 - 1/n)K$ and let $Z_n$ be a random vector, independent of $R$, distributed uniformly on $K_n$. We claim that $Z_n \stackrel{d}{\rightarrow} Z/\| Z \|_K$ as $n \rightarrow \infty$. To see this, let $Z'_n$ be a random vector distributed uniformly on $(1 - 1/n)K$ and let $W_n$ be a Bernoulli$(q_{n,p})$ random variable, where $q_{n,p} := 1 - (1-1/n)^p$, independent of $(Z_n, Z_n')$. Since $Z$ is uniformly distributed on $K$, we have that $Z \stackrel{d} = W_n Z_n + (1 - W_n) Z'_n$ and thus $\frac{Z}{\| Z \|_K} \stackrel{d} = W_n \frac{Z_n}{\| Z_n \|_K} + (1 - W_n) \frac{Z'_n}{\|Z'_n \|_K}$.
  
  We observe that $Z'_n/\|Z'_n \|_K \stackrel{d}= Z/\| Z \|_K$ since $Z'_n \stackrel{d} = (1 - 1/n) Z$.  Now, writing $\psi_{Z/\|Z\|_K}$ and $\psi_{Z_n/\|Z_n \|_K}$ for the characteristic functions of $Z/\|Z\|_K$ and $Z_n/\|Z_n\|_K$ respectively, we have that $\psi_{Z/\|Z\|_K}(t) = q_{n,p} \psi_{Z_n/\|Z_n\|_K}(t) + (1 - q_{n,p}) \psi_{Z/\|Z\|_K}(t)$ for all $t \in \mathbb{R}^p$. We deduce that $Z_n/ \|Z_n\|_K \stackrel{d}= Z/\|Z\|_K$. Since
  \[
    \biggl\| Z_n - \frac{Z_n}{\|Z_n\|_K} \biggr\|_K = 1 - \|Z_n\|_K \leq \frac{1}{n} \rightarrow 0,
  \]
  it follows that $Z_n \stackrel{d} \rightarrow Z/\| Z\|_K$, as claimed.

  Define $X_n := Z_n R$, so that $X_n$ has density 
  \begin{align*}
    f_n(x) = \frac{1}{ q_{n,p}\lambda_p(K) } \int_0^\infty \frac{h(r)}{r^p} \mathbbm{1}_{\{ x/r \in K_n\}} \, dr = \frac{1}{ q_{n,p} \lambda_p(K) } \int_{\|x\|_K}^{\|x\|_K/(1-1/n)} \frac{h(r)}{r^p }\, dr 
  \end{align*}
for any $x \in \mathbb{R}^p$.  We deduce that whenever $x \in \mathbb{R}^p$ is non-zero and such that $\phi$ is continuous at $\|x\|_K$,
  \begin{align*}
    \lim_{n \rightarrow \infty} f_n(x) = \lim_{n \rightarrow \infty} \frac{1}{pq_{n,p}} \int_{\|x\|_K}^{\|x\|_K/(1-1/n)} r^{-1}e^{\phi(r)}\, dr = e^{\phi(\| x \|_K)} = f(x).
  \end{align*}
Thus, since $\phi$ is continuous Lebesgue almost everywhere, by Scheff\'e's lemma, $X_n$ converges in distribution to a random variable $X$ with density~$f$.  We conclude that $(Z/\|Z\|_K) R \stackrel{d}= X$, as desired. 
\end{proof}
\begin{remark}
  Alternatively, we may prove the first claim in Proposition~\ref{Prop:NormDirectionIndep} by defining, for any $t \in (0, 1)$, an operator $A_t \,:\, K \backslash\{0\} \rightarrow K \backslash t K$ of the form $A_t(x) = \bigl\{ 1 - t^p + \bigl(t/\| x \|_K \bigr)^p \bigr\}^{1/p} x $ and then showing that $A_t(Z)$ is uniformly distributed on $K \backslash t K$ for any $t \in (0,1)$. One can then show that if $Z_n \sim \mathrm{Unif}(K_n)$ then $Z_n \stackrel{d} \rightarrow Z / \| Z \|_K$ since $\lim_{t \nearrow 1} A_t(Z) = Z/\|Z\|_K$. 
\end{remark}
\begin{remark}
  In fact, from this proof, we see that Proposition~\ref{Prop:NormDirectionIndep} holds more generally whenever $K$ is compact and star-shaped at the interior point $0$, and $\phi$ is continuous Lebesgue almost everywhere.
\end{remark}


\begin{proof}[Proof of Proposition~\ref{Prop:ProjectionExistence}]
  (i) Fix $\phi \in \Phi_{a}$. Observe that if $\lim_{r \rightarrow \infty} \phi(r) = c > -\infty$, then $L(\phi, Q) \leq \phi(a) - e^c \int_{a}^\infty r^{p-1} \, dr = -\infty$.  Otherwise $\lim_{r \rightarrow \infty} \phi(r) = -\infty$, and then there exist $\alpha > 0, \beta \in \mathbb{R}$ such that $\phi(r) \leq -\alpha r+ \beta$. Hence,
  \[
    L( \phi, Q) \leq \int_{a}^\infty \phi \, dQ + 1 \leq - \alpha \int_{a}^\infty r \, dQ(r) + \beta + 1 = - \infty.
  \]

(ii) Now suppose that $Q(\{a\}) = 1$ and let $e^{\phi_n(r)} := n \mathbbm{1}_{\{ r \in [a, a + n^{-1}]\}}$. Then, 
\begin{align*}
  L(\phi_n, Q) &= \log n - np\lambda_p(K) \int_{a}^{a + n^{-1}} r^{p-1}\, dr + 1 \\
  &\geq \log n - p\lambda_p(K)(a + n^{-1})^{p-1} + 1 \rightarrow \infty
\end{align*}
as $n \rightarrow \infty$.

(iii) Finally, suppose that $Q \in \mathcal{Q}_{a}$.  For $\phi(r) := -r$, we have
\begin{align*}
  L(\phi, Q) &= - \int_{a}^\infty r\, dQ(r) - p\lambda_p(K)\int_{a}^\infty r^{p-1} e^{-r} \, dr + 1 \\
  &\geq - \int_{a}^\infty r\,dQ(r)  - p\lambda_p(K)\Gamma(p) + 1 > - \infty,
  \end{align*}
so $\sup_{\phi \in \Phi_{a}} L(\phi, Q) > -\infty$.

For $\delta, \epsilon > 0$, let $\mathcal{Q}_a(\delta,\epsilon) := \bigl\{Q \in \mathcal{Q}_a:Q\bigl((a+\delta,\infty)\bigr) > \epsilon\bigr\}$.  Then, since $Q(\{a\}) < 1$, we have $Q \in \mathcal{Q}_a(\delta,\epsilon)$ for some $\delta,\epsilon > 0$.  We also write $M := \phi(a)$ and $M' := \phi(a+\delta)$. Then by the concavity of $\phi$,
\begin{equation}
\label{Eq:L}
L(\phi, Q) \leq M (1-\epsilon) + M'\epsilon - \int_{a}^{a+\delta} r^{p-1} \exp \biggl(M - \frac{r - a}{\delta}(M - M') \biggr) \, dr + 1.
\end{equation}
If $M > 0$ and $(M - M')\epsilon \leq 2M$, then 
\begin{align*}
L(\phi, Q) &\leq M - e^M p\lambda_p(K)\int_{a}^{a+\delta} r^{p-1} \exp \biggl( - \frac{2M(r - a)}{\epsilon\delta} \biggr) \, dr + 1 \\
           &\leq M - e^M \biggl(\frac{\epsilon \delta}{2M} \biggr)^p p\lambda_p(K)\int_0^{2M/\epsilon} \biggl( \frac{2M a}{\epsilon\delta} + s \biggr)^{p-1}e^{-s}  \, ds + 1  \\
   &\leq M - e^M \biggl( \frac{\epsilon\delta}{2M} \biggr)^p p\lambda_p(K) \int_0^{2M/\epsilon} s^{p-1} e^{-s}  \, ds + 1. 
\end{align*}

On the other hand, if $M > 0$ and $(M - M')\epsilon > 2M$, then from~\eqref{Eq:L} we see that $L(\phi, Q) \leq -M+1$.  We deduce that there exists $M^*> 0$, depending only on $\delta$, $\epsilon$ and $p$, such that 
\[
  \sup_{\phi \in \Phi_{a}} L(\phi, Q) = \sup_{\phi \in \Phi_{a}:\phi(a) \leq M^*} L(\phi, Q) < \infty.
\]
The existence of $\phi^*$ then follows from the proof of Theorem 2.2 in \cite{dumbgen2011approximation}.

(iv) By the change of variable formula \citep[e.g.][Theorem~16.13]{billingsley1995probability}, we have $\int_{\mathbb{R}^p} \phi( \| x \|) \, dP(x) = \int_{[0,\infty)} \phi(r) \, dQ(r)$ for all $\phi \in \Phi$. The result then follows from~(iii), specialised to the case $a = 0$.
\end{proof}

\begin{proof}[Proof of Proposition~\ref{Prop:ProjectionBasicProperties}]

  (i) For any $\phi \in \Phi_{a}$, we may define $\phi_\alpha \in \Phi_{\alpha a}$ by $\phi_\alpha(r) := \phi(r/\alpha) - p \log \alpha$.  The map $\phi \mapsto \phi_\alpha$ is a bijection from $\Phi_{a}$ to $\Phi_{\alpha a}$.  Let $\phi^* = \phi_{a}^*(Q)$. Then, for any $\phi \in \Phi_{a}$, we have 
  \begin{align*}
    \int_{[\alpha a,\infty)} \phi^{*}_\alpha\, dQ_\alpha - p \lambda_p(K) \int_{\alpha a}^\infty r^{p-1} e^{\phi^{*}_\alpha (r)} \, dr
    &= \int_{[a,\infty)} \phi^* \, dQ - p \log \alpha - p \lambda_p(K) \int_{a}^\infty \! s^{p-1}e^{\phi^{*}(s)} \, ds \\
    &\geq \int_{[a,\infty)} \phi \, dQ - p \log \alpha - p \lambda_p(K) \int_{a}^\infty s^{p-1} e^{\phi(s)}  \, ds \\
    &=  \int_{[\alpha a,\infty)} \phi_\alpha\, dQ_\alpha - p \lambda_p(K) \int_{\alpha a}^\infty r^{p-1}e^{\phi_\alpha(r)} \, dr.
    \end{align*}
    This establishes that $\phi^{*}_\alpha = \phi_{\alpha a}^*(Q_\alpha)$ and thus proves scale equivariance.

    (ii) For any $t > 0$, we have
\begin{equation}
\label{Eq:DomConv}
0 \geq \frac{1}{t}\bigl\{L(\phi^*+t\Delta,Q) -L(\phi^*,Q)\bigr\} = \int_{[a,\infty)} \Delta \, dQ - \frac{1}{t}\int_{a}^\infty (e^{t\Delta(r)} - 1)h^*(r) \, dr.
\end{equation}
Choose $t_0 > 0$ small enough that $\phi^* + t_0 \Delta \in \Phi_{a}$.  Since $p \lambda_p(K) \int_{a}^\infty r^{p-1} e^{\phi^*(r)} \, dr < \infty$, we must have $\phi^*(r) \rightarrow -\infty$ as $r \rightarrow \infty$, and hence, by reducing $t_0 > 0$ if necessary, we may assume that $p \lambda_p(K) \int_{a}^\infty \Delta(r) r^{p-1} e^{\phi^*(r)+t_0\Delta(r)} \, dr < \infty$.  Now, for $t \in (0,t_0]$,
\[
\frac{1}{t}(e^{t\Delta(r)} - 1) \leq \frac{1}{t_0}(e^{t_0\Delta(r)} - 1)\mathbbm{1}_{\{\Delta(r) \geq 0\}} + \Delta(r)\mathbbm{1}_{\{\Delta(r) < 0\}}.
\]
Hence, if $\int_{a}^\infty \Delta(r) h^*(r) \, dr > -\infty$, then we may apply the dominated convergence theorem to~\eqref{Eq:DomConv} to take the limit as $t \searrow 0$ and reach the desired conclusion.  On the other hand, if $\int_{a}^\infty \Delta(r) h^*(r) \, dr = -\infty$, then for every $t \in (0,t_0]$,
\begin{align*}
\frac{1}{t} \int_{a}^\infty &(e^{t\Delta(r)} - 1)h^*(r) \, dr \\
&\leq \int_{a}^\infty \biggl(\frac{1}{t_0}(e^{t_0\Delta(r)} - 1)\mathbbm{1}_{\{\Delta(r) \geq 0\}} + \Delta(r)\mathbbm{1}_{\{\Delta(r) < 0\}}\biggr)h^*(r) \, dr = -\infty.
\end{align*}
The result follows.

(iii) Letting $\Delta(r) = -r$, this is a consequence of~(ii). 

(iv) Letting $\Delta(r) =\log \bigl(h_0(r)/h^*(r)\bigr)$, this also follows from an application of~(ii).
\end{proof}

\begin{proof}[Proof of Proposition~\ref{Prop:Continuity}]
  The proof is very similar to (in fact, somewhat more straightforward than) the proof of \citet[][Theorem~4.5]{dumbgen2011approximation}, so we focus on the main differences.  We first observe that if $X_n \sim P_n$ and $X \sim P$ are defined on the same probability space, then 
\[
\mathbb{E}\bigl|\|X_n\|_K - \|X\|_K\bigr| \leq \mathbb{E}\|X_n-X\|_K \leq \mathbb{E}\|X_n-X\| \sup_{x \neq 0} \frac{\|x\|_K}{\|x\|} \leq d_\mathrm{W}(P_n,P)\sup_{x \neq 0} \frac{\|x\|_K}{\|x\|}.
\]
Taking $\epsilon > 0$ such that $B_p(0,\epsilon) \subseteq K$, we have that $\sup_{x \neq 0} \|x\|_K/\|x\| \leq 1/\epsilon < \infty$.  Hence, writing $Q_n$ and $Q$ for the distributions of $\|X_n\|$ and $\|X\|$ respectively, we deduce that $d_\mathrm{W}(Q_n,Q) \leq d_\mathrm{W}(P_n,P)/\epsilon \rightarrow 0$.  It follows that $\int_0^\infty r \, dQ_n(r) \rightarrow \int_0^\infty r \, dQ(r) < \infty$, and $\limsup_{n \rightarrow \infty} Q_n(\{0\}) \leq Q(\{0\}) < 1$, so $Q_n \in \mathcal{Q}$ for $n \geq n_0$, say.  For such $n$, we write $\phi_n^* := \phi^*(Q_n)$ and $\phi^* := \phi^*(Q)$.  Let $n_0 \leq n_1 < n_2 < \ldots$ be an arbitrary, strictly increasing sequence of positive integers.  By extracting a further subsequence if necessary, we may assume that $L(\phi_{n_k}^*,Q_{n_k}) \rightarrow \alpha \in [-\infty,\infty]$.  First note that, by considering the function $\phi(r) = -r$,
\begin{align*}
  \alpha &\geq \lim_{k \rightarrow \infty} \int_0^\infty -r \, dQ_{n_k}(r) - p\lambda_p(K)(p-1)! + 1 \\
  &= \int_0^\infty -r \, dQ(r) - p\lambda_p(K)(p-1)! + 1 > -\infty.
\end{align*}
Our next claim is that $\limsup_{k \rightarrow \infty} \sup_{r \in [0,\infty)} \phi_{n_k}^*(r) < \infty$.  To see this, recall the definition of the classes $\mathcal{Q}(\delta,\epsilon) \equiv \mathcal{Q}_0(\delta,\epsilon)$ from the proof of Proposition~\ref{Prop:ProjectionExistence}, and let $\delta_0,\epsilon_0 > 0$ be such that $Q \in \mathcal{Q}(\delta_0,\epsilon_0)$.  Since $\liminf_{k \rightarrow \infty} Q_{n_k}\bigl((\delta_0,\infty)\bigr) \geq Q\bigl((\delta_0,\infty)\bigr) > \epsilon_0$, we see from the proof of Proposition~\ref{Prop:ProjectionExistence} that our claim follows.  This means that $\alpha < \infty$. 

Let $r_0 := \sup\bigl\{r \in [0,\infty):Q\bigl([0,r)\bigr) < 1\bigr\}$.  Our next claim is that $\liminf_{k \rightarrow \infty} \phi_{n_k}^*(r) > -\infty$ for all $r \in [0,r_0)$.  To see this, note by our first claim that we may assume without loss of generality that there exists $M^* \geq \max(\alpha,0)$ such that $\sup_{k \in \mathbb{N}} \sup_{r \in [0,\infty)} \phi_{n_k}^*(r) \leq M^*$.  Then, for any $r \in [0,r_0)$, 
\[
L(\phi_{n_k}^*,Q_{n_k}) \leq \phi_{n_k}^*(0)\bigl\{1 - Q_{n_k}\bigl((r,\infty)\bigr)\bigr\} + \phi_{n_k}^*(r)Q_{n_k}\bigl((r,\infty)\bigr).
\]
Since $Q\bigl((r,\infty)\bigr) > 0$, we deduce that
\begin{align*}
\liminf_{k \rightarrow \infty} \phi_{n_k}^*(r) \geq \liminf_{k \rightarrow \infty} \frac{L(\phi_{n_k}^*,Q_{n_k}) - \phi_{n_k}^*(0)\bigl\{1 - Q_{n_k}\bigl((r,\infty)\bigr)\bigr\}}{Q_{n_k}\bigl((r,\infty)\bigr)} \geq -\frac{M^* - \alpha}{Q\bigl((r,\infty)\bigr)} > -\infty,
\end{align*}
as required.  

These two claims allow us to extract a further subsequence $(\phi_{n_{k(\ell)}}^*)$ that converges in an appropriate sense to a limit $\phi^* \in \Phi$ (in particular, this convergence occurs Lebesgue almost everywhere).  It turns out that $\phi^* = \phi^*(Q)$, that $L(\phi_{n_{k(\ell)}}^*,Q) \rightarrow L(\phi^*,Q)$, and, writing $f_\ell^* := f^*(P_{n_{k(l)}})$ , we have $\int_{\mathbb{R}^p} |f_\ell^* - f^*| \rightarrow 0$.  The desired total variation convergence~\eqref{Eq:TVConv} follows.  See the proof of Theorem~4.5 of \citet{dumbgen2011approximation} for details.

For the final claim, note that our previous argument allows us to conclude that $f^*(P_n)$ converges to $f^*(P)$ Lebesgue almost everywhere.  The conclusion therefore follows from \citet[][Theorem~10.8]{rockafellar1997convex}.
\end{proof}

\section{Proofs from Section~\ref{Sec:Kknown}}

\begin{proof}[Proof of Theorem~\ref{Thm:KKnownWorstCase}]
Let $h_0$ denote the density of $\|X_1\|_K$, so that $h_0(r) = p\lambda_p(K)r^{p-1} e^{\phi(r)}$ for some $\phi \in \Phi$ by Proposition~\ref{Prop:NormDirectionIndep}.  Let $Z_i := \|X_i\|_K$ for $i \in [n]$, and write $\mathbb{Q}_n$ for the empirical distribution of $Z_1,\ldots,Z_n$.  Now let $\hat{h}_n := h^*(\mathbb{Q}_n)$, so that $\hat{h}_n(r) = p\lambda_p(K)r^{p-1}e^{\hat{\phi}_n(r)}$, where $\hat{\phi}_n := \phi^*(\mathbb{Q}_n) = \varphi^*(\mathbb{P}_n)$, by Proposition~\ref{Prop:ProjectionExistence}(iv).  Then $\hat{h}_n(Z_i)/h_0(Z_i) = \hat{f}_n(X_i)/f_0(X_i)$ for $i \in [n]$, so $d^2_X(\hat{f}_n,f_0) = d^2_X(\hat{h}_n,h_0)$.  By scale equivariance of $\phi^*$ (Proposition~\ref{Prop:ProjectionBasicProperties}(i)), together with the scale invariance of the loss function, we may assume without loss of generality that $\sigma_{h_0} = 1$.  By Lemma~\ref{Lem:MLEMeanVarPreservation}, there exist universal constants $C_\mu, C_\sigma, C > 0$ such that $\mathbb{P}\bigl(\hat{h}_n \notin \mathcal{H}_0(h_0, C_\mu, C_{\sigma})\bigr) \leq C/n$.  Moreover, by Lemma~\ref{Lem:NonlocalBracketingEntropy}, there exists a universal constant $K > 0$ such that for every $\delta > 0$,
\[
  \int_0^\delta H^{1/2}_{[]}\bigl(\epsilon, \mathcal{H}_0(h_0, C_\mu, C_\sigma), d_{\mathrm{H}}\bigr) \, d\epsilon \leq K \delta^{3/4}.
\]
Define $\Psi(\delta) := \max(K\delta^{3/4},\delta)$, so that $\delta \mapsto \Psi(\delta)/\delta^2$ is decreasing.  By choosing $\delta_* := K_0n^{-2/5}$ for a suitably large universal constant $K_0 > 0$, we may apply \citet[][Theorem 10]{kim2016adaptationsupp} (a minor restatement of \citet[][Corollary~7.5]{vandegeer2000empirical}), to deduce that there exists a universal constant $K_* > 0$ such that for $n \geq 8$,
\begin{align*}
\mathbb{E} d_X^2&( \hat{h}_n, h_0) \leq \int_0^{16 \log n}  \mathbb{P}\big( \bigl\{d_X^2(\hat{h}_n, h_0) \geq t\bigr\} \cap
                             \bigl\{\hat{h}_n \in \mathcal{H}_0(h_0, C_\mu, C_\sigma)\bigr\} \big) \, dt \\
                             &\hspace{2cm} +16 \log n \, \mathbb{P}\bigl( \hat{h}_n \notin \mathcal{H}(h_0, C_\mu, C_\sigma)\bigr) +
                             \int_{16 \log n}^\infty \mathbb{P}\big( d_X^2(\hat{h}_n, h_0) \geq t \big) \, dt
                             \\
&\leq \delta_*^2 + K_* \int_{\delta^2_*}^{16 \log n}  \exp\Bigl( - \frac{n t}{K_*} \Bigr) \, dt + \frac{16C \log n}{n} +
                             \int_{16 \log n}^\infty \mathbb{P}\bigg( \max_{i=1,\ldots,n} \log \frac{\hat{h}_n(Z_i)}{h_0(Z_i)}  \geq t \bigg) \, dt \\
&\lesssim n^{-4/5},
\end{align*}
where, to obtain the final inequality, we have applied Lemmas~\ref{Lem:ExtremeEventControl} and \ref{Lem:ExtremeEventControl2}.
\end{proof}

\begin{proof}[Proof of Theorem~\ref{Thm:KknownAdaptation}]
Fix $h_* \in \mathcal{H}^{(k)}$ where $h_*(r) = p\lambda_p(K)r^{p-1}e^{\phi_*(r)}$ and $\phi_* \in \Phi^{(k)}$, let $I_1,\ldots, I_k$ be the $k$ intervals on which $\phi_*$ is affine, with $I_j = [a_{j-1},a_j]$, and let $r_0 := \sup\bigl\{r \in [0,\infty):\phi_*(r) > -\infty\bigr\}$.  We work throughout on the probability 1 event that $\{Z_1,\ldots,Z_n\} \cap \{a_0,a_1,\ldots,a_k\} = \emptyset$.  Define $\mathcal{Z}_j := \{ i \,:\, Z_i \in I_j \}$ and $n_j := | \mathcal{Z}_j |$.  Let $J := \{ j \in [k] \,:\, n_j \geq 8\}$ be the set of indices of intervals with at least eight data points, and let $J^c := [k] \setminus J$.  Define $\tilde{\Phi}$ to be the set of upper semi-continuous functions $\phi:[0,\infty) \rightarrow [-\infty,\infty)$ such that $\phi\big|_{I_j}$ is decreasing and concave for each $j \in [k]$, and such that $\phi(r) = -\infty$ for $r > r_0$.  Note that a function $\phi \in \tilde{\Phi}$ need not be globally decreasing and, in fact, need not be continuous on $[0,r_0]$.  
Given any $\phi \in \tilde{\Phi}$ and $j \in J$, let $\phi^{(j)}(r) := \phi(r) + \log (n/n_j)$ for $r \in I_j$,  
and let $\Phi_{I_j} := \{ \phi|_{I_j}:\phi \in \Phi\}$.  Now, for $j \in J$, define 
\[
\tilde{\phi}_n^{(j)} := \argmax_{\phi_j \in \Phi_{I_j}} \biggl\{\frac{1}{n_j} \sum_{i \in \mathcal{Z}_j} \phi_j(Z_i) - p\lambda_p(K)\int_{I_j} r^{p-1}e^{\phi_j(r)}\, dr\biggr\},
\]
and let $\tilde{\phi}_n(r) := \tilde{\phi}_n^{(j)}(r) - \log (n/n_j)$ whenever $r \in I_j$ for some $j \in J$, and $\tilde{\phi}_n(r) := -\infty$ otherwise.
Then, for any $\phi \in \tilde{\Phi}$,
  \begin{align}
\label{Eq:LongDisplay}
    \frac{1}{n} \sum_{i=1}^n \tilde{\phi}_n(Z_i) &- p\lambda_p(K)\int_0^\infty r^{p-1} e^{\tilde{\phi}_n(r)} \, dr \nonumber \\
    &= \sum_{j \in J} \frac{n_j}{n}\biggl(  \frac{1}{n_j} \sum_{i \in \mathcal{Z}_j} \tilde{\phi}_n(Z_i) -
      \frac{n}{n_j} p\lambda_p(K)\int_{I_j} r^{p-1} e^{\tilde{\phi}_n(r)} \, dr \biggr) \nonumber \\
  &= \sum_{j \in J} \frac{n_j}{n}
    \biggl(  \frac{1}{n_j} \sum_{i \in \mathcal{Z}_j} \tilde{\phi}_n^{(j)}(Z_i) -
    p\lambda_p(K)\int_{I_j} r^{p-1} e^{\tilde{\phi}_n^{(j)}(r)} \, dr \biggr) - \sum_{j \in J} \frac{n_j}{n} \log \frac{n}{n_j} \nonumber \\
  &\geq \sum_{j \in J} \frac{n_j}{n}
    \biggl(  \frac{1}{n_j} \sum_{i \in \mathcal{Z}_j} \phi^{(j)}(Z_i) -
    p\lambda_p(K)\int_{I_j} r^{p-1} e^{\phi^{(j)}(r)} \, dr \biggr) - \sum_{j \in J} \frac{n_j}{n} \log \frac{n}{n_j} \nonumber \\
  &=  \frac{1}{n} \sum_{i=1}^n \phi(Z_i) - p\lambda_p(K)\int_0^\infty r^{p-1} e^{\phi(r)} \, dr. 
  \end{align}
Arguing similarly to the second paragraph of Section~\ref{Sec:Projection}, it follows that the function $\tilde{h}_n$ defined by $\tilde{h}_n(r) := p\lambda_p(K)r^{p-1}e^{\tilde{\phi}_n(r)}$ is a density.  Moreover, for $j \in J$, the function $\tilde{h}_n^{(j)} := \frac{n}{n_j}\tilde{h}_n|_{I_j}$ is a density.  Writing $p_j := \int_{I_j} h_0$, and $h_0^{(j)} := \frac{1}{p_j}h_0|_{I_j}$, we deduce from~\eqref{Eq:LongDisplay} that 
  \begin{align}
\label{Eq:dx2}
    \mathbb{E}_{h_0} &d_X^2(\hat{h}_n,h_0)
    \leq \mathbb{E}_{h_0} \frac{1}{n} \sum_{j \in J} \sum_{i \in \mathcal{Z}_j} \log \frac{\tilde{h}_n(Z_i)}{h_0(Z_i)} + \frac{7k}{n} \mathbb{E}_{h_0} \max_{i \in [n]} \biggl(\log \frac{\hat{h}_n(Z_i)}{h_0(Z_i)}\biggr) \nonumber \\
       &= \mathbb{E}_{h_0} \sum_{j \in J} \frac{n_j}{n}
      \biggl( \frac{1}{n_j} \sum_{i \in \mathcal{Z}_j} \log \frac{\tilde{h}_n^{(j)}(Z_i) }{h_0^{(j)}(Z_i)} \biggr) + \mathbb{E}_{h_0} \sum_{j \in J} \frac{n_j}{n} \log \frac{n_j}{n p_j} + \frac{7k}{n} \mathbb{E}_{h_0} \max_{i \in [n]} \biggl(\log \frac{\hat{h}_n(Z_i)}{h_0(Z_i)}\biggr).
  \end{align}
Now, since $n_j \sim \mathrm{Bin}(n,p_j)$ and $\log x \leq x-1$ for $x > 0$, we have
  \begin{align}
\label{Eq:SecondTerm}
    \mathbb{E}_{h_0} \sum_{j \in J}  \frac{n_j}{n} \log \frac{n_j}{n p_j} \leq \mathbb{E}_{h_0} \sum_{j \in J} \frac{n_j}{n} \biggl( \frac{n_j}{np_j} - 1 \biggr) &\leq \sum_{j=1}^k \frac{n^2 p_j^2 + np_j(1-p_j)}{n^2 p_j} - \sum_{j=1}^k \frac{np_j - 7}{n} \nonumber \\
    &= \frac{1}{n}\sum_{j =1}^k \bigl\{(1-p_j) + 7\bigr\} \leq \frac{8k}{n}.
  \end{align}
  For the third term in~\eqref{Eq:dx2}, we have by Lemmas~\ref{Lem:ExtremeEventControl} and \ref{Lem:ExtremeEventControl2}, that for $n \geq 8$ (cf.~\eqref{Eq:Max}),
  \begin{equation}
    \label{Eq:Max2}
    \mathbb{E}_{h_0} \max_{i \in [n]} \biggl(\log \frac{\hat{h}_n(Z_i)}{h_0(Z_i)}\biggr) \lesssim \log(en).
  \end{equation}
Finally, to bound the first term in~\eqref{Eq:dx2}, for $j\in[k]$, let us first define $q_j := \int_{I_j} h_*(r) \, dr$, $h_*^{(j)} := h_*/q_j$, $\nu^2_{*,j} := d_{\mathrm{H}}^2(h_0^{(j)}, h_*^{(j)})$, $\nu^2_* := \sum_{j=1}^k p_j \nu^2_{*,j}$, $J_0 := \{j \in J:n_j\nu_{*,j}^2 \geq 1\}$ and temporarily assume that $k \leq e^{-1/4}n$.  Note that $h_*^{(j)} \in \mathcal{H}_{a_{j-1}}^{(1)}$, and $h_0^{(j)} \ll h_*^{(j)}$ for $j \in J$.  It follows by Proposition~\ref{Prop:AdaptiveRateAffine}, applied conditionally on $n_1,\ldots,n_k$, a simple extension of Jensen's inequality using the fact that $x \mapsto \log^{5/4} x$ is concave on $[e^{1/4},\infty)$ \citep[e.g.][Lemma~2]{han2017isotonic} and the fact that $k \mapsto k \log^{5/4}(en/k)$ is increasing for $k \in [1,e^{-1/4}n]$,
  \begin{align}
    \label{Eqn:TripleJensenPrecursor}
    \mathbb{E}_{h_0} \sum_{j \in J} \frac{n_j}{n}
    \biggl( \frac{1}{n_j} \sum_{i \in \mathcal{Z}_j} &\log \frac{\tilde{h}_n^{(j)}(Z_i) }{h_0^{(j)}(Z_i)} \biggr) 
    \lesssim \mathbb{E}_{h_0} \sum_{j \in J} \frac{n_j}{n} \biggl(\frac{\nu_{*,j}^{2/5}}{n_j^{4/5}}\log \frac{en_j}{\nu_{*,j}} + \frac{1}{n_j} \log^{5/4} (en_j) \biggr) \nonumber \\
    &\lesssim \frac{1}{n} \mathbb{E}_{h_0} \sum_{j \in J} n_j^{1/5} \nu_{*,j}^{2/5} \bigl\{\log (n_j \nu_{*,j}^2) + \log(e/\nu_{*,j}^3) \bigr\} + \frac{k}{n} \log^{5/4} \Bigl(\frac{en}{k}\Bigr) .
  \end{align} 
To bound the first term of~\eqref{Eqn:TripleJensenPrecursor}, observe that by two applications of Jensen's inequality,
  \begin{align}
    \label{Eqn:Jensen1}
    \frac{1}{n} \mathbb{E}_{h_0} \sum_{j \in J} n_j^{1/5} \nu_{*,j}^{2/5} \log(n_j \nu_{*,j}^2)
    &\leq \frac{1}{n} \mathbb{E}_{h_0} \biggl\{|J_0|^{4/5}\biggl(
      \sum_{j\in J_0} n_j \nu_{*,j}^2 \biggr)^{1/5} \log \biggl(  \frac{\sum_{j \in J_0} n_j \nu_{*,j}^2 }{|J_0|} \biggr)\biggr\} \nonumber \\
&\leq \frac{1}{n} \mathbb{E}_{h_0} \biggl\{|J_0|^{4/5}\biggl(
      \sum_{j=1}^k n_j \nu_{*,j}^2 \biggr)^{1/5} \log \biggl(  \frac{\sum_{j =1}^k n_j \nu_{*,j}^2 }{|J_0|} \biggr)\biggr\} \nonumber \\
&\leq \frac{k^{4/5}}{n} \log \Bigl(\frac{en}{k}\Bigr)\mathbb{E}_{h_0} \biggl\{\biggl(
      \sum_{j=1}^k n_j \nu_{*,j}^2 \biggr)^{1/5} \biggr\} \nonumber \\
    &\leq \frac{k^{4/5}}{n^{4/5}} \nu_*^{2/5} \log \Bigl(\frac{en}{k}\Bigr). 
  \end{align}
But
\begin{align}
\label{Eq:nustar}
\nu_*^2 = \sum_{j=1}^k p_j \nu_{*,j}^2 \leq 2 \wedge \sum_{j=1}^k p_j d_{\mathrm{KL}}^2(h_0^{(j)}, h_*^{(j)}) &\leq 2 \wedge \biggl\{\sum_{j=1}^k p_j d_{\mathrm{KL}}^2(h_0^{(j)}, h_*^{(j)}) + \sum_{j=1}^k p_j \log \frac{p_j}{q_j}\biggr\} \nonumber \\
&= 2 \wedge d_{\mathrm{KL}}^2(h_0,h_*).
\end{align}
Moreover, by three further applications of Jensen's inequality, and using the fact that $x \mapsto \log^5 x$ is concave for $x \geq e^5$, we have
  \begin{align}
\label{Eq:TripleJensen}
    \frac{1}{n} \mathbb{E}_{h_0} \sum_{j \in J} &n_j^{1/5} \nu_{*,j}^{2/5} \log(e/\nu_{*,j}^3) \leq \frac{3}{n} \mathbb{E}_{h_0} \sum_{j =1}^k n_j^{1/5} \nu_{*,j}^{2/5} \log \frac{2^{1/2}e^5}{\nu_{*,j}} \nonumber \\
&\leq \frac{3k^{4/5}}{n} \mathbb{E}_{h_0} \biggl\{\biggl(\sum_{j=1}^k n_j \nu_{*,j}^2 \log^5 \frac{2^{1/2}e^5}{\nu_{*,j}}\biggr)^{1/5}\biggr\} \leq \frac{3k^{4/5}}{n^{4/5}}\biggl(\sum_{j=1}^k p_j \nu_{*,j}^2 \log^5 \frac{2^{1/2}e^5}{\nu_{*,j}}\biggr)^{1/5} \nonumber \\
                                                                                              &\leq \frac{3k^{4/5}}{n^{4/5}}\nu_*^{2/5}\log\biggl(2^{1/2}e^5\sum_{j=1}^k \frac{p_j\nu_{*,j}}{\nu_*^2}\biggr) \leq \frac{3k^{4/5}}{n^{4/5}}\nu_*^{2/5}\log \Bigl(\frac{2e^5}{\nu_*^2}\Bigr) \nonumber \\
&\leq \frac{3k^{4/5}}{n^{4/5}}\{2 \wedge d_{\mathrm{KL}}^2(h_0,h_*)\}^{1/5}\log \Bigl(\frac{2e^5}{2 \wedge d_{\mathrm{KL}}^2(h_0,h_*)}\Bigr),
  \end{align}
where the last step follows from the fact that $x \mapsto x^{1/5} \log \frac{2e^5}{x}$ is increasing for $x \leq 2$.  Combining~\eqref{Eq:dx2},~\eqref{Eq:SecondTerm},\eqref{Eq:Max2},~\eqref{Eqn:TripleJensenPrecursor},~\eqref{Eqn:Jensen1},~\eqref{Eq:nustar} and~\eqref{Eq:TripleJensen}, the result follows in the case $k \leq e^{-1/4}n$.

Now suppose $k > e^{-1/4} n$. Then, by Theorem~\ref{Thm:KKnownWorstCase},
\begin{align*}
  \mathbb{E}_{h_0} d^2_X(\hat{h}_n, h_0) \lesssim n^{-4/5} \lesssim \frac{k}{n} \log^{5/4} \frac{en}{k},
\end{align*}
which completes the proof.
\end{proof}

\begin{proof}[Proof of Proposition~\ref{Prop:AdaptiveRateAffine}]
By the scale equivariance described in Proposition~\ref{Prop:ProjectionBasicProperties}(i), we may assume without loss of generality that $\sigma_{h_0} = 1$.  Define $\nu := \inf \bigl\{ d_{\mathrm{H}}(h_0, h) \,:\, h \in \mathcal{H}_{a}^{(1)}, h_0 \ll h \bigr\} \in [0,2^{1/2}]$. By Lemma~\ref{Lem:LocalBracketing}, if $\delta \in (0, 2^{-9}-\nu)$, then for every $\epsilon > 0$, it holds that 
\begin{align*}
  H_{[]}\bigl(\epsilon, \mathcal{H}(h_0, C_\mu, C_\sigma) \cap \mathcal{H}(h_0, \delta), d_{\mathrm{H}}\bigr) 
  \leq  H_{[]}(\epsilon, \mathcal{H}(h_0, \delta), d_{\mathrm{H}}) \lesssim \frac{(\delta + \nu)^{1/2}}{\epsilon^{1/2}} \log^{5/4} \frac{1}{\delta}.
\end{align*}
On the other hand, if $\delta \geq 2^{-9} - \nu$, then by Lemma~\ref{Lem:NonlocalBracketingEntropy}, for every $\epsilon > 0$, we have that 
\begin{align*}
    H_{[]}\bigl(\epsilon, \mathcal{H}(h_0, C_\mu, C_\sigma) \cap \mathcal{H}(h_0, \delta), d_{\mathrm{H}}\bigr) \leq  H_{[]}\bigl(\epsilon, \mathcal{H}(h_0, C_\mu, C_\sigma), d_{\mathrm{H}}\bigr) \lesssim \frac{1}{\epsilon^{1/2}} \lesssim \frac{(\delta + \nu)^{1/2}}{\epsilon^{1/2}}.
\end{align*}
It follows that 
\begin{align*}
  \int_0^\delta H_{[]}^{1/2}\bigl(\epsilon, \mathcal{H}(h_0, \delta) \cap
    \mathcal{H}(h_0, C_\mu, C_\sigma), d_{\mathrm{H}}\bigr) \, d\epsilon &\lesssim (\delta+\nu)^{1/4}
    \bigl\{ \log^{5/8} (1/\delta) \vee 1 \bigr\}
    \int_0^{\delta} \epsilon^{-1/4} \, d\epsilon \\
  &\lesssim \delta^{3/4} (\delta+\nu)^{1/4} \bigl\{ \log^{5/8} (1/\delta) \vee 1 \bigr\}.
\end{align*}
Define $\Psi(\delta) :=  C \delta^{3/4} (\delta+\nu)^{1/4}\{ \log^{5/8} (1/\delta) \vee 1\}$, where the universal constant $C > 0$ is chosen such that 
\[
\Psi(\delta) \geq \max\biggl\{\int_0^\delta H_{[]}^{1/2}\bigl(\epsilon, \mathcal{H}(h_0, \delta) \cap
    \mathcal{H}(h_0, C_\mu, C_\sigma), d_{\mathrm{H}}\bigr) \, d\epsilon \, , \, \delta\biggr\},
\]
Set $\delta_n := K\bigl\{\frac{\nu^{2/5}}{n^{4/5}} \log \frac{en}{\nu} + \frac{1}{n} \log^{5/4}(en)\bigr\}^{1/2}$ for a universal constant $K > 0$ to be chosen later.  Then, because $\Psi(\delta)/\delta^2$ is non-increasing, we have
\begin{align*}
  \inf_{\delta \geq \delta_n} \frac{n^{1/2}\delta^2}{\Psi(\delta)} = & \frac{n^{1/2}\delta_n^2}{\Psi(\delta_n)} = \frac{n^{1/2}\delta_n^{5/4}}{C(\delta_n+\nu)^{1/4}\{ \log^{5/8} (1/\delta_n) \vee 1\}}.
\end{align*}
By choosing the universal constant $K > 0$ sufficiently large, we can ensure that this ratio is larger than the universal constant required to apply Theorem~10 in the online supplement of \citet{kim2016adaptationsupp} (a minor restatement of \citet[][Corollary~7.5]{vandegeer2000empirical}).  We deduce from this result that there exists a universal constant $C > 0$ such that for $\delta \geq \delta_n$,
\begin{equation}
\label{Eq:vdg}
\mathbb{P}\Bigl(\bigl\{d_X^2(\hat{h}_n,h_0) \geq \delta^2\bigr\} \cap \bigl\{\hat{h}_n \in \mathcal{H}(h_0,C_\mu,C_\sigma)\bigr\}\Bigr) \leq C\exp\Bigl(-\frac{n\delta^2}{C}\Bigr).
\end{equation}
Moreover, by Lemmas~\ref{Lem:ExtremeEventControl} and \ref{Lem:ExtremeEventControl2}, for $n \geq 8$, 
\begin{align}
\label{Eq:Max}
  \int_{16 \log n}^\infty \mathbb{P}\bigg( \max_{i\in[n]} \log &\frac{\hat{h}_n(Z_i)}{h_0(Z_i)}  \geq t \bigg) \, dt
  = \log (en) \int_{16}^\infty  \mathbb{P}\bigg( \max_{i\in[n]} \log \frac{\hat{h}_n(Z_i)}{h_0(Z_i)}  \geq s\log n \bigg) \, ds \nonumber \\
  &\lesssim \log (en) \int_{16}^\infty \Bigl(\frac{6}{n^{s/64}}\Bigr)^n + 2n^{-s n^{1/2}/128} + n^{-(s/2-1)} \, ds \lesssim \frac{\log (en)}{n}.
\end{align}
It follows from~\eqref{Eq:vdg},~\eqref{Eq:Max} and Lemma~\ref{Lem:MLEMeanVarPreservation} that for $n \geq 8$,
\begin{align*}
   \frac{1}{n}\sum_{i=1}^n \mathbb{E} \log &\frac{\hat{h}_n(Z_i)}{h_0(Z_i)} 
  \leq  \int_0^{16 \log n}  \mathbb{P}\Big( \bigl\{d_X^2(\hat{h}_n, h_0) \geq t\bigr\} \cap
         \bigl\{\hat{h}_n \in \mathcal{H}(h_0, C_\mu, C_\sigma)\bigr\} \Big) \, dt  \\
  & \hspace{1cm}+ 16 \log n \, \mathbb{P}\bigl( \hat{h}_n \notin \mathcal{H}(h_0, C_\mu, C_\sigma)\bigr) +
    \int_{16 \log n}^\infty \mathbb{P}\big( d_X^2(\hat{h}_n, h_0) \geq t \big) \, dt
  \\
  &\leq  \delta_n^2 + C \int_{\delta_n^2}^{\infty} \exp\Bigl(-\frac{nt}{C}\Bigr) \, dt + \frac{C \log n}{n} +
         \int_{16 \log n}^\infty \mathbb{P}\bigg( \max_{i\in[n]} \log \frac{\hat{h}_n(Z_i)}{h_0(Z_i)}  \geq t \bigg) \, dt \\
&\lesssim \frac{\nu^{2/5}}{n^{4/5}}\log\frac{en}{\nu} + \frac{\log^{5/4}(en)}{n},
\end{align*}
as required.
\end{proof}

\section{Proofs from Section~\ref{Sec:KEstimated}}

\begin{proof}[Proof of Proposition~\ref{Prop:KEstimatedGeneralRisk}]
  First, we note that since $\hat{K}, \hat{\mu}$ depend only on $X_{n+1},\ldots, X_{2n}$, they are independent of $X_1, \ldots, X_n$.  Moreover, since $f_0 \in \mathcal{F}^{K, \mu}_p$, we may, by Proposition~\ref{Prop:ExistencePhi}, rescale $K$ if necessary to assume without loss of generality that
  \begin{align}
    \mathbb{E}_{f_0}\bigl( \| X_1 - \mu \|^2_K \bigr) = p. \label{Eqn:MomentAssumption2}
  \end{align}
  Once we prove the proposition with this assumption, the more general conclusion follows immediately from the fact that both $\| \hat{\mu} - \mu \|_K / \mathbb{E}^{1/2}_{f_0}( \| X_1 - \mu \|_K^2)$ and $\inf_{\alpha > 0} d_{\mathrm{scale}}(\alpha \hat{K}, K)$ remain unchanged if we rescale $K$.  Under assumption~\eqref{Eqn:MomentAssumption2}, the event $\mathcal{E}_{c_1,c_2}$ defined in~\eqref{Eqn:Ec1c2} then takes the form
  \begin{align}
   \mathcal{E}_{c_1,c_2} = \bigl\{ \| \hat{\mu} - \mu \|_K p \log(ep) \leq c_1 \bigr\} \cup \Bigl\{ \inf_{\alpha > 0} d_{\mathrm{scale}}( \alpha \hat{K}, K) < c_2 \Bigr\},
  \end{align}
and we choose universal constants $c_1, c_2 > 0$ small enough that the conclusions of Lemmas~\ref{Lem:MeanVarianceControl}, \ref{Lem:EmpiricalProcess}, \ref{Lem:WorstTVBound}, \ref{Lem:dHWorstBound2}, \ref{Lem:KLWorstBound}, \ref{Lem:dHBestCase1}, \ref{Lem:dHBestCase2}, and \ref{Lem:KLBestCase} hold.

Note that, by Proposition~\ref{Prop:ProjectionBasicProperties}(i), $\hat{h}_n$ is scale equivariant. Thus, $\hat{f}_n$ is by construction also scale equivariant and consequently, we may rescale $\hat{K}$ if necessary to assume that, on event $\mathcal{E}_{c_1,c_2}$, we have that $d_{\mathrm{scale}}(\hat{K}, K) < c_2$.  Now let $a_n := n^{-4/5} + d_{\textrm{KL}}^2(\tilde{h}_n, \tilde{h}_0) + d_{\textrm{H}}^2(\check{f}_n, f_0) + d_{\textrm{H}}^2\bigl( \tilde{h}_0, h_0 \frac{\lambda_p(\hat{K})}{\lambda_p(K)} \bigr)$.  By Lemmas~\ref{Lem:MeanVarianceControl},~\ref{Lem:EmpiricalProcess},~\ref{Lem:WorstTVBound},~\ref{Lem:dHWorstBound2}, and~\ref{Lem:KLWorstBound}, together with the fact that $d_{\mathrm{H}}^2(f,g) \leq d_{\mathrm{TV}}(f,g)$ for all densities $f,g$, there exists universal constants $C_\mu > 0, C_\sigma > 1, C > 0$ such that on the event $\mathcal{E}_{c_1,c_2}$,
\begin{align}
  \label{Eq:ConditionalHellinger}
  \mathbb{E}_{f_0}&\bigl\{d_{\mathrm{H}}^2 (\hat{f}_n, f_0) \bigm| \hat{\mu},\hat{K}\bigr\} \nonumber \\
    &\quad \leq a_n + \int_{a_n}^{\infty} \mathbb{P}_{f_0}\bigl( \{d_{\textrm{H}}^2 (\hat{f}_n, f_0) > s\} \cap \{ \hat{h}_n \in \mathcal{H}(h_0, C_\mu, C_\sigma)\} \bigm| \hat{\mu}, \hat{K} \bigr) \, ds  \nonumber \\
    &\hspace{7cm} + \mathbb{P}_{f_0}\bigl\{ \hat{h}_n \notin \mathcal{H}(h_0, C_\mu, C_\sigma) \,| \,\hat{\mu}, \hat{K} \bigr\} \nonumber \\
    &\quad \lesssim  a_n + \int_{0}^{\infty} e^{-n s/C} \, ds + 1/n \nonumber \\
    &\quad \lesssim n^{-4/5} +  p^{1/2} \| \hat{\mu} - \mu\|_K+ p(\tau^* - \tau_{*}),
  \end{align}
  where $\tau^* := \sup_{x \in \mathbb{R}^p \setminus \{0\} } \| x \|_{\hat{K}}/\| x \|_K$ and $\tau_* := \inf_{x \in \mathbb{R}^p \setminus \{0\} } \| x \|_{\hat{K}}/\|x \|_K$.  We now claim that $(\tau^* - \tau_*) \leq 2 d_{\mathrm{scale}}(\hat{K}, K)$. To see this, fix any $\epsilon > d_{\mathrm{scale}}(\hat{K}, K)$ and $x \in \mathbb{R}^p \setminus \{0\}$. Then,
  \begin{align*}
    \frac{ \|x \|_{\hat{K}}}{\|x \|_K} &= (1 + \epsilon) \Bigl \| \frac{1}{1+\epsilon} \frac{x}{\|x\|_K} \Bigr\|_{\hat{K}} \leq 1 + \epsilon, \\
                                     \frac{ \|x \|_{\hat{K}}}{\|x \|_K} &= \frac{1}{1 + \epsilon}  \Bigl \| (1+\epsilon) \frac{x}{\|x\|_K} \Bigr\|_{\hat{K}} \geq \frac{1}{1 + \epsilon}.
  \end{align*}
  Since $\epsilon > d_{\mathrm{scale}}(\hat{K}, K)$ and $x \in \mathbb{R}^p \setminus \{0\}$ were chosen arbitrarily, we have that $\tau^* - 1 \leq d_{\mathrm{scale}}(\hat{K}, K)$ and that $1 - \tau_* \leq d_{\mathrm{scale}}(\hat{K}, K)$ as desired.  The first part of the proposition then follows.

The second claim follows from the same argument except that we apply Lemmas~\ref{Lem:dHBestCase1}, \ref{Lem:dHBestCase2}, and~\ref{Lem:KLBestCase} in the final inequality in~\eqref{Eq:ConditionalHellinger}.

Finally, for the third claim of the proposition, we define
  \[
    \tilde{a}_n := \frac{1}{n}\log^{5/4}(en) + d_{\textrm{KL}}^2(\tilde{h}_n, \tilde{h}_0) + d_{\textrm{H}}^2(\tilde{h}_0, h_0) + d_{\textrm{H}}^2(\check{f}_n, f_0) + d_{\textrm{H}}^2 \biggl( \tilde{h}_0, h_0 \frac{\lambda_p(\hat{K})}{\lambda_p(K)} \biggr).
  \]
 We may then apply the second claim of Lemma~\ref{Lem:EmpiricalProcess} and Lemmas~\ref{Lem:MeanVarianceControl}, \ref{Lem:dHBestCase1}, \ref{Lem:dHBestCase2}, and~\ref{Lem:KLBestCase} in the final inequality in~\eqref{Eq:ConditionalHellinger} to obtain the desired conclusion.
\end{proof} 

\begin{proof}[Proof of Proposition~\ref{Prop:dscaleAffineBound}]
  Since $\hat{K}$ is location invariant, $\hat{\mu}$ is location equivariant, and $\| \hat{\mu} - \mu \|_K$ is also location invariant, we assume without loss of generality that $\mu = 0$. Moreover, $\hat{K}$ is scale equivariant in the sense that if $\tilde{\Sigma} \in \mathbb{S}^{p \times p}$ and we define $\tilde{X}_i := \tilde{\Sigma}^{1/2} X_i$ for $i \in [n]$, let $\hat{\Sigma}' := n^{-1}\sum_{i=1}^n \tilde{X}_i \tilde{X}_i^\top$ and $\hat{K}' := \hat{\Sigma}^{\prime 1/2} K_0$, then $\hat{K}' = \tilde{\Sigma}^{1/2} \hat{K}$. Since $\inf_{\alpha > 0} d_{\mathrm{scale}}(\alpha \hat{K}, K)$ is also scale invariant in the sense that  $\inf_{\alpha > 0} d_{\mathrm{scale}}(\alpha \hat{K}, K) =  \inf_{\alpha > 0} d_{\mathrm{scale}}(\alpha \tilde{\Sigma}^{1/2} \hat{K}, \tilde{\Sigma}^{1/2} K)$ for any $\tilde{\Sigma} \in \mathbb{S}^{p \times p}$, we assume without loss of generality that $\Sigma = I_p$. Thus, there exists $\alpha_0 > 0$ such that $K = \alpha_0 K_0$.
  
For each $j \in [p]$, the random variable $X_{1j}$ has a univariate log-concave density with mean $0$ and variance $1$. By e.g. \citet[][Proposition~S2(iii)]{feng2018multivariate}, we have $\mathbb{E}(|X_{1j}|^k) \leq 2e k!$ for all integers $k \geq 2$. Then, by Bernstein's inequality, there exists a universal constant $C_1 > 0$ such that
  \[
    \mathbb{P}\biggl\{ \max_{j \in [p]}\biggl| \frac{1}{n} \sum_{i=1}^n X_{ij} \biggr| > C_1 \frac{\log^{1/2}(en)}{n^{1/2}}\biggr\} \leq \frac{p}{n^2}.
  \]
  Let $\mathcal{E}_1:= \{ \| \hat{\mu} \| \leq C_1 (p/n)^{1/2} \log^{1/2} (en) \}$, so that $\mathbb{P}(\mathcal{E}_1^c) \leq p/n^2 \leq n^{-1}$.
  
By \citet[][Theorem 4.1]{adamczak2010quantitative}, there exists a universal constant $C_2 > 0$ such that, with probability at least $1 - 1/n$, 
 \begin{align}
    \biggl\| \frac{1}{n}  \sum_{i=1}^n X_i X_i^\top  - I_p \biggr\|_{\mathrm{op}} \leq C_2 \sqrt{\frac{p}{n}} \log^3(en).  \label{eqn:op_norm_small}
 \end{align}
 Let $\mathcal{E}_2$ denote the event that~\eqref{eqn:op_norm_small} holds. We work on the event $\mathcal{E}_1 \cap \mathcal{E}_2$ for the rest of this proof, so that $\mathbb{P}(\mathcal{E}_1^c \cup \mathcal{E}_2^c) \leq 2/n$.

 Now
  \begin{align}
    \| \hat{\Sigma}^{1/2} - I_p \|_{\mathrm{op}} &\leq 
    \| \hat{\Sigma} - I_p \|_{\mathrm{op}} 
    \leq \biggl\| \frac{1}{n} \sum_{i=1}^n X_i X_i^\top - I_p \biggr\|_{\mathrm{op}} + \biggl\| \frac{1}{n} \sum_{i=1}^n X_i \biggr\|_2^2    \nonumber \\
    &\leq C \sqrt{\frac{p}{n}} \log^3 (en) \label{eqn:op_norm_small2}
  \end{align}
  for some universal constant $C > 0$. 
Fix an arbitrary $x \in \mathbb{R}^p$ with $\| x \|_{\alpha_0 \hat{K}} = 1$. Then $\| \hat{\Sigma}^{-1/2} x \|_K = 1$ since $\alpha_0 \hat{K} =  \alpha_0 \hat{\Sigma}^{1/2} K_0 = \hat{\Sigma}^{1/2} K$. Thus,
\begin{align*}
  \| x \|_K &\leq  \|\hat{\Sigma}^{-1/2} x \|_K + \|x - \hat{\Sigma}^{-1/2} x \|_K \leq 1 + \frac{1}{r_1}\|x - \hat{\Sigma}^{-1/2} x \|_2 \\          &\leq 1 + \frac{1}{r_1}\|\hat{\Sigma}^{1/2}-I_p\|_{\mathrm{op}}\|\hat{\Sigma}^{-1/2} x \|_2 \leq 1 + r_0 \|\hat{\Sigma}^{1/2}-I_p\|_{\mathrm{op}}.
\end{align*}
Hence $\alpha_0 \hat{K} \subseteq ( 1 + r_0 \| \hat{\Sigma}^{1/2} - I_p \|_{\mathrm{op}} ) K $. Likewise,
\[
  \| x \|_K \geq \|\hat{\Sigma}^{-1/2} x \|_K - \| x - \hat{\Sigma}^{-1/2} x \|_K  \geq 1 - r_0 \| \hat{\Sigma}^{1/2} - I_p \|_{\mathrm{op}} \geq (1 + 2 r_0 \|\hat{\Sigma}^{1/2} - I_p \|_{\mathrm{op}})^{-1},
\]
where the final inequality follows from~\eqref{eqn:op_norm_small2} and the assumption that $C r_0 \sqrt{\frac{p}{n}} \log^3 (en) \leq 1/2$. Thus, we also have that $K \subseteq (1 + 2 r_0 \| \hat{\Sigma}^{1/2} - I_p \|_{\mathrm{op}}) \alpha_0 \hat{K}$. Therefore, by~\eqref{eqn:op_norm_small2} again,
\[
  d_{\mathrm{scale}}(\alpha_0 \hat{K}, K) \leq 2 r_0 \| \hat{\Sigma}^{1/2} - I_p \|_{\mathrm{op}} \lesssim r_0 \sqrt{\frac{p}{n}} \log^3 (en),
\]
Moreover, since $\Sigma = I_p$ and $B_p(0,r_1) \subseteq K \subseteq B_p(0, r_2)$, we find that
\begin{align*}
 p^{1/2} \frac{ \| \hat{\mu}\|_K}{\mathbb{E}_{f_0}^{1/2}( \|X_1\|^2_K )} \leq r_0 \| \hat{\mu}\|_2 \lesssim r_0 \frac{p^{1/2}}{n^{1/2}} \log^{1/2} (en),
\end{align*}
as desired.
\end{proof}

\begin{proof}[Proof of Proposition~\ref{Prop:dscaleOverall}]
Algorithm~\ref{alg:estimateK} is scale equivariant in the sense that for any $\alpha > 0$, if we let $X'_i := \alpha X_i$ for $i \in [n+M]$ and let $\hat{K}, \hat{K}'$ be the resulting outputs of Algorithm~\ref{alg:estimateK} on inputs $\{X_i\}_{i=1}^{n+M}$ and $\{X'_i\}_{i=1}^{n+M}$ respectively, then $\hat{K}' = \alpha \hat{K}$. Since the left-hand side of~\eqref{Eqn:dscaleOverall} is invariant to scaling of $\hat{K}$, we assume without loss of generality that $\mathbb{E}_{f_0}(\| X_1 \|_2) = 1$. We also assume that $\mathbb{E}_{f_0}( \| X_1 \|_K ) = 1$, which can be done without loss of generality by Proposition~\ref{Prop:ExistencePhi} and the fact that the left-hand side of~\eqref{Eqn:dscaleOverall} is invariant to the scaling of $K$.
  
Define $\tilde{K} := \mathrm{conv} \bigl\{ \frac{\theta_1}{\| \theta_1 \|_K}, \ldots, \frac{\theta_m}{\| \theta_m \|_K} \bigr\}$, and define the events
\begin{align*}
  \mathcal{E}^* &:= \biggl\{ d_{\mathrm{scale}}(\hat{K},\tilde{K}) \leq  8 r_0 \Bigl( \frac{M \log^5 (en)}{n} \Bigr)^{1/2} + 4 r_0^2 \Bigl( \frac{\log^{p+1} (en)}{M} \Bigr)^{1/(p-1)}\biggr\} \\
                  \tilde{\mathcal{E}}^* &:= \biggl\{ d_{\mathrm{scale}}(\tilde{K}, K) \leq 64 r^2_0 \Bigl( \frac{\log M}{M} \Bigr)^{1/(p-1)} \biggr\}.
\end{align*}
Under the assumption that $r_0^2 n^{-\frac{1}{p+1}} \log^2 (en) \leq 1/64$, we have that $M/\log M > r_0^{2(p-1)} 64^{p-1}$. Thus, on the event $\mathcal{E}^* \cap \tilde{\mathcal{E}}^*$, we have that $d_{\mathrm{scale}}(\tilde{K}, K) < 1$ and that, by Lemma~\ref{lem:dscale_properties}, 
\begin{align*}
  d_{\mathrm{scale}}(\hat{K},K) &\leq 2d_{\mathrm{scale}}(\hat{K},\tilde{K}) + 2 d_{\mathrm{scale}}(\tilde{K},K) \\
                              & \lesssim r_0 \biggl( \frac{M \log^5 (en)}{n} \biggr)^{1/2} + r_0^2 \biggl( \frac{\log^{p+1} (en)}{M} \biggr)^{1/(p-1)} + r_0^2 \biggl( \frac{\log M}{M} \biggr)^{1/(p-1)} \\
  &\lesssim r_0^2 n^{-1/(p+1)} \log^3 (en).
\end{align*}
By Lemmas~\ref{Lem:Brunel} and~\ref{Lem:tmDeviation}, it holds that $\mathbb{P}(\mathcal{E}^* \cap \tilde{\mathcal{E}}^*) \geq 1 - C_{p,r_0} n^{- p/(p+1)}$ for some $C_{p,r_0} > 0$, depending only on $p$ and $r_0$, as desired.
\end{proof}

\section{Auxiliary lemmas}
\label{Sec:AuxLemma}

\subsection{Auxiliary lemmas for Section~\ref{Sec:Basic}}

Our first result, amongst other things, reveals the density of the random variable $\|X - \mu \|_K$ when $X$ has a density belonging to $\mathcal{F}^{\mathcal{K}}_p$. 
\begin{lemma}
  \label{Lem:ChangeOfVar}
  Let $\mu \in \mathbb{R}^p$, $K \in \mathcal{K}$, and let $g \,:\, [0, \infty) \rightarrow \mathbb{R}$ be integrable. Let $B \in \mathcal{B}([0,\infty))$ and let $A := \{ x \in \mathbb{R}^p \,:\, \|x - \mu \|_K \in B\}$. Then
  \[
    \int_A g(\|x -\mu \|_K) \,dx = p \lambda_p(K) \int_B r^{p-1}g(r) \, dr.
  \]
\end{lemma}
\begin{proof}
 By transforming $x-\mu$ to $x$ if necessary, we may assume without the loss of generality that $\mu = 0$.  Define a measures $\mu_1^K,\lambda_1^K$ on $\bigl([0, \infty), \mathcal{B}([0,\infty))\bigr)$ by $\mu_1^K(B') := \lambda_p( \{ x \in \mathbb{R}^p \,:\, \|x \|_K \in B'\})$ and $\lambda_1^K(B') := p\lambda_p(K)\int_{B'} r^{p-1} \, dr$ for $B' \in \mathcal{B}([0,\infty))$. We claim that $\mu_1^K = \lambda_1^K$.  
 To prove the claim, first consider the case where $B' = [0,c)$ for some $c > 0$.  Then
 \begin{equation}
   \label{Eq:c}
   \mu_1^K\bigl([0,c)\bigr) = c^p \lambda_p(K) = \lambda_p(K) \int_0^c pr^{p-1}  \,dr = \lambda_1^K\bigl([0,c)\bigr).
 \end{equation}
 For $n \in \mathbb{N}$, define $\mathcal{B}_n' := \bigl\{ B'' \in \mathcal{B}([0,n)) \,:\, \mu_1^K(B'') = \lambda_1^K(B'') \bigr\}$. Now $\mu_1^K,\lambda^K_1$ are finite measures on $[0, n)$, and $\mathcal{B}'_n$ is a $\sigma$-algebra that contains $\{[0,c):c \in (0,n]\}$ by~\eqref{Eq:c}.  Hence $\mathcal{B}_n' = \mathcal{B}([0,n))$.  For a general $B' \in \mathcal{B}([0,\infty))$, we have that
  \[
    \mu_1^K(B') = \lim_{n \rightarrow \infty} \mu_1^K\bigl(B' \cap [0, n)\bigr) = \lim_{n \rightarrow \infty} \lambda^K_1 \bigl(B' \cap [0,n)\bigr) = \lambda_1^K( B'),
    \]
and we deduce that $\mu_1^K = \lambda_1^K$, as claimed.
    
  Now suppose that $g$ is a non-negative measurable function, fix $B \in \mathcal{B}([0,\infty))$ and $A = \{x \in \mathbb{R}^p \,:\, \|x \|_K \in B\}$. Let $s_1 \leq s_2 \leq \ldots$ be a sequence of non-negative, simple functions on $[0,\infty)$ such that $s_k \nearrow g$, so that $\int_A s_k( \|x \|_K) \, dx = p\lambda_p(K)\int_B r^{p-1}s_k(r) \, dr$ by our claim.  By two applications of the monotone convergence theorem, we conclude that $\int_A g(\|x \|_K) \, dx = p \lambda_p(K) \int_B r^{p-1}g(r) \, dr$.  The case where $g$ is integrable can be handled by applying this result to the positive and negative parts of $g$.
\end{proof}

\subsection{Auxiliary lemmas for Section~\ref{Sec:Kknown}}

The aim of the next three results is to elucidate the way in which the first two moments of the empirical distribution $\mathbb{Q}_n$ of a set of $n$ data points in $[0,\infty)$ change under the projection $h_a^*$.  These results enable us to show that if the data are drawn independently from a common distribution on $[0,\infty)$, then with high probability, the first two moments of $\hat{h}_n := h_a^*(\mathbb{Q}_n)$ are close to their population analogues.

Our first lemma concerns bounds on $\mu_{\hat{h}_n}$, and is expressed in terms of the function $\rho \equiv \rho_{a,p}:[0,\infty) \rightarrow (0,\infty)$ defined by
\begin{equation}
  \label{Eq:RhoDefinition}
\rho(s) := \frac{(a + s)^{p-1} p s}{(a + s)^p - a^p}.
\end{equation}
Basic properties of the function $\rho$ are given in Lemma~\ref{Lem:RhoProperties}.
\begin{lemma}
  \label{Lem:MLEMeanPreservation}
Fix $a \geq 0$, and suppose that $Z_1,\ldots,Z_n$ are real numbers in the interval $[a,\infty)$ that are not all equal to $a$.  Let $\mathbb{Q}_n$ be the empirical distribution corresponding to $Z_1,\ldots,Z_n$.  Let $\hat{h}_n := h_{a}^*(\mathbb{Q}_n)$, so that $\hat{h}_n(r) = p\lambda_p(K)r^{p-1}e^{\hat{\phi}_n(r)}$ for $r \geq a$, for some $\hat{\phi}_n \in \Phi$. Then, writing $\bar{Z} := n^{-1}\sum_{i=1}^n Z_i$, as well as $r_0 := \sup \{ r \in [a, \infty) : \hat{\phi}_n(r) = \hat{\phi}_n(a) \}$ and $s_0 := r_0  - a$, we have 
\[
\bar{Z}  - \min \biggl\{ \frac{\bar{Z} - a}{\rho(s_0)}, \frac{s_0}{\rho(s_0)} \biggr\}  \leq \mu_{\hat{h}_n} \leq \bar{Z}.
\]
\end{lemma}
\begin{proof}[Proof of Lemma~\ref{Lem:MLEMeanPreservation}]
  By~\eqref{Eq:phibar}, we have that $r_0 = Z_i$ for some $i$ because $\mathbb{Q}_n$ is an empirical distribution. Moreover, the right derivative of $\hat{\phi}_n$ at $r_0$ is strictly negative.  Hence, by Proposition~\ref{Prop:ProjectionBasicProperties}(ii), applied to the functions $\Delta(r) := \pm(r-r_0)_+$, we have
\begin{align}
\label{Eq:PosPart}
\int_{a}^\infty (r-r_0)_+ \hat{h}_n(r) \, dr + r_0 &= \frac{1}{n} \sum_{i=1}^n (Z_i - r_0)_+ + r_0 \nonumber \\
                       &= \frac{1}{n} \sum_{i=1}^n (Z_i - r_0) + \frac{1}{n} \sum_{i=1}^n (Z_i - r_0)_- + r_0 \nonumber \\
                       &= \bar{Z} + \frac{1}{n} \sum_{i=1}^n (r_0 - Z_i)_+.    
\end{align}
Now, since $\hat{\phi}_n(r) = \hat{\phi}_n(a)$ for all $r \in [a,r_0]$, we have 
\[
1 \geq p\lambda_p(K)e^{\hat{\phi}_n(a)} \int_{a}^{r_0} r^{p-1} \, dr = \lambda_p(K)e^{\hat{\phi}_n(a)}(r_0^p - a^p).
\]
We deduce that
\begin{align}
\label{Eq:NegPart}
\int_{a}^\infty (r-r_0)_- \hat{h}_n(r) \, dr &= p\lambda_p(K)e^{\hat{\phi}_n(a)} \int_{a}^{r_0} (r_0 - r) r^{p-1} \, dr \nonumber \\
  &\leq \frac{p}{r_0^p - a^p}\biggl( \frac{r_0^{p+1}}{p} - \frac{r_0 a^p}{p} - \frac{r_0^{p+1}}{p+1} + \frac{a^{p+1}}{p+1} \biggr) \nonumber \\
  &=  r_0 - \frac{p}{p+1} \frac{r_0^{p+1} - a^{p+1}}{r_0^p - a^p} \leq \frac{s_0}{\rho(s_0)},
\end{align}
where we used Lemma~\ref{Lem:EvilInequality} to obtain the final bound.  From~\eqref{Eq:PosPart} and~\eqref{Eq:NegPart}, we find that 
\begin{equation}
\label{Eq:muhatlower}
\mu_{\hat{h}_n} = \int_{a}^\infty (r-r_0)_+ \hat{h}_n(r) \, dr - \int_{a}^\infty (r-r_0)_- \hat{h}_n(r) \, dr + r_0 \geq \bar{Z} + \frac{1}{n} \sum_{i=1}^n (r_0 - Z_i)_+ - \frac{s_0}{\rho(s_0)}.
\end{equation}
In particular, $\mu_{\hat{h}_n} \geq \bar{Z} - \frac{s_0}{\rho(s_0)}$. 

Now, for $i=1,\ldots,n$, let $\tilde{Z}_i := \min(Z_i,s_0+a)$.  Then $n^{-1} \sum_{i=1}^n (r_0 - Z_i)_+ = s_0 - n^{-1} \sum_{i=1}^n (\tilde{Z}_i - a) \geq 0$ and $n^{-1} \sum_{i=1}^n \tilde{Z}_i \leq \bar{Z}$.  Hence
\begin{align}
\label{Eq:muhatlower2}
  \frac{s_0}{\rho(s_0)} - \frac{1}{n} \sum_{i=1}^n &(r_0 - Z_i)_+ \nonumber \\
&=  \frac{s_0 - n^{-1} \sum_{i=1}^n (\tilde{Z}_i - a)}{\rho(s_0)}
    + \frac{n^{-1} \sum_{i=1}^n (\tilde{Z}_i - a)}{\rho(s_0)} - \biggl(s_0 - \frac{1}{n} \sum_{i=1}^n (\tilde{Z}_i - a)\biggr) \nonumber \\
  &\leq  \biggl(s_0 - \frac{1}{n} \sum_{i=1}^n (\tilde{Z}_i - a)\biggr) \biggl( \frac{1}{\rho(s_0)} - 1 \biggr)
    + \frac{\bar{Z} - a}{\rho(s_0)} \leq \frac{\bar{Z} - a}{\rho(s_0)},
\end{align}
where the final inequality follows from Lemma~\ref{Lem:RhoProperties}(iii).  The second lower bound for $\mu_{\hat{h}_n}$ follows from~\eqref{Eq:muhatlower} and~\eqref{Eq:muhatlower2}.  The upper bound on $\mu_{\hat{h}_n}$ follows from Proposition~\ref{Prop:ProjectionBasicProperties}(iii).
\end{proof}
We now study bounds for $\sigma_{\hat{h}_n}$, and their consequences for $\sup_{r \geq a} \log \hat{h}_n(r)$.
\begin{lemma}
  \label{Lem:MLEVarPreservation}
  Let $a \geq 0$ and let $Q \in \mathcal{Q}_a$.  Suppose that there exists $A > 0$ such that $\sup \bigl\{ \frac{Q(D)}{\lambda_1(D)} \,:\, D \subseteq \mathbb{R} \, \textrm{compact, convex},\, Q(D) \geq 1/2 \bigr\} \leq A$. Let $h^* := h_{a}^*(Q)$, and $\ell_{h^*} := \int_{a}^\infty \log h^* \, dQ$.  Then there exists a universal constant $C_{\sigma} > 0$ such that
  \[
 \frac{1}{C_\sigma} \min(1, e^{\ell_{h^*} - \log A})  \leq A \sigma_{h^*} \leq C_\sigma e^{-(\ell_{h^*} - \log A)}.
  \]
Moreover,
  \begin{equation}
\label{Eq:Sup}
    \sup_{r \geq a} \log h^*(r) - \log A \leq \max\bigl\{4\log 2, -(\ell_{h^*}-\log A) + 12 \log 2 - 2\bigr\}.
  \end{equation}
\end{lemma}

\begin{proof}
[Proof of Lemma~\ref{Lem:MLEVarPreservation}]
We have
\begin{align*}
     \sup_{r \geq a} \log h^*(r) \geq \int_{a}^\infty \log h^* \, dQ \geq \ell_{h^*}.
\end{align*}
Since $h^*$ is upper semi-continuous, there exists $r_0 \geq a$ such that $\log h^*(r_0) \geq \ell_{h^*}$. By Lemma~\ref{Lem:VarianceSupRelation}(ii), we then have that $\sigma_{h^*} \leq C' e^{- \ell_{h^*}}$ for some universal constant $C' > 0$. 

To provide the lower bound for $A \sigma_{h^*}$, we first prove~\eqref{Eq:Sup}. To that end, let $M := \sup_{r \geq a} \log h^*(r)$ and let $t := e^M/(2^4 A)$. If $t < 1$, then the claim immediately follows, so let us thus assume $t \geq 1$. Define $D_{M - t} := \{ r \in [a, \infty) \,:\, \log h^*(r) \geq M - t \}$ and suppose first that $Q(D_{M - t}) < 1/2$. Then, 
\begin{align*}
  \ell_{h^*} &\leq MQ(D_{M-t}) + (M-t)Q(D_{M-t}^c) \\
  &\leq M - t + t Q(D_{M-t}) 
    < M - \frac{t}{2} \\
  &\leq M - \log(2^5 A) - e^{ M - \log(2^5 A)} + \log(2^5 A) \\
  &\leq (\log 2 - 2) - (M - \log(2^5 A) - \log 2) + \log(2^5 A),
\end{align*}
where the final inequality follows because the univariate function $g(x) = x - e^x$ is concave and thus $g(x) \leq g(x_0) + g'(x_0)(x - x_0)$ for any $x, x_0 \in \mathbb{R}$; we take $x$ and $x_0$ to be $M - \log(2^5 A)$ and $\log 2$ respectively. Thus,
\begin{align}
  M - \log A \leq - (\ell_{h^*}  - \log A) + 12 \log 2 - 2. \label{eqn:MUpperBound}
\end{align}
If on the other hand, $Q(D_{M-t}) \geq 1/2$, then we may use \citet[][Lemma~4.1]{dumbgen2011approximation} and the assumption that $t \geq 1$ to obtain
\begin{align*}
  \ell_{h^*}
  &\leq M - t + t Q(D_{M-t})
    \leq M - t + t A \lambda_1(D_{M-t}) \\
  & \leq M - t + 4t^2 e^{-M} A \leq M - \frac{t}{2} \\
  & \leq (\log 2 - 2) - (M - \log(2^5 A) - \log 2) + \log(2^5 A),
\end{align*}
and thus~\eqref{eqn:MUpperBound} also follows. Therefore~\eqref{Eq:Sup} holds in all cases. By Lemma~\ref{Lem:VarianceSupRelation}(i), there exists a universal constant $C'' > 0$ such that $A \sigma_{h^*} \geq C'' \min(1, e^{\ell_{h^*} - \log A})$, as desired. 
\end{proof}

\begin{lemma}
  \label{Lem:QLebesgueRatioBound}
  Let $h$ be a density on $\mathbb{R}$ and suppose $\| h \|_{\mathrm{esssup}} < \infty$. Let $Z_1,\ldots,Z_n \stackrel{\mathrm{iid}}\sim h$ with empirical distribution $\mathbb{Q}_n$. We have that 
  \[
  \mathbb{P}\biggl( \sup \biggl\{ \frac{\mathbb{Q}_n(C)}{\lambda_1(C)} \,:\, C\, \textrm{compact, convex},\, \mathbb{Q}_n(C) \geq 1/2 \biggr\} > 2 \| h \|_{\mathrm{esssup}} \biggr) \leq 2 e^{-n/128}.
  \]
\end{lemma}

\begin{proof}
We may assume that $n \geq 8$ since the bound is trivially true otherwise, and we also assume that $Z_1,\ldots,Z_n$ are distinct (an event of probability 1). Let $H$ and $\mathbb{H}$ denote the distribution function of $h$ and $\mathbb{Q}_n$ respectively. Define the event $E := \{ \| H - \mathbb{H} \|_\infty \leq 1/16 \}$ and observe that $\mathbb{P}(E^c) \leq 2e^{-n/128}$ by the Dvoretzky--Kiefer--Wolfowitz inequality. On the event $E$, for any $a, b \in \mathbb{R}$ with $a < b$ and $\mathbb{Q}_n([a, b]) \geq 1/2$, we have that
  \begin{align*}
    \frac{1}{2} &\leq \mathbb{Q}_n([a,b]) \leq \frac{1}{n} + \mathbb{Q}_n((a, b]) \\
                &= \frac{1}{n} + \mathbb{H}(b) - \mathbb{H}(a) - (H(b) - H(a)) + \int_a^b h(t) dt \\
                &\leq \frac{1}{n} + 2 \| \mathbb{H} - H \|_\infty + \| h \|_{\textrm{esssup}} | b - a| 
                \leq \frac{1}{4} + \| h \|_{\textrm{esssup}} | b - a |.
  \end{align*}
  Thus, we have that $ \| h \|_{\textrm{esssup}} | b - a | \geq 1/4$ and hence
  \begin{align*}
    \mathbb{Q}_n([a,b]) \leq \frac{1}{4}  + \| h \|_{\textrm{esssup}} | b - a | \leq 2 \| h \|_{\textrm{esssup}} | b - a |,
  \end{align*}
  as desired.
\end{proof}
We are now in a position to argue that, with high probability, $\hat{h}_n$ belongs to a subclass of $\mathcal{H}_a$ with restricted first two moments.  These moment restrictions are important for enabling us to obtain the bracketing entropy bounds that drive the rates of convergence of the $K$-homothetic log-concave MLE.  For $C_\mu,C_\sigma > 0$, $a_0 \geq 0$ and $h_0 \in \mathcal{H}_{a_0}$, let
\begin{align}
\mathcal{H}_{a_0}(h_0,C_\mu,C_\sigma) := \biggl\{h \in \mathcal{H}_{a_0} : |\mu_h - \mu_{h_0}| \leq C_\mu\sigma_{h_0} \, , \, \frac{1}{C_\sigma} \leq \frac{\sigma_h}{\sigma_{h_0}} \leq C_{\sigma}\biggr\}. \label{Eqn:MomentClass}
\end{align}
\begin{lemma}
  \label{Lem:MLEMeanVarPreservation}
Let $a_0 \geq 0$, fix a density $h_0 \in \mathcal{H}_{a_0}$ with $\sigma_{h_0}=1$, and suppose that $Z_1,\ldots,Z_n \stackrel{\mathrm{iid}}{\sim} h_0$, with empirical distribution $\mathbb{Q}_n$.  Writing $\hat{h}_n := h_{a_0}^*(\mathbb{Q}_n)$, there exist universal constants $C_\mu, C_\sigma, C > 0$ such that
  \begin{align*}
   \mathbb{P}\bigl(\hat{h}_n \notin \mathcal{H}_{a_0}(h_0,C_\mu,C_\sigma)\bigr) \leq \frac{C}{n}.
  \end{align*}
\end{lemma}
\begin{proof}[Proof of Lemma~\ref{Lem:MLEMeanVarPreservation}]
  We may assume that $n \geq 500$.  Let $E := \{ | \bar{Z} - \mu_{h_0} | \leq 1\}$, so that $\mathbb{P}(E^c) \leq 1/n$, by Chebychev's inequality.  On the event $E$, we have  $\mu_{\hat{h}_n} \leq \bar{Z} \leq \mu_{h_0} + 1$ by Lemma~\ref{Lem:MLEMeanPreservation}.  Recall the definition of $s_0$ from Lemma~\ref{Lem:MLEMeanPreservation}.  If $s_0 \leq 1$, then by Lemma~\ref{Lem:MLEMeanPreservation}, on the event $E$,
  \[
    \mu_{\hat{h}_n} \geq \bar{Z} - \frac{s_0}{\rho(s_0)} \geq \bar{Z} - 1 \geq \mu_{h_0} - 2,
  \]
  where the middle inequality uses the fact that $\rho(s_0) \geq 1$ (Lemma~\ref{Lem:RhoProperties}(iii)). If $s_0 > 1$, then by Lemma~\ref{Lem:MLEMeanPreservation}
, on the event $E$,
\begin{align*}
  \mu_{\hat{h}_n} \geq \bar{Z} - \frac{\bar{Z} - a_0}{\rho(s_0)} &\geq \mu_{h_0} - 1 - \frac{\mu_{h_0} + 1 - a_0}{\rho(s_0)} \\
                               & \geq \mu_{h_0} - 2 - \frac{\mu_{h_0} - a_0}{\rho(s_0)}.
\end{align*}
Hence, if $\rho(s_0)/s_0 < 2^{-7}$, then by Lemma~\ref{Lem:UnitVarianceMeanBound}, we have $\mu_{\hat{h}_n} - \mu_{h_0} \geq -2 - 2^{12} \geq -2^{13}$.  On the other hand, if $\rho(s_0)/s_0 \geq 2^{-7}$, then by Lemma~\ref{Lem:MLEMeanPreservation},
\[
\mu_{\hat{h}_n} - \mu_{h_0} \geq \bar{Z} - \mu_{h_0} - \frac{s_0}{\rho(s_0)} \geq -1 - 2^7 \geq -2^8.
\]
It follows that there exists a universal constant $C_\mu > 0$ such that
\[
\mathbb{P}\bigl(|\mu_{\hat{h}_n} - \mu_{h_0}| > C_\mu\bigr) \leq \mathbb{P}(E^c) \leq 1/n.
\]
To bound $\sigma_{\hat{h}_n}$, define the event $E' := \bigl\{ \int_{a_0}^\infty \log \hat{h}_n \,d \mathbb{Q}_n \geq -3 \bigr \}$. By Lemma~\ref{Lem:DifferentialEntropyBound} and \citet[][Theorem~1.1]{bobkov2011concentration}, for $n \geq 500$,
  \begin{align*}
  \mathbb{P}(E')  
    &\geq \mathbb{P}\biggl( \frac{1}{n} \sum_{i=1}^n \log h_0(Z_i) \geq -3 \biggr) \\
    & \geq \mathbb{P}\biggl(\biggl|\frac{1}{n} \sum_{i=1}^n \log h_0(Z_i) - \int_{a_0}^\infty h_0 \log h_0\biggr| \leq 1\biggr) 
    \geq 1 - 2e^{-n^{1/2}/16} \geq 1- \frac{250}{n}.
  \end{align*}
Define event $E'' := \bigl\{ \sup \{ \frac{\mathbb{Q}_n(C)}{\lambda_1(C)} \,:\, C \textrm{ compact, convex},\, \mathbb{Q}_n(C) \geq 1/2 \} \leq 2 \| h_0 \|_\infty \bigr\}$. By Lemma~\ref{Lem:QLebesgueRatioBound}, we have that $\mathbb{P}(E'') \geq 1 - 2 e^{-n/128}$. On the event $E' \cap E''$, by Lemma~\ref{Lem:MLEVarPreservation} and~\ref{Lem:VarianceSupRelation}, there exists a universal constant $C_{\sigma} \geq 1$ such that $C_{\sigma}^{-1} \leq \sigma_{\hat{h}_n} \leq C_\sigma$. The desired result follows from a union bound.
\end{proof}
The mean $\mu_h$ of any $h \in \mathcal{H}_a$ is constrained because $h(r) = p\lambda_p(K)r^{p-1}e^{\phi(r)}$ for $r \geq a$ and some decreasing function $\phi$. The next lemma formalises this notion.
\begin{lemma}
  \label{Lem:UnitVarianceMeanBound}
  Let $a \geq 0$ and let $h \in \mathcal{H}_{a}$.  Then, writing $s^* := \sup \{ s \in (0,\infty) : \frac{\rho_{a/\sigma_h,p}(s)}{s} \geq 2^{-7} \}$, we have
\[
a \leq \mu_h \leq a + 2^{12}\sigma_h\rho_{a/\sigma_h,p}(s^*).
\]
\end{lemma}
\begin{remark}
Even though we cannot obtain an analytic expression for $\rho(s^*)$, we can apply the bounds developed in Lemma~\ref{Lem:RhoProperties} to control $\mu_h$.  For example, since $\rho(s) \leq p$ for any $a \geq 0$ by Lemma~\ref{Lem:RhoProperties}(iii), we have that $\mu_h - a \lesssim \sigma_h p$. When $a = 0$, this bound is sharp up to the universal constant, because taking $e^{\phi(r)} = pb^{-p}\mathbbm{1}_{\{r \in [0,b]\}}$, where $b = (p+1)(p+2)^{1/2}/p^{1/2}$, yields $\mu_h = p^{1/2}(p+2)^{1/2}$ and $\sigma_h=1$.
\end{remark}
\begin{proof}
We initially assume that $\sigma_h = 1$ and write $\rho(\cdot) = \rho_{a,p}(\cdot)$.  If $h \in \mathcal{H}_{a}$, then $h(r) = 0$ for $r \in (-\infty, a)$ and thus, $\mu_h \geq a$.  For the upper bound on $\mu_h$, we first observe that $h(\mu_h) \geq 2^{-7}$ by Lemma~\ref{Lem:VarianceSupRelation}.  Since $\rho(s)/s$ is decreasing by Lemma~\ref{Lem:RhoProperties}(ii) and $1 \leq \rho(s) \leq p$ by Lemma~\ref{Lem:RhoProperties}(iii), we have that $s^* \in (0,\infty)$.  Suppose for a contradiction that $\mu_h > a + 2^{12}\rho(s^*)$.  Then, since $\rho$ is increasing  by Lemma~\ref{Lem:RhoProperties}(i), we have $\rho(s)/(\mu_h - a) < 2^{-12}$ for all $s \leq s^*$. Moreover, by definition of $s^*$, we have $\rho(s)/s < 2^{-7}$ for all $s > s^*$. Hence $\sup_{s \in (0,\infty)} \min \bigl(32 \frac{\rho(s)}{\mu_h - a}, \frac{\rho(s)}{s} \bigr) < 2^{-7}$.  But then Lemma~\ref{Lem:MeanConstrainedSupBound} establishes a contradiction, so $\mu_h \leq a + 2^{12} \rho(s^*)$, as desired.  For general $\sigma_h \in (0,\infty)$, we apply the above argument to $g(\cdot) := \sigma_hh(\sigma_h\cdot)$, which satisfies $\sigma_g = 1$ and $g \in \mathcal{H}_{a/\sigma_h}$.
\end{proof}

The next lemma is used in the proof of Lemma~\ref{Lem:UnitVarianceMeanBound}.
\begin{lemma}
  \label{Lem:MeanConstrainedSupBound}
For any $a \geq 0$, $h \in \mathcal{H}_{a}$ and any $s \in (0, \infty)$, 
For any $a \geq 0$, $h \in \mathcal{H}_a$,
   \[
   \sup_{s \geq a} h(a + s) \leq \sup_{s \in [0, \infty)} \min\biggl\{ \frac{32\rho(s)}{\mu_h - a}, \frac{\rho(s)}{s} \biggr\}.
   \]
\end{lemma}
\begin{proof}
  Let us fix $h \in \mathcal{H}_{a}$ and define $\tilde{h}(s) := h(a + s)$. Observe that $\tilde{h}$ is a density of the form $\tilde{h}(s) = p\lambda_p(K)(a + s)^{p-1}e^{\phi(s)}$ for some $\phi \in \Phi_0$.  We will show that $\tilde{h}(s') \leq \rho(s')/s'$ for all $s' \in (0,\infty)$ and if $s_0 \in [0,\infty)$ is such that $\tilde{h}(s_0) = \sup_{s \in [0,\infty)} \tilde{h}(s)$, then $\tilde{h}(s_0) \leq 32\rho(s_0)/\mu_{\tilde{h}}$ (such an $s_0$ exists by the upper semi-continuity of $\phi$). The lemma then follows since $\mu_{\tilde{h}} = \mu_h - a$. 

  To this end, fix $s' \in (0, \infty)$ and define $\tilde{h}_{s'}(s) := p\lambda_p(K)\alpha(a+s)^{p-1}\mathbbm{1}_{\{s \in [0, s']\}}$ where $\alpha^{-1} := p\lambda_p(K)\int_0^{s'} (a + s)^{p-1} \, ds = \lambda_p(K)\{(a + s')^p - a^p\}$.  Then
\begin{align*}
1 = \int_0^{s'} \tilde{h}_{s'} \geq \int_0^{s'} \tilde{h} \geq p\lambda_p(K)e^{\phi(s')} \int_0^{s'} (a + s)^{p-1} \, ds = \frac{e^{\phi(s')}}{\alpha}.
  \end{align*}
Hence $\tilde{h}(s') \leq \tilde{h}_{s'}(s') = p\lambda_p(K)\alpha (a + s')^{p-1} = \rho(s')/s'$, as desired. 

To prove the second claim, fix $s_0 \in [0, \infty)$ such that $\tilde{h}(s_0) = \sup_{s \in [0, \infty)} \tilde{h}(s)$. Observe that if $s_0 \geq \mu_{\tilde{h}}/32$, then $\rho(s_0)/s_0 \leq 32\rho(s_0)/\mu_{\tilde{h}}$ and lemma follows.  We may therefore assume that $s_0 \in [0, \mu_{\tilde{h}}/32)$, define $M := \log \tilde{h}(s_0)$ and fix $m \in (-\infty, M-2]$.  For $t \in [m,M]$, let $D_t := \{s \in [0,\infty) \,:\, \log \tilde{h}(s) \geq t\}$.  Since $\tilde{h}$ is itself a log-concave density, we have that, any $t \in [m,M]$ and $s \in D_m$, 
\[
\log \tilde{h} \biggl(\frac{t-m}{M-m}s_0 + \frac{M-t}{M-m}s\biggr) \geq \frac{(t-m)M}{M-m} + \frac{(M-t)m}{M-m} = t.
\]
Hence
\[
\lambda_1(D_t) \geq \lambda_1 \biggl(\frac{t-m}{M-m}s_0 + \frac{M-t}{M-m}D_m\biggr) = \frac{M-t}{M-m}\lambda_1(D_m).
\]
Using Fubini's theorem as in \citet[][Lemma~4.1]{dumbgen2011approximation}, we can now compute
\begin{align*}
1 &\geq \int_{D_m} \tilde{h}(s) - e^m \, ds \geq \int_{D_m} \int_m^M e^t \mathbbm{1}_{\{\log \tilde{h}(s) \geq t\}} \, dt \, ds \\
&= \int_m^M e^t \lambda_1(D_t) \, dt \geq \frac{\lambda_1(D_m)}{M-m}\int_m^M (M-t)e^t \, dt = \frac{\lambda_1(D_m)e^M}{M-m}\int_0^{M-m} xe^{-x} \, dx \\
&\geq \frac{\lambda_1(D_m)e^M}{2(M-m)}.
\end{align*}                
Since $D_m$ is an interval containing $s_0$, we conclude that $\log \tilde{h}(s) \leq m$ whenever $|s-s_0| > 2(M-m)e^{-M}$.  Thus, for $|s-s_0| > 4e^{-M}$, we have
\[
\log \tilde{h}(s) \leq \inf\bigl\{m \in (-\infty,M-2] :2(M-m)e^{-M} < |s-s_0|\bigr\} = M - \frac{|s-s_0|e^M}{2}.
\]
Assume for the sake of contradiction that $s_0 - 4 e^{-M} > 0$.  Then
\begin{align}
  \mu_{\tilde{h}} = \int_0^\infty s \tilde{h}(s) \, ds
  &\leq \int_0^{s_0 - 4e^{-M}} s \exp\biggl\{M - \frac{(s_0-s)e^M}{2}\biggr\} \, ds + \int_{s_0 - 4e^{-M}}^{s_0 + 4e^{-M}} s \, ds  \sup_{s \in [0,\infty)} \tilde{h}(s)  \nonumber \\  
  &\hspace{3cm} + \int_{s_0 + 4e^{-M}}^\infty s \exp\biggl\{M - \frac{(s-s_0)e^M}{2}\biggr\} \, ds \nonumber \\ 
  &\leq 2\int_2^{s_0e^M/2} \biggl(s_0 - \frac{2t}{e^M}\biggr) e^{-t} \, dt + 8s_0 + 2\int_2^\infty \biggl(s_0 + \frac{2t}{e^M}\biggr) e^{-t} \, dt \nonumber \\  
  &\leq 12s_0. \nonumber
\end{align}
This is a contradiction since we assumed that $\mu_{\tilde{h}} \geq 32 s_0$. Thus $s_0 - 4e^{-M} \leq 0$. We deduce that 
\begin{align*}
  \mu_{\tilde{h}} &\leq \int_0^{s_0 + 4e^{-M}} s\,ds \sup_{s \in [0,\infty)} \tilde{h}(s) + s_0 + 2e^{-M} \\
                  &\leq \frac{e^M}{2}(s_0 + 4e^{-M})^2 + s_0 + 2e^{-M} \leq 7 s_0 + 10 e^{-M} \leq 7 s_0 + 10 e^{-M}\rho(s_0).
\end{align*}
Since $s_0 \leq \mu_{\tilde{h}}/32$, we obtain $e^M \leq 13\rho(s_0)/\mu_{\tilde{h}}$ as desired. 
\end{proof}
The next lemma is a very slight generalisation of \cite[][Theorem~4]{kim2016global} and can be proved in the same manner, with  minor modifications to handle the general mean and variance perturbation.  The proof is omitted for brevity.
\begin{lemma} \citet[][Theorem 4]{kim2016global}
  \label{Lem:NonlocalBracketingEntropy}
Fix $C_\mu > 0, C_\sigma \geq 1$ and $h_0 \in \mathcal{H}_a$.  There exists $C > 0$, depending only on $C_\mu, C_\sigma$, such that for every $\epsilon > 0$,
  \[
    H_{[]}\bigl( \epsilon, \mathcal{H}_{a_0}(h_0, C_\mu, C_{\sigma}), d_{\mathrm{H}}\bigr) \leq \frac{C}{\epsilon^{1/2}}.
  \]
\end{lemma}

\begin{lemma}
  \label{Lem:ExtremeEventControl}
Let $a \geq 0$ and let $Z_1,\ldots, Z_n \stackrel{\mathrm{iid}}{\sim} h_0 \in \mathcal{H}_{a}$ with $\sigma_{h_0}=1$ and empirical distribution $\mathbb{Q}_n$.  Let $\hat{h}_n := h^*(\mathbb{Q}_n)$.  Then for $n \geq 8$ and $t \geq 8$,
\[
\mathbb{P}\biggl(\sup_{r \geq a} \log \hat{h}_n(r) > t\log n\biggr) \leq \biggl( \frac{6}{n^{t/32}} \biggr)^n + 2n^{-t n^{1/2}/64}.
\]
\end{lemma}
\begin{proof}
  Let $\mathbb{Q}_n$ denote the empirical distribution of $Z_1,\ldots,Z_n$ and define the event
  \[
    E_1 := \biggl\{ \sup \biggl\{ \frac{\mathbb{Q}_n(D)}{\lambda_1(D)} \,:\, D \subseteq \mathbb{R} \textrm{ compact, convex, } \mathbb{Q}_n(D) \geq \frac{1}{2} \biggr\} \leq \frac{n^{t/4}}{2^4}
  \biggr\}.
\]
Observe that $E_1$ occurs only if, for every closed interval $D$ of length at most $2^5 n^{-t/4}$, it holds that $\mathbb{Q}_n(D) < 1/2$. Since $\| h_0 \|_{\infty} \leq 1$ by \citet[][Proposition~S2(iii)]{feng2018multivariate},  we have that for $n,t \geq 8$,

\begin{align*}
  \mathbb{P}(E_1^c) &\leq \mathbb{P}\Biggl(\bigcup_{i=1}^n \, \bigcup_{\substack{ S \subseteq \{1,\ldots,n\} \setminus \{i\} \\ |S| = \lceil n /2 \rceil - 1} } \, \bigcap_{j \in S} \{Z_j \in [Z_i,Z_i+2^5n^{-t/4}]\} \Biggr)  \\
                    &\leq n \binom{n-1}{ \lceil n/2 \rceil -1 } \mathbb{P} \biggl( \bigcap_{j=2}^{ \lceil n/2\rceil } \{ Z_j \in [Z_1, Z_1 + 2^5 n^{-t/4}] \} \biggr) \\
                    &\leq n \binom{n-1}{ \lceil n/2 \rceil - 1 } (2^5 n^{-t/4})^{\lceil n/2 \rceil - 1} \leq 2^{5n/2} n^{- \frac{t}{4} \lceil \frac{n}{2} \rceil + \frac{t}{4} + \lceil \frac{n}{2} \rceil} 
  \leq \biggl( \frac{6}{n^{t/32}} \biggr)^n
\end{align*}

Now let $E_{2} := \bigl\{\int_{a}^\infty  \log h_0 \, d\mathbb{Q}_n > -(t/2) \log n + 4 \log 2 - 2 \bigr\}$, so by Lemma~\ref{Lem:DifferentialEntropyBound} and \citet[][Theorem~1.1]{bobkov2011concentration},
  \begin{align*}
\mathbb{P}(E_2^c) &\leq \mathbb{P}\biggl( \int_a^\infty \log h_0 \, d\mathbb{Q}_n - \int_a^\infty h_0 \log h_0  \leq \frac{1}{2} + \frac{1}{2}\log(2\pi) - \frac{t}{2} \log n + 4 \log2 - 2\biggr) \\
 &\leq \mathbb{P}\biggl(\biggl|\int_a^\infty \log h_0 \, d\mathbb{Q}_n - \int_a^\infty h_0 \log h_0\biggr| \geq \frac{t}{4} \log n\biggr) \leq 2n^{-t n^{1/2}/64}.
\end{align*}
Applying Lemma~\ref{Lem:MLEVarPreservation} with $h^* = \hat{h}_n$, $A = n^{t/4}/2^4$ and observing that $\ell_{h^*} \geq \int_a^\infty \log h_0 \, d\mathbb{Q}_n \geq -(t/2) \log n + 4 \log 2 - 2$ on $E_1 \cap E_2$, we find that on this event, for $n,t \geq 8$,
\[
\sup_{r \geq a} \log \hat{h}_n(r) \leq (4 \log 2 + \log A) \vee (- \ell_{h^*} + 2 \log A + 12 \log 2 -2) \leq t \log n,
\]
as desired.
\end{proof}

\begin{lemma} 
  \label{Lem:ExtremeEventControl2}
Let $a \geq 0$, let $h_0 \in \mathcal{H}_{a}$ with $\sigma_{h_0}=1$, and suppose that $Z_1,\ldots,Z_n \stackrel{\mathrm{iid}}{\sim} h_0$.  Writing $Z_{(1)} := \min_i Z_i$ and $Z_{(n)} := \max_i Z_i$, there exists a universal constant $C > 0$ such that for $t \geq 4$,
  \[
\mathbb{P}\biggl(\sup_{r \in [Z_{(1)}, Z_{(n)}]} \log \frac{1}{h_0(r)} \geq t \log n\biggr) \leq Cn^{-(t-1)}.
  \]
\end{lemma}
\begin{proof}
This result follows from the proof of \citet[][Lemma~2]{kim2016adaptationsupp}.
\end{proof}

\subsubsection{Local bracketing entropy results}

For $h_0 \in \mathcal{H}_a$ and $\delta > 0$, let 
\[
  \mathcal{H}(h_0, \delta) := \bigl\{ h \in \mathcal{H}_a : h \ll h_0, d_{\mathrm{H}}(h,h_0) \leq \delta \bigr\}.
\]
Recall from the introduction that $\mathcal{F}_p$ denotes the class of all upper semi-continuous, log-concave densities on $\mathbb{R}^p$. For $f_0 \in \mathcal{F}_1$, we define
\begin{align}
  \tilde{\mathcal{F}}(f_0, \delta) := \bigl\{ f \in \mathcal{F}_1 \,:\, \log (f/f_0) \textrm{ concave}, f \ll f_0, d_{\mathrm{H}}(f,f_0) \leq \delta\bigr\}, \label{eqn:local_density_ball}
\end{align}
where we adopt the convention that $0/0 := 0$.  
\begin{lemma}
  \label{Lem:LocalBracketing}
  Fix $a \geq 0$ and let $h_0 \in \mathcal{H}_{a}$. Assume that $\nu := \inf \{ d_{\mathrm{H}}(h_0, h) \,:\, h \in \mathcal{H}_{a}^{(1)}, h_0 \ll h \} < 2^{-9}$. Then there exists a universal constant $C > 0$ such that, for all $\delta \in (0, 2^{-9} - \nu)$ and all $\epsilon > 0$,
  \[
    H_{[]}\bigl(\epsilon, \mathcal{H}(h_0, \delta), d_{\mathrm{H}}\bigr) \leq
    C \biggl( \frac{\delta + \nu}{\epsilon} \biggr)^{1/2} \log^{5/4} \frac{1}{\delta}.
  \]
\end{lemma}
\begin{proof}
Fix $h_1 \in \mathcal{H}_a^{(1)}$ with $h_0 \ll h_1$ and let $\nu_1 := d_{\mathrm{H}}(h_0, h_1)$.  Then, by the triangle inequality,
  \[
    \mathcal{H}(h_0, \delta) \subseteq \mathcal{H}(h_1, \delta + \nu_1).
  \]
Since $h_1(r) = p\lambda_p(K)r^{p-1} e^{\phi_1(r)}$ where $\phi_1 \in \Phi^{(1)}$, we have that $ \log (h/h_1)$ is concave for any $h \in  \mathcal{H}(h_1, \delta + \nu_1)$. We therefore have
  \[
    \mathcal{H}(h_1, \delta + \nu_1) \subseteq \tilde{\mathcal{F}}(h_1, \delta + \nu_1)
  \]
  where the right-hand side is defined in~\eqref{eqn:local_density_ball}. It then follows from this and Lemma~\ref{Lem:GeneralLocalBracketing} that
  \[
  H_{[]}\bigl(\epsilon, \mathcal{H}(h_0, \delta), d_{\mathrm{H}}\bigr) \lesssim \biggl( \frac{\delta + \nu_1}{\epsilon} \biggr)^{1/2} \log^{5/4} \frac{1}{\delta}.
  \]
  Since the choice of $h_1 \in \mathcal{H}_a^{(1)}$ with $h_0 \ll h_1$ was arbitrary, the bound continues to hold when $\nu_1$ is replaced with $\nu$.
\end{proof}

Recall that $\mathcal{F}_1$ denotes the set of all upper semi-continuous log-concave densities on $\mathbb{R}$.
\begin{lemma}
  \label{Lem:HellingerMomentBound}
  Let $f, g \in \mathcal{F}_1$. Then there exist universal constants $C'_\mu > 0, C'_\sigma \geq 1$ such that if $d_{\mathrm{H}}(f,g) \leq 2^{-7}$, then 
  \[
    \frac{1}{C'_\sigma} \leq \frac{\sigma_g}{\sigma_f} \leq C'_\sigma,
    \qquad
    |\mu_g - \mu_f| \leq C'_\mu. 
  \]
\end{lemma}
\begin{proof}
Since the Hellinger distance is affine invariant, we may assume without loss of generality that $\mu_f = 0$ and $\sigma_f = 1$.  By \citet[][Theorem 5.14(a) and~(d)]{lovasz2007geometry}, $f(x) \geq 2^{-8}$ for $x \in [-1/9, 1/9]$. We claim that $g(x) \geq 2^{-12}$ for some $x \in [-1/9, 1/9]$. To see this, suppose for a contradiction that $g(x) < 2^{-10}$ for all $x \in [-1/9, 1/9]$. Then
  \begin{align*}
    \int_{-\infty}^\infty (f^{1/2} - g^{1/2})^2 \geq \int_{-1/9}^{1/9} ( 2^{-4} - 2^{-5} )^2 \, dx \geq (2/9) 2^{-10} > 2^{-14},
  \end{align*}
a contradiction.  By Lemma~\ref{Lem:VarianceSupRelation}, it follows that $\sigma_g \leq C_\sigma'$ for some universal constant $C'_{\sigma} > 0$.  The lower bound on $\sigma_g$ follows by symmetry.

Now assume without loss of generality that $\mu_g \geq 0$.  By the first part and \citet[][Proposition~S2(iii)]{feng2018multivariate}, $g(x) \leq C_\sigma'e^{- \mu_g/C_\sigma' + 1}$ for all $x \leq 0$.  It follows that if $\mu_g \geq C_\sigma'(1 + 10\log 2 + \log C_\sigma')$, then 
\[
\int_{-\infty}^\infty (f^{1/2} - g^{1/2})^2 \geq \int_{-1/9}^0 ( 2^{-4} - 2^{-5} )^2 \, dx \geq (1/9) 2^{-10} > 2^{-14},
\]
a contradiction.  The result follows.
\end{proof}
We now prove a general result on the local bracketing entropy of log-concave densities.

\begin{lemma}
  \label{Lem:GeneralLocalBracketing}
  Let $\delta \in (0,2^{-8}]$ and $f_0 \in \mathcal{F}_1$. Then there exists a universal constant $C > 0$ such that, for every $\epsilon > 0$,
  \[
    H_{[]}\bigl(\epsilon, \tilde{\mathcal{F}}(f_0, \delta), d_{\mathrm{H}}\bigr) \leq C\frac{\delta^{1/2}}{\epsilon^{1/2}} \log^{5/4} \frac{1}{\delta}.
  \]
\end{lemma}
\begin{proof} 
  In this proof, we let $C>0$ be a generic universal constant whose value may vary from instance to instance. For a set $S \subseteq \mathbb{R}$, we define $\tilde{\mathcal{F}}(f_0, \delta, S) := \{ f|_S \,:\, f \in \tilde{\mathcal{F}}(f_0, \delta)\}$ and abuse notation slightly to define 
  $H_{[]}(\epsilon, \tilde{\mathcal{F}}(f_0, \delta), d_{\mathrm{H}}, S) := H_{[]}(\epsilon, \tilde{\mathcal{F}}(f_0, \delta, S), d_{\mathrm{H}}, S)$. We note that, for any $\epsilon_1, \epsilon_2 > 0$ and disjoint Borel measurable sets $S_1, S_2 \subseteq \mathbb{R}$,
  \[
    H_{[]}\bigl( (\epsilon_1^2 + \epsilon_2^2)^{1/2}, \tilde{\mathcal{F}}(f_0, \delta), d_{\mathrm{H}}, S_1 \cup S_2\bigr) \leq H_{[]}\bigl(\epsilon_1, \tilde{\mathcal{F}}(f_0, \delta), d_{\mathrm{H}}, S_1\bigr) +
    H_{[]}\bigl(\epsilon_2, \tilde{\mathcal{F}}(f_0, \delta), d_{\mathrm{H}}, S_2\bigr).
  \]
  Since $d_{\mathrm{H}}$ is location and scale equivariant, we assume without loss of generality that $\mu_{f_0}=0$ and $\sigma_{f_0} = 1$. Define $a_L := \inf \{ r \in \mathbb{R} \,:\, f_0(r) \geq \delta^2\}$ and $a_R := \sup \{ r \in \mathbb{R}\,:\, f_0(r) \geq \delta^2\}$; it holds by the fact that $\delta \leq 2^{-9}$ and Lemma~\ref{Lem:LogConcaveCentralMass} that $a_L \leq -1/9$ and $a_R \geq 1/9$. By Lemma~\ref{Lem:HellingerMomentBound} and \citet[][Proposition~S2(iii)]{feng2018multivariate}, there exist $\alpha > 0$ and $\beta \in \mathbb{R}$ such that for any $f \in \tilde{\mathcal{F}}(f_0, \delta)$ and $x \in \mathbb{R}$, we have $f(x) \leq e^{-\alpha |x| + \beta}$.  Let $b_L := - \alpha^{-1} \log (e^\beta/\delta^4)$ and $b_R := \alpha^{-1} \log (e^\beta/\delta^4)$.  Then, for any $f \in \tilde{\mathcal{F}}(f_0, \delta)$, $f(r) < \delta^4$ for $r \in (-\infty,b_L) \cup (b_R,\infty)$, and $[a_L, a_R] \subseteq [b_L, b_R]$.

First, we will bracket the region $[b_L, a_L) \cup (a_R, b_R]$.  To this end, fix $\epsilon > 0$, and let $K_L := \min \{ k \in \mathbb{N}\,:\, e^{- \alpha ( k - a_L) + \beta} \leq \delta^4 \}$ and $K_R := \min \{ k \in \mathbb{N}\,:\, e^{ - \alpha ( k + a_R)+\beta} \leq \delta^4\}$, so that $\max(K_L,K_R) \lesssim \log(1/\delta)$.  By these definitions, $a_L - K_L \leq b_L$ and $a_R + K_R \geq b_R$. We segment $[b_L, a_L)$ into subintervals $S_k$ for $k=1,\ldots,K_L$, where
\begin{align*}
  S_k &:= \bigl[a_L - k, a_L - (k-1)\bigr),  \quad k=1,\ldots,K_L-1, \\
  S_{K_L} &:= \bigl[b_L, a_L - (K_L - 1)\bigr).
\end{align*}
Define $\bar{\epsilon} := \epsilon/(4 K_L^{1/2})$.  For any $r \in [b_L,a_L)$, we have that $f_0(r) \leq \delta^2$ because $r < a_L$ and, moreover, $e^{-\alpha|r|+\beta} \geq \delta^4$ because $r \geq b_L$. Now, by Lemma~\ref{Lem:EnvelopeUpperBound}, $f(r) \leq f_0(r) e^{ C \delta \log \frac{1}{\delta}} \lesssim f_0(r)$ for any $f \in \tilde{\mathcal{F}}(f_0, \delta)$ and $r \in [b_L,a_L)$.  Hence, by Lemma~\ref{Prop:SegmentBracket1}, 
  \begin{align}
\label{Eq:bLaL}
    H_{[]}\biggl(\frac{\epsilon}{4}, \tilde{\mathcal{F}}(f_0, \delta), d_{\mathrm{H}}, [b_L, a_L]\biggr)
    &\leq  \sum_{k=1}^{K_L}
    H_{[]}(\bar{\epsilon}, \tilde{\mathcal{F}}(f_0, \delta), d_{\mathrm{H}}, S_k) \nonumber \\
    & \lesssim \sum_{k=1}^{K_L} \frac{\delta^{1/2}}{\bar{\epsilon}^{1/2}} \lesssim \frac{\delta^{1/2}}{\epsilon^{1/2}} \log^{5/4} \frac{1}{\delta}.
  \end{align}
By symmetry, we obtain the same bound for $H_{[]}\bigl(\epsilon/4, \tilde{\mathcal{F}}(f_0, \delta), d_{\mathrm{H}}, [a_R, b_R]\bigr)$.
  
Now we bracket the region $(-\infty, b_L) \cup (b_R, \infty)$.  For $k \in \mathbb{N}$, define $S_k := [b_L - k, b_L - (k-1))$ and set $\epsilon_k := C \epsilon e^{-\alpha(k - b_L)/4}$ where $C > 0$ is a constant chosen such that $\sum_{k=1}^\infty \epsilon^2_k \leq \epsilon^2/16$.  Then,  
by Lemma~\ref{Prop:SegmentBracket1} again, 
  \begin{align}
\label{Eq:bL}
H_{[]}\bigl(\epsilon/4, \tilde{\mathcal{F}}(f_0, \delta), d_{\mathrm{H}}, (-\infty, b_L)\bigr) &\leq \sum_{k=1}^\infty H_{[]}(\epsilon_k, \tilde{\mathcal{F}}(f_0, \delta), d_{\mathrm{H}}, S_k) 
    \lesssim \sum_{k=1}^\infty \frac{e^{-\alpha(k - b_L)/4}}{\epsilon_k^{1/2}} \nonumber \\
    &\lesssim \frac{1}{\epsilon^{1/2}}\sum_{k=1}^\infty e^{-\alpha(k - b_L)/8} \lesssim \frac{\delta^{1/2}}{\epsilon^{1/2}}.
  \end{align}
The same bound holds for $H_{[]}\bigl( \epsilon/4, \tilde{\mathcal{F}}(f_0, \delta), d_{\mathrm{H}}, [b_R, \infty)\bigr)$.
    
Next, we bracket the region $[a_L, -1/16]$.  To this end, we write $s_0 := a_L$ and partition $[s_0,-1/16]$ into segments $[s_0, s_1)$, $[s_1, s_2),\ldots,[s_{J-2},s_{J-1})$, $[s_{J-1},s_J]$ (where $s_J := -1/16$) as follows:
  \begin{enumerate}
  \item Choose $s_1 > s_0$ such that $\int_{s_0}^{s_1} f_0(t) \, dt = 4 \delta^2$.
  \item For each $j \geq 2$, if there exists $t_0 < -1/16$ such that $\int_{s_{j-1}}^{t_0} f_0(t) \, dt \geq 2\int_{-\infty}^{s_{j-1}} f_0(t) \, dt$, then choose $s_j$ such that $\int_{s_{j-1}}^{s_j} f_0(t) \, dt = 2\int_{-\infty}^{s_{j-1}} f_0(t) \, dt$. Otherwise, set $J := j$ and choose $s_J = -1/16$.
  \end{enumerate}
  Define $\phi_0 := \log f_0$ and write $\mathrm{Range}_j(\phi_0) := \sup_{t \in [s_{j-1},s_j]} \phi_0(t) - \inf_{t \in [s_{j-1},s_j]} \phi_0(t)$.  We make the following six claims:
  \begin{enumerate}
  \item[(1)] $s_1 < -1/16$;
  \item[(2)] $(s_1 - s_0)\sup_{t \in [s_0, s_1]} f_0(t) \lesssim \delta^2 \log \frac{1}{\delta}$;
  \item[(3)] $\int_{s_J}^{\infty} f_0(t) \, dt > 2^{-11}$;
  \item[(4)] $(s_j - s_{j-1})\sup_{t \in [s_{j-1}, s_j]} f_0(t) \lesssim \bigl\{\mathrm{Range}_j(\phi_0) +\log 2\bigr\}\int_{-\infty}^{s_{j-1}} f_0(t) \, dt$ for $j=2,\ldots,J-1$;
  \item[(5)] $\int_{-\infty}^{s_J} f_0(t) \, dt \geq 2^{-13}$;
  \item[(6)] $J \lesssim \log(1/\delta)$.
  \end{enumerate}
To verify claim~(1), observe by \citet[][Theorem~5.14(a) and (d)]{lovasz2007geometry} that $f_0(t) \geq 2^{-8}$ for all $t \in [-1/9, 1/9]$.  Hence $\int_{s_0}^{-1/16} f_0(t)\, dt \geq (1/9 - 1/16) 2^{-8} > 4\delta^2$, so $s_1 < -1/16$.

For claim~(2), note that $-2 \log (1/\delta) \leq \phi_0(t) \leq 0$ for $t \in  [s_0, s_1]$ by \citet[][Proposition~S2(iii)]{feng2018multivariate}.  Thus by the second part of Lemma~\ref{lem:integral_approximation} and the definition of $s_1$, we have
\[
  (s_1 - s_0) \sup_{t \in [s_0, s_1]} f_0(t) \lesssim \log (1/\delta) \int_{s_0}^{s_1} f_0(t) \, dt \lesssim \delta^2 \log (1/\delta).
\]
  For claim (3), we have $\int_{s_J}^{\infty} f_0(t) \, dt \geq \int_{-1/16}^{1/9} f_0(t) \, dt \geq 2^{-8} (1/9 + 1/16) > 2^{-11}$.

  For claim (4), observe that $2 \int_{-\infty}^{s_{j-1}} f_0(t) \, dt = \int_{s_{j-1}}^{s_j} f_0(t) \, dt$ for $j=2,\ldots,J$. Hence
  \begin{align*}
    \frac{(s_j - s_{j-1}) \sup_{t \in [s_{j-1}, s_j]} f_0(t) }{\int_{-\infty}^{s_{j-1}} f_0(t) \, dt } \leq
    \frac{2(s_j - s_{j-1}) \sup_{t \in [s_{j-1}, s_j]} f_0(t) }{\int_{s_{j-1}}^{s_j} f_0(t) \, dt } 
    \lesssim \mathrm{Range}_j(\phi_0) + \log 2,
  \end{align*}
  where the final bound follows from Lemma~\ref{lem:integral_approximation}.

  For claim (5), we have 
\begin{align*}
\int_{-\infty}^{s_J} f_0 \geq \int_{-1/9}^{-1/16} f_0 \geq 2^{-8}(1/9-1/16) \geq 2^{-13}.
\end{align*}

 Finally, for claim~(6), we have 
\[
1 \geq \int_{-\infty}^{s_{J-1}} f_0 = \int_{-\infty}^{s_{J-2}} f_0 + \int_{s_{J-2}}^{s_{J-1}} f_0 = 3  \int_{-\infty}^{s_{J-2}} f_0 = 3^{J-2}\int_{-\infty}^{s_1} f_0 \geq 3^{J-2} \cdot 4\delta^2, 
\]
so $J \lesssim \log(1/\delta)$.

  Now set $\tilde{\epsilon} := \epsilon/(2J^{1/2})$.  Then, by Lemma~\ref{Prop:SegmentBracket1}, claim~(2), and Lemma~\ref{Lem:EnvelopeUpperBound},
  \begin{align}
\label{Eq:s0}
    H_{[]}\bigl(\tilde{\epsilon}, \tilde{\mathcal{F}}(f_0, \delta), d_{\mathrm{H}}, [s_0, s_1]\bigr)
    &\lesssim \frac{(s_1 - s_0)^{1/4}}{\tilde{\epsilon}^{1/2}}\sup_{e^\phi \in \tilde{\mathcal{F}}(f_0, \delta) }\sup_{t \in [s_0, s_1]} e^{\phi(t)/4} \nonumber \\
    &\lesssim \frac{\delta^{1/2}}{\tilde{\epsilon}^{1/2}} \log^{1/4}(1/\delta)\sup_{e^\phi \in \tilde{\mathcal{F}}(f_0, \delta) }\sup_{t \in [s_0, s_1]} e^{(\phi(t) - \phi_0(t))/4} \nonumber \\
    &\lesssim  \frac{\delta^{1/2}}{\epsilon^{1/2}} \log^{1/2} \frac{1}{\delta}.
  \end{align}

Now let $j \in \{2,\ldots,J\}$, and observe by claim~(1) that $s_1,\ldots,s_J$ are strictly increasing.  Let $\check{\mathcal{F}}(f_0,\delta) := \bigl\{e^{\phi - \phi_0}:e^\phi \in \tilde{\mathcal{F}}(f_0,\delta)\bigr\}$.  Let $\{(\check{\psi}_{\ell}^L,\check{\psi}_{\ell}^U):\ell\in [N]\}$ be an $\tilde{\epsilon}$-Hellinger bracketing set for $\check{\mathcal{F}}(f_0,\delta)$ with $\log N = H_{[]}\bigl(\tilde{\epsilon},\check{\mathcal{F}}(f_0,\delta),d_{\mathrm{H}},[s_{j-1},s_j]\bigr)$; we define $\{(\tilde{\psi}_\ell^L,\tilde{\psi}_\ell^U):\ell=1,\ldots,N\}$ by $\tilde{\psi}_\ell^L := f_0 \check{\psi}_\ell^L$ and $\tilde{\psi}_\ell^U := f_0 \check{\psi}_\ell^U$.  Then
\[
\int_{s_{j-1}}^{s_j} \bigl\{(\tilde{\psi}_\ell^U)^{1/2} - (\tilde{\psi}_\ell^L)^{1/2}\bigr\}^2 \leq \sup_{t \in [s_{j-1},s_j]} f_0(t) \int_{s_{j-1}}^{s_j} \bigl\{(\check{\psi}_\ell^U)^{1/2} - (\check{\psi}_\ell^L)^{1/2}\bigr\}^2 \leq \tilde{\epsilon}^2 \sup_{t \in [s_{j-1},s_j]} f_0(t).
\]
Moreover, if $\check{\psi}_\ell^L \leq e^{\phi-\phi_0} \leq \check{\psi}_\ell^U$, then $\tilde{\psi}_\ell^L \leq e^{\phi} \leq \tilde{\psi}_\ell^U$.  We deduce that $\bigl\{(\tilde{\psi}_\ell^L,\tilde{\psi}_\ell^U):\ell \in [N]\bigr\}$ form an $\tilde{\epsilon} \sup_{t \in [s_{j-1},s_j]} f_0(t)^{1/2}$-Hellinger bracketing set for $\tilde{\mathcal{F}}(f_0,\delta,[s_{j-1},s_j])$.  

Now, on $[s_{j-1}, s_j]$, the conditions of Lemma~\ref{Lem:EnvelopeLowerBound} are fulfilled with $r=s_j$ because $\int_{-\infty}^{s_{j-1}} f_0(t)\, dt \geq \int_{s_0}^{s_1} f_0(t)\, dt = 4\delta^2$ and $\int_{s_j}^\infty f_0(t)\, dt > 2^{-11} \geq 4\delta^2$ by claim~(3). Thus, we may combine Lemma~\ref{Prop:SegmentBracket2} with Lemmas~\ref{Lem:EnvelopeLowerBound} and~\ref{Lem:EnvelopeUpperBound} with claim~(4) to obtain
  \begin{align}
\label{Eq:KeyDisplay}
 H_{[]}\bigl(\tilde{\epsilon}, \tilde{\mathcal{F}}&(f_0, \delta), d_{\mathrm{H}}, [s_{j-1}, s_j]\bigr) \nonumber \\
&\leq H_{[]}\biggl(\frac{\tilde{\epsilon}}{\sup_{t \in [s_{j-1},s_j]} e^{\phi_0(t)/2}}, \check{\mathcal{F}}(f_0, \delta), d_{\mathrm{H}}, [s_{j-1}, s_j]\biggr) \nonumber \\
&\lesssim \frac{1}{\tilde{\epsilon}^{1/2}}\sup_{t \in [s_{j-1},s_j]} e^{\phi_0(t)/4}\biggl[|s_{j-1}|\delta + \frac{\delta}{\bigl\{\int_{-\infty}^{s_{j-1}} f_0(t) \, dt \bigr\}^{1/2}}\biggr]^{1/2}(s_j - s_{j-1})^{1/4} \nonumber \\
    & \lesssim \frac{\delta^{1/2}}{\tilde{\epsilon}^{1/2}}\bigl\{\mathrm{Range}_j(\phi_0) + \log 2\bigr\}^{1/4}\biggl[|s_{j-1}|^{1/2} \biggl\{\int_{-\infty}^{s_{j-1}} f_0(t) \, dt \biggr\}^{1/4} + 1\biggr] \nonumber \\
    &\lesssim \frac{\delta^{1/2}}{\tilde{\epsilon}^{1/2}}\bigl\{\mathrm{Range}_j(\phi_0) + \log 2\bigr\}^{1/4},
  \end{align}
where the final bound follows because $s_{j-1}^2 \int_{-\infty}^{s_{j-1}} f_0(t) \, dt  \leq 1$ by Markov's inequality.  

By symmetry, we obtain the same bracketing entropy bound over $[1/16, a_R]$.  For the region $[-1/16, 1/16]$, since $\int_{-\infty}^{-1/16} f_0(t) \,dt \geq 2^{-13}$ and $\int_{1/16}^\infty f_0(t) \,dt \geq 2^{-13}$ by claim~(5), we may argue as in~\eqref{Eq:KeyDisplay} to obtain
\begin{equation}
\label{Eq:1/16}
 H_{[]}\biggl( \frac{\epsilon}{2}, \tilde{\mathcal{F}}(f_0, \delta), d_{\mathrm{H}}, [-1/16, 1/16] \biggr) \lesssim \frac{\delta^{1/2} }{\epsilon^{1/2}}.
\end{equation}
Now, since $\phi_0$ is unimodal and $0 \geq \phi_0 \geq -2 \log (1/\delta)$ on $[a_L, a_R]$, it holds that
\[
  \sum_{j=2}^J \{\mathrm{Range}_j(\phi_0) + \log 2\} \lesssim \log(1/\delta).
\]
Thus, by Jensen's inequality,
\begin{equation}
\label{Eq:Jensen}
\sum_{j=2}^J \bigl\{\mathrm{Range}_j(\phi_0)+\log 2\bigr\}^{1/4} \lesssim J \biggl(\frac{\log(1/\delta)}{J}\biggr)^{1/4} \lesssim \log(1/\delta). 
\end{equation}
We conclude from~\eqref{Eq:Jensen},~\eqref{Eq:1/16},~\eqref{Eq:KeyDisplay},~\eqref{Eq:s0},~\eqref{Eq:bL},~\eqref{Eq:bLaL} and claim~(6) that
  \begin{align*}
    & H_{[]} \bigl( \epsilon, \tilde{\mathcal{F}}(f_0, \delta), d_{\mathrm{H}} \bigr) \\
    &\leq H_{[]}\biggl( \frac{\epsilon}{2}, \tilde{\mathcal{F}}(f_0, \delta), d_{\mathrm{H}}, [-1/16, 1/16] \biggr) + 2\sum_{j=1}^J H_{[]}\biggl( \frac{\epsilon}{2J^{1/2}}, \tilde{\mathcal{F}}(f_0, \delta), d_{\mathrm{H}}, [s_{j-1}, s_J] \biggr) \\
&\hspace{2cm}+ 2 H_{[]}\biggl( \frac{\epsilon}{4}, \tilde{\mathcal{F}}(f_0, \delta), d_{\mathrm{H}}, [b_L,a_L] \biggr) + 2H_{[]}\biggl( \frac{\epsilon}{4}, \tilde{\mathcal{F}}(f_0, \delta), d_{\mathrm{H}}, (-\infty, b_L) \biggr)
\\
  & \lesssim \frac{\delta^{1/2}}{\epsilon^{1/2}} \log^{5/4}\frac{1}{\delta},
  \end{align*}
as required.
\end{proof}

\begin{lemma}
  \label{Lem:EnvelopeLowerBound}
  Let $e^{\phi} \in \tilde{\mathcal{F}}(e^{\phi_0}, \delta)$ for some concave $\phi_0:\mathbb{R} \rightarrow [-\infty,\infty)$ and $\delta \in (0,\infty)$ and let $r \in \mathbb{R}$. If $\int_{-\infty}^r e^{\phi_0(t)} \, dt \wedge \int_r^{\infty} e^{\phi_0(t)} \, dt \geq 4\delta^2$, then

  \[ 
    \frac{\phi(r) - \phi_0(r)}{2} \geq
    - \frac{2\delta}{\bigl\{ \int_{-\infty}^r e^{\phi_0(t)} \, dt \wedge \int_r^{\infty} e^{\phi_0(t)} \, dt \bigr\}^{1/2}}.
  \]
\end{lemma}
\begin{proof}
As a shorthand, let us write $\psi := (\phi - \phi_0)/2$. Assume without loss of generality that $\psi(r) < 0$ (because otherwise the result is immediate).  Since $\psi$ is concave, either $\psi(t) \leq \psi(r)$ for all $t \in (-\infty, r)$ or for all $t \in (r, \infty)$.  In the former case,
  \[
    \delta^2 \geq \int_{-\infty}^\infty e^{\phi_0(t)} (1 - e^{\psi(t)})^2 \, dt \geq (1 - e^{\psi(r)})^2 \int_{-\infty}^r e^{\phi_0(t)} \, dt.
  \]
Hence
\[
   \psi(r) \geq \log \biggl(1 -   \frac{\delta}{\bigl\{\int_{-\infty}^r e^{\phi_0(t)} \, dt \bigr\}^{1/2}} \biggr) \geq - \frac{2\delta}{\bigl\{\int_{-\infty}^r e^{\phi_0(t)} \, dt\bigr\}^{1/2}}.
\]
On the other hand, if $\psi(t) < \psi(r)$ for $t \in (r, \infty)$, we can apply an almost identical argument to see that
  \[
    \psi(r) \geq -\frac{2\delta}{\bigl\{ \int_r^{\infty} e^{\phi_0(t)} \, dt \bigr\}^{1/2}},
  \]
as required.
\end{proof}

\begin{lemma}
  \label{Lem:EnvelopeUpperBound}
  Let $f_0 = e^{\phi_0} \in \mathcal{F}_1$ with $\mu_{f_0}=0$ and $\sigma_{f_0}=1$, and let $e^{\phi} \in \tilde{\mathcal{F}}(f_0, \delta)$ for some $\delta \in (0,2^{-8}]$.  Then there exists a universal constant $C > 0$ such that for any $r \in \mathbb{R}$,
  \[
    \frac{\phi(r) - \phi_0(r)}{2} \leq C (|r|+1) \delta. 
  \]
\end{lemma}
\begin{proof}
Again, we write $\psi := (\phi - \phi_0)/2$.  Since we seek an upper bound for $\psi$, we may assume without loss of generality that $\psi$ is upper semi-continuous, and by symmetry, it suffices to prove the bound at a fixed $r \geq 0$.  Further, we assume without loss of generality that $\psi(r) > 0$ (because otherwise the result is immediate).  

Let $r_0 := r + 1$.  Define $S := \bigl\{ t \in [-r_0, r_0] \,:\,  \psi(t) \geq \frac{\psi(r)}{2^{16}r_0} \bigr\}$, $s_L := \inf S$, and $s_R := \sup S$.  We note that $r \in S$ since $2^{16}r_0 \geq 2$. Then, since $e^x - 1 \geq x$ for any $x \geq 0$, we have
  \begin{align}
    \delta^2 &\geq  \int_{s_L}^{s_R} e^{\phi_0(t)} (e^{\psi(t)} - 1)^2 \, dt \geq  \int_{s_L}^{s_R} e^{\phi_0(t)} \psi(t)^2  \, dt \geq  \left(\frac{\psi(r)}{2^{16}r_0}\right)^2
               \int_{s_L}^{s_R} e^{\phi_0(t)} \, dt.
                \label{eqn:upper_bound1}
  \end{align}
Now, define $S' := \bigl\{ t \in [-r_0, r_0] \,:\, \psi(t) \geq - \frac{\psi(r)}{2^{16}r_0} \bigr\}$, $s'_L := \inf S'$, and $s'_R := \sup S'$. Then 
\begin{align}
  \label{eqn:upper_bound2}
    \delta^2 \geq \int_{[-r_0, r_0]\setminus [s'_L, s'_R]} e^{\phi_0(t)} ( 1 - e^{\psi(t)})^2 \, dt \geq \bigl\{ 1 - e^{-\frac{\psi(r)}{2^{16} r_0}} \bigr\}^2 \int_{[-r_0, r_0]\setminus [s'_L, s'_R]} e^{\phi_0(t)} \, dt.
  \end{align}
 As a shorthand, let us define
  \begin{align*}
    T_1 &:= \int_{s_L}^{s_R} e^{\phi_0(t)} \, dt, \quad
    T_2 := \int_{s'_L}^{s_L} e^{\phi_0(t)} \, dt +  \int_{s_R}^{s'_R} e^{\phi_0(t)} \, dt, \\
    T_3 &:= \int_{-r_0}^{s'_L} e^{\phi_0(t)} \, dt + \int_{s'_R}^{r_0} e^{\phi_0(t)} \, dt.
  \end{align*}
Inequalities~\eqref{eqn:upper_bound1} and \eqref{eqn:upper_bound2} yield
  \begin{align}
    \psi(r) &\leq  2^{16}r_0 \frac{ \delta}{ T_1^{1/2} } \label{Eqn:PsiUpperBound1} \\
    \psi(r) &\leq - 2^{16}r_0 \log \biggl( 1 - \frac{\delta}{T_3^{1/2}} \biggr) \quad \text{ if $T_3 > 0$}. \label{Eqn:PsiUpperBound2}
  \end{align}
 Since $T_1 + T_2 + T_3 = \int_{-r_0}^{r_0} e^{\phi_0(t)} \, dt$, we have that $T_1 + T_2 + T_3 \geq 2^{-12}$ by Lemma~\ref{Lem:LogConcaveCentralMass}.

We claim that $T_2 \leq 2^{-13}$.  To see this, note that by concavity of $\psi$ (cf.~the proof of Theorem~1 of \citet{CSS2010}),
  \begin{align*}
    s'_R - s'_L \leq (s_R - s_L) \frac{ 1 + \frac{1}{2^{16}r_0}}
      {1 - \frac{1}{2^{16}r_0}} \leq (s_R - s_L) \biggl( 1 + \frac{4}{2^{16}r_0} \biggr).
  \end{align*}
By \citet[][Proposition~S2(iii)]{feng2018multivariate}, $\sup_{t \in \mathbb{R}} e^{\phi_0(t)} \leq 1$, hence,
  \begin{align*}
    T_2 = \int_{[s'_L, s_L] \cup [s_R, s'_R]} e^{\phi_0(t)} \, dt \leq  s'_R - s_R + s_L - s'_L \leq (s_R - s_L) \frac{4}{2^{16} r_0 } \leq 2^{-13}.
  \end{align*}
It follows that either $T_1 \geq 2^{-14}$ or $T_3 \geq 2^{-14}$. If $T_1 \geq 2^{-14}$, then the result follows from~\eqref{Eqn:PsiUpperBound1}.  On the other hand, if $T_3 \geq 2^{-14}$, then we obtain the desired result from the fact that $\delta \leq 2^{-8}$, that $-\log(1 - x) \leq 2x$ for all $x \in [0, 1/2]$, and~\eqref{Eqn:PsiUpperBound2}.
\end{proof}

\begin{lemma}
  \label{Lem:LogConcaveCentralMass}
  Let $f \in \mathcal{F}_1$ be such that $\mu_f = 0$ and $\sigma_f=1$. We have that $a_L := \inf \{ r \,:\, f(r) \geq 2^{-8}\} \leq -1/9$, $a_R := \sup \{ r \,:\, f(r) \geq 2^{-8} \} \geq 1/9$, and
  \begin{equation}
\label{Eq:Int}
    \int_{a_L}^{a_R} f(r) \, dr \geq \int_{-1/9}^{1/9} f(r) \,dr \geq  2^{-12}.
  \end{equation}
\end{lemma}
\begin{proof}
  By \citet[][Theorem~5.14(a) and 5.14(d)]{lovasz2007geometry}, $2^{-7} \leq f(0) \leq 2^4$ and $f(r) \geq 2^{-8}$ for all $r \in [-1/9, 1/9]$. We immediately obtain that $a_L \leq -1/9$ and $a_R \geq 1/9$. Moreover,
  \[
    \int_{a_L}^{a_R} f(r) \,dr \geq \int_{-1/9}^{1/9} f(r) \, dr \geq 2^{-8} (2/9) \geq 2^{-12},
    \]
as required.
\end{proof}

\begin{lemma}
  \label{lem:integral_approximation}
Let $\phi_0:\mathbb{R} \rightarrow [-\infty,\infty)$ be a concave function, let $a, b \in \mathbb{R}$ where $a < b$, and let
\[
q(s) := \left\{ \begin{array}{ll} (1 - e^{-s})/s & \mbox{if $s \neq 0$} \\
1 & \mbox{if $s = 0$.} \end{array} \right.
\]
Then for any $t^* \in [a, b]$, it holds that
\begin{align*}
    \int_a^b e^{\phi_0(t)} \, dt \geq (t^*-a) e^{\phi_0(t^*)} q\bigl(\phi_0(t^*) - \phi_0(a)\bigr) + (b - t^*) e^{\phi_0(t^*)}q\bigl(\phi_0(t^*) - \phi_0(b)\bigr).
  \end{align*}
Moreover, if $\phi_0(t^*) \leq \max\{\phi_0(a), \phi_0(b)\} + \tau$ for some $\tau \geq \log 2$, then
  \[
  \int_a^b e^{\phi_0(t)} \, dt \geq (b-a) e^{\phi_0(t^*)} \frac{1}{2 \tau}.
  \]
\end{lemma}
\begin{proof}
  Let us first suppose that $t^* > a$. We have $\phi_0(t) \geq s \phi_0(a) + (1 - s) \phi_0(t^*)$ for $t \in [a,t^*]$, where $s := (t^* - t)/(t^* - a)$.  Hence
  \begin{align}
    \int_a^{t^*} e^{\phi_0(t)} \, dt \geq (t^* - a) \int_0^1 e^{ s \phi_0(a) + (1 - s) \phi_0(t^*)} \, ds = (t^*- a)e^{\phi_0(t^*)}q\bigl(\phi_0(t^*) - \phi_0(a)\bigr).
    \label{Eqn:qBound1}
  \end{align}
  We can bound $\int_{t^*}^b e^{\phi_0(t)} \,dt$ when $t^* < b$ by a similar argument to yield the first conclusion. 

For the second part, observe that $q$ is strictly decreasing, so from~\eqref{Eqn:qBound1},
\[
\int_a^{t^*} e^{\phi_0(t)} \, dt \geq (t^*-a)e^{\phi_0(t^*)}q(\tau) \geq (t^*-a)e^{\phi_0(t^*)}\frac{1}{2\tau},
\]
for $\tau \geq \log 2$. We may bound $\int_{t^*}^b e^{\phi_0(t)} \,dt$ by a similar argument to obtain the desired conclusion.
\end{proof}

The following two lemmas are from \citet{kim2016adaptation}, though the first is only a minor restatement of \citet[][Theorem~4.1]{doss2016global}. For $a < b$ and $-\infty \leq B_1 < B_2 < \infty$, we define $\tilde{\mathcal{F}}([a,b], B_1, B_2)$ to be the set of log-concave functions $f:[a,b] \rightarrow [e^{B_1},e^{B_2}]$. 
\begin{lemma}
  \label{Prop:SegmentBracket1}
There exists a universal constant $C > 0$ such that 
\[
H_{[]}( \epsilon, \tilde{\mathcal{F}}([a,b],-\infty, B), d_{\mathrm{H}}, [a,b]) \leq C \frac{ e^{B/4} (b-a)^{1/4}}{\epsilon^{1/2}}
\]
for every $a < b$ and $B, \epsilon > 0$
\end{lemma}
\begin{lemma}
  \label{Prop:SegmentBracket2}
There exists a universal constant $C > 0$ such that 
\[
H_{[]}( \epsilon, \tilde{\mathcal{F}}([a,b],B_1,B_2), d_{\mathrm{H}}, [a,b]) \leq C (B_2 - B_1)^{1/2} \frac{ e^{B_2/4} (b-a)^{1/4}}{\epsilon^{1/2}}
\]
for every $a < b$, $\epsilon > 0$ and $-\infty \leq B_1 \leq B_2 < \infty$. 
\end{lemma}

\subsection{Auxiliary lemmas for Subsection~\ref{Subsec:GeneralApproach}}

\label{Subsec:AuxGeneralApproach}

We first describe the common setting for all the lemmas in this subsection and define the notation used throughout. We fix $K,\hat{K} \in \mathcal{K}$, $\mu, \hat{\mu} \in \mathbb{R}^p$ 
Let $f_0 \in \mathcal{F}^{K, \mu}_p$ be of the form $f_0(x) = e^{\phi_0(\| x - \mu \|_K)}$ for some $\phi_0 \in \Phi$ and let $X_1, \ldots, X_n \stackrel{\mathrm{iid}} \sim f_0$. We assume in this subsection that
\begin{align}
\mathbb{E}_{f_0}( \| X_1 - \mu \|^2_K) = p. \label{Eqn:MomentAssumption}
\end{align}
This assumption can be made without loss of generality as shown in the proof of Proposition~\ref{Prop:KEstimatedGeneralRisk}.  Let $h_0$ be a density on $[0,\infty)$ be of the form $h_0(r) = p \lambda_p(K) r^{p-1} e^{\phi_0(r)}$. Define $\check{f}_n \,:\, \mathbb{R}^p \rightarrow [0,\infty)$ by $\check{f}_n(x) := e^{ \phi_0 ( \| x - \hat{\mu} \|_{\hat{K}}) }$. We note that $\check{f}_n$ is not necessarily a density.  Let
\begin{align}
  \tau^* :=  1 \vee \sup_{x \in \mathbb{R}^{p} \setminus \{0\}} \frac{\| x \|_{\hat{K}}}{\| x \|_K},\quad \tau_* := 1 \wedge \inf_{x \in \mathbb{R}^{p} \setminus \{0\}} \frac{\| x \|_{\hat{K}}}{\| x \|_K}.
\end{align}
We also define a deformation $\tilde{\phi}_0 \in \Phi$ of $\phi_0$ by
\begin{align}
  \tilde{\phi}_0(r) := \left\{ \begin{array}{ll} \phi_0(0) & \textrm{if } r \leq \| \hat{\mu} - \mu \|_{\hat{K}} \\
                                 \phi_0\bigl( \frac{r - \| \hat{\mu} - \mu \|_{\hat{K}} }{\tau^*} \bigr) & \textrm{otherwise,}
                               \end{array} \right. \label{eqn:phi_tilde_definition}
\end{align}
and write $\tilde{h}_0(r) := \gamma^{-1}p \lambda_{p}(\hat{K}) r^{p-1} e^{\tilde{\phi}_0(r)}$ where $\gamma := p \lambda_{p}(\hat{K})\int_0^\infty r^{p-1} e^{\tilde{\phi}_0(r)} \, dr$ so that $\tilde{h}_0$ is a density. 

For $i \in [n]$, let $Z_i := \|X_i - \mu\|_K$ and $\tilde{Z}_i := \| X_i - \hat{\mu} \|_{\hat{K}}$. Then $Z_i$ has density $h_0$ and, by Lemma~\ref{Lem:ContourLevelExistence} below, $\tilde{Z}_1$ has a density which we denote by $\tilde{h}_n$. We let $Q_0$ and $\tilde{Q}_n$ denote the probability distributions induced by $h_0$ and $\tilde{h}_n$ respectively and let $\tilde{\mathbb{Q}}_n$ denote the empirical distribution corresponding to $\tilde{Z}_1, \ldots, \tilde{Z}_n$, so that $\hat{h}_n := h^*(\tilde{\mathbb{Q}}_n)$.  We write $\hat{h}_n(r) = p\lambda_p(\hat{K})r^{p-1}e^{\hat{\phi}_n(r)}$ for some $\hat{\phi}_n \in \Phi$, and set $\hat{f}_n := e^{\hat{\phi}_n}$.   Similarly to~\eqref{Eqn:MomentClass} we write $\mathcal{H}(h_0, C_\mu, C_\sigma) := \bigl\{ h \in \mathcal{H} \,:\, | \mu_h - \mu_{h_0} | \leq C_\mu\sigma_{h_0}, \, C_\sigma^{-1} \leq \sigma_h/\sigma_{h_0} \leq C_\sigma \bigr\}$.
\begin{lemma}
  \label{Lem:MeanVarianceControl}
  There exist universal constants $c_1, c_2, C, C_\mu > 0$ and $C_\sigma > 1$ such that if $\| \hat{\mu} - \mu \|_K p \log(ep) \leq c_1$ and $p(\tau^* - \tau_*) \leq c_2$, then
  \[
\mathbb{P}_{f_0}\bigl(\hat{h}_n \notin \mathcal{H}(h_0, C_\mu, C_\sigma) \bigr) \leq \frac{C}{n}.
  \]
\end{lemma}

\begin{proof}
  The proof is similar to that of Lemma~\ref{Lem:MLEMeanVarPreservation}. 

  As a preliminary step, we first claim that there exist universal constants $c_1, c_2,C_\sigma' > 0$ such that if $\| \hat{\mu} - \mu \|_K p \log(ep) \leq c_1$ and $p(\tau^* - \tau_*) \leq c_2$, then the following statements hold simultaneously:
  \begin{align}
    &\textrm{(a) $\gamma \leq 2$},\,\, \textrm{(b) $\mathbb{E}\bigl\{ (\tilde{Z}_1 - Z_1)^2 \bigr\} \leq \sigma_{h_0}^2/16$},\\
    &\textrm{(c) $\frac{1}{2} \leq \frac{\sigma^2_{\tilde{h}_n}}{\sigma^2_{h_0}} \leq 2$},\, 
    \textrm{and (d) $\frac{1}{C_\sigma'} \leq \frac{\sigma_{\tilde{h}_0}}{\sigma_{h_0}} \leq C_\sigma'$}.
    \label{eqn:var_bound_claims}
  \end{align}
  Provided we choose universal constants $c_1,c_2 > 0$ sufficiently small, claim (a) follows from Lemma~\ref{Lem:GammaBound}, while claim (d) follows from the second claim of Lemma~\ref{Lem:dHWorstBound2} and Lemma~\ref{Lem:HellingerMomentBound}. For claim (b), observe that since $\mu_{h_0} \lesssim p \sigma_{h_0}$ (which holds by Lemma~\ref{Lem:UnitVarianceMeanBound}), we have that $\sigma_{h_0}^2 \gtrsim (\sigma_{h_0}^2 + \mu_{h_0}^2)(1 + p^2)^{-1} \gtrsim p^{-1}$. Thus, by Lemma~\ref{Lem:KPerturbation}, we may reduce the values of $c_1, c_2$ if necessary to obtain
  \begin{align}
    \mathbb{E} \bigl\{(\tilde{Z}_1 - Z_1)^2 \bigr\}
    &\leq 2 (\tau^* - \tau_*)^2 \mathbb{E} Z^2_1 +
    2 (\tau^* + \tau_*)^2 \| \hat{\mu} - \mu \|_K^2 \nonumber \\
    &\leq 2 c_2 p^{-1} +
      2(2 + c_2/p)^2 c_1 p^{-1}  \leq \sigma_{h_0}^2/16.
      \label{eqn:ZDiffVarianceBound}
  \end{align}
For claim (c), observe that $\textrm{Var}( \tilde{Z}_1 - Z_1) \leq \mathbb{E} \bigl\{(\tilde{Z}_1 - Z_1)^2 \bigr\} \leq \sigma_{h_0}^2/16$. We therefore have by Cauchy--Schwarz and claim (b) that,
  \begin{align}
    \sigma^2_{\tilde{h}_n} = \textrm{Var}(\tilde{Z}_1) &= \textrm{Var}\bigl(Z_1 + (\tilde{Z}_1 - Z_1) \bigr) \nonumber \\
    &\leq  \sigma_{h_0}^2 + 2 \sigma_{h_0} \textrm{Var}( \tilde{Z}_1 - Z_1)^{1/2} + \textrm{Var}( \tilde{Z}_1 - Z_1) \leq 2 \sigma_{h_0}^2,
  \end{align}
  and
  \begin{align*}
    \sigma^2_{\tilde{h}_n} &\geq \sigma_{h_0}^2 - 2 \sigma_{h_0} \textrm{Var}(\tilde{Z}_1 - Z_1)^{1/2} \geq \frac{\sigma_{h_0}^2}{2}.
  \end{align*}                  
  This establishes claim~(c).
  
  Define $E_1 := \bigl\{ | n^{-1} \sum_{i=1}^n \tilde{Z}_i - \mathbb{E}\tilde{Z}_1| \leq \sqrt{2} \sigma_{h_0} \bigr\}$; from claim~(c) of~\eqref{eqn:var_bound_claims} and Chebychev's inequality, it holds that $\mathbb{P}(E^c_1) \leq 1/n$. By Lemma~\ref{Lem:MLEMeanPreservation} (with $a = 0$) and claim (b) of~\eqref{eqn:var_bound_claims}, there exists a universal constant $C_\mu > 0$ such that on $E_1$, 
\begin{align}
  \mu_{\hat{h}_n} & \leq \frac{1}{n} \sum_{i=1}^n \tilde{Z}_i \leq \mathbb{E} \tilde{Z} + \sqrt{2}\sigma_{h_0} \leq \mathbb{E} Z_1 + \mathbb{E}| \tilde{Z}_1 - Z_1 | + \sqrt{2} \sigma_{h_0} \nonumber \\
  &\leq \mu_{h_0} + \mathbb{E}\bigl\{(\tilde{Z}_1 - Z_1)^2\bigr\}^{1/2} + \sqrt{2} \sigma_{h_0} \leq \mu_{h_0} + C_\mu\sigma_{h_0}.
\end{align}
Similarly, by Lemma~\ref{Lem:MLEMeanPreservation} again and Lemma~\ref{Lem:UnitVarianceMeanBound},
\begin{align}
  \mu_{\hat{h}_n} &\geq \biggl(1 - \frac{1}{p} \biggr) 
    \frac{1}{n} \sum_{i=1}^n \tilde{Z}_i \geq \biggl(1 - \frac{1}{p} \biggr)\bigl\{\mathbb{E} Z_1 - \mathbb{E}| \tilde{Z}_1 - Z_1 | - \sqrt{2} \sigma_{h_0} \bigr\} \geq \mu_{h_0} - C_\mu \sigma_{h_0}.
  \end{align}
We therefore conclude that 
  \begin{align}
    \mathbb{P}\bigl( |\mu_{\hat{h}_n} - \mu_{h_0}| > C_\mu\sigma_{h_0} \bigr) \leq \mathbb{P}(E_1^c) \leq \frac{1}{n}.
    \label{eqn:C_mu}
  \end{align}

  We now consider $\sigma_{\hat{h}_n}^2$. By Lemma~\ref{Lem:PerturbedDensityUpperBound},~\ref{Lem:VarianceSupRelation}, and claims~(a) and (d) of~\eqref{eqn:var_bound_claims}, there exists a universal constant $C' > 0$ such that $\| \tilde{h}_n \|_{\mathrm{esssup}} \leq \gamma \| \tilde{h}_0 \|_\infty \leq C' \sigma_{h_0}^{-1}$.  Define the event
  \[
    E_2 := \biggl\{ \sup \biggl\{ \frac{\tilde{\mathbb{Q}}_n(D)}{\lambda_1(D)} \,:\, D \subseteq \mathbb{R} \textrm{ compact, convex},\, \tilde{\mathbb{Q}}_n(D) \geq 1/2 \biggr\} \leq \frac{2 C'}{\sigma_{h_0}} \biggr\}.
    \]
Then $\mathbb{P}(E_2^c) \leq 2 e^{-n/128} \leq 2^{10}/n$ by Lemma~\ref{Lem:QLebesgueRatioBound}.

  We now obtain a lower bound for $\int_0^\infty \log \hat{h}_n \, d \tilde{\mathbb{Q}}_n$. Note that by \citet[][Theorem~8.6.5]{cover2006elements} and by claim (c) of~\eqref{eqn:var_bound_claims}, we have that
  \begin{align}
    \int_0^\infty \tilde{h}_n(r) \log \tilde{h}_n(r) \,dr \geq -\log(4 \pi e)/2 - \log \sigma_{h_0}.
    \label{Eqn:htildeEntropyBound}
  \end{align}
Now write $g(r) := \sigma_{h_0} \tilde{h}_n(\sigma_{h_0} r)$ for $r \geq 0$ as a shorthand and observe that
  \begin{align}
    \int_0^\infty \tilde{h}_n(r) \log^2 (\sigma_{h_0} \tilde{h}_n(r)) \, dr &= \int_0^\infty g(s) \log^2 g(s) \, ds  =  \int_0^\infty \int_{-\infty}^{\log g(s)} e^t (t^2 + 2t) \, dt \, ds \nonumber \\
    &\leq \int_0^\infty \int_{-\infty}^\infty \mathbbm{1}_{\{ t \leq \log g(s) \}} e^t (t^2 + 2 |t|) \, dt \, ds \nonumber \\
    &= \int_{-\infty}^\infty e^t (t^2 + 2|t|) \lambda_1 \bigl\{ s > 0 \,:\, \log g(s) \geq t \bigr \} \, dt \nonumber \\
    &\leq \int_{-\infty}^\infty e^t (t^2 + 2|t|) \lambda_1 \bigl\{ s > 0 \,:\, \log \bigl(\sigma_{h_0}\tilde{h}_0(\sigma_{h_0}s)\bigr) \geq t - \log \gamma \bigr\} \, dt. \label{eqn:fubini_intermediate1}
  \end{align}
By claim (d) of~\eqref{eqn:var_bound_claims} and \citet[][Proposition~S2(iii)]{feng2018multivariate}, there exist universal constants $C'_{\sigma} > 1, \mu' \in [0,\infty)$ such that $\sigma_{h_0}\tilde{h}_0(\sigma_{h_0}s) \leq C_{\sigma}^{\prime} e^{- C_{\sigma}^{\prime -1} | s - \mu' | + 1}$ for all $s \in [0,\infty)$. Thus, by claim (a) of~\eqref{eqn:var_bound_claims}, for any $t \in \mathbb{R}$,
  \begin{align}
    \lambda_1 \bigl\{ s > 0 \,:\, \log \bigl(\sigma_{h_0}\tilde{h}_0(\sigma_{h_0}s)\bigr) &\geq t - \log \gamma \bigr\} \nonumber \\
    &\leq \lambda_1 \bigl\{ s > 0 \,:\, - C_{\sigma}^{\prime -1} | s - \mu'| + 1 + \log C_{\sigma}' \geq t - \log 2 \} \nonumber \\
    &\leq 2 C_{\sigma}'(\log C_{\sigma}' + 1 + \log 2 - t) \vee 0.
      \label{eqn:fubini_intermediate2}
  \end{align}
  Combining~\eqref{eqn:fubini_intermediate1} and~\eqref{eqn:fubini_intermediate2}, we deduce that there exists a universal constant $C'' > 0$ such that
  \begin{equation}
    \label{Eq:C''}
    \int_0^\infty \tilde{h}_n(r) \log^2( \sigma_{h_0} \tilde{h}_n(r)) \, dr \leq C''.
  \end{equation}
Therefore, we have by~\eqref{Eq:C''} that 
  \begin{align}
    \int_0^\infty \tilde{h}_n(r) \log^2 \tilde{h}_n(r) \, dr &- \biggl(\int_0^\infty \tilde{h}_n(r) \log \tilde{h}_n(r) \, dr \biggr)^2 \nonumber \\
    &= \int_0^\infty \tilde{h}_n(r) \log^2( \sigma_{h_0} \tilde{h}_n(r)) \,dr  -  2 (\log \sigma_{h_0}) \int_0^\infty \tilde{h}_n(r) \log \tilde{h}_n(r) \, dr  \nonumber \\
    &\hspace{3cm}   - \log^2 \sigma_{h_0} - \biggl(\int_0^\infty \tilde{h}_n(r) \log \tilde{h}_n(r) \, dr \biggr)^2 \nonumber \\
    &\leq C'' - \biggl( \log \sigma_{h_0} + \int_0^\infty \tilde{h}_n(r) \log \tilde{h}_n(r) \,dr \biggr)^2   
     \leq C''. \label{Eqn:htildeLogSquareBound}
  \end{align}
Defining the event
\begin{align*}
  E_3 := \biggl\{\biggl|\int_0^\infty \log \tilde{h}_n\, d \tilde{\mathbb{Q}}_n - \int_0^\infty \tilde{h}_n \log \tilde{h}_n\biggr| \leq \sqrt{C''} \biggr\},
\end{align*}
we have by~\eqref{Eqn:htildeLogSquareBound} and Chebychev's inequality that $\mathbb{P}(E_3^c ) \leq 1/n$. Let $\mathcal{R}_0 := \{r \in [0,\infty):\tilde{h}_n(r) \leq \gamma \tilde{h}_0(r)\}$ and let $E_4$ denote the event that $\tilde{\mathbb{Q}}_n(\{0\}) < 1$ and $\tilde{\mathbb{Q}}_n(\mathcal{R}_0) = 1$. It holds then by Lemma~\ref{Lem:PerturbedDensityUpperBound} that $\mathbb{P}(E_4^c) = 0$. 

Moreover, on $E_3 \cap E_4$, 
  \begin{align}
    \int_0^\infty \log \hat{h}_n \, d \tilde{\mathbb{Q}}_n
    &\geq \int_0^\infty \log \tilde{h}_0 \, d \tilde{\mathbb{Q}}_n
      = \int_0^\infty \log \frac{\tilde{h}_0}{\tilde{h}_n} \, d \tilde{\mathbb{Q}}_n
      + \int_0^\infty \log \tilde{h}_n \, d \tilde{\mathbb{Q}}_n \nonumber \\
    &\geq - \log 2 + \int_0^\infty \tilde{h}_n \log \tilde{h}_n \nonumber \\
    &\qquad - 
      \biggl| \int_0^\infty \log \tilde{h}_n\, d \tilde{\mathbb{Q}}_n - \int_0^\infty \tilde{h}_n \log \tilde{h}_n \biggr|
      \nonumber \\
    &\geq
      C_3 - \log \sigma_{h_0},
      \label{eqn:var_final_bound1}
  \end{align}
for some universal constant $C_3 > 0$.  We conclude that on the event $E_2 \cap E_3 \cap E_4$, by Lemma~\ref{Lem:MLEVarPreservation}, there exists a universal constant $C_{\sigma} \geq 1$ such that
  \[
    C_\sigma^{-1} \leq \frac{\sigma_{\hat{h}_n}}{\sigma_{h_0}} \leq C_\sigma.
  \]
A union bound yields the desired result.
\end{proof}

  For a Borel measurable function $g : [0, \infty) \rightarrow \mathbb{R}$, define 
\[
  \rho_1^2(g) := 2 \int_0^\infty (e^{|g|} - 1 - |g|) \, d \tilde{Q}_n.
\]
Even though $\rho_1$ is not a norm, we can define the $\epsilon$-generalised bracketing entropy of a class $\mathcal{G}$ of Borel measurable, real-valued functions by treating it as a norm, and continue to denote this by $H_{[]}(\epsilon, \mathcal{G}, \rho_1)$.
\begin{lemma}
  \label{Lem:EmpiricalProcess}
 Set
  \[
    a_n := n^{-4/5}+ d_{\mathrm{KL}}^2(\tilde{h}_n, \tilde{h}_0) + d_{\mathrm{H}}^2(\check{f}_n,f_0) + d_{\mathrm{H}}^2 \biggl( \tilde{h}_0, h_0  \frac{\lambda_p(\hat{K})}{\lambda_p(K)}\biggr).
  \]
  There exist universal constants $c_1, c_2, C > 0$ such that if $\| \hat{\mu} - \mu \|_K p \log(ep) \leq c_1$ and $p(\tau^* - \tau_*) \leq c_2$, then, for any $t \geq a_n^{1/2}$,
  \begin{align}
    \mathbb{P}_{f_0}\bigl( \{d_{\mathrm{H}}^2(\hat{f}_n, f_0) \geq Ct^2\} \cap \{ \hat{h}_n \in \mathcal{H}(h_0, C_\mu, C_\sigma) \} \bigr) \lesssim e^{-nt^2/C},
      \label{Eqn:dHBound}
  \end{align}
 where $C_\mu > 0$ and $C_{\sigma} > 1$ are taken from Lemma~\ref{Lem:MeanVarianceControl}.  Morever, if additionally $f_0(x) = e^{-a\|x - \mu \|_K + b}$ for some $a > 0$ and $b \in \mathbb{R}$, and writing
\[
  \tilde{a}_n := \frac{1}{n} \log^{5/4} (en) + d_{\mathrm{KL}}^2(\tilde{h}_n, \tilde{h}_0) + d_{\mathrm{H}}^2(\tilde{h}_0,h_0) + d_{\mathrm{H}}^2(\check{f}_n,f_0) + d_{\mathrm{H}}^2 \biggl(\tilde{h}_0,h_0  \frac{\lambda_p(\hat{K})}{\lambda_p(K)}\biggr),
\]
we have that~\eqref{Eqn:dHBound} holds for any $t \geq \tilde{a}_n^{1/2}$.
\end{lemma}

\begin{proof}
Assume $\| \hat{\mu} - \mu \|_K p \log(ep) \leq c_1$ and $p(\tau^* - \tau_*) \leq c_2$ for universal constants $c_1, c_2 > 0$ chosen such that (a) $\tilde{h}_0 \in \mathcal{H}(h_0, C_\mu, C_\sigma)$ where $C_\mu > 0$ and $C_{\sigma} > 1$ are taken from Lemma~\ref{Lem:MeanVarianceControl} and (b) $\gamma \leq 2$. The existence of such a choice of $c_1, c_2$ is guaranteed by Lemma~\ref{Lem:HellingerMomentBound}, Lemma~\ref{Lem:dHWorstBound2}, and Lemma~\ref{Lem:GammaBound}.

We make two observations before proceeding with the main proof. First, for $\delta > 0$, define $\mathcal{H}(\tilde{h}_0, \delta) := \bigl\{ h \in \mathcal{H}(h_0, C_\mu, C_\sigma) : d_{\mathrm{H}}(h,\tilde{h}_0) \leq \delta\bigr\}$, and observe that $\tilde{h}_0 \in \mathcal{H}(\tilde{h}_0, \delta)$ for every $\delta > 0$.  Second, we may assume that $\tilde{\mathbb{Q}}_n(\{0\}) = 0$  (since this is a probability 1 event), and thus, $\tilde{\mathbb{Q}}_n \in \mathcal{Q}_0$ and $\hat{h}_n := h^*(\tilde{\mathbb{Q}}_n) = \argmax_{h \in \mathcal{H}} \int_0^\infty \log h\, d \tilde{\mathbb{Q}}_n$ (recall the definition of $\mathcal{H} := \mathcal{H}_0$ in~\eqref{Eqn:HclassDefn}). Therefore, since $\tilde{h}_0 \in \mathcal{H}$, we have by Proposition~\ref{Prop:ProjectionExistence} and Lemma~\ref{Lem:PerturbedDensityUpperBound} that with probability~1,
\[
  \infty > \int_0^\infty \log \hat{h}_n \, d \tilde{\mathbb{Q}}_n \geq \int_0^\infty \log \tilde{h}_0 \, d \tilde{\mathbb{Q}}_n \geq -\log \gamma + \int_0^\infty \log \tilde{h}_n \, d \tilde{\mathbb{Q}}_n > -\infty.
\]
Hence, again with probability 1,
  \begin{align}
    \int_0^\infty \log \frac{\hat{h}_n + \tilde{h}_0}{2 \tilde{h}_0} \, d \tilde{\mathbb{Q}}_n \geq 0.
    \label{Eqn:BasicIneq}
  \end{align}
Now we proceed to the proof of the Lemma. We have
   \begin{align}
    d_{\mathrm{H}}^2(\hat{f}_n, f_0) \leq 2 d_{\mathrm{H}}^2(\hat{f}_n, \check{f}_n) + 2 d_{\mathrm{H}}^2(\check{f}_n,f_0). \label{eqn:first_decomp}
  \end{align}
To bound the first term, we have by Lemma~\ref{Lem:ChangeOfVar} that
\begin{align}
  \label{Eq:NextDecomp}
    d_{\mathrm{H}}^2(\hat{f}_n, \check{f}_n)
    &= \int_{\mathbb{R}^p} \bigl\{e^{\hat{\phi}_n( \| x - \hat{\mu} \|_{\hat{K}} )/2} -  e^{\phi_0( \| x - \hat{\mu} \|_{\hat{K}} )/2} \bigr\}^2 \, dx \nonumber \\
    &= p\lambda_p(\hat{K})\int_0^\infty r^{p-1} \bigl\{e^{\hat{\phi}_n(r)/2 } - e^{\phi_0(r)/2}\bigr\}^2 \, dr \nonumber \\
    &= d_{\mathrm{H}}^2\biggl( \hat{h}_n, h_0 \frac{\lambda_p(\hat{K})}{\lambda_p(K)} \biggr) \nonumber \\
    &\leq 2 d_{\mathrm{H}}^2(\hat{h}_n, \tilde{h}_0) + 2 d_{\mathrm{H}}^2 \biggl(\tilde{h}_0 , h_0  \frac{\lambda_p(\hat{K})}{\lambda_p(K)}\biggr).
  \end{align}
 By \citet[][Lemma~4.2]{vandegeer2000empirical}, the fact that KL divergence is no smaller than the squared Hellinger distance, and~\eqref{Eqn:BasicIneq},
\begin{align}
  \label{Eqn:BasicInequality}
    d_{\mathrm{H}}^2(\hat{h}_n, \tilde{h}_0) &\leq 16 d_{\mathrm{H}}^2\biggl( \tilde{h}_0, \frac{\hat{h}_n + \tilde{h}_0}{2}\biggr) \nonumber \\
                                             &\leq 32 d_{\mathrm{KL}}^2\biggl( \tilde{h}_n, \frac{\hat{h}_n + \tilde{h}_0}{2}\biggr) + 32 d_{\mathrm{KL}}^2(\tilde{h}_n, \tilde{h}_0) \nonumber \\
                                             &= -32 \int_0^\infty \log \frac{\hat{h}_n +  \tilde{h}_0 }{2 \tilde{h}_0 } \, d \tilde{Q}_n + 64d^2_{\mathrm{KL}}( \tilde{h}_n, \tilde{h}_0 ) \nonumber \\
                                             &\leq 32 \int_0^\infty \log \frac{\hat{h}_n +  \tilde{h}_0 }{2 \tilde{h}_0 } \, d (\tilde{\mathbb{Q}}_n - \tilde{Q}_n) + 64d^2_{\mathrm{KL}}( \tilde{h}_n, \tilde{h}_0 ).
\end{align}
Since $\mathcal{H}(\tilde{h}_0, \delta)$ is nonempty for any $\delta > 0$, the bracketing entropy $H_{[]}\bigl(\epsilon, \mathcal{H}(\tilde{h}_0, \delta), d_{\mathrm{H}}\bigr)$ is well-defined and non-negative for any $\epsilon > 0$.  We may therefore define $\Psi \,:\, [0,\infty) \rightarrow [0,\infty)$ by $\Psi(\delta) := \delta \vee \int_0^{\delta} H_{[]}^{1/2}\bigl(\epsilon/2^{1/2}, \mathcal{H}(\tilde{h}_0, \delta), d_{\mathrm{H}}\bigr) \, d\epsilon$ and let $\delta_n := \inf \bigl\{ \delta \in [0, \infty) : \frac{\sqrt{n} \delta^2}{\Psi(\delta)} \geq 2^{9} C_0 \bigr\}$ for a universal constant $C_0 > 0$ specified in \citet[][Theorem 5.11]{vandegeer2000empirical}. By Lemma~\ref{Lem:NonlocalBracketingEntropy}, it holds that $\delta_n^2 \lesssim n^{-4/5}$. Moreover, if $f_0(\cdot) = e^{-a\|\cdot - \mu\|_K + b}$ for some $a > 0$ and $b \in \mathbb{R}$, then, by Lemma~\ref{Lem:LocalBracketing}, we have that $\delta_n^2 \lesssim n^{-1} \log^{5/4} (en) + d_{\mathrm{H}}^2(\tilde{h}_0, h_0)$.

  For $\delta > 0$, define $\mathcal{G}(\delta) := \bigl\{ \frac{1}{2}\log \frac{h + \tilde{h}_0 }{2 \tilde{h}_0} : h \in \mathcal{H}(\tilde{h}_0, \delta) \bigr\}$. Let $(h_U, h_L)$ be an element from the $\epsilon/\sqrt{2}$-Hellinger bracketing set of $\mathcal{H}(\tilde{h}_0, \delta)$. Define $g_U := \frac{1}{2} \log \frac{h_U + \tilde{h}_0}{2\tilde{h}_0}$ and $g_L := \frac{1}{2} \log \frac{h_L + \tilde{h}_0}{2\tilde{h}_0}$. We have by \citet[][Lemmas~7.1 and~4.2]{vandegeer2000empirical}, Lemma~\ref{Lem:PerturbedDensityUpperBound} and the fact that $\gamma \leq 2$ that
\begin{align}
  \label{Eq:PreviousDisplay}
    \int_0^\infty \rho^2_1(g_U - g_L) \, d \tilde{Q}_n
    &\leq \int_0^\infty (e^{g_U - g_L} - 1)^2 \, d\tilde{Q}_n \nonumber \\
    &= \int_0^\infty \biggl\{ \biggl(\frac{h_U + \tilde{h}_0}{h_L + \tilde{h}_0}\biggr)^{1/2}   - 1 \biggr\}^2 \tilde{h}_n \nonumber \\
    &= \int_0^\infty \frac{1}{2} \bigl\{ (h_U + \tilde{h}_0)^{1/2}  - (h_L + \tilde{h}_0)^{1/2} \bigr\}^2 \frac{2\tilde{h}_n}{h_L + \tilde{h}_0} \nonumber \\
    &\leq \gamma d_{\mathrm{H}}^2(h_U, h_L) \leq \gamma \epsilon^2/2 \leq \epsilon^2.
  \end{align}
  Moreover, if $h \in \mathcal{H}(\tilde{h}_0, \delta)$ and $g = \frac{1}{2} \log \frac{h + \tilde{h}_0}{2\tilde{h}_0}$, then a virtually identical calculation to~\eqref{Eq:PreviousDisplay} shows that $\sup_{g \in \mathcal{G}(\delta)} \rho_1^2(g) \leq \delta^2$ for every $\delta > 0$. We therefore conclude that
  \begin{align}
    H_{[]}\bigl(\epsilon, \mathcal{G}(\delta), \rho_1\bigr) \leq H_{[]}\bigl(\epsilon/2^{1/2}, \mathcal{H}(\tilde{h}_0, \delta), d_{\mathrm{H}}\bigr). \label{Eqn:GBracketing}
  \end{align}
Fix any $t > \bigl\{\delta^2_n + 2^7 d^2_{\mathrm{KL}}(\tilde{h}_n, \tilde{h}_0)\bigr\}^{1/2}$, where we note that this lower bound is finite by Lemma~\ref{Lem:PerturbedDensityUpperBound}. For each $s \in \mathbb{N} \cup \{0\}$, define the events $\mathcal{A}_s := \bigl\{ 2^{s} t < d_{\mathrm{H}} \bigl(\hat{h}_n, \tilde{h}_0\bigr) \leq 2^{s+1} t \bigr\}$. Writing $\hat{g}_n := \frac{1}{2} \log \frac{\hat{h}_n + \tilde{h}_0}{2 \tilde{h}_0}$, we note that on $\mathcal{A}_s \cap \{\hat{h}_n \in \mathcal{H}(h_0,C_\mu,C_\sigma)\}$, by~\eqref{Eqn:BasicInequality}, 
  \begin{equation}
    \int_0^\infty \hat{g}_n \, d(\tilde{\mathbb{Q}}_n - \tilde{Q}_n) \geq 2^{2s - 6} t^2 - d_{\mathrm{KL}}^2 (\tilde{h}_n, \tilde{h}_0) \geq 2^{2s - 7} t^2 \quad \textrm{and } \quad \rho_1^2(\hat{g}_n) \leq 2^{2s+2} t^2. \label{Eqn:OnePeel}
  \end{equation}
Moreover, since $\delta \mapsto \Psi(\delta)/\delta^2$ is decreasing, we have
  \[
    \frac{\sqrt{n} 2^{2s - 7} t^2 }{\Psi( 2^{s+1} t )} \geq \frac{\sqrt{n} 2^{-9} \delta_n^2}{\Psi(\delta_n)} \geq C_0,
  \]
  and thus, by~\eqref{Eqn:GBracketing}, 
  \begin{align}
    \sqrt{n} 2^{2s - 7} t^2 \geq C_0 \Psi(2^{s + 1} t) \geq C_0
    \Bigl( 2^{s+1}t \vee \int_0^{2^{s+1} t} H_{[]}^{1/2}(\epsilon, \mathcal{G}(\delta), \rho_1) \, d\epsilon \Bigr).
    \label{Eqn:EntropyCheck}
  \end{align}
Therefore, by~\eqref{Eqn:OnePeel},~\eqref{Eqn:EntropyCheck}, and \citet[][Theorem 5.11]{vandegeer2000empirical}, there exists a universal constant $C > 0$ such that
  \begin{align}
    \mathbb{P}_{f_0}\bigl(\{d_{\mathrm{H}}^2(\hat{h}_n, \tilde{h}_0) > t^2\}  \,\cap\, \{ \hat{h}_n &\in \mathcal{H}(h_0, C_\mu, C_\sigma) \} \bigr)  \nonumber \\
    &\leq \sum_{s=0}^\infty \mathbb{P}_{f_0}\bigl(\mathcal{A}_s  \cap \{ \hat{h}_n \in \mathcal{H}(h_0, C_\mu, C_\sigma) \} \bigr) \nonumber \\
    &\leq \sum_{s=0}^\infty \mathbb{P}_{f_0} \biggl\{ \sup_{g \in \mathcal{G}(2^{s+1} t)} \int_0^\infty g \, d(\tilde{\mathbb{Q}}_n - \tilde{Q}_n) \geq 2^{2s-7} t^2\biggr\}  \nonumber \\
    &\lesssim \sum_{s=0}^\infty \exp \biggl( - \frac{ n 2^{2s}t^2 }{C } \biggr) \lesssim e^{-nt^2/C}. \label{Eq:FinalVdgBound}
  \end{align}
The lemma follows from~\eqref{eqn:first_decomp},~\eqref{Eq:NextDecomp}, ~\eqref{Eq:FinalVdgBound}, and the fact that $\delta_n^2 \lesssim n^{-4/5}$ in general and $\delta_n^2 \lesssim n^{-1} \log^{5/4} (en) + d^2_{\mathrm{H}}(\tilde{h}_0, h_0)$ when $f_0$ has the form $f_0(\cdot) = e^{-a\|\cdot - \mu \|_K + b}$ for some $a > 0$ and $b \in \mathbb{R}$.
\end{proof}

\begin{lemma}
  \label{Lem:KPerturbation}
  Let $x_0, \mu, \hat{\mu} \in \mathbb{R}^p$ and let $K, \hat{K} \in \mathcal{K}$. Writing $\xi := \hat{\mu} - \mu$, we have 
  \begin{align*}
    \bigl| &\| x_0 - \hat{\mu} \|_{\hat{K}} - \| x_0 - \mu \|_K \bigr| \\
&\leq \min\biggl\{\bigl( \tau^*  - \tau_* \bigr) \| x_0 - \mu \|_K + (\tau^* + \tau_*) \|\xi\|_K \, , \, \biggl( \frac{1}{\tau_*}  - \frac{1}{\tau^*}\biggr) \| x_0 - \hat{\mu} \|_{\hat{K}} + \biggl(\frac{1}{\tau_*} + \frac{1}{\tau^*}\biggr) \|\xi\|_{\hat{K}}\biggr\}.
  \end{align*} 
\end{lemma}

\begin{proof}
 By the subadditivity of $\| \cdot \|_K$ (Proposition~\ref{Prop:BasicK}(iv)), we have 
  \begin{align*}
    \| x_0 - \hat{\mu}\|_{\hat{K}} &\leq \tau^* \| x_0 - \hat{\mu} \|_K \leq \tau^* \| x_0 - \mu \|_K + \tau^* \|\xi\|_K
  \end{align*}
and 
  \begin{align*}
    \| x_0 - \hat{\mu} \|_{\hat{K}}
    &\geq \tau_* \| x_0 - \hat{\mu} \|_K 
    \geq \tau_* \| x_0 - \mu \|_K - \tau_* \|\xi\|_K.
  \end{align*}
This yields the bound with the first term in the minimum, and the bound with the second term follows analogously. 
\end{proof}

\begin{lemma}
  \label{Lem:PerturbedDensityUpperBound}
  For every $x \in \mathbb{R}^p$, we have that $\tilde{\phi}_0(\|x - \hat{\mu}\|_{\hat{K}}) \geq \phi_0(\|x - \mu\|_K)$. Moreover, for $\lambda_1$-almost every $r \in [0,\infty)$, we have that
    $\tilde{h}_n(r) \leq \gamma \tilde{h}_0(r)$.
\end{lemma}

\begin{proof}
Let $x \in \mathbb{R}^p$. If $\| x - \hat{\mu} \|_{\hat{K}} \leq \| \hat{\mu} - \mu \|_{\hat{K}}$, then, since $\phi_0$ is decreasing, $\tilde{\phi}_0(\|x - \hat{\mu}\|_{\hat{K}}) = \phi_0(0) \geq \phi_0(\|x - \mu\|_K)$.  On the other hand, if $\| x - \hat{\mu} \|_{\hat{K}} > \| \hat{\mu} - \mu \|_{\hat{K}}$, then
 \[
    \| x - \mu \|_K \geq \frac{\| x - \mu \|_{\hat{K}}}{\tau^*} \geq \frac{\| x - \hat{\mu} \|_{\hat{K}} - \|\hat{\mu} - \mu\|_{\hat{K}}}{\tau^*}.
  \]
Hence, $\tilde{\phi}_0(\|x - \hat{\mu}\|_{\hat{K}}) = \phi_0 \Bigl( \frac{\| x - \hat{\mu} \|_{\hat{K}} - \|\hat{\mu} - \mu\|_{\hat{K}}}{\tau^*} \Bigr) \geq \phi_0(\|x - \mu \|_K)$. This proves the first claim of the lemma.

For $0 < r_1 \leq r_2$, we write $B_{\hat{K}}(\hat{\mu}; r_1, r_2) := \{ x \in \mathbb{R}^p : \| x - \hat{\mu} \|_{\hat{K}} \in (r_1, r_2]\}$. Then, by the first claim of the lemma and Lemma~\ref{Lem:ChangeOfVar}, for any $r \in [0,\infty)$ and $\epsilon > 0$, 
  \begin{align}
\frac{1}{\epsilon}\mathbb{P}_{f_0}\bigl(\|X-\hat{\mu}\|_{\hat{K}} \in (r,r+\epsilon]\bigr) &= \frac{1}{\epsilon}
    \int_{B_{\hat{K}}(\hat{\mu}; r, r+\epsilon)} e^{ \phi_0( \| x - \mu \|_K ) } \,dx \leq
                                                                                        \frac{1}{\epsilon}\int_{B_{\hat{K}}(\hat{\mu}; r, r+\epsilon)} e^{ \tilde{\phi}_0( \| x - \hat{\mu} \|_{\hat{K}} ) } \,dx \nonumber \\
                                                                                      &= \frac{1}{\epsilon} p\lambda_p(\hat{K})\int_r^{r+\epsilon} s^{p-1}e^{\tilde{\phi}_0(s)} \, ds.
  \end{align}
We now take the limit as $\epsilon \searrow 0$ on both sides.  On the left-hand side, we may apply Lemma~\ref{Lem:ContourLevelExistence} to conclude that the limit is $\tilde{h}_n(r)$ for $\lambda_1$-almost all $r \in [0,\infty)$.  On the right-hand side, the limit is $\gamma h_0(r)$ whenever $r$ is a continuity point of $\tilde{\phi}_0$ (i.e.~$\lambda_1$-almost everywhere, since $\tilde{\phi}_0$ is decreasing).
\end{proof}

\begin{definition}
Let $\nu_p$ be a signed measure on $(\mathbb{R}^p, \mathcal{B}(\mathbb{R}^p))$ and let $K \in \mathcal{K}$.  We refer to the signed measure $\nu_1^K$ on $\mathbb{R}$ defined by $\nu_1^K(E) := \nu_p( \{x \in \mathbb{R}^p \,:\, \|x \|_K \in E \} )$ for $E \in \mathcal{B}(\mathbb{R})$ as the $K$-\emph{contour measure} of $\nu_p$.
\end{definition}

The following lemma, among other things, implies that if a random vector $X$ has a density on $\mathbb{R}^p$, then $\|X \|_K$ has a density on $[0,\infty)$ as well.

\begin{lemma}
  \label{Lem:ContourLevelExistence}
Let $\nu_p$ be a signed measure on $\mathbb{R}^p$ with $\nu_p \ll \lambda_p$. Let $K \in \mathcal{K}$ and let $\nu_1^K$ be the $K$-contour measure of $\nu_p$. Then $\nu_1^K \ll \lambda_1$. 
\end{lemma}
\begin{proof}
Define $g : \mathbb{R}^p \rightarrow [0, \infty)$ by $g(x) := \|x \|_K$. We first claim that for any Borel measurable $E \subseteq [0, \infty)$ such that $\lambda_1(E) = 0$,  we have $\lambda_p\bigl( g^{-1}(E)\bigr) = 0$.

Let $E \subseteq [0, \infty)$ be a Borel measurable set such that $\lambda_1(E) = 0$, let $n \in \mathbb{N}$, and let $E_n := E \cap [0, n]$. 
Now let $\epsilon > 0$ be fixed. Since $\lambda_1(E_n) = 0$, there exist disjoint intervals $\{(b_m-\epsilon_m,b_m]:m \in \mathbb{N}\}$ such that $E_n \setminus \{0\} \subseteq \cup_{m=1}^\infty (b_m-\epsilon_m,b_m]$ and $\lambda_1\bigl( \cup_{m=1}^\infty (b_m-\epsilon_m,b_m]\bigr) \leq \epsilon$.  
Then by the mean value theorem, for any $M \in \mathbb{N}$,
\begin{align*}
  \lambda_p\biggl(g^{-1}\Bigl( \bigcup_{m=1}^M (b_m-\epsilon_m,b_m]\Bigr)\biggr)
  &\leq \lambda_p(K) \sum_{m=1}^M \frac{b_m^p - (b_m - \epsilon_m)^p}{\epsilon_m} \epsilon_m \\
  &\leq pn^{p-1} \lambda_p(K) \sum_{m=1}^M \epsilon_m \leq pn^{p-1}\lambda_p(K) \epsilon.
\end{align*}
We therefore deduce that 
\[
\lambda_p\bigl(g^{-1}(E_n)\bigr) \leq \lim_{M \rightarrow \infty} \lambda_p\biggl(g^{-1}\Bigl( \bigcup_{m=1}^M (b_m-\epsilon_m,b_m]\Bigr)\biggr) \leq pn^{p-1} \lambda_p(K) \epsilon,
\]
Since $\epsilon > 0$ was arbitrary, $\lambda_p\bigl(g^{-1}(E_n)\bigr) = 0$, so $\lambda_p\bigl(g^{-1}(E)\bigr) = \lim_{n \rightarrow \infty} \lambda_p\bigl(g^{-1}(E_n)\bigr) = 0$, which establishes the claim.

Hence, if $E$ is a Borel measurable subset of $[0,\infty)$ with $\lambda_1(E) = 0$, then since $\nu_p \ll \lambda_p$, we have $\nu_1^K( E ) = \nu_p\bigl( g^{-1}(E)\bigr) = 0$, as required.
\end{proof}

\subsubsection{Auxiliary lemmas for the worst case risk bounds of Section~\ref{Sec:KEstimated}}

We continue to use the setting and notation defined in Subsection~\ref{Subsec:AuxGeneralApproach}.

\begin{lemma}
  \label{Lem:WorstTVBound}
If $\| \hat{\mu} - \mu \|_Kp \log (ep) \leq 1$ and $p(\tau^* - \tau_{*}) \leq 1/2$, then 
  \[
    d_{\mathrm{TV}}(\check{f}_n,f_0) \lesssim p^{1/2} \| \hat{\mu} - \mu \|_K + p(\tau^* - \tau_{*}).
    \]
  \end{lemma}
  
  \begin{proof}
By Fubini's theorem,
\begin{align}
\label{Eq:Fubini}
  d_{\mathrm{TV}}(\check{f}_n,f_0)
  &= \int_{\mathbb{R}^p} \bigl| e^{\phi_0(\| x - \hat{\mu} \|_{\hat{K}} ) } - e^{\phi_0(\| x - \mu \|_K)} \bigr| \, dx \nonumber \\
  &= \int_{\mathbb{R}^p} \biggl| \int_{-\infty}^{\phi_0(0)} e^t \mathbbm{1}_{\{ \phi_0(\| x - \hat{\mu} \|_{\hat{K}} ) \geq t \}} \,dt - \int_{-\infty}^{\phi_0(0)} e^t \mathbbm{1}_{\{\phi_0(\| x - \mu \|_K) \geq t \}} \, dt \biggr| \,dx \nonumber \\
  &\leq \int_{\mathbb{R}^p} \int_{-\infty}^{\phi_0(0)} e^t \bigl| \mathbbm{1}_{\{ \phi_0(\| x - \hat{\mu} \|_{\hat{K}} ) \geq t \}} - \mathbbm{1}_{\{\phi_0(\| x - \mu \|_K) \geq t \}} \bigr| \,dt \, dx \nonumber \\
  &=   \int_{-\infty}^{\phi_0(0)} e^t \int_{\mathbb{R}^p} \bigl| \mathbbm{1}_{\{ \phi_0(\| x - \hat{\mu} \|_{\hat{K}} ) \geq t \}} - \mathbbm{1}_{\{\phi_0(\| x - \mu \|_K) \geq t \}} \bigr| \,dx \, dt .
\end{align}
For any $\mu' \in \mathbb{R}^p$ and $r \in [0,\infty)$, define $B_K(\mu', r) := \{x \,:\, \| x - \mu'\|_K \leq r\}$ and define $B_{\hat{K}}(\mu', r)$ analogously. Define $r_0 \,:\, (-\infty, \phi_0(0)] \rightarrow [0, \infty)$ by $r_0(t) := \sup \{ r \geq 0 : \phi_0(r) \geq t \}$ for any $t \in (-\infty, \phi_0(0)]$. Since $r_0(t) \geq r$ if and only if $\phi_0(r) \geq t$ for any $r \in [0,\infty)$ and $t \in (-\infty, \phi_0(0)]$, we have  
\begin{align}
  & \int_{\mathbb{R}^p} \bigl| \mathbbm{1}_{\{ \phi_0(\| x - \hat{\mu} \|_{\hat{K}} ) \geq t \}} - \mathbbm{1}_{\{\phi_0(\| x - \mu \|_K) \geq t \}} \bigr| \,dx \nonumber \\
  &\leq \int_{\mathbb{R}^p} \bigl| \mathbbm{1}_{\{ \phi_0(\| x - \hat{\mu} \|_{\hat{K}} ) \geq t \}} - \mathbbm{1}_{\{\phi_0(\| x - \hat{\mu} \|_{K}) \geq t \}} \bigr| \,dx + \int_{\mathbb{R}^p} \bigl| \mathbbm{1}_{\{\phi_0(\| x - \hat{\mu} \|_{K}) \geq t \}} -  \mathbbm{1}_{\{ \phi_0(\| x - \mu \|_{K} ) \geq t \}} \bigr| \,dx \nonumber \\
   &= \lambda_p \bigl( B_{\hat{K}}(\hat{\mu}; r_0(t)) \triangle B_{K}(\hat{\mu}; r_0(t)) \bigr) + \lambda_p \bigl( B_K(\hat{\mu} ; r_0(t)) \triangle B_K(\mu ; r_0(t)) \bigr). \label{eqn:volume_decomp}
\end{align}
 Let us denote $\xi := \hat{\mu} - \mu$ and $K + [0,\xi] := \{ x \in \mathbb{R}^p \,:\, x = x' + \alpha \xi \, \textrm{ for } x' \in K, \alpha \in [0,1] \}$ so that, for any $r \geq 0$, we have $rK \cup (rK + \xi) \subseteq rK + [0,\xi]$.  To upper bound the first term of ~\eqref{eqn:volume_decomp}, we obtain from the translation invariance of Lebesgue measure, its corresponding scaling property and Proposition~\ref{Prop:BasicK}(iv) that
\begin{align}
  \lambda_p \bigl( B_{\hat{K}}(\hat{\mu}; r_0(t)) \triangle B_{K}(\hat{\mu}; r_0(t)) \bigr)
  &= \lambda_p \bigl( B_{\hat{K}}(0; r_0(t)) \triangle B_{K}(0; r_0(t)) \bigr) \nonumber \\
  &= r_0(t)^p \bigl\{ \lambda_p(\hat{K} \cup K) - \lambda_p(\hat{K} \cap K) \bigr\} \nonumber \\
  &\leq r_0(t)^p \bigl\{ \lambda_p( \tau_*^{-1} K ) - \lambda_p( \tau^{*-1} K) \bigr\} \nonumber \\
  &= r_0(t)^p (\tau_*^{-p} - \tau^{*-p}) \lambda_p(K). \label{eqn:second_term_bound}
\end{align}
To upper bound the second term of \eqref{eqn:volume_decomp}, note that by the translation invariance of Lebesgue measure again, 
\begin{align}
  \lambda_p \bigl( B_K(\hat{\mu} ; r_0(t)) \triangle B_K(\mu ; r_0(t)) \bigr)
  &=
    \lambda_p \bigl( B_K( \xi; r_0(t)) \triangle B_K(0 ; r_0(t)) \bigr) \nonumber \\
  &\leq 2 \bigl\{ \lambda_p( r_0(t) K + [0, \xi] ) - \lambda_p(r_0(t) K)  \bigr\} \nonumber \\
  &\leq 2 \bigl\{ \lambda_p\bigl( (r_0(t)+ \| \xi\|_K ) K \bigr) - \lambda_p\bigl(r_0(t)K\bigr) \bigr\} \nonumber \\
  &\leq 2\bigl\{  \bigl( (r_0(t) + \| \xi \|_K)^p - r_0(t)^p  \bigr) \lambda_p(K) \bigr\}.  \label{eqn:cavalieri_bound2} 
\end{align}
Combining \eqref{Eq:Fubini},~\eqref{eqn:volume_decomp},~\eqref{eqn:second_term_bound} and~\eqref{eqn:cavalieri_bound2}, we obtain
\begin{align}
  d_{\mathrm{TV}}(\check{f}_n,f_0)
  &\leq p \lambda_p(K)\Bigl( \frac{\tau_*^{-p} - \tau^{* -p}}{p} \Bigr) \int_{-\infty}^{\phi_0(0)} e^t r_0(t)^p \, dt \nonumber \\
&\hspace{3cm}+ 2p \lambda_p(K) \int_{-\infty}^{\phi_0(0)} e^t \biggl\{ \frac{(r_0(t) - \| \xi \|_K)^p - r_0(t)^p}{p} \biggr\} \, dt. \label{eqn:overall_bound2}
\end{align}
Recall that $\mu_{h_0} = p \lambda_p(K) \int_0^\infty r^pe^{\phi_0(r)} \, dr$ and $\sigma_{h_0}^2 = p \lambda_p(K) \int_0^\infty (r - \mu_{h_0})^2 r^{p-1}e^{\phi_0(r)} \,dr$. We now make a few observations that are used repeatedly in the remainder of this proof:
\begin{enumerate}[(i)]
\item By \citet{feng2018multivariate}, it holds that $p \lambda_p(K)r^{p-1}e^{\phi_0(r)} \leq \frac{1}{\sigma_{h_0}}e^{- \frac{1}{\sigma_{h_0}} | r - \mu_{h_0} | + 1}$ for all $r \in [0,\infty)$.
\item By~\eqref{Eqn:MomentAssumption} and Lemma~\ref{Lem:ChangeOfVar}, it holds that $p \lambda_p(K) \int_0^\infty r^{p + 1} e^{\phi_0(r)} \, dr = p$. Thus, $\sigma_{h_0} \leq 1$, by \citet[][Lemma~1]{bobkov2003spectral}.
\item Since $\mu_{h_0}^2 + \sigma_{h_0}^2 = p$, we have $\mu_{h_0} \in [(p-1)^{1/2}, p^{1/2}]$.
\item From (ii) and (iii) as well as Lemma~\ref{Lem:UnitVarianceMeanBound}, there exists a universal constant $c^*  \in (0, 1/2]$ such that $\sigma_{h_0} \geq c^*/p^{1/2}$.
\end{enumerate}
Suppose first that $p \geq 2$. Define $r' := \frac{c^*}{\log (ep)}$; we observe that $r' \leq 1/2 \leq \mu_{h_0}/2$ since $c^* \leq 1/2$. We also define $\phi^* \,:\, [0, \infty) \rightarrow \mathbb{R}$ by
\[
  \phi^*(r) := \left\{ \begin{array}{ll}
                         \log\bigl(\frac{1}{\sigma_{h_0}p\lambda_p(K)r^{p-1}}\bigr) -\frac{c^*}{2\sigma_{h_0}} | r - \mu_{h_0} | + 1 & \mbox{for $r \in [r',\infty)$}\\
                           \max \bigl\{ \phi^*(r'), \, \phi_0(r') + r' \frac{\phi_0(r') - \phi_0(\mu_{h_0})}{\mu_{h_0}-r'} \bigr\} & \mbox{for $r \in [0,r')$.}
                       \end{array} \right.
  \]
We then have that
  \begin{align}
    \phi^{* \prime}(r) = \left\{ \begin{array}{ll} 0 & \mbox{for $r \in [0,r')$} \\
\frac{c^*}{2\sigma_{h_0}} - \frac{p-1}{r} & \mbox{for $r \in (r', \mu_{h_0})$} \\
-\frac{c^*}{2\sigma_{h_0}} - \frac{p-1}{r} & \mbox{for $r \in (\mu_{h_0},\infty).$} \end{array} \right. \label{eqn:phi_star_deriv1}
\end{align}
Since $c^*/(2\sigma_{h_0}) \leq p^{1/2}/2 \leq (p-1)/\mu_{h_0}$, we see that $\phi^*$ is a decreasing function and, since $\phi_0(0) \leq \phi_0(r') + r' \frac{\phi_0(r') - \phi_0(\mu_{h_0})}{\mu_{h_0} - r'}$, we also have that $\phi^*(r) \geq \phi_0(r)$ for all $r \in [0,\infty)$. We now claim that
\begin{align}
  p \lambda_p(K)r^{\prime p-1} e^{\phi^*(0)} \lesssim 1 \label{eqn:discontinuity_term}.
\end{align}
To see this, note that if $\phi^*(0) = \phi^*(r')$, then $p \lambda_p(K)r^{\prime p-1} e^{\phi^*(0)} \leq \frac{1}{\sigma_{h_0}} e^{-\frac{c^*}{2\sigma_{h_0}} (\mu_{h_0} - r') + 1} \lesssim 1$ as required. On the other hand, suppose that $\phi^*(0) =  \phi_0(r') + r' \frac{\phi_0(r') - \phi_0(\mu_{h_0})}{\mu_{h_0} - r'}$.  Then by the proof of Lemma~\ref{Lem:VarianceSupRelation},
\[
  p \lambda_p(K) \mu_{h_0}^{p-1} e^{\phi_0(\mu_{h_0})} \geq \frac{2^{-7}}{\sigma_{h_0}},
\]
and moreover,
\[
p \lambda_p(K) r^{\prime p-1} e^{\phi_0(r')} \leq \frac{1}{\sigma_{h_0}} e^{ - \frac{c^*}{2\sigma_{h_0}} ( \mu_{h_0} - r') + 1}.
\]
From the fact that $r' \log (1/r') \rightarrow 0$ as $p \rightarrow \infty$, we deduce that there exists a universal constants $C^*,C^{**} > 0$ such that 
\begin{align}
  r' \frac{\phi_0(r') - \phi_0(\mu_{h_0})}{\mu_{h_0} - r'}
  & \leq r' \frac{ - \frac{c^*}{2\sigma_{h_0}}( \mu_{h_0} - r') + 1 + 7 \log 2 + (p-1)\log (\mu_{h_0}/r^{\prime}) }{\mu_{h_0} - r'} \nonumber \\
  &\leq C^* + 2(p-1)^{1/2}r' \log\mu_{h_0} + 2 (p-1)^{1/2}r' \log \frac{1}{r'} 
   \leq C^{**} +  p^{1/2}.
    \label{eqn:phi_slope_bound}
\end{align}
Therefore, we also obtain in this case that
\begin{align*}
  p \lambda_p(K) r^{\prime p-1} e^{\phi^*(0)} &\leq e^{C^{**} +  p^{1/2}} \frac{1}{\sigma_{h_0}} e^{-\frac{1}{\sigma_{h_0}}(\mu_{h_0} - r') + 1} \lesssim 1.
\end{align*}
Define $r^* \,:\, (-\infty, \phi^*(0)] \rightarrow [0, \infty)$ by $r^*(t) := \sup \{ r \geq 0 \,:\, \phi^*(r) \geq t\}$. It follows from the fact $\phi^*(r) \geq \phi_0(r)$ for all $r \in [0,\infty)$ that $r_0(t) \leq r^*(t)$ for all $t \in (-\infty, \phi_0(0)]$.  Hence, by a change of a variable and our assumption on $\|\xi \|_K p \log(ep)$, 
\begin{align*}
  2p \lambda_p(K) \int_{-\infty}^{\phi_0(0)}
  & e^t \biggl\{ \frac{ (r_0(t) + \| \xi \|_K )^p - r_0(t)^p }{p} \biggr\} \, dt \\
  &\leq  2p \lambda_p(K) \int_{-\infty}^{\phi^*(0)} e^t \biggl\{ \frac{ (r^*(t) + \| \xi \|_K )^p - r^*(t)^p }{p} \biggr\} \, dt \\
  &\leq 2p \lambda_p(K) \int_{-\infty}^{\phi^*(r')} r^*(t)^pe^t \biggl\{ \frac{(1 + \frac{\| \xi \|_K}{r^*(t)})^p -1 }{p} \biggr\} \, dt \\
  &\hspace{4cm} + 2 p \lambda_p(K) (e^{\phi^*(0)} - e^{\phi^*(r')} ) \biggl\{ \frac{(r' + \| \xi \|_K)^p - r^{\prime p}}{p} \biggr\}
  \\
  &\lesssim \| \xi \|_Kp \lambda_p(K) \int_{r'}^\infty r^{p-1}e^{\phi^*(r)} |\phi^{* \prime}(r)| \, dr +  \| \xi \|_K p \lambda_p(K) e^{\phi^*(0)} r^{\prime p-1} \\
  &\lesssim \frac{\| \xi \|_K}{\sigma_{h_0}} \int_{r'}^\infty e^{ - \frac{c^*}{2\sigma_{h_0}} | r - \mu_{h_0} | + 1}  \biggl( \frac{1}{\sigma_{h_0}} + \frac{p-1}{r} \biggr) \, dr + \| \xi \|_K,
\end{align*}
where the final inequality follows from~\eqref{eqn:discontinuity_term}. Now define $r'' := \mu_{h_0}/2$ and note that
\begin{align*}
  \frac{1}{\sigma_{h_0}} \int_{r'}^\infty & e^{-\frac{1}{\sigma_{h_0}} | \mu_{h_0} - r | + 1 } \frac{p-1}{r} \,dr \\
  &\leq \frac{(p-1)(r'' - r')}{\sigma_{h_0} r'} e^{- \frac{\mu_{h_0}}{2\sigma_{h_0}} + 1} +
  \frac{p-1}{\sigma_{h_0} r''} \int_{r''}^\infty e^{- \frac{1}{\sigma_{h_0}} |\mu_{h_0} - r| + 1} \, dr \lesssim p^{1/2}.
  \end{align*}
Hence
  \begin{align}
   2p \lambda_p(K) \int_{-\infty}^{\phi_0(0)}
    e^t \biggl\{ \frac{ (r_0(t) + \| \xi \|_K )^p - r_0(t)^p }{p} \biggr\} \, dt  \lesssim p^{1/2} \| \xi \|_K.
    \label{eqn:integral_first_part}
  \end{align}
Returning to~\eqref{eqn:overall_bound2}, by a very similar argument, we also have that
  \begin{align}
    p \lambda_p(K) \int_{-\infty}^{\phi_0(0)} r_0(t)^pe^t \,dt
    &\leq p \lambda_p(K) \int_{-\infty}^{\phi^*(0)} r^*(t)^pe^t \,dt \nonumber \\
    &\leq p \lambda_p(K) \int_{r'}^{\infty} r^pe^{\phi^*(r)} | \phi^{*\prime}(r)| \,dr +
      p \lambda_p(K) r^{\prime p}e^{\phi^*(0)} \lesssim p. \label{Eqn:IntegralSecondPart}
  \end{align}
It follows from \eqref{eqn:overall_bound2},~\eqref{eqn:integral_first_part} and~\eqref{Eqn:IntegralSecondPart} that
  \begin{align*}
    d_{\mathrm{TV}}(\check{f}_n,f_0)& \lesssim p^{1/2} \| \xi \|_K + (\tau_{*}^{-p} - \tau^{* - p } ) \lesssim p^{1/2} \| \xi \|_K + p(\tau^* - \tau_*),
  \end{align*}
  as desired. If $p = 1$, then we define
  \begin{align}
    \phi^*(r) := \left \{ \begin{array}{cc}
                            \log \frac{1}{\sigma_{h_0} \lambda_1(K)} - \frac{1}{\sigma_{h_0}}(r - \mu_{h_0}) + 1 & \textrm{ if } r \in [\mu_{h_0}, \infty) \\
                            \phi^*(\mu_{h_0}) & \textrm{ if } r \in [0, \mu_{h_0})
                          \end{array} \right.
                                            \label{Eqn:p1PhiDefinition}
  \end{align}
and observe, as in the case when $p \geq 2$, that $\phi^*$ is decreasing, that $\phi^*(r) \geq \phi_0(r)$ for all $r \in [0, \infty)$, and that $\phi^{* \prime}(r) = 0$ for $r \in [0, \mu_{h_0})$ and $\phi^{* \prime}(r) = -\frac{1}{\sigma_{h_0}}$ for $r \in (\mu_{h_0}, \infty)$. By applying the same argument as for the case where $p \geq 2$, we obtain the conclusion of the lemma. 
\end{proof}

\begin{lemma}
  \label{Lem:dHWorstBound2}
If $ \| \hat{\mu} - \mu \|_{K} p \log (ep) \leq 1$ and $p(\tau^* - \tau_{*}) \leq 1/2$, then
  \begin{align*}
    d_{\mathrm{H}}^2 \Bigl(\tilde{h}_0, h_0 \frac{\lambda_p(\hat{K})}{\lambda_p(K)} \Bigr) \lesssim p^{1/2} \| \hat{\mu} - \mu \|_K + p(\tau^* - 1),
  \end{align*}
  and 
  \begin{align*}
    d_{\mathrm{H}}^2 \bigl(\tilde{h}_0,h_0 \bigr) \lesssim p^{1/2} \| \hat{\mu} - \mu \|_K + p(\tau^* - \tau_{*}).
  \end{align*}
\end{lemma}

\begin{proof}
Since Lemma~\ref{Lem:PerturbedDensityUpperBound} implies that $\gamma \geq 1$, we have
  \begin{align*}
    d_{\mathrm{H}}^2 \Bigl(\tilde{h}_0 , h_0 \frac{\lambda_p(\hat{K})}{\lambda_p(K)} \Bigr)
    &\leq \int_0^\infty \biggl| h_0(r) \frac{\lambda_p(\hat{K})}{\lambda_p(K)} - \tilde{h}_0(r) \biggr| \, dr  \\
    &=   p \lambda_p(\hat{K}) \int_0^\infty r^{p-1}
      \biggl| e^{\phi_0(r)} - \frac{1}{\gamma} e^{\tilde{\phi}_0(r)} \biggr| \, dr \\
    &\leq p \lambda_p(\hat{K}) \int_0^\infty r^{p-1} 
      \bigl( e^{\tilde{\phi}_0(r)} - e^{\phi_0(r)} \bigr) \, dr + \gamma - 1.
  \end{align*}
  We then follow the proof of Lemma~\ref{Lem:GammaBound} and use~\eqref{Eqn:GammaBound1} and Lemma~\ref{Lem:VolumeRatio} to obtain the first statement of the lemma. For the second statement of the lemma, observe that we may use Lemma~\ref{Lem:VolumeRatio} and our assumption on $p(\tau^* - \tau_*)$ to obtain
  \begin{align*}
    d_{\mathrm{H}}^2 \bigl(\tilde{h}_0,h_0 \bigr)
    &\leq 2 d_{\mathrm{H}}^2 \biggl(\tilde{h}_0, h_0 \frac{\lambda_p(\hat{K})}{\lambda_p(K)} \biggr) +
    2 \biggl( 1 - \sqrt{ \frac{\lambda_p(\hat{K})}{\lambda_p(K)} } \biggr)^2 \\
    &\leq 2 d_{\mathrm{H}}^2 \biggl(\tilde{h}_0, h_0 \frac{\lambda_p(\hat{K})}{\lambda_p(K)} \biggr) +
      2 ( \tau_*^{-p/2} - \tau^{* - p/2})^2 \\
    &\lesssim p^{1/2} \| \hat{\mu} - \mu \|_K + p(\tau^* - \tau_{*}),
  \end{align*}
  as desired.
\end{proof}

\begin{lemma}
  \label{Lem:KLWorstBound}
If $ \| \hat{\mu} - \mu\|_{K} p \log (ep) \leq 1$ and $p(\tau^* - \tau_{*}) \leq 1/2$, then
  \begin{align*}
    d^2_{\mathrm{KL}}(\tilde{h}_n, \tilde{h}_0) \lesssim p^{1/2} \| \hat{\mu} - \mu \|_K + p(\tau^* - \tau_{*}).
  \end{align*}
\end{lemma}
\begin{proof}
By Lemma~\ref{Lem:PerturbedDensityUpperBound}, the fact that $\log(1+x) \leq x$ for all $x \in [0, \infty)$, and the fact that $\gamma \geq 1$, 
  \begin{align*}
    \int_{\mathbb{R}^p} & e^{\phi_0(\|x - \mu\|_K)} \log \frac{ \gamma e^{\phi_0(\|x - \mu\|_K)}}{e^{\tilde{\phi}_0(\|x - \hat{\mu}\|_{\hat{K}})}} \, dx \leq \int_{\mathbb{R}^p} e^{\phi_0(\|x - \mu\|_K)} \Bigl| \log \frac{ e^{\tilde{\phi}_0(\|x - \hat{\mu}\|_{\hat{K})}}}{e^{\phi_0(\|x - \mu\|_{K})}} \Bigr| \, dx + \log \gamma\\
    &\leq \int_{\mathbb{R}^p} \bigl(e^{\tilde{\phi}_0(\|x - \hat{\mu}\|_{\hat{K}})} - e^{\phi_0(\|x - \mu\|_K)} \bigr) \, dx + \log \gamma \\
    &\leq \int_{\mathbb{R}^p} \bigl(e^{\tilde{\phi}_0(\|x - \hat{\mu}\|_{\hat{K}})} - e^{\phi_0(\|x - \hat{\mu}\|_{\hat{K}})} \bigr) \, dx + \int_{\mathbb{R}^p} \bigl|  e^{\phi_0(\|x - \hat{\mu}\|_{\hat{K}})} - e^{\phi_0(\|x - \mu\|_K)} \bigr| \, dx + \log \gamma \\
                     &\leq p \lambda_p(\hat{K}) \int_0^\infty r^{p-1} \bigl( e^{\tilde{\phi}_0(r)} - e^{\phi_0(r)} \bigr) \, dr +  \int_{\mathbb{R}^p} \bigl|  e^{\phi_0(\|x - \hat{\mu}\|_{\hat{K}})} - e^{\phi_0(\|x - \mu\|_K)} \bigr| \, dx + \gamma - 1.
  \end{align*}
  By Lemma~\ref{Lem:WorstTVBound}, Lemma~\ref{Lem:GammaBound}, and~\eqref{Eqn:GammaBound1} in the proof of Lemma~\ref{Lem:GammaBound}, we have that
  \begin{align}
    \int_{\mathbb{R}^p}  e^{\phi_0(\|x - \mu\|_K)} \log \frac{ \gamma e^{\phi_0(\|x - \mu\|_K)}}{e^{\tilde{\phi}_0(\|x - \hat{\mu}\|_{\hat{K}})}} \, dx \lesssim
    p^{1/2} \| \hat{\mu} - \mu \|_K + p(\tau^* - \tau_*) < \infty.
    \label{Eqn:KLBound1}
  \end{align}
  We observe that $\tilde{h}_n$ is the density of the $\hat{K}$-contour measure (see definition above Lemma~\ref{Lem:ContourLevelExistence}) of the probability measure induced by $f_0(\cdot + \hat{\mu})$. By~\eqref{Eqn:KLBound1} and Lemma~\ref{lem:contour_KL}, it holds that
  \begin{align*}
d^2_{\mathrm{KL}}(\tilde{h}_n, \tilde{h}_0) &\leq \int_{\mathbb{R}^p}  e^{\phi_0(\|x - \mu\|_K)} \log \frac{ \gamma e^{\phi_0(\|x - \mu\|_K)}}{e^{\tilde{\phi}_0(\|x - \hat{\mu}\|_{\hat{K}})}} \, dx\\
    &\lesssim   p^{1/2} \| \hat{\mu} - \mu \|_K + p(\tau^* - \tau_*),
  \end{align*}
  as desired.  
\end{proof}
\begin{lemma}
  \label{Lem:GammaBound}
  If $\| \hat{\mu} - \mu \|_{K} p \log (ep) \leq 1$ and $p(\tau^* - \tau_{*}) \leq 1/2$, then
  \begin{align*}
    \gamma - 1 \lesssim  p^{1/2} \| \hat{\mu} - \mu\|_{K} + p (\tau^* - \tau_{*}) .
  \end{align*}
\end{lemma}
\begin{proof}
By the definition of $\gamma$ and Lemma~\ref{Lem:VolumeRatio}, 
  \begin{align}
    \gamma - 1 &= p \int_0^\infty r^{p-1}
                   \bigl\{ \lambda_p(\hat{K}) e^{\tilde{\phi}_0(r)} -
                   \lambda_p(K) e^{\phi_0(r)} \bigr\} \, dr \nonumber \\
                 &=  p \lambda_p(\hat{K}) \int_0^\infty r^{p-1}
                   \bigl( e^{\tilde{\phi}_0(r)} - e^{\phi_0(r)} \bigr) \,dr \nonumber \\
                 &\hspace{3cm}+ p \lambda_p(K) \int_0^\infty r^{p-1} 
                   \biggl( \frac{\lambda_p(\hat{K})}{\lambda_p(K)} e^{\phi_0(r)} - e^{\phi_0(r)} \biggr) \, dr \nonumber \\
       &\leq 2 p \lambda_p(K) \int_0^\infty r^{p-1} 
      \bigl( e^{\tilde{\phi}_0(r)} - e^{\phi_0(r)} \bigr) \,dr + 2p(1-\tau_{*}). \label{eqn:gamma_bound_overall}
  \end{align}
We define $r_0, \tilde{r}_0 : (-\infty, \phi_0(0)] \rightarrow [0, \infty)$ by $r_0(t) := \sup \{ r \geq 0 : \phi_0(r) \geq t \}$ and $\tilde{r}_0(t) := \sup \{ r \geq 0 : \tilde{\phi}_0(r) \geq t \}$; we also write $\xi := \hat{\mu} - \mu$.  Observe that $\tilde{\phi}_0\bigl( \tau^*(r_0(t) + \| \xi \|_{\hat{K}} + \epsilon) \bigr) < t$ for any $t \in (-\infty, \phi_0(0)]$ and $\epsilon > 0$, so $\tilde{r}_0(t) \leq \tau^*(r_0(t) + \|\xi \|_{\hat{K}})$. By Fubini's theorem, 
\begin{align}
  \label{Eq:gammabound1}
  p \lambda_p(K) \int_0^\infty r^{p-1} \bigl( e^{\tilde{\phi}_0(r)} - e^{\phi_0(r)} \bigr) \, dr &=  p \lambda_p(K) \int_0^\infty r^{p-1} \int_{-\infty}^{\phi_0(0)} e^t \mathbbm{1}_{\{ t \in ( \phi_0(r), \tilde{\phi}_0(r)] \}} \,dt \,dr \nonumber \\
   &= p \lambda_p(K)  \int_{-\infty}^{\phi_0(0)} e^t \int_0^\infty r^{p-1} \mathbbm{1}_{\{ t \in ( \phi_0(r), \tilde{\phi}_0(r)] \}} \,dr \,dt \nonumber \\
  &=  p \lambda_p(K) \int_{-\infty}^{\phi_0(0)} e^t  \int_0^\infty r^{p-1} \mathbbm{1}_{\{ r \in ( r_0(t), \tilde{r}_0(t)] \}} \,dr \,dt \nonumber \\
                   &= p \lambda_p(K) \int_{-\infty}^{\phi_0(0)} e^t  \biggl( \frac{\tilde{r}_0(t)^p - r_0(t)^p }{p} \biggr) \, dt \nonumber \\
                   &\leq  p \lambda_p(K) \int_{-\infty}^{\phi_0(0)} e^t \biggl( \frac{ \tau^{* p} (r_0(t) + \|\xi\|_{\hat{K}})^p - r_0(t)^p}{p} \biggr) \,dt.
\end{align}
Let us first assume $p \geq 2$ and define $r'$, $\phi^*$, and $r^*$ as in the proof of Lemma~\ref{Lem:WorstTVBound}. Recall from the proof of Lemma~\ref{Lem:WorstTVBound} that $r^*(t) \geq r_0(t)$ for all $t \in (-\infty, \phi_0(0)]$. Since $\tau^* \geq 1$, it also holds that $\tau^{*p} (r_0(t) + \|\xi\|_{\hat{K}})^p - r_0(t)^p \leq \tau^{*p} (r^*(t) + \|\xi\|_{\hat{K}})^p - r^*(t)^p$ for all $t \in (-\infty, \phi_0(0)]$.  We may now follow the proof of Lemma~\ref{Lem:WorstTVBound} and apply the assumption on $\|\xi \|_K p \log(ep)$ and a change of variable to obtain
\begin{align}
 p \lambda_p(K) \int_{-\infty}^{\phi_0(0)} &e^t \biggl( \frac{ \tau^{* p} (r_0(t) + \|\xi\|_{\hat{K}})^p - r_0(t)^p}{p} \biggr) \,dt \nonumber \\
  &\leq  p \lambda_p(K)  \int_{-\infty}^{\phi^*(0)} e^t\biggl( \frac{ \tau^{* p} (r^*(t) + \| \xi \|_{\hat{K}})^p - r^*(t)^p}{p} \biggr) \,dt  \nonumber \\
               &\leq p \lambda_p(K) r^{\prime p}e^{\phi^*(0)}\biggl( \frac{\tau^{* p} (1 + \| \xi \|_{\hat{K}}/r')^p - 1 }{p} \biggr) \nonumber \\
              &\quad + p \lambda_p(K) \int_{r'}^\infty |\phi^{* \prime}(r)| r^pe^{\phi^*(r)}\biggl( \frac{\tau^{* p} ( 1 + \| \xi \|_{\hat{K}}/r)^p - 1}{p} \biggr) \,dr. \label{eqn:gamma_integral_overall}
\end{align}
Now, in order to bound the first term of~\eqref{eqn:gamma_integral_overall}, we use~\eqref{eqn:discontinuity_term}, the assumptions on $\tau^*$ and $\|\xi\|_K$, and the fact that $\| \xi \|_{\hat{K}} \leq \| \xi \|_K \tau^* \leq 2 \| \xi \|_K$ to obtain
\begin{align}
  \label{Eq:gammabound2}
  p \lambda_p(K) e^{\phi^*(0)}r^{\prime p}
  &\biggl( \frac{\tau^{* p} (1 + \| \xi \|_{\hat{K}}/r')^p - 1 }{p} \biggr) \nonumber \\
  &\lesssim \frac{(\tau^{*p} - 1)r'}{p} (1 + \| \xi \|_{\hat{K}}/r')^p +
    \| \xi \|_{\hat{K}} \lesssim  (\tau^* - 1) + \| \xi \|_K.
\end{align}
To bound the second term of~\eqref{eqn:gamma_integral_overall}, we again follow the proof of Lemma~\ref{Lem:WorstTVBound} and use~\eqref{eqn:integral_first_part} and~\eqref{Eqn:IntegralSecondPart}:
\begin{align}
  \label{Eq:gammabound3}
  p \lambda_p(K) \int_{r'}^\infty |\phi^{* \prime}(r)| &r^p e^{\phi^*(r)}
    \biggl( \frac{\tau^{* p} ( 1 + \| \xi \|_{\hat{K}}/r)^p - 1}{p} \biggr) \,dr \nonumber \\
  &\quad \lesssim   \frac{(\tau^{*p} - 1)}{p} (1 + \|\xi\|_{\hat{K}}/r')^p p \lambda_p(K)  \int_{r'}^\infty  |\phi^{* \prime}(r)| r^p e^{\phi^*(r)} \, dr  \nonumber \\
  &\qquad \qquad  +
    p \lambda_p(K)  \int_{r'}^\infty |\phi^{*\prime}(r)| r^p e^{\phi^*(r)}  \biggl( \frac{( 1 + \| \xi \|_{\hat{K}}/r)^p - 1}{p} \biggr) \,dr \nonumber \\
  &\quad \lesssim p (\tau^* - 1) + p^{1/2} \| \xi \|_{\hat{K}} \lesssim p(\tau^* - 1) + p^{1/2} \| \xi \|_K.
\end{align}
From~\eqref{Eq:gammabound1},\eqref{eqn:gamma_integral_overall},~\eqref{Eq:gammabound2} and~\eqref{Eq:gammabound3}, we deduce that 
\begin{align}
p \lambda_p(K) \int_0^\infty r^{p-1} \bigl( e^{\tilde{\phi}_0(r)} - e^{\phi_0(r)} \bigr) \, dr \lesssim p^{1/2} \| \xi \|_K + p(\tau^* - 1) \label{Eqn:GammaBound1}.
\end{align}
The desired result therefore follows from~\eqref{eqn:gamma_bound_overall}.

If $p = 1$, then we define $\phi^*$ as in~\eqref{Eqn:p1PhiDefinition} and obtain the same bound. Thus, in all cases, 
\[
  \gamma - 1 \lesssim p^{1/2} \| \xi \|_K + p (\tau^* - \tau_*),
\]
as required.
\end{proof}

\begin{lemma}
  \label{Lem:VolumeRatio}
  If $p(1 - \tau_*) \leq 1/2$, then
  \[
    1 - p(\tau^* - 1) \leq \tau^{* -p} \leq  \frac{\lambda_p(\hat{K})}{\lambda_p(K)} \leq \tau_*^{-p} \leq  1 + 2p(1 - \tau_*).
  \]  
\end{lemma}

\begin{proof}
We first prove the upper bound on $\lambda_p(\hat{K})/\lambda_p(K)$. Since $-p\log(1-x/p) \leq -\log(1-x)$ for $x \in (0,1)$, we have
  \begin{align*}
  \frac{\lambda_p(\hat{K})}{\lambda_p(K)} &=
  \frac{\lambda_p\bigl(\{ x : \|x\|_{\hat{K}} \leq 1\}\bigr)}{\lambda_p\bigl(\{x : \|x\|_K \leq 1\}\bigr)} \leq
    \frac{\lambda_p\bigl(\{ x : \|x\|_K \leq \tau^{-1}_{*}\}\bigr)}{\lambda_p\bigl(\{ x : \|x\|_K \leq 1\}\bigr)} = \tau^{-p}_{*} \\
    &= \frac{1}{\bigl\{1-p(1-\tau_*)/p\bigr\}^p} \leq \frac{1}{1-p(1-\tau_*)} \leq 1 + 2p(1 - \tau_*),
  \end{align*}
when $p(1 - \tau_*) \leq 1/2$, as required. 

  For the lower bound on $\lambda_p(\hat{K})/\lambda_p(K)$, by the fact that $-p\log(1+x/p) \geq \log(1-x)$ for $x \in (0,1)$, we have
    \begin{align*}
  \frac{\lambda_p(\hat{K})}{\lambda_p(K)} &=
  \frac{\lambda_p\bigl(\{ x : \|x\|_{\hat{K}} \leq 1\}\bigr)}{\lambda_p\bigl(\{x : \|x\|_K \leq 1\}\bigr)} \geq
      \frac{\lambda_p\bigl(\{ x : \|x\|_{\hat{K}} \leq 1\}\bigr)}{\lambda_p\bigl(\{ x : \|x\|_{\hat{K}} \leq \tau^* \}\bigr)} = \tau^{* -p} \\
      &= \frac{1}{\bigl\{1 + p(\tau^*-1)/p\bigr\}^p} \geq 1 - p(\tau^*-1),
  \end{align*}
  as required.
\end{proof}

\begin{lemma}
  \label{lem:contour_KL}
  Let $K \in \mathcal{K}$ and let $f$ be a density on $\mathbb{R}^p$ with corresponding distribution $\nu$. Let $h : [0, \infty) \rightarrow [0, \infty)$ be the density of the $K$-contour measure (see Definition before Lemma~\ref{Lem:ContourLevelExistence}) of $\nu$. Let $g : [0, \infty) \rightarrow [0, \infty)$ be a continuous function such that $\int_{\mathbb{R}^p} g(\| x \|_K) \, dx = 1$ and suppose that 
\begin{equation}
\label{Eq:KL}
\int_{\mathbb{R}^p} f(x) \log \frac{f(x)}{g( \| x \|_K)} \,dx < \infty.
\end{equation}
Then
  \[
    \int_{\mathbb{R}^p} f(x) \log \frac{f(x)}{g( \| x \|_K)} \,dx \geq
    \int_0^\infty h(r) \log \frac{h(r)}{p \lambda_p(K)r^{p-1}g(r)} \, dr .
  \]
  
\end{lemma}

\begin{proof}
  By condition~\eqref{Eq:KL}, we may define a signed measure $\pi$ on $(\mathbb{R}^p, \mathcal{B}(\mathbb{R}^p))$ by $\pi(A) := \int_A f(x) \log \frac{f(x)}{g(\|x \|_K)} \,dx$ for any Borel measurable $A$. Let $\pi^K_1$ denote the $K$-contour measure of~$\pi$. Since $\pi \ll \lambda_p$, we have by Lemma~\ref{Lem:ContourLevelExistence} that $\pi_1^K \ll \lambda_1$. Let $\tilde{L} \,:\, [0, \infty) \rightarrow [0,\infty)$ denote the Radon--Nikodym derivative of $\pi_1^K$ with respect to $\lambda_1$. For any $r \in (0, \infty)$ and $\epsilon \in (0,r)$, define $A_{r,\epsilon} := (r+\epsilon)K \setminus rK$.

  We observe that, by~\eqref{Eq:KL} and the fact that the integral in that assumption is also non-negative by the Gibbs inequality, $\pi$ is locally bounded on $\mathbb{R}^p$ in the sense that $|\pi(B)| < \infty$ for every bounded $B \in \mathcal{B}(\mathbb{R}^p)$ and therefore, $\pi^K_1$ is locally bounded on $[0,\infty)$. Then, by the Lebesgue differentiation theorem \citep[][Theorem 3.21]{folland2013real}, for $\lambda_1$-almost every $r \in (0,\infty)$, 
 \[
\tilde{L}(r) = \lim_{\epsilon \searrow 0} \frac{1}{\epsilon} \int_{A_{r,\epsilon}} f(x) \log \frac{f(x)}{g(\|x\|_K)} \, dx.
\]
Moreover, we claim that $\pi \ll \nu$. To see this, let $A \in \mathcal{B}(\mathbb{R}^p)$ be such that $\nu(A) = 0$. Then, by definition of $\nu$, we have that $f$ is $\lambda_1$-almost everywhere $0$ on $A$ and thus we conclude that $\pi(A) = 0$ as well, which establishes the claim. By~\eqref{Eq:KL} again, it must be that $\nu_1^K( g^{-1}(0)) = \nu( \{ x \,:\, g(\|x\|_K) = 0 \}) = 0$ and thus $\pi_1^K( g^{-1}(0) ) = 0$. Hence, we have that $\tilde{L}(r) = 0$ for $\lambda_1$-almost every $r \in [0,\infty)$ such that $g(r) = 0$.
  
For $r > 0$, let us define
\[
L(r) := \left\{ \begin{array}{ll} \limsup_{\epsilon \searrow 0} \frac{1}{\epsilon} \int_{A_{r,\epsilon}} f(x)\log \frac{f(x)}{g(r)} \, dx & \mbox{if $g(r) > 0$} \\
0 & \mbox{if $g(r) = 0$,} \end{array} \right.
\]
and define $L(0) = 0$. For any $r$ such that $g(r) > 0$, we have by continuity of $g$ that
\[
\lim_{\epsilon \searrow 0} \frac{1}{\epsilon} \int_{A_{r,\epsilon}} f(x)\log \frac{g(\|x\|_K)}{g(r)} \, dx = 0,
\]
so $L(r) = \tilde{L}(r)$ $\lambda_1$-almost everywhere. 

We also observe that $\nu$ is locally bounded since $f$ is a density. Thus, $\nu^K_1$ is locally bounded and consequently, there exists $E \in \mathcal{B}\bigl((0,\infty)\bigr)$ such that $\lim_{\epsilon \searrow 0} \epsilon^{-1} \int_{A_{r,\epsilon}} f(x) \, dx = h(r)$ for every $r \in E$ and $\lambda_1\bigl([0,\infty) \setminus E\bigr) = 0$.  Fix $\epsilon > 0$ and $r \in E$ with $g(r) > 0$.  Since $y \mapsto y \log y$ is convex on $[0, \infty)$ (with the convention that $0\log 0 = 0$), we have by Jensen's inequality that
  \begin{align}
\label{eqn:contour_limit}
    \frac{1}{\lambda_p(A_{r,\epsilon})} \int_{A_{r,\epsilon}} f(x) \log \frac{f(x)}{g(r)} \, dx
    &\geq \biggl( \frac{1}{\lambda_p(A_{r,\epsilon})}
      \int_{A_{r,\epsilon}} f(x) \, dx \biggr) \log \frac{\int_{A_{r,\epsilon}} f(x) \, dx}{ \lambda_p(A_{r,\epsilon}) g(r)}.
  \end{align}

Hence, from \eqref{eqn:contour_limit} and the equality case of Minkowski's first inequality for convex bodies \citep[][Section 5]{gardner2002brunn},
  \begin{align*}
   L(r) \geq \limsup_{\epsilon \searrow 0} \biggl\{\frac{1}{\epsilon} \int_{A_{r,\epsilon}} f(x) \, dx \log
    \frac{\epsilon^{-1} \int_{A_{r,\epsilon}} f(x) \, dx}{ \epsilon^{-1} \lambda_p(A_{r,\epsilon}) g(r)}\biggr\} = h(r) \log \frac{h(r)}{  p \lambda_p(K)r^{p-1}g(r)}.
  \end{align*}
We therefore conclude that
  \[
    \int_{\mathbb{R}^p} f(x) \log \frac{f(x)}{g(\|x\|_K)}\, dx = \int_0^\infty \tilde{L}(r) \,dr = \int_0^\infty L(r) \, dr \geq
    \int_0^\infty h(r) \log \frac{h(r)}{p \lambda_p(K) r^{p-1} g(r)} \, dr,
  \]
  as required.
\end{proof}

\subsubsection{Auxiliary lemmas for the adaptive risk bounds of Section~\ref{Sec:KEstimated}}

\begin{lemma}
  \label{Lem:dHBestCase1}
  Suppose that $p(\tau^* - \tau_{*}) \leq 1/2$ and that $\| \hat{\mu}- \mu \|_K \leq p^{1/2}$. If $\phi_0'$ is absolutely continuous and differentiable, and there exists $D_0 > 0$ such that $\inf_{r \in [0,\infty)} \phi_0''(r) \geq - D_0$, then
 \[
   d_{\mathrm{H}}^2( \check{f}_n, f_0) \lesssim
   (D_0^2 + 1) \bigl\{ p \| \hat{\mu} - \mu \|^2_K + p^2 (\tau^* - \tau_{*})^2 \bigr\}.
 \]
\end{lemma}

\begin{proof}
We have by absolute continuity of $\phi_0'$ that
\begin{equation}
  \label{Eq:phi0prime}
\phi_0'(r) = \int_0^r \phi_0''(s) \, ds + \phi_0'(0) \geq - D_0 r + \phi_0'(0).
\end{equation}
For any $a \in [1/2, 2]$, the function $r \mapsto a p \lambda_p(K) (ar)^{p-1}e^{\phi_0(ar)}$ is a density on $[0,\infty)$. Moreover, since the function $r \mapsto \sup_{a \in [1/2, 2]} a p \lambda_p(K) (ar)^{p-1}e^{\phi_0(ar)}$ is finite when integrated over $[0,\infty)$, we may, by~\citet[][Exercise 1.6.3]{ash1999probability}, differentiate under the integral to obtain
    \begin{align}
      \label{Eq:DiffIntegral}
  0 &= p \lambda_p(K)\frac{\partial}{\partial a} \int_0^\infty a (ar)^{p-1}e^{\phi_0(ar)} \, dr \biggm|_{a=1} \nonumber \\
    &= p \lambda_p(K) \int_0^\infty \phi_0'(r) r^p e^{\phi_0(r)} \,dr + p^2 \lambda_p(K) \int_0^\infty r^{p-1}e^{\phi_0(r)} \, dr \nonumber \\
  &\leq \phi_0'(0) (p-1)^{1/2} + p,
\end{align}
where the final inequality follows from \citet[][Lemma~1]{bobkov2003spectral}.  From~\eqref{Eq:phi0prime} and~\eqref{Eq:DiffIntegral}, we deduce that 
\begin{equation}
  \label{Eq:phi0prime2}
  \phi_0'(r)^2 \lesssim D_0^2 r^2 + p. 
\end{equation}
Using the inequality $|e^z-1| \leq |z|(1 + e^{z})$ for $z \in \mathbb{R}$, and writing $z(x) := \{\phi_0(\|x - \hat{\mu}\|_{\hat{K}}) - \phi_0(\| x - \mu\|_K)\}/2$, we have
\begin{align}
  d_{\mathrm{H}}^2 (\check{f}_n,f_0)
  = \int_{\mathbb{R}^p} e^{\phi_0(\| x - \mu \|_K)}
    \bigl( e^{z(x)} - 1 \bigr)^2 \, dx
  &\leq 2\int_{\mathbb{R}^p}  e^{\phi_0(\| x - \mu \|_K)}z(x)^2\{1 + e^{2z(x)}\} \, dx.
    \label{eqn:dH_best1_overall}
\end{align} 
As a shorthand, write $\Delta := (\tau^* - \tau_{*})$, $\Delta' := (\tau_{*}^{-1} - \tau^{*-1})$ and $\xi := \hat{\mu} - \mu$. By Lemma~\ref{Lem:KPerturbation}, Taylor's theorem, and the facts that $\tau^* \in [1,2]$ and $\tau_* \leq [1/2,1]$ by our assumption on $p(\tau^* - \tau_*)$, we have that
\begin{align}
  \label{Eq:BigMin}
 |z(x)| \leq &
    \biggl\{ | \phi_0'(\|x - \mu \|_K) |
    (\Delta \| x - \mu\|_K + 3 \| \xi \|_K) +
    \frac{D_0}{2}  (\Delta \| x - \mu\|_K + 3 \| \xi \|_K) ^2 \biggr\} \wedge \nonumber \\
  & \biggl\{  | \phi_0'(\|x - \hat{\mu} \|_{\hat{K}}) |
    (\Delta' \| x - \hat{\mu}\|_{\hat{K}} + 3 \| \xi \|_{\hat{K}})
    + \frac{D_0}{2}  (\Delta' \| x - \hat{\mu}\|_{\hat{K}} + 3 \| \xi \|_{\hat{K}})^2 \biggr\}. 
\end{align}
By~\eqref{Eq:BigMin}, Lemma~\ref{Lem:ChangeOfVar},~\eqref{Eq:phi0prime2} and Lemma~\ref{Lem:PhiMoment}, we have that 
\begin{align}
  \int_{\mathbb{R}^p}
  &z(x)^2 e^{\phi_0(\| x - \mu \|_K)} \, dx \nonumber \\
  &\leq p \lambda_p(K)\int_0^\infty \biggl\{ |\phi_0'(r)| (\Delta r + 3 \| \xi \|_K) + \frac{D_0}{2} ( \Delta r + 3 \| \xi \|_K)^2 \biggr\}^2 r^{p-1}e^{\phi_0(r)} \, dr  \nonumber \\
  &\lesssim p \lambda_p(K)\int_0^\infty \bigl\{ \phi'_0(r)^2 \Delta^2 r^2 + \phi'_0(r)^2 \| \xi \|_K^2 + D_0^2 \Delta^4 r^4 + D_0^2 \| \xi \|_K^4 \bigr\} r^{p-1}e^{\phi_0(r)} \, dr \nonumber \\
  &\lesssim (D_0^2+1) \Delta^2 p^2 + (D_0^2+1) p \|\xi\|_K^2 + D_0^2 \Delta^4 p^2 + D_0^2 \|\xi\|_K^4 \nonumber \\
  & \lesssim (D_0^2+1) (p \|\xi\|_K^2 + \Delta^2p^2). \label{eqn:dH_best1_first_term} 
\end{align}
Similarly, by Lemma~\ref{Lem:VolumeRatio},
\begin{align}
  \int_{\mathbb{R}^p}
  &\bigl\{ \phi_0(\|x-\mu\|_K) - \phi_0(\|x-\hat{\mu}\|_{\hat{K}})\bigr\}^2e^{\phi_0(\| x - \hat{\mu} \|_{\hat{K}})} \, dx  \nonumber \\
  &\leq p \lambda_p(\hat{K}) \int_0^\infty
    \biggl\{ | \phi_0'(r)| ( \Delta' r + 3 \| \xi \|_{\hat{K}}) + \frac{D_0}{2}( \Delta' r + 3 \| \xi \|_{\hat{K}})^2 \biggr\}^2
    r^{p-1}e^{\phi_0(r)} \, dr \nonumber \\
  &\lesssim  p \lambda_p(K) \int_0^\infty
    \bigl\{ \phi'_0(r)^2 \Delta^{\prime 2}r^2 +
    \phi'_0(r)^2 \| \xi \|_{\hat{K}}^2 +
    D_0^2 \Delta^{\prime 4} r^4 + D_0^2 \| \xi \|_{\hat{K}}^4 \bigr\}
   r^{p-1}e^{\phi_0(r)} \, dr \nonumber \\
  &\lesssim (D_0^2 + 1) \Delta^{\prime 2}p^2 + (D_0^2 + 1) p \| \xi \|_{\hat{K}}^2 + D_0^2 \Delta^{\prime 4}p^2 + D_0^2 \| \xi \|_{\hat{K}}^4 \nonumber \\
   &\lesssim (D_0^2+1) (p\|\xi\|_K^2 + \Delta^2p^2),\label{eqn:dH_best1_second_term}
\end{align}
where the final inequality follows because $\| \xi \|_{\hat{K}} \leq \| \xi \|_K \tau^*$ and $\Delta' = \Delta/(\tau^* \tau_{*})$.  Combining \eqref{eqn:dH_best1_overall}, \eqref{eqn:dH_best1_first_term} and \eqref{eqn:dH_best1_second_term} yields the desired conclusion. 
\end{proof}

\begin{lemma}
  \label{Lem:dHBestCase2}
  There exists universal constants $c_1, c_2 > 0$ such that if $p(\tau^* - \tau_{*}) \leq c_2$ and $\| \hat{\mu}- \mu \|_K p \log(ep) \leq c_1$ and if $\phi_0'$ is absolutely continuous and differentiable and there exists $D_0 > 0$ such that $\inf_{r \in [0,\infty)} \phi_0''(r) \geq - D_0$, then
 \[
   d_{\mathrm{H}}^2 \biggl( \tilde{h}_0,h_0 \frac{\lambda_p(\hat{K})}{\lambda_p(K)}\biggr) \lesssim
   (D_0^2 + 1) \bigl\{ p \| \hat{\mu} - \mu\|_K^2 + p^2 (\tau^* - \tau_*)^2 \bigr\} ,
 \]
 and
 \[
   d_{\mathrm{H}}^2 (\tilde{h}_0 , h_0 ) \lesssim
   (D_0^2 + 1) \bigl\{ p \| \hat{\mu} - \mu\|_K^2 + p^2 (\tau^* - \tau_*)^2 \bigr\}.
 \]
 \end{lemma}
\begin{proof}
We first note that~\eqref{Eq:phi0prime2} holds since we have the same assumptions on $\phi_0$ as Lemma~\ref{Lem:dHBestCase1}. Define $\psi \,:\, \mathbb{R}^p \rightarrow [0,\infty)$ by
  \begin{align}
    \psi(x) := \biggl\{ \begin{array}{ll}
                          0 & \textrm{if } \|x\|_{\hat{K}} \leq \|\hat{\mu} - \mu\|_{\hat{K}} \\
                          \frac{\|x\|_{\hat{K}} - \|\hat{\mu} - \mu \|_{\hat{K}}}{\tau^*} & \textrm{otherwise,}
                        \end{array} \label{eqn:psi_defn}
    \end{align}
so that $\tilde{\phi}_0(\|x \|_{\hat{K}}) = \phi_0(\psi(x))$. For $x \in \mathbb{R}^p$, write $z(x) := \{ \phi_0(\psi(x)) - \log \gamma - \phi_0(\| x \|_{\hat{K}}) \}/2$. By Lemma~\ref{Lem:ChangeOfVar} and the inequality $|e^z - 1| \leq |z|(1 + e^z)$ for $z \in \mathbb{R}$, 
  \begin{align}
    d_{\mathrm{H}}^2\biggl( \tilde{h}_0, h_0 \frac{\lambda_p(\hat{K})}{\lambda_p(K)} \biggr)
    &= \int_{\mathbb{R}^p} 
      \Bigl( e^{z(x) } - 1 \Bigr)^2 e^{\phi_0(\|x\|_{\hat{K}})} \, dx \nonumber \\
    &\leq 2 \int_{\mathbb{R}^p}  z(x)^2 e^{\phi_0(\|x\|_{\hat{K}})} \, dx
      + \frac{2}{ \gamma} \int_{\mathbb{R}^p} z(x)^2  e^{\tilde{\phi}_0(\|x\|_{\hat{K}})} \, dx. \label{eqn:dH_best_overall2}
 \end{align}                                                                                    
 Let us write $\xi := \hat{\mu} - \mu$. Since $0 \leq \| x \|_{\hat{K}} - \psi(x) \leq (1 - \tau^{*-1})\|x \|_{\hat{K}} + \| \xi \|_{\hat{K}}$, we have by Taylor's theorem that
 \begin{align*}
   | z(x) |
   &\leq  \phi_0(\psi(x)) - \phi_0(\|x\|_{\hat{K}}) + |\log \gamma | \\
   &\leq |\phi_0'(\|x \|_{\hat{K}}) | \bigl\{  (1 - \tau^{*-1})\|x \|_{\hat{K}} + \| \xi \|_{\hat{K}} \bigr\} + \frac{D_0}{2} \bigl\{ (1 - \tau^{*-1})\|x \|_{\hat{K}} + \| \xi \|_{\hat{K}} \bigr\}^2 + |\log \gamma|.
 \end{align*}
Now, write $\Delta := \tau^* - \tau_*$, and note that $\Delta \geq 1 - 1/\tau^*$. By Lemma~\ref{Lem:GammaBound} and the fact that $\gamma \geq 1$ (a consequence of Lemma~\ref{Lem:PerturbedDensityUpperBound}), we have that
 \begin{align}
  \log^2 \gamma \leq (\gamma - 1)^2 \lesssim p^2 \Delta^2 + p \| \xi \|^2_K. \label{eqn:log_gamma_bound}
 \end{align}
 Therefore, we have by Lemma~\ref{Lem:ChangeOfVar},~\eqref{Eq:phi0prime2}, Lemma~\ref{Lem:PhiMoment},~\eqref{eqn:log_gamma_bound}, and Lemma~\ref{Lem:VolumeRatio} that
 \begin{align}
 \label{Eq:BestTerm1}
   &\int_{\mathbb{R}^p} 
    z(x)^2 e^{\phi_0(\|x \|_{\hat{K}})} \, dx \nonumber \\
   &\lesssim p \lambda_p(\hat{K}) \int_0^\infty \bigl\{ \phi_0'(r)^2 (r^2\Delta^2 + \| \xi \|_{\hat{K}}^2)
     + D_0^2 \Delta^4 r^4 + D_0^2 \| \xi \|_{\hat{K}}^4 + p^2 \Delta^2 + p \|\xi\|^2_K  \bigr\}r^{p-1} e^{\phi_0(r)} \, dr
       \nonumber \\
    &\lesssim  (D_0^2+1) p^2 \Delta^2 + (D_0^2+1) p \| \xi \|_K^2 + D_0^2 p^2 \Delta^4 + D_0^2 \| \xi \|_K^4  \nonumber \\
    & \lesssim (D_0^2+1) (p^2 \Delta^2 + p \| \xi \|_K^2).
 \end{align}
Similarly, but using Lemma~\ref{Lem:PhiTildeMoment} instead of Lemma~\ref{Lem:PhiMoment}, we have that
 \begin{align}
 \label{Eq:BestTerm2}
   \frac{1}{\gamma}\int_{\mathbb{R}^p} 
   & z(x)^2 e^{\tilde{\phi}_0(\|x \|_{\hat{K}})} \, dx \nonumber \\
   &\lesssim  p \lambda_p(\hat{K}) \int_0^\infty \bigl\{ \phi_0'(r)^2 (r^2\Delta^2 + \| \xi \|_{\hat{K}}^2)
     + D_0^2 \Delta^4 r^4 + D_0^2 \| \xi \|_{\hat{K}}^4 + p^2 \Delta^2 + p \|\xi\|_K^2  \bigr\}
     \nonumber \\
   &\qquad \qquad r^{p-1} e^{\tilde{\phi}_0(r)}  \, dr \nonumber \\
   &\lesssim (D_0^2+1) p^2 \Delta^2 + (D_0^2+1) p \| \xi \|_K^2 + D_0^2 p^2 \Delta^4 + D_0^2 \| \xi \|_K^4 \nonumber \\
   &\lesssim (D_0^2+1) (p^2 \Delta^2 + p \| \xi \|_K^2).
 \end{align}
 The first statement of the lemma then follows from~\eqref{eqn:dH_best_overall2},~\eqref{Eq:BestTerm1} and~\eqref{Eq:BestTerm2}. For the second statement, note that, by our assumption on $p(\tau^* - \tau_*)$,
 \begin{align*}
    d_{\mathrm{H}}^2 \bigl(\tilde{h}_0,h_0 \bigr)
    &\leq 2 d_{\mathrm{H}}^2 \biggl(h_0 \frac{\lambda_p(\hat{K})}{\lambda_p(K)}, \tilde{h}_0 \biggr) +
    2 \biggl( 1 - \Bigl( \frac{\lambda_p(\hat{K})}{\lambda_p(K)} \Bigr)^{1/2} \biggr)^2 \\
    &\leq 2 d_{\mathrm{H}}^2 \biggl(h_0 \frac{\lambda_p(\hat{K})}{\lambda_p(K)}, \tilde{h}_0 \biggr) +
      2 ( \tau_*^{-p/2} - \tau^{* - p/2})^2 \\
    &\lesssim (D_0^2+1)\bigl\{p \| \hat{\mu} - \mu \|^2_K + p^2(\tau^* - \tau_{*})^2\bigr\},
  \end{align*}
 as desired.
\end{proof}

\begin{lemma}
  \label{Lem:KLBestCase}
  There exists universal constants $c_1,c_2 > 0$ such that if $p(\tau^* - \tau_{*}) \leq c_2$ and that $\| \hat{\mu}- \mu \|_K p \log(ep) \leq c_1$ and if $\phi_0'$ is absolutely continuous and differentiable and there exists $D_0 > 0$ such that $\inf_{r \in [0,\infty)} \phi_0''(r) \geq - D_0$, then
  \[
    d^2_{\mathrm{KL}}(\tilde{h}_n, \tilde{h}_0) \lesssim
  (D_0^2 + 1) \bigl\{ p \| \hat{\mu} - \mu \|_K^2 + p^2 (\tau^{*} - \tau_{*})^2 \bigr\}.
  \]
\end{lemma}
\begin{proof}
  First, we observe that, by Lemma~\ref{Lem:PerturbedDensityUpperBound} and~\ref{Lem:GammaBound},
  \begin{align}
    \int_{\mathbb{R}^p} e^{\phi_0(\|x - \mu\|_K)} \log
    \frac{e^{\phi_0(\|x - \mu\|_K)}}{\gamma^{-1} e^{\tilde{\phi}_0(\|x - \hat{\mu}\|_{\hat{K}})}} \,dx
    \leq \log \gamma < \infty. \label{Eqn:FiniteKL}
  \end{align}
We also observe that $\tilde{h}_n$ is the density of the $\hat{K}$-contour measure of the probability distribution induced by $f_0(\cdot + \hat{\mu})$.
  
As a shorthand, for $x \in \mathbb{R}^p$, let
\[
m(x) := \frac{e^{\phi_0(\|x - \mu \|_{K})} - \gamma^{-1} e^{\tilde{\phi}_0(\|x - \hat{\mu}\|_{\hat{K}})}}{\gamma^{-1} e^{\tilde{\phi}_0(\|x - \hat{\mu} \|_{\hat{K}})} }.
\]
Then, by Lemma~\ref{lem:contour_KL} (which is applicable by~\eqref{Eqn:FiniteKL}) and the fact that $(1+z) \log(1+z) \leq z + z^2$ for any $z \in (-1, \infty)$, we have
  \begin{align*} 
    d^2_{\textrm{KL}}(\tilde{h}_n, \tilde{h}_0)
    &= \int_0^\infty \tilde{h}_n(r) \log \frac{\tilde{h}_n(r)}{\tilde{h}_0(r)} \,dr \\
    &\leq \int_{\mathbb{R}^p} e^{\phi_0(\|x - \mu\|_K)} \log \frac{e^{\phi_0(\|x - \mu\|_K)}}{\gamma^{-1} e^{\tilde{\phi}_0(\|x - \hat{\mu} \|_{\hat{K}})}} \, dx \\
    &=  \frac{1}{\gamma}\int_{\mathbb{R}^p} e^{\tilde{\phi}_0(\|x - \hat{\mu} \|_{\hat{K}})} (1 + m(x)) \log (1 + m(x))\, dx \\
    &\leq \frac{1}{\gamma}\int_{\mathbb{R}^p}  e^{\tilde{\phi}_0(\|x - \hat{\mu}\|_{\hat{K}})} \biggl( \frac{e^{\phi_0(\|x - \mu\|_K)}}{\gamma^{-1} e^{\tilde{\phi}_0(\|x - \hat{\mu}\|_{\hat{K}})}}  - 1 \biggr)^2 \, dx.
  \end{align*}
Define $z(x) := \phi_0(\|x - \mu\|_K) + \log \gamma - \tilde{\phi}_0(\|x - \hat{\mu}\|_{\hat{K}})$ for $x \in \mathbb{R}^p$. By the fact that $(e^z - 1) \leq |z|(1 + e^{z})$ for all $z \in \mathbb{R}$, by Lemma~\ref{Lem:PerturbedDensityUpperBound}, and by Lemma~\ref{Lem:GammaBound},
  \begin{align*}
    \frac{1}{\gamma}\int_{\mathbb{R}^p}  & e^{\tilde{\phi}_0(\|x - \hat{\mu}\|_{\hat{K}})} \bigl(e^{z(x)}  - 1 \bigr)^2 \, dx \leq \frac{2}{\gamma} \int_{\mathbb{R}^p} e^{\tilde{\phi}_0(\|x - \hat{\mu}\|_{\hat{K}})} z(x)^2 \bigl(e^{2z(x)} + 1\bigr) \, dx \\
    &\leq \frac{2(\gamma^2+1)}{\gamma} \int_{\mathbb{R}^p} e^{\tilde{\phi}_0(\|x - \hat{\mu}\|_{\hat{K}})}z(x)^2 \, dx \\
    &\lesssim  \frac{1}{\gamma} \int_{\mathbb{R}^p} e^{\tilde{\phi}_0(\|x - \hat{\mu}\|_{\hat{K}})} 
      \bigl\{ \phi_0(\|x - \mu\|_K) -  \phi_0(\|x - \hat{\mu}\|_{\hat{K}}) \bigr\}^2 \, dx \\
      & \qquad \qquad + \frac{1}{\gamma}\int_{\mathbb{R}^p}  e^{\tilde{\phi}_0(\|x - \hat{\mu}\|_{\hat{K}})} \bigl\{ \phi_0(\|x - \hat{\mu}\|_{\hat{K}}) - \tilde{\phi}_0(\|x - \hat{\mu}\|_{\hat{K}}) + \log \gamma \bigr\}^2 \, dx \\
      &\lesssim (D_0^2 + 1)\bigl\{p^2 (\tau^* - \tau_*)^2 + p \| \hat{\mu} - \mu \|_K \bigr\},
  \end{align*}
  where we used~\eqref{eqn:dH_best1_second_term} and~\eqref{Eq:BestTerm2} to obtain the final inequality. 
 The lemma therefore follows from our assumption on $p(\tau^* - \tau_*)$.
\end{proof}
\begin{lemma}
  \label{Lem:PhiMoment}
  For any $s \in \mathbb{N}$, 
  \[
    \int_0^\infty  r^s h_0(r) \, dr \lesssim_s p^{s/2}.
  \]
\end{lemma}
\begin{proof}
By \citet[][Proposition~S2(iii)]{feng2018multivariate} we have $h_0(r) \leq \frac{1}{\sigma_{h_0}} e^{- \frac{1}{\sigma_{h_0}} |r-\mu_{h_0}| + 1}$. Moreover, by \citet[][Lemma~1]{bobkov2003spectral}, $\mu_{h_0} \in [(p-1)^{1/2}, p^{1/2}]$ and $\sigma_{h_0} \leq 1$.  Hence
  \begin{align*}
    \int_0^\infty r^s h_0(r) \, dr 
    &\leq \int_0^\infty r^s \frac{1}{\sigma_{h_0}}e^{- \frac{1}{\sigma_{h_0}} |r-\mu_{h_0}| + 1} \, dr \\
    &\lesssim_s \int_0^\infty \bigl(|r - \mu_{h_0}|^s + \mu_{h_0}^s \bigr)
      \frac{1}{\sigma_{h_0}} e^{- \frac{1}{\sigma_{h_0}} | r-\mu_{h_0}| + 1} \, dr \lesssim_s p^{s/2},
  \end{align*}
 as desired. 
\end{proof}
\begin{lemma}
  \label{Lem:PhiTildeMoment}
 There exists a universal constant $c_1, c_2 > 0$ such that if $\| \hat{\mu} - \mu \|_K p \log (ep) \leq c_1$ and $p(\tau^* - \tau_*) \leq c_2$, then
  \[
    \int_0^\infty r^s \tilde{h}_0(r) \, dr \lesssim_s p^{s/2}.
  \]
\end{lemma}

\begin{proof}
Assume that $\| \hat{\mu} - \mu \|_K p \log (ep) \leq c_1$ and $p(\tau^* - \tau_*) \leq c_2$ where $c_1, c_2 > 0$ are universal constants chosen such that $ d_{\mathrm{H}}^2(\tilde{h}_0,h_0) \leq 2^{-16}$; the existence of such a choice of $c_1,c_2$ is guaranteed by Lemma~\ref{Lem:dHWorstBound2}. Consequently, by Lemma~\ref{Lem:HellingerMomentBound}, there exists universal constants $C_{\mu}' > 0$ and $C_{\sigma}' > 1$ such that $C_{\sigma}^{\prime -1} \leq \sigma_{\tilde{h}_0}/\sigma_{h_0} \leq C_{\sigma}'$ and that $| \mu_{\tilde{h}_0} - \mu_{h_0} | \leq C_{\mu}'$. Moreover, by \citet[][Lemma~1]{bobkov2003spectral}, $\mu_{h_0} \in [(p-1)^{1/2}, p^{1/2}]$ and $\sigma_{h_0} \leq 1$.

Therefore, by \citet[][Proposition~S2(iii)]{feng2018multivariate}, we have that
\begin{align*}
  \int_0^\infty r^s \tilde{h}_0(r) \, dr
  &\leq \int_0^\infty r^s \frac{1}{\sigma_{\tilde{h}_0}} e^{ - \frac{1}{\sigma_{\tilde{h}_0}} | r - \mu_{\tilde{h}_0}| + 1} \, dr\\
  &\leq  C_{\sigma}'\int_0^\infty r^s \frac{1}{\sigma_{h_0}}
    e^{ - \frac{1}{C_{\sigma}' \sigma_{h_0}} | r - \mu_{\tilde{h}_0}| +1} \, dr \\
  &\lesssim_s  C_{\sigma}' \int_0^\infty \bigl(| r - \mu_{\tilde{h}_0}|^s + \mu_{\tilde{h}_0}^s \bigr)
     \frac{1}{\sigma_{h_0}} e^{ - \frac{1}{C_{\sigma}' \sigma_{h_0}} | r - \mu_{\tilde{h}_0}| +1} \, dr \lesssim_s p^{s/2},
    \end{align*}
    as required.
    \end{proof}

\subsection{Auxiliary lemmas for Section~\ref{Sec:KNonparametric}}
\label{Subsec:KNonparametricAux}

We first describe the common setting for all the lemmas in this subsection as well as define some notation and quantities used throughout.  We fix $n \geq 8$, $p \geq 2$ and $K \in \mathcal{K}$; we assume, for some fixed $0 < r_1 \leq r_2 < \infty$, that $B_p(0,r_1) \subseteq K \subseteq  B_p(0,r_2)$ and write $r_0 := r_2/r_1$. We suppose that $X_1, \ldots, X_{n+M} \stackrel{\mathrm{iid}}{\sim} f_0 \in \mathcal{F}^K_p$. Recall from Algorithm~\ref{alg:estimateK} then that $\theta_m = X_{n+m}/\|X_{n+m}\|$ for $m \in [M]$.  Define
\begin{align}
  M &:= \bigl\lceil n^{\frac{p-1}{p+1}} \bigr \rceil, \; \epsilon_1 := \Bigl( \frac{1}{M \log^2 n} \Bigr)^{1/(p-1)}, \; \epsilon_2 := \Bigl( \frac{\log^{p+1} n}{M} \Bigr)^{1/(p-1)}, \nonumber \\
  k &:= \lfloor \log n \rfloor, \; \tilde{n} := \Bigl \lceil \frac{n}{M \log^4 n} \Bigr \rceil, \;
 C_p := \frac{\kappa_{p-1}}{p \kappa_p} 2^{p-1},\;
c_p := \frac{\kappa_{p-1}}{p \kappa_p} \Bigl( \frac{3}{4}\Bigr)^{(p-1)/2}. \label{Eqn:KNonparametricDefinitions}
\end{align}
For $m \in [M]$, we also define
\[
  \mathcal{X}_m^k := \Bigl\{ x \in \mathbb{R}^p \,:\, x^\top \theta_m \geq \textrm{ $k$-th max} \{ x^\top \theta_{m'} \,:\, m' \in [M] \} \Bigr\}.
\]
Note that with this definition, for any $m \in [M]$, the random quantity $\mathcal{I}^k_m$ as defined in Algorithm~\ref{alg:estimateK} satisfies $\mathcal{I}^k_m = | \{ i \in [n] \,:\, X_i \in \mathcal{X}_m^k \} |$.  For $\epsilon \in (0, 1]$, define the spherical cone with centre $\theta \in \mathbb{S}^{p-1}$ as $S(\theta, \epsilon) := \{ x \in \mathbb{R}^p \,:\, x^\top \theta \geq \bigl(1 - \frac{1}{2} \epsilon^2\bigr)\| x \|_2 \}$. We then define the events
\begin{align}
  \mathcal{E} := \bigcap_{m \in [M]}\bigl\{S(\theta_m, \epsilon_1) \subseteq \mathcal{X}_m^k \subseteq S(\theta_m, \epsilon_2) \bigr \}, \qquad  
  \mathcal{E}_1 := \Bigl\{ \min_{m \in [M]} | \mathcal{I}^k_m | \geq \tilde{n} \Bigr\}. \label{Eqn:EventDefinitions}
\end{align}
For $\epsilon > 0$, we say that a finite set $\mathcal{N}_{\epsilon} \subseteq \mathbb{S}^{p-1}$ is an $\epsilon$-net of $\mathbb{S}^{p-1}$ if for every $x \in \mathbb{S}^{p-1}$, there exists $y \in \mathcal{N}_{\epsilon}$ such that $x \in \mathbb{S}^{p-1} \cap S(y, \epsilon)$. 

The key results of this subsection are Lemmas~\ref{Lem:tmDeviation} and~\ref{Lem:Brunel}, both of which are used in the proof of Proposition~\ref{Prop:dscaleOverall}. 
\begin{lemma}
  \label{Lem:tmDeviation}
Suppose that $\mathbb{E}_{f_0}(\| X_1 \|_K) = \mathbb{E}_{f_0}(\|X_1 \|)= 1$ and that $r_0 M^{-1/(p-1)} \log^3 n \leq 1/2$. For $m \in [M]$, let $t_m$ be defined as in Algorithm~\ref{alg:estimateK}.  Then there exists $C_{1,p,r_0} > 0$, depending only on $p$ and $r_0$, such that, with probability at least $1 - C_{1,p,r_0}/n$, 
  \[
    \max_{m \in [M]} \biggl| t_m - \frac{1}{\| \theta_m \|_K} \biggr| \leq 2 r_0^2 \biggl( \frac{\log^{p+1} n}{M} \biggr)^{1/(p-1)} + 8 r_0 \biggl( \frac{M \log^5 n}{n} \biggr)^{1/2}.
  \]
\end{lemma}
\begin{proof}
  For $m \in [M]$, let $s_m := |\mathcal{I}_m^k|^{-1} \sum_{i \in \mathcal{I}_m^k} \| X_i \|_K$.  We have that $r_0 \epsilon_2 \leq 1/2$ under our assumptions on $r_0 M^{-1/(p-1)}\log^3 n$. Thus, on the event $\mathcal{E}$, we have by Lemma~\ref{Lem:KNorm2NormRatio} that
  \begin{align*}
    t_m &= \frac{1}{|\mathcal{I}_m^k|} \sum_{i \in \mathcal{I}_m^k} \| X_i \| \leq s_m \sup_{x \in \mathcal{X}_m^k} \frac{\| x \|}{\| x \|_K} \leq \frac{s_m}{(1-r_0\epsilon_2)\| \theta_m \|_K}
         \leq \frac{s_m}{\| \theta_m \|_K}( 1 + 2 r_0 \epsilon_2),
  \end{align*}
and that
  \[
     t_m \geq \frac{s_m}{\| \theta_m \|_K} \bigl( 1 - 2 r_0\epsilon_2 \bigr),
   \]
and therefore
   \begin{align}
     \biggl| t_m - \frac{s_m}{\| \theta_m \|_K} \biggr| \leq 2 r_0 \epsilon_2 \frac{s_m}{\| \theta_m \|_K}.
     \label{eqn:tm_inequality1}
   \end{align}
Define the event $\mathcal{E}_2 := \bigl\{ \max_{m \in [M]} | s_m - \mathbb{E}s_m| \leq 4 \bigl( \frac{M \log^5 n}{n} \bigr)^{1/2} \bigr\}$. To bound $\mathbb{P}(\mathcal{E}_2^c)$, choose $N_0 \subseteq [n]$ with $n_0 := |N_0| \geq \tilde{n}$. Let $\mathcal{E}_m(N_0)$ be the event that $\mathcal{I}_m^k = N_0$.  Then, by Proposition~\ref{Prop:NormDirectionIndep} and \citet[][Theorem 5]{karlin1961moment}, 
 for any $s \geq 2$,
  \begin{align}
    \mathbb{E}(\| X_1 \|_K^s) \leq s!. 
  \end{align}
  By Proposition~\ref{Prop:NormDirectionIndep} again, Bernstein's inequality \citep[][Corollary 2.11]{boucheron2013concentration}, and the fact that $n_0^{-1/2}\log^{1/2} n \leq 1/2$ under our assumption on $n$ and $r_0 M^{- 1/(p-1)} \log^3n$, we therefore have that for each $m \in [M]$,
  \begin{align*}
    &\mathbb{P}\biggl( \biggl\{ |s_m - \mathbb{E} s_m| > 4\biggl( \frac{\log n}{n_0} \biggr)^{1/2} \biggr\} \,\cap\, \mathcal{E}_m(N_0) \biggr) \\
    &\qquad =  \mathbb{P}\biggl( \biggl\{ \biggl| \frac{1}{n_0} \sum_{i \in N_0} \|X_i\|_K - \mathbb{E} \|X_1\|_K \biggr| > 4\biggl( \frac{\log n}{n_0} \biggr)^{1/2} \biggr\} \,\cap\, \mathcal{E}_m(N_0) \biggr) \\
    &\qquad = \mathbb{P}\biggl( \biggl| \frac{1}{n_0} \sum_{i \in N_0} \|X_i\|_K - \mathbb{E} \|X_1\|_K \biggr| > 4\biggl( \frac{\log n}{n_0} \biggr)^{1/2} \biggr) \mathbb{P}( \mathcal{E}_m(N_0))  \leq \frac{2}{n^2} \mathbb{P}(\mathcal{E}_m(N_0)).
  \end{align*}  
  Hence, for each $m \in [M]$,
  \begin{align*}
   & \mathbb{P}\biggl( |s_m - \mathbb{E} s_m| > 4 \biggl( \frac{M\log^5 n}{n} \biggr)^{1/2} \, \cap \, \mathcal{E}_1   \biggr) \\
    &\quad \leq \sum_{ \substack{N_0 \subseteq \{1,\ldots,n\} : \\ |N_0| \geq \tilde{n}}} \mathbb{P}\biggl( \biggl\{ |s_m - \mathbb{E} s_m| > 4\biggl( \frac{M \log^5 n}{n} \biggr)^{1/2} \biggr\} \,\cap \, \mathcal{E}_m(N_0) \biggr) \\
    &\quad \leq \frac{2}{n^2} \sum_{ \substack{N_0 \subseteq \{1,\ldots,n\} : \\ |N_0| \geq \tilde{n}}} \mathbb{P}\bigl(\mathcal{E}_m(N_0) \bigr) = \frac{2}{n^2}\mathbb{P}(|\mathcal{I}_m^k|\geq \tilde{n}) \leq \frac{2}{n^2}.
  \end{align*}
  Thus, using a union bound over $m \in [M]$, Lemma~\ref{Lem:EnoughPoints} (which we may apply since $\tilde{n} > 1$ under our assumption on $r_0 M^{-1/(p-1)} \log^3 n$), and the fact that $M \leq n$,
    \begin{align}
      \mathbb{P}( \mathcal{E}_2^c ) \leq \sum_{m=1}^M \mathbb{P}\biggl( |s_m - \mathbb{E} s_m| > 4 \biggl( \frac{M\log^5 n}{n} \biggr)^{1/2} \, \cap \, \mathcal{E}_1  \biggr) + \mathbb{P}(\mathcal{E}_1^c)  
      \leq \frac{2}{n} + \frac{C_{3, p, r_0}}{n}. \label{Eqn:E2ProbBound}
    \end{align}
    Now, under the assumption that $\mathbb{E}_{f_0}(\|X_1 \|_K) =1$, we have $\mathbb{E} s_m = 1$.  Moreover,
    \[
      1 = \mathbb{E}_{f_0}(\|X_1\|) \geq r_1\mathbb{E}_{f_0}(\|X_1\|_K) = r_1,
    \]
so $r_2 \leq r_0$.   Thus, on the event $\mathcal{E} \cap \mathcal{E}_2$, for each $m \in [M]$,
  \begin{align*}
    \biggl| t_m - \frac{1}{\| \theta_m \|_K} \biggr|
    &\leq \biggl| t_m - \frac{s_m}{\|\theta_m\|_K} \biggr| + \frac{1}{\|\theta_m\|_K} \bigl| s_m - \mathbb{E} s_m \bigr| \\
    &\leq 2 r_0 \epsilon_2 \frac{s_m}{\|\theta_m \|_K} + \frac{4}{r_1} \biggl( \frac{M \log^5 n}{n} \biggr)^{1/2} \\
    &\leq 2 r_0 \epsilon_2 \frac{1}{\|\theta_m \|_K} + 2 r_0 \epsilon_2 \frac{1}{\|\theta_m\|_K} | s_m - \mathbb{E} s_m | + 4r_2\biggl( \frac{M \log^5 n}{n} \biggr)^{1/2}\\
    &\leq 2 r_0^2 \biggl( \frac{\log^{p+1} n}{M} \biggr)^{1/(p-1)} + 8 r_0 \biggl( \frac{M \log^5 n}{n} \biggr)^{1/2}.
  \end{align*}
Finally, by Lemma~\ref{Lem:BoundingCaps} and~\eqref{Eqn:E2ProbBound}, there exists $C_{1,p,r_0} > 0$, depending only on $p$ and $r_0$, such that $\mathbb{P}(\mathcal{E} \cap \mathcal{E}_2) \geq 1 - C_{1,p,r_0}/n$ as required.
\end{proof}

\begin{lemma}
  \label{Lem:BoundingCaps}
  If $\epsilon_2 \leq 1$, then there exists $C_{2,p,r_0} > 0$, depending only on $p$ and $r_0$, such that
  \[
    \mathbb{P}(\mathcal{E}^c) \leq \frac{C_{2, p, r_0}}{n}.
  \]
\end{lemma}

\begin{proof}
Let $\tilde{\epsilon}_2 := \epsilon_2/(4\log n)$. By Lemma~\ref{lem:cap_covering}, there exists an $\tilde{\epsilon}_2$-net $\mathcal{N}_{\tilde{\epsilon}_2}$ of $\mathbb{S}^{p-1}$ such that $| \mathcal{N}_{\tilde{\epsilon}_2} | \leq C'_p \tilde{\epsilon}_2^{-(p-1)}$ for some $C'_p > 0$ depending only on $p$. Let $\mathcal{E}_{\mathrm{net},1}$ be the event that for every $y \in \mathcal{N}_{\tilde{\epsilon}_2}$, there exists some $m \in [M]$ such that $\theta_m \in \mathbb{S}^{p-1} \cap S(y, \tilde{\epsilon}_2)$. By Lemma~\ref{Lem:CapProbabilityBound} and a union bound, we have that
  \begin{align}
    \mathbb{P}( \mathcal{E}^c_{\mathrm{net},1} )
    &\leq (1 - c_p r_0^{-p} \tilde{\epsilon}_2^{p-1})^M |\mathcal{N}_{\tilde{\epsilon}_2} | \nonumber \\
    &\leq \exp( - c_p r_0^{-p} \tilde{\epsilon}_2^{p-1} M ) C'_p \tilde{\epsilon}_2^{-(p-1)} \nonumber \\
    &\leq \exp( - c_p r_0^{-p} 4^{-(p-1)}  \log^2 n ) C'_p 4^{p-1}  \frac{n}{\log^2 n} \leq \frac{C_{2,p,r_0}'}{n}, \label{eqn:cap_prob1}
  \end{align}
say. Now for any $y \in \mathbb{S}^{p-1}$, there exist $y_1, \ldots, y_k \in \mathbb{S}^{p-1}$ such that $S(y_j, \tilde{\epsilon}_2) \subseteq S(y, \epsilon_2/2)$ for $j \in [k]$, that $y \notin S(y_j, \tilde{\epsilon}_2)$ for any $j$ and that $S(y_j, \tilde{\epsilon}_2) \cap S(y_{j'}, \tilde{\epsilon}_2) = \emptyset$ for $j \neq j'$. Thus, on the event $\mathcal{E}_{\mathrm{net},1}$, for any $\theta_m$, there exist $m_1, \ldots, m_{k} \in [M]$ not equal to $m$ such that $\theta_{m_1}, \ldots, \theta_{m_k} \in \mathbb{S}^{p-1} \cap S(\theta_m, \epsilon_2/2)$. Thus, on the event $\mathcal{E}_{\mathrm{net},1}$, we have $\mathcal{X}_m^k \subseteq S(\theta_m,\epsilon_2)$ for every $m$.

Now let $\tilde{\epsilon}_1 := 2 \epsilon_1$ and note that $2\epsilon_1 \leq \epsilon_2 \leq 1$. For any $m \in [M]$, let $A_m := \{m' \in [M] : \theta_{m'} \in S(\theta_m, \tilde{\epsilon}_1) \}$. Let $\mathcal{E}_{\mathrm{net},2}$ be the event that $\max_{m \in [M]} |A_m| \leq k$.  Then by Lemma~\ref{Lem:CapProbabilityBound} again,
  \begin{align}
    \mathbb{P}(\mathcal{E}_{\mathrm{net},2}^c)
    &\leq \bigl(C_p r_0^{p} \tilde{\epsilon}_1^{p-1}\bigr)^{k} \binom{M-1}{k-1} M \nonumber \\
    &\leq  \bigl(C_p r_0^{p} \tilde{\epsilon}_1^{p-1}\bigr)^{k-1} M^{k} \nonumber \\
    &\leq \biggl( C_p r_0^{p}  2^{p-1} \frac{1}{\log^2 n} \biggr)^{k-1} n \leq \frac{C''_{2,p, r_0}}{n},  \label{eqn:cap_prob2}
  \end{align}
say.  
On the event $\mathcal{E}_{\mathrm{net},2}$, we have $S(\theta_m, \epsilon_1) \subseteq \mathcal{X}_m^k$ for every $m \in [M]$.  Setting $C_{2,p,r_0} := C'_{2,p,r_0} + C''_{2,p,r_0}$ and applying a union bound, we obtain the conclusion of the lemma as desired.   
\end{proof}
\begin{lemma}
  \label{Lem:EnoughPoints}
  If $\epsilon_2 \leq 1$ and $\tilde{n} > 1$, then there exists $C_{3, p, r_0} > 0$, depending only on $p$ and $r_0$, such that
  \begin{align}
    \mathbb{P}(\mathcal{E}^c_1) \leq \frac{C_{3, p, r_0}}{n}.
  \end{align}
\end{lemma}
\begin{proof} 
Since $\tilde{n} > 1$, we have that $\frac{n}{M} \geq \log^4 n$. Thus, we have by Lemma~\ref{Lem:CapProbabilityBound} that for $n \geq 8$ (so that $\tilde{n} < n/2$),
  \begin{align}
    \mathbb{P}( \mathcal{E}_1^c \cap \mathcal{E} )
    &\leq M\bigl(1 - c_p r_0^{-p} \epsilon_1^{p-1}\bigr)^{n - (\tilde{n}-1)}
      \binom{n}{\tilde{n}-1} \nonumber \\
    &\leq \exp\biggl( - c_p r_0^{-p} \epsilon_1^{p-1} \frac{n}{2} + \frac{n}{M \log^3 n} + \log M \biggr) \nonumber \\
    &\leq \exp \biggl( - c_p r_0^{-p} \frac{n}{2 M \log^2 n} + \frac{2n}{M \log^3 n}\biggr) \nonumber \\
    &\leq \exp \biggl\{  - \biggl(\frac{c_p r_0^{-p}}{2 \log^2 n} - \frac{2}{\log^3n} \biggr) \frac{n}{M}   \biggr\} \leq \frac{C_{3,p,r_0}'}{n}, \label{eqn:cap_prob3} 
  \end{align}
say.  Hence, by Lemma~\ref{Lem:BoundingCaps},
 \begin{align*}
    \mathbb{P}(\mathcal{E}_1^c) \leq \mathbb{P}(\mathcal{E}_1^c \cap \mathcal{E}) + \mathbb{P}(\mathcal{E}^c) \leq \frac{C_{3,p,r_0}'}{n} + \frac{C_{2,p,r_0}}{n},
  \end{align*}
  as required.
\end{proof}

\begin{lemma}
  \label{Lem:KNorm2NormRatio}
  Let $\theta \in \mathbb{S}^{p-1}$ and let $\epsilon \in (0, 1]$. Then, for all $x \in S(\theta, \epsilon)$, we have
  \[
    \bigl( 1 - r_0\epsilon \bigr) \|x\|_2 \leq \frac{\|x \|_K}{\|\theta\|_K} \leq \bigl( 1 + r_0\epsilon\bigr)\|x\|_2.
  \]
\end{lemma}

\begin{proof}
 The claim is immediately true if $x = 0$ or if $x = \theta$. Thus, let $x \in S(\theta, \epsilon)$ be such that $x \neq 0$ and $x\neq \theta$. Let $\tilde{x} := x/\|x\|_2$, so that $\| \tilde{x} - \theta\|_2 \leq \epsilon$. Then
  \begin{align*}
    \| \theta \|_K - \frac{\epsilon}{r_1} &\leq  \| \theta \|_K - \epsilon  \biggl( \sup_{z \in \mathbb{R}^p\backslash \{0\}} \frac{ \|z \|_K}{\|z\|_2} \biggr) \leq   \|\theta \|_K - \| \tilde{x} - \theta \|_K  \\
    &\leq \|\tilde{x}\|_K 
    \leq \| \theta \|_K + \| \tilde{x} - \theta \|_K
               \leq \| \theta \|_K + \epsilon \biggl( \sup_{z \in \mathbb{R}^p \backslash\{0\}} \frac{ \|z \|_K}{\|z\|_2} \biggr)
               \leq \| \theta \|_K + \frac{\epsilon}{r_1}.
  \end{align*}
The claim then follows from the fact that $1/\| \theta \|_K \leq \| \theta \|_2/\|\theta\|_K \leq r_2$.
\end{proof}
\begin{lemma}
\label{Lem:CapProbabilityBound}
  Let $X$ be a random vector uniformly distributed on $K$. Let $\theta \in \mathbb{S}^{p-1}$ be fixed and let $\epsilon \in (0,1]$. Then,
  \[
    c_p r_0^{-p}  \epsilon^{p-1}
    \leq \mathbb{P}\bigl( X \in S(\theta, \epsilon) \bigr) \leq
    C_p r_0^{p} \epsilon^{p-1},
  \]
  where $C_p, c_p$ are as defined in~\eqref{Eqn:KNonparametricDefinitions}. 
\end{lemma}
\begin{proof}
  Fix $\epsilon \in (0, 1]$ and let $S_{\mathrm{b}}(\theta, \epsilon) := \{ x \in B_p(0,1) \,:\, x^\top \theta = (1 - \epsilon^2/2) \}$ so that $S_{\mathrm{b}}(\theta, \epsilon)$ is the base of the spherical cap $B_p(0,1) \cap S(\theta, \epsilon)$. Define also $S_{\mathrm{t}}(\theta, \epsilon) := S(\theta, \epsilon) \cap \{ x \in \mathbb{R}^p \,:\, x^\top \theta = 1\}$. We observe then that
  \begin{align*}
    \lambda_{p-1}\bigl(S_{\mathrm{b}}(\theta, \epsilon)\bigr) &= \kappa_{p-1} \biggl\{\epsilon \Bigl( 1 - \frac{\epsilon^2}{4}\Bigr)^{1/2}\biggr\}^{p-1} 
    \geq \kappa_{p-1}  \Bigl(\frac{3}{4}\Bigr)^{(p-1)/2}\epsilon^{p-1} = c_pp\kappa_p\epsilon^{p-1}, \\
    \lambda_{p-1}\bigl(S_{\mathrm{t}}(\theta, \epsilon)\bigr) &=  \kappa_{p-1} \biggl\{ \frac{\epsilon (1 - \epsilon^2/4)^{1/2}}{1 - \epsilon^2/2} \biggr\}^{p-1} \leq \kappa_{p-1} 2^{p-1} \epsilon^{p-1} = C_pp\kappa_p\epsilon^{p-1}.
  \end{align*}
We therefore have that
  \begin{align*}
    \mathbb{P}\bigl(X \in S(\theta, \epsilon)\bigr)
    &= \frac{\lambda_{p}\bigl( K \cap S(\theta, \epsilon)\bigr)}{\lambda_{p}(K)} 
    \leq \frac{\lambda_{p}\bigl( B_p(0, r_2) \cap S(\theta, \epsilon) \bigr) }{r_1^p \kappa_p}\\
    &= \frac{\lambda_{p-1}\bigl( S_{\mathrm{t}}(\theta, \epsilon)  \bigr) r_0^p  }{ p\kappa_p} \leq C_p r_0^p \epsilon^{p-1}. 
  \end{align*}
For the lower bound, we have
  \begin{align*}
    \mathbb{P}\bigl(X \in S(\theta, \epsilon) \bigr)
    &\geq \frac{\lambda_{p}\bigl( B_p(0, r_1) \cap S(\theta, \epsilon) \bigr) }{r_2^p \kappa_p}
      = \frac{\lambda_{p-1}\bigl( S_{\mathrm{b}}(\theta, \epsilon)  \bigr) r_0^{-p}  }{ p\kappa_p} \geq c_p r_0^{-p} \epsilon^{p-1},
  \end{align*}
  as desired.
\end{proof}
The following lemma is well-known and follows from the fact that the surface area of a spherical cap $\mathbb{S}^{p-1} \cap S(\theta, \epsilon)$ scales as $\epsilon^{p-1}$ up to a multiplicative constant depending only on~$p$ \citep{li2011concise}. We omit the proof for brevity.
\begin{lemma}
  \label{lem:cap_covering}
There exists $C'_p > 0$, depending only on $p$, such that, for every $\epsilon \in (0,1]$, there exists an $\epsilon$-net $\mathcal{N}_{\epsilon}$ of $\mathbb{S}^{p-1}$ of cardinality $|\mathcal{N}_{\epsilon}| \leq C'_p \epsilon^{-(p-1)}$.
\end{lemma}
For $K \in \mathcal{K}$, let $h_K$ be the support function of $K$, i.e., $h_K(u) := \sup \{ x^\top u \,:\, x \in K\}$ for $u \in \mathbb{S}^{p-1}$. For any $u \in \mathbb{S}^{p-1}$ and $\epsilon \in (0,1)$, define $C_K(u,\epsilon) := \{ x \in K \,:\, u^\top x \geq h_K(u) - \epsilon\}$. For $A, B \subseteq \mathbb{R}^p$, we define the Hausdorff distance $d_{\mathrm{Haus}}(A, B) := \inf \{ \epsilon > 0 \,:\, A \subseteq B + \epsilon B_p(0,1), B \subseteq A + \epsilon B_p(0,1) \}$. 
\begin{theorem} \citep[][Theorem 1]{brunel2018uniform}
  \label{Thm:BrunelTheorem}
  Let $K \in \mathcal{K}$ with $K \subseteq B_p(0,1)$ and let $\nu$ be probability distribution supported on $K$. Suppose there exist $\alpha \geq 1$, $L > 0$ and $\epsilon_0 > 0$ such that $\nu\bigl(C_K(u,\epsilon)\bigr) \geq L\epsilon^\alpha$ for all $\epsilon \in (0,\epsilon_0]$.  Let $Y_1, \ldots, Y_M$ be independent random vectors with distribution $\nu$ and let $\tilde{K} := \mathrm{conv} \bigl\{ Y_1, \ldots, Y_M\}$. Let $\tau_1 := \max(1, p/(\alpha L))$ and $a_M := ( \tau_1M^{-1} \log M)^{1/\alpha}$. If $4 a_M \leq \epsilon_0$, then
  \[
    \mathbb{P}\bigl( d_{\mathrm{Haus}}(\tilde{K}, K) \geq 4 a_M\bigr) \leq 12^p M^{ - L \tau_1}.
  \]
\end{theorem}
\begin{lemma}
  \label{Lem:Brunel}
Let $Z_1, \ldots, Z_M$ be independent random vectors distributed uniformly on $K$.  Let $Y_m := Z_m/\|Z_m\|_K$ for $m \in [M]$, and let $\tilde{K} := \mathrm{conv}\{ Y_1, \ldots, Y_M\}$. If $M/\log M \geq r_0^{2(p-1)} 64^{p-1}$, then
  \[
    \mathbb{P}\biggl( d_{\mathrm{scale}}(\tilde{K},K) > 64 r^2_0 \biggl( \frac{\log M}{M} \biggr)^{1/(p-1)} \biggr) \leq 12^p M^{-p/(p-1)}.
  \]
\end{lemma}
\begin{proof}
Since $d_{\mathrm{scale}}$ is scale invariant, we assume without loss of generality that $r_2 = 1$ so that $r_0 = 1/r_1$. Let $\nu$ denote the distribution of $Y_1$.  We claim that $\nu$ satisfies the hypothesis of Theorem~\ref{Thm:BrunelTheorem} with $L = r_1^{p-1}/(2p)$, $\epsilon_0 = r_1$, and $\alpha = p-1$.

  To see this, let $\epsilon \in (0, r_1]$ and $u \in \mathbb{S}^{p-1}$ be arbitrary and let $x^* \in \partial K$ be such that $h_K(u) = u^\top x^*$. We write $C^0_K(u,\epsilon) = \mathrm{conv}\bigl\{ \{0\} \cup C_K(u,\epsilon) \bigr\}$. Let $D := \{ x \in K \,:\, x^\top u = h_K(u) - \epsilon \}$ be the base of $C_K(u, \epsilon)$ and define $D^0 = \mathrm{conv}\bigl\{ \{0\} \cup D\bigr\}$ and $D^{x^*} := \mathrm{conv}\bigl\{ \{x^*\} \cup D\bigr\}$. Let $C^*$ be the conical pyramid connecting $B_p(0, r_1) \cap \{ x \in \mathbb{R}^p \,:\, x^\top u = 0 \}$ and $x^*$ and let $D^* = C^* \cap D$. Then, by an application of Corollary~\ref{Cor:ConvexContraction}, the fact that $h_K(u) \leq \sup \{ \|x \|_2 \,:\, x \in K \} \leq 1$,  the fact that $\epsilon \leq r_1 \leq h_K(u)$, and the fact that $\kappa_{p-1}/\kappa_p \geq 1/2$, 
  \begin{align*}
    \nu(C_K(u, \epsilon))
    &= \frac{\lambda_{p}\bigl( C^0_K(u, \epsilon) \bigr)}{\lambda_{p}(K)}
      \geq \frac{\lambda_{p}(D^0 \cup D^{x^*})}{\kappa_p}
      = \frac{\lambda_{p-1}(D) h_K(u)}{p \kappa_p} \\
      &\geq \frac{\lambda_{p-1}(D^*) h_K(u)}{p \kappa_p} 
    \geq  \frac{\kappa_{p-1} h_K(u)}{p \kappa_p} \biggl(r_1\frac{\epsilon}{h_K(u)} \biggr)^{p-1}
      \geq   \frac{1}{2p} r_1^{p-1} \epsilon^{p-1}.
  \end{align*}
  Let $\tau_1$ and $a_M$ be defined as in Theorem~\ref{Thm:BrunelTheorem}. Since $p/(\alpha L) =  \frac{2p^2}{p-1} r_1^{-(p-1)} \geq 1,$ we have that $\tau_1 = \max(1, p/(\alpha L)) = p/(\alpha L)$. Moreover, $p/(\alpha L) \leq 8^{p-1}/r_1^{p-1}$ and so $a_M \leq \frac{8}{r_1} \bigl( \frac{\log M}{M} \bigr)^{1/(p-1)}$. By the assumption on $M/\log M$, we have that $4 a_M \leq \epsilon_0$. Thus, writing
  \[
    \mathcal{E}_{\mathrm{Haus}} := \biggl\{d_{\mathrm{Haus}}(\tilde{K}, K) < \frac{32}{r_1} \Bigl(\frac{\log M}{M}\Bigr)^{1/(p-1)}\biggr\},
  \]
we have by Theorem~\ref{Thm:BrunelTheorem} that
  \begin{align}
    \mathbb{P}(\mathcal{E}_{\mathrm{Haus}}^c) \leq 12^p M^{-p/(p-1)}.
    \label{eqn:brunel}
  \end{align}
We now work on the event $\mathcal{E}_{\mathrm{Haus}}$, which by our assumption on $M/\log M$ implies that $d_{\mathrm{Haus}}(\tilde{K}, K) < r_1/2$.  We can therefore fix $\epsilon \in (r_1^{-1} d_{\mathrm{Haus}}(\tilde{K}, K), 1/2]$.  Since $r_2 = 1$, we have $K \subseteq \tilde{K} + \epsilon r_1 B_p(0,1) \subseteq \tilde{K} + \epsilon K$. Applying this recursively, we obtain $K \subseteq \bigl( \sum_{r=0}^R \epsilon^r  \bigr)\tilde{K} + \epsilon^{R+1} K$ for any $R \in \mathbb{N}$. Because $\sum_{r=0}^R \epsilon^k \leq 1 + 2 \epsilon$, it holds that for any $x \in K$, there exists $\{ y_R, z_R\}_{R \in \mathbb{N}}$ such that $y_R \in (1 + 2\epsilon) \tilde{K}$, $z_R \in \epsilon^{R+1} K$ and $x = y_R + z_R$. Thus, $\| x - y_R \| \leq \epsilon^{R+1}$ for all $R \in \mathbb{N}$ and so $x$ is a limit point of $(1 + 2\epsilon) \tilde{K}$. Since $\tilde{K}$ is closed, we conclude that $x \in (1 + 2\epsilon) \tilde{K}$ and hence that $K \subseteq (1 + 2 \epsilon) \tilde{K}$. On the other hand, we have that $\tilde{K} \subseteq K$ by definition and so $d_{\mathrm{scale}}(\tilde{K}, K) \leq 2 \epsilon$. Since $\epsilon \in (r_1^{-1} d_{\mathrm{Haus}}(\tilde{K}, K),  1/2]$ was chosen arbitrarily, we have that on the event $\mathcal{E}_{\mathrm{Haus}}$,
  \[
    d_{\mathrm{scale}}(\tilde{K}, K) \leq \frac{2}{r_1} d_{\mathrm{Haus}}(\tilde{K}, K) < 64 r_0^2 \Bigl( \frac{\log M}{M} \Bigr)^{1/p-1},
  \]
  as desired.
\end{proof}

Recall that for $m \in (0,p]$, the $m$-dimensional Hausdorff outer measure of $E \subseteq \mathbb{R}^p$ is defined as
\begin{align}
  \lambda_{m,p}(E) := \lim_{\delta \searrow 0} \frac{\pi^{m/2}}{2^m \Gamma(m/2 + 1)} \inf \biggl\{
    \sum_{i=1}^\infty \textrm{diam}(E_i)^m \,:\,
  E \subseteq \bigcup_{i=1}^\infty E_i, \, \textrm{diam}(E_i) < \delta,\, E_i \subseteq \mathbb{R}^p \biggr\}.
  \label{eqn:hausdorff_measure}
\end{align}
Note that by this definition, if $E \subseteq \mathbb{R}^p$ is Lebesgue measurable, then $\lambda_{p,p}(E) = \lambda_p(E)$ \citep[e.g.,][Section 4.3]{mattila1999geometry}.  We extend the definition of $\lambda_{p-1}$ to Borel subsets of $\mathbb{R}^p$ by writing $\lambda_{p-1}(A) := \lambda_{p-1, p}(A)$ for a Borel measurable $A \subseteq \mathbb{R}^p$. The next lemma shows that the Hausdorff measure contracts under a $1$-Lipschitz map.  

\begin{lemma}
  \label{Lem:Contraction}
  Let $\phi \,:\, \mathbb{R}^p \rightarrow \mathbb{R}^p$ be a $1$-Lipschitz mapping in the sense that for any $x, y \in \mathbb{R}^p$, $\| \phi(x) - \phi(y) \|_2 \leq \| x - y \|_2$. For $m \leq p$, let $\lambda_{m,p}$ be the $m$-dimensional Hausdorff outer measure as defined in~\eqref{eqn:hausdorff_measure}. Then, for any $A \subseteq \mathbb{R}^p$,
  \[
    \lambda_{m,p}(A) \geq \lambda_{m,p}(\phi(A)).
  \]
\end{lemma}

\begin{proof}
For any $E \subseteq \mathbb{R}^p$, we have that $\textrm{diam}(E) \geq \textrm{diam}(\phi(E))$. The claim follows immediately.
\end{proof}

\begin{corollary}
  \label{Cor:ConvexContraction}
  Let $A, B \subseteq \mathbb{R}^p$ be compact convex sets with non-empty interiors. If $B \subseteq A$, then, $\lambda_{p-1}(\partial A) \geq \lambda_{p-1}(\partial B)$.
\end{corollary}
\begin{proof}
  For any $x \in \mathbb{R}^p$, let $\phi(x)$ denote the Euclidean projection onto $B$. To see that $\phi$ is $1$-Lipschitz, let $x, y \in \mathbb{R}^p$ and let $x' := \phi(x)$ and $y' := \phi(y)$. If $x'=y'$ then certainly $\|x'-y'\|_2 \leq \|x-y\|_2$.  On the other hand, if $x' \neq y'$, then, by the optimality conditions of $x', y'$, 
  \[
  0 \leq (x - x')^\top (x' - y') + (y - y')^\top (y' - x') = (x - y)^\top (x' - y') - \| x' - y'\|_2^2.
  \]
  We may then apply the Cauchy--Schwarz inequality to obtain that $\| x' - y'\|_2 \leq \|x - y\|_2$.

If $x \in \partial A$, then $\phi(x) \in \partial B$ because otherwise we can find a point in $B$ on the line segment between $x$ and $\phi(x)$ that is closer to $x$ than is $\phi(x)$. Thus, $\phi(\partial A) \subseteq \partial B$.  For any $x' \in \partial B$, by the separating hyperplane theorem \citep[e.g.][Theorem~11.6]{rockafellar1997convex}, there exists $\alpha \in \mathbb{R}^p$ and $b \in \mathbb{R}$ such that $\alpha^\top x' + b = 0$ and $\alpha^\top z + b \leq 0$ for all $z \in B$. Since $B \subseteq A$, there exists $x \in \partial A$ such that $x = x' + c \alpha$ for some $c \in [0, \infty)$. Moreover, for any $z \in B$,
  \[
(x-x')^\top (z-x') = c\alpha^\top(z'-x') \leq 0,
  \]
so $\phi(x) = x'$. Thus, $\phi(\partial A) = \partial B$ and the result follows from Lemma~\ref{Lem:Contraction}.
\end{proof}

Recall the definition of $d_{\mathrm{scale}}$ from~\eqref{Eq:dscale}.

\begin{lemma}
  \label{lem:dscale_properties}
We have
  \begin{enumerate}
  \item $d_{\mathrm{scale}}(K', K) = d_{\mathrm{scale}}(K, K')$ for any $K, K' \in \mathcal{K}$.
  \item For any $K, K', K'' \in \mathcal{K}$ such that $d_{\mathrm{scale}}(K, K') < 1$ or $d_{\mathrm{scale}}(K', K'') < 1$, it holds that $d_{\mathrm{scale}}(K, K'') \leq 2 d_{\mathrm{scale}}(K, K') + 2 d_{\mathrm{scale}}(K', K'')$.
  \end{enumerate}
\end{lemma}
\begin{proof}
The first claim follows from the fact that $\frac{1}{1+\epsilon} K \subseteq K'$ if and only if $K \subseteq (1 + \epsilon) K'$.

For the second claim, let $K, K', K'' \in \mathcal{K}$, and suppose without loss of generality that $d_{\mathrm{scale}}(K, K') < 1$.  Then there exist $\epsilon \in (0,1]$ and $\epsilon' > 0$ such that $(1 + \epsilon)^{-1} K \subseteq K' \subseteq (1 + \epsilon) K$ and $(1 + \epsilon')^{-1} K'' \subseteq K' \subseteq (1 + \epsilon') K''$. Since
  \[
    \frac{1}{(1+\epsilon)(1+\epsilon')} K'' \subseteq \frac{1}{1+\epsilon} K' \subseteq K \subseteq (1+\epsilon)K' \subseteq (1+\epsilon)(1+\epsilon') K'',
    \]
    we have that $d_{\mathrm{scale}}(K, K'') \leq \epsilon + \epsilon' + \epsilon \epsilon' \leq 2 \epsilon + 2 \epsilon'$. Since this holds for any $\epsilon > d_{\mathrm{scale}}(K, K')$ and $\epsilon' > d_{\mathrm{scale}}(K', K'')$, the claim follows.
  \end{proof}

\section{Technical lemmas}

  \begin{lemma}
    \label{Lem:VarianceSupRelation}
    Let $h \in \mathcal{F}_1$.  
    \begin{enumerate}[(i)]
    \item There exists a universal constant $c > 0$ such that $\sigma_h \geq c/h(\mu_h)$;
    \item There exists a universal constant $C > 0$ such that $\sigma_h \leq C/\sup_{r \in \mathbb{R}} h(r)$.
    \end{enumerate}
  \end{lemma}

  \begin{proof}
(i)  Let $f(r) := \sigma_h h(\sigma_h r + \mu_h)$, so $f$ is a log-concave density with $\mu_f = 0$ and $\sigma_f=1$. Then
    \[
 \sigma_h = \frac{f(0)}{h(\mu_h)} \geq \frac{2^{-7}}{h(\mu_h)},
    \]
where the final inequality follows from \citet[][Theorem~5.14(d)]{lovasz2007geometry}.

(ii) Since $h$ is upper semi-continuous and $h(r) \rightarrow 0$ as $r \rightarrow \pm \infty$, there exists $r_0 \in \mathbb{R}$ such that $h(r_0) = \sup_{r \in \mathbb{R}} h(r)$.  Thus, with $f$ as above,
    \[
\sigma_h \leq \frac{f\bigl((r_0-\mu_h)/\sigma_h\bigr)}{h(r_0)} \leq \frac{\|f\|_\infty}{\sup_{r \in \mathbb{R}} h(r)} \leq \frac{2^9}{\sup_{r \in \mathbb{R}} h(r)},
    \]
by \citet[][Theorem~5.14(b) and (d)]{lovasz2007geometry}.
  \end{proof}

\begin{lemma}
  \label{Lem:DifferentialEntropyBound}
  For any $f \in \mathcal{F}_1$ with $\sigma_f = 1$, we have that
  \[
    -\frac{1}{2} - \frac{1}{2}\log (2\pi) \leq    \int_{-\infty}^\infty f \log f \leq 9 \log 2.
  \]
\end{lemma}
\begin{proof}
By the location invariance of the entropy functional, we may assume without loss of generality that $\mu_f = 0$.  Since $\|f\|_\infty \leq 2^9$ \citep[][Theorem~5.14(b) and (d)]{lovasz2007geometry}, we have
  \begin{align*}
    \int_{-\infty}^\infty f\log f \leq 9 \log 2.
  \end{align*}
Now let $g(x) := (2\pi)^{-1/2}e^{-x^2/2}$, so by the non-negativity of Kullback--Leibler divergence, 
  \begin{align*}
    \int_{-\infty}^\infty f \log f \geq \int_{-\infty}^\infty f \log g = - \frac{1}{2}\int_{-\infty}^\infty x^2 f(x) \, dx - \frac{1}{2}\log (2\pi) = -\frac{1}{2} - \frac{1}{2}\log (2\pi),
  \end{align*}
as required.
\end{proof}
The next lemma provides basic properties of the function $\rho$ defined in~\eqref{Eq:RhoDefinition}.
\begin{lemma}
  \label{Lem:RhoProperties}
  For any $p \in \mathbb{N}$ and $a \geq 0$, we have
  \begin{enumerate}[(i)]
        \item $\rho(s)$ is increasing;
        \item $\frac{\rho(s)}{s}$ is decreasing;
        \item $\rho(s) \in [1,p]$ for all $s \in (0,\infty)$.
    \end{enumerate}
\end{lemma}
\begin{proof}
$(i)$ If $a > 0$, define $\alpha := (a+s)/a \in (1,\infty)$, an increasing function of $s$.  Then
  \begin{align*}
    \rho(s) & =  \frac{(a + s)^{p-1} p s}{(a + s)^p - a^p} = \frac{ \alpha^{p-1} p (\alpha - 1)}{\alpha^p - 1} 
              = p \biggl( 1 - \frac{\alpha^{p-1} - 1}{\alpha^p - 1} \biggr),
  \end{align*}
which is increasing in $\alpha$, as required.  If $a=0$, then $\rho(s) = p$, so the claim is also true.

$(ii)$ If $a > 0$, then
  \begin{align*}
    \frac{\rho(s)}{s} = \frac{p}{a} \frac{\alpha^{p-1} }{\alpha^p - 1},
  \end{align*}
which is decreasing in $\alpha$. If $a = 0$, then $\rho(s)/s = p/s$ and the claim also follows.

$(iii)$ When $a=0$, the result follows because $\rho(s) = p$ for all $s \in (0,\infty)$.  When $a > 0$, the lower bound follows from $(i)$ and the fact that $\lim_{s \searrow 0} \rho(s) = 1$, while the upper bound follows from $(i)$ and the fact that $\lim_{s \rightarrow \infty} \rho(s) = p$.
%
\end{proof}

\begin{lemma}
  \label{Lem:EvilInequality}
  For any $p \geq 1$ and $r_0 > a \geq 0$, we have that
  \[
    r_0 \leq \frac{p}{p+1} \frac{r_0^{p+1} - a^{p+1}}{r_0^p - a^p}
    + \frac{r_0^p - a^p}{r_0^{p-1}  p}.
  \]
\end{lemma}
\begin{proof}
Writing $x := 1 - (a/r_0)^p$, we are required to prove that for $x \in (0,1]$,
  \begin{align}
    1 \leq \frac{p}{p+1} \frac{1 - (1-x)^{(p+1)/p}}{x} + \frac{x}{p}.
  \end{align}
The inequality holds for sufficiently small $x > 0$ and at $x = 1$, and also when $p=1$. To finish the proof, it suffices to show that $t_p(x) := \frac{1 - (1 - x)^{(p+1)/p}}{x}$ is concave on $(0,1)$ when $p \geq 2$.  But for $x \in (0,1)$, 
\begin{align*}
p^2(1-x)^{1-1/p}x^3 t_p''(x) &= 2p^2\bigl\{1 - (1-x)^{1/p}\}(1-x)^{1-1/p} - p(2-x)x - x^2 \\
&< 2p^2\biggl\{\frac{x}{p} - \frac{(p-1)x^2}{2p^2}\biggr\} - p(2-x)x - x^2 = 0,
\end{align*}
as required.
\end{proof}

\begin{lemma}
  \label{lem:linearize_power}
  Let $p \geq 1$ and $x \in [-1/p, 1/p]$. Then
  \[
    | (1 + x)^p - (1 + px) | \leq \frac{e}{2}(px)^2.
  \]
\end{lemma}

\begin{proof}
  By Taylor's theorem, there exists $\bar{x}$ on the line segment joining $0$ and $x$ such that
  \[
    (1 + x)^p = 1 + px + (1 + \bar{x})^{p-2} \frac{p (p-1)}{2} x^2.
  \]
  The lemma follows by noting that $(1 + \bar{x})^{p-2} \leq e$.
\end{proof}

\newpage

\bibliographystyle{dcu}
\bibliography{refs}

\end{document}